\documentclass[11pt]{amsart}

\usepackage{amsmath, amssymb, amsthm, hyperref, graphicx,   verbatim, enumerate, xcolor, tikzsymbols, stmaryrd, epsfig, xypic}
\usepackage{mathrsfs}  
\usepackage{mathabx}
 \usepackage{times}
 \usepackage{comment}
 
 \DeclareFontFamily{U}{mathc}{}
\DeclareFontShape{U}{mathc}{m}{it}%
{<->s*[1.03] mathc10}{}

\DeclareMathAlphabet{\mathscr}{U}{mathc}{m}{it}

\usepackage[paper=a4paper]{geometry}
\usepackage[all]{xy}
\usepackage[utf8]{inputenc}

\usepackage[english]{babel}
 
\newcommand{\ee}{\mathbb{E}}

\newcommand{\e}{\varepsilon}

\newcommand{\fr}{\partial}

\newcommand{\set}[1]{\left\{#1\right\}}
\newcommand{\norm}[1]{{\left\Vert#1\right\Vert}}
\newcommand{\abs}[1]{\left\vert#1\right\vert}

\newcommand{\pu}{{\mathbb{P}^1}}

\newcommand{\cX}{\mathcal{X}}
\newcommand{\rest}[1]{ \arrowvert_{#1}}

\newcommand{\tendvers}{\underset{n\to\infty}{\longrightarrow}}
\newcommand{\unsur}[1]{\frac{1}{#1}}

\newcommand{\qq}{\mathcal{Q}}
\newcommand{\rond}{\!\circ\!}

\newcommand{\lrpar}[1]{\left(#1\right)}

\newcommand{\loc}{\mathrm{loc}}
\newcommand{\inv}{^{-1}}

\DeclareMathOperator{\supp}{Supp}

\DeclareMathOperator{\inter}{inter}

\DeclareMathOperator{\leb}{Leb}

\DeclareMathOperator{\length}{length}

\DeclareMathOperator{\id}{id}
\DeclareMathOperator{\jac}{Jac}
\DeclareMathOperator{\diam}{Diam}
\DeclareMathOperator{\dist}{dist}
\DeclareMathOperator{\hexp}{\Phi}
\DeclareMathOperator{\osc}{osc}
\DeclareMathOperator{\Lim}{Lim}
\DeclareMathOperator{\area}{area}

\newcommand{\C}{\mathbf{C}}
\newcommand{\R}{\mathbf{R}}
\newcommand{\Q}{\mathbf{Q}}
\newcommand{\Z}{\mathbf{Z}}
\newcommand{\N}{\mathbf{N}}

\newcommand{\bfe}{{\mathbf{e}}}

\newcommand{\Hyp}{\mathbb{H}}
\renewcommand{\P}{\mathbb{P}}
\newcommand{\pp}{\mathbb{P}}  

\newcommand{\NS}{{\mathrm{NS}}}
\newcommand{\Pic}{{\mathrm{Pic}}}

\newcommand{\sA}{{\mathcal A}}

\newcommand{\Mer}{{\mathcal M}}

\newcommand{\M}{{\mathbf{M}}}

\DeclareMathOperator{\Pent}{{Pent}}

\DeclareMathOperator{\Jac}{{Jac}}
\DeclareMathOperator{\Cur}{{Cur}}
\DeclareMathOperator{\Lip}{{Lip}}

\DeclareMathOperator{\Vect}{Vect}

\DeclareMathOperator{\dens}{dens}
\DeclareMathOperator{\Ax}{Geo}

\newcommand{\g}{{\mathrm{g}}}
 \newcommand{\Diff}{{\mathsf{Diff}}}

\newcommand{\disk}{\mathbb{D}}

\newcommand{\Aut}{\mathsf{Aut}}
\newcommand{\Isom}{\mathsf{Isom}}

 \newcommand{\Bir}{{\mathsf{Bir}}}

 \newcommand{\PGL}{{\sf{PGL}}}
\newcommand{\PSL}{{\sf{PSL}}}
\renewcommand{\O}{{\sf{O}}}
\newcommand{\SO}{{\sf{SO}}}

\newcommand{\End}{{\sf{End}}}
\newcommand{\GL}{{\sf{GL}}}
\newcommand{\SL}{{\sf{SL}}}
\newcommand{\Aff}{{\sf{Aff}}}

\newcommand{\U}{{\sf{U}}}

\newcommand{\vol}{{\sf{vol}}}
\newcommand{\euc}{{\sf{euc}}}

\newcommand{\moco}{{\mathrm{modc}}}

\newcommand{\Ker}{{\mathrm{Ker}}}
\newcommand{\Kah}{{\mathrm{Kah}}}
\newcommand{\Kahbar}{{\overline{{\mathrm{Kah}}}}}
\newcommand{\Amp}{{\mathrm{Amp}}}
\newcommand{\Nef}{{\mathrm{Nef}}}

\newcommand{\Psef}{{\mathrm{Psef}}}
\newcommand{\Zar}{{\mathsf{Zar}}}

\newcommand{\m}{\mathscr{m}}
\newcommand{\F}{F}
\newcommand{\X}{\mathscr{X}}
\newcommand{\cx}{\mathscr{x}}
\newcommand{\cy}{\mathscr{y}}

\newcommand{\romain}[1]{{\color{blue}*}\marginpar{\tiny  \color{blue} RD: #1}}
\newcommand{\serge}[1]{{\color{red}*}\marginpar{\tiny  \color{red} SC: #1}}
\newcommand{\bleu}[1]{\textcolor{blue}{#1}}

\theoremstyle{plain}

\newtheorem{thm}{Theorem}[section]
\newtheorem{cor}[thm]{Corollary}
\newtheorem{pro}[thm]{Proposition}
\newtheorem{lem}[thm]{Lemma}

\newtheorem{mthm}{Theorem}

\newtheorem{mcor}{Corollary}

\theoremstyle{definition}

\newtheorem{eg}[thm]{Example}
\newtheorem{rem}[thm]{Remark}

\numberwithin{equation}{section}       

\excludecomment{vlongue}
\includecomment{vcourte}

\begin{document}
 
\setlength{\baselineskip}{0.46cm}        

\setlength{\textwidth}{13.2cm}                       
\setlength{\textheight}{21.2cm}                     
\setlength{\topmargin}{0.15cm}                     
\setlength{\headheight}{0.7cm}                     
\setlength{\headsep}{0.7cm}                         
\setlength{\oddsidemargin}{1.2cm}                
\setlength{\evensidemargin}{0.4cm}              


\title[Random dynamics on complex surfaces]{Random dynamics on   real and complex projective  surfaces}

\author{Serge Cantat}
\address{Serge Cantat, IRMAR, campus de Beaulieu,
b\^atiments 22-23
263 avenue du G\'en\'eral Leclerc, CS 74205
35042  RENNES C\'edex}
\email{serge.cantat@univ-rennes1.fr}

\author{Romain Dujardin}
\address{ Sorbonne Universit\'e, CNRS, Laboratoire de Probabilit\'es, Statistique  et Mod\'elisation  (LPSM), F-75005 Paris, France}
\email{romain.dujardin@sorbonne-universite.fr}

 \begin{abstract}
We initiate the study of random iteration of automorphisms of real and complex projective surfaces, as well as compact K\"ahler surfaces, 
focusing on the classification of stationary measures. We show that, in a number 
of cases, such stationary measures are invariant,  and 
provide criteria  for uniqueness,  smoothness  and rigidity  of  invariant probability measures. 
This involves a variety of tools from complex and algebraic geometry,    
 random products of matrices,  non-uniform hyperbolicity,  as well as     recent results of 
 Brown and Rodriguez Hertz  on random iteration of surface diffeomorphisms. 
\end{abstract}

 \maketitle

 \setcounter{tocdepth}{1}
\tableofcontents

\begin{vcourte}
\begin{small}
Note: A longer version of this paper  is maintained on the personal website of the authors (see \cite{cantat-dujardin:vlongue}). It contains complementary 
examples and preliminaries, and the  proofs of some known facts are included for convenience. 
\end{small}
\end{vcourte}

\newpage

 \section{Introduction}\label{sec:introduction}
 \subsection{Random dynamical systems}\label{par:RDS_intro}
Consider a compact  manifold $M$ and a probability measure $\nu$ on $\Diff(M)$; to simplify the exposition 
we assume throughout this introduction that  the support $\supp(\nu)$ is finite. The data 
$(M,\nu)$ defines a \textbf{random dynamical system}, obtained by randomly composing independent 
diffeomorphisms with distribution $\nu$. In this paper, these random dynamical systems are studied 
from the point of view of \emph{ergodic theory}, that is, we are mostly interested in understanding the  \emph{asymptotic distribution} of orbits. 
 
A probability measure $\mu$ on $M$ is {\bf{$\nu$-invariant}}
if $f_*\mu=\mu$ for $\nu$-almost every $f\in \Diff(M)$, and it is 
{\bf{$\nu$-stationary}} if it is invariant on average: $\int f_*\mu \, d\nu(f)=\mu$. A simple fixed 
point argument shows that stationary measures always exist. On the other hand, the existence of an invariant measure 
should hold only under special circumstances, for instance  when the group $\Gamma_\nu$ generated 
by $\supp(\nu)$ is amenable, or has a finite orbit, or 
preserves an invariant volume form.
 
According to Breiman’s law of large numbers,
for every $x\in M$ and $\nu^\N$-almost every 
$(f_j)\in \Diff(M)^\N$, every cluster value of the sequence of empirical measures
\begin{equation}
\frac{1}{n}\sum_{j=0}^{n-1}\delta_{f_j\circ \cdots \circ f_0(x)}
\end{equation}
is a stationary measure. Thus, a classification of stationary measures gives an essentially 
complete understanding 
of the asymptotic     distribution of such random orbits, 
as   $n$ goes to $+\infty$. 
\begin{vcourte}
Our goal is to combine  
algebraic and holomorphic dynamics  
with recent results  in random dynamics   to study the case when $M$ is a real
or complex projective surface and the action is by algebraic diffeomorphisms.
But let us first highlight  
a nice example to which our techniques apply; 
rooted in elementary euclidean geometry, it leads to interesting algebro-geometric constructions (see Remarks~\ref{rem:hessian} and \ref{rem:pentagon_enriques}). 
\end{vcourte}

\begin{vlongue}
When $\Gamma_\nu$ is a cyclic group, the set of invariant measures 
is typically too large to be amenable to a complete description. On the other hand a number of
recent works have  shown that stationary measures, even if they 
always exist, tend to  satisfy some \emph{rigidity properties} when $\Gamma_\nu$ is large.  Our goal in this 
article is to combine tools from  algebraic and holomorphic dynamics together with  these
recent results from random dynamics to study the case when $M$ is a real
or complex projective surface and the action is by algebraic diffeomorphisms. 
Before describing the state of the art and stating  a few precise results, let us highlight  
a nice geometric  example to which our techniques can be applied;
rooted in elementary euclidean geometry, it illustrates our main results in 
the case of K3 and Enriques surfaces. 
\end{vlongue}

\subsection{Randomly folding pentagons}\label{par:pentagons_intro}
 Let $\ell_0, \ldots, \ell_4$
be five positive real numbers   such that there exists a pentagon with side lengths $\ell_i$. Here a pentagon 
is just an ordered set of points $(a_i)_{i=0, \ldots, 4}$ in the Euclidean plane,
such that $\dist(a_i,a_{i+1})=\ell_i$ for $i=0,\ldots, 4$ (with $a_5 = a_0$ by definition); pentagons are not assumed to be convex, and 
two distincts sides $[a_i, a_{i+1}]$ and $[a_j, a_{j+1}]$ may 
intersect at  a point which is not one of the $a_i$'s.
Let  $\Pent(\ell_0, \ldots, \ell_4)$ be the set of pentagons with side lengths $\ell_i$. 
This set may be defined by polynomial equations of the form 
$\dist(a_i, a_{i+1})^2 = \ell_i^2$, so it is naturally a real  algebraic variety. 

For every  $i$,  $a_i$ is one of the two intersection points $\set{a_i, a'_i}$ of the circles
of respective  centers $a_{i-1}$ and $a_{i+1}$ and   radii $\ell_{i-1}$ and $\ell_i$. 
The transformation exchanging  these two points $a_i$ and $a_i’$,  while keeping 
 the other vertices fixed, defines an involution  $s_i$ of $\Pent(\ell_0, \ldots, \ell_4)$. 
 It commutes with the action of 
 positive isometries of the plane, 
  hence, it induces an involution $\sigma_i$ on the quotient space 
\begin{equation}
\Pent^0(\ell_0, \ldots, \ell_4)=\Pent(\ell_0, \ldots, \ell_4)/(\SO_2(\R)\ltimes \R^2).
\end{equation}
Each element of $\Pent^0(\ell_0, \ldots, \ell_4)$ admits a unique representative with 
$a_0=(0,0)$ and $a_1=(\ell_0, 0)$, so as before  
  $\Pent^0(\ell_0, \ldots, \ell_4)$ is a real 
algebraic variety, which is easily seen to be of dimension 2 (see \cite{Curtis-Steiner,Shimamoto-Vanderwaart}). 
When 
smooth (see Lemma~\eqref{lem:pentagon_smoothness} and the comments preceding Remark~\ref{rem:topology} for a discussion),  $\Pent^0(\ell_0, \ldots, \ell_4)$ is a real K3 surface 
on which the $\sigma_i$ 
act by algebraic diffeomorphisms, 
preserving a canonically defined area form  (see \S\ref{par:pentagons}); 
and for a general choice of lengths, the group generated by these involutions 
generates  a rich dynamics. 
Now, start
with some pentagon 
and at every unit of time,  apply randomly one of the~$\sigma_i$. This creates a random sequence of pentagons, and our results explain how this 
 sequence is asymptotically distributed on $\Pent^0(\ell_0, \ldots, \ell_4)$. (The dynamics of the folding maps acting on plane quadrilaterals
 was studied for instance in~\cite{esch-rogers, benoist-hulin}.)

\subsection{Stiffness}\label{par:intro_stiffness} Let us  present a few 
 landmark results 
 about stationary measures.
  First, suppose  that 
$\nu$ is a finitely supported 
probability measure on $\SL_2(\C)$, which we view as 
   acting by  linear  projective transformations
on $M = \P^1(\C)$. 
\begin{vcourte}
Suppose that the group $\Gamma_\nu$ generated by the support of $\nu$ 
is \textbf{non-elementary}, that is, $\Gamma_\nu$ is unbounded  
and  acts strongly irreducibly on $\C^2$. 
\end{vcourte}
\begin{vlongue}
Suppose that the group $\Gamma_\nu$ generated by the support of $\nu$ 
is \textbf{non-elementary}, that is, $\Gamma_\nu$ is non-compact  and  acts strongly irreducibly on $\C^2$
(in the non-compact case, this simply means that $\Gamma_\nu$ does not have any orbit of cardinality $1$ or $2$ in $\P^1(\C)$).
\end{vlongue}
Then, \emph{there is a unique $\nu$-stationary (probability) 
measure $\mu$ on $\P^1(\C)$, and this measure is not invariant}. 
This is one  instance of a more general result due 
to Furstenberg~\cite{Furstenber:1963}.

\begin{vlongue}
Temporarily leaving the setting
of diffeomorphisms, let us consider the semigroup of transformations of the circle 
$\R/\Z$ generated by $m_2$ and $m_3$, where $m_d(x)   = dx \mod 1$. Since the multiplications by 2 and 3 
commute, the so-called   Choquet-Deny theorem asserts that any stationary measure is invariant.  
Furstenberg's famous ``$\times 2\times 3$ conjecture'' asserts that 
any atomless probability  measure  $\mu$ invariant under $m_2$ and $m_3$ is  the Lebesgue measure (see~\cite{furstenberg-disjointness}). 
This question is still open so far, and has attracted a lot of attention. 
Rudolph \cite{rudolph} proved  that the answer is positive when $\mu$ is of positive entropy with respect to $m_2$ or $m_3$.

Back to diffeomorphisms, let $\nu$ be a finitely supported measure on $\SL_2(\Z)$, and 
 consider the action of $\SL_2(\Z)$ on  the torus $M=\R^2/\Z^2$. 
\end{vlongue}

\begin{vcourte}
Now, let $\nu$ be a finitely supported measure on $\SL_2(\Z)$, and 
consider the action of $\SL_2(\Z)$ on  the torus $M=\R^2/\Z^2$. 
\end{vcourte}
In that case, the Haar measure 
of $\R^2/\Z^2$, as well as the 
atomic measures equidistributed on finite orbits $\Gamma_\nu(x,y)$, for $(x,y)\in \Q^2/\Z^2$, are examples of $\Gamma_\nu$-invariant measures.  
 By using Fourier analysis and additive combinatorics techniques, Bourgain, Furman, Lindenstrauss
and Mozes~\cite{BFLM} proved that  \emph{if $\Gamma_\nu$ is non-elementary, then 
every stationary measure $\mu$ on $\R^2/\Z^2$ is $\Gamma_\nu$-invariant and is 
 a convex combination of the above mentioned invariant measures.}
\begin{vlongue}
This can be viewed as an affirmative answer to  a non-Abelian version 
of the $\times2\times 3$ conjecture.   
\end{vlongue}   
This  property of automatic invariance of stationary measures was called  \textbf{stiffness}
 (or 
 $\nu$-stiffness) by Furstenberg~\cite{furstenberg_stiffness}, 
 who conjectured it to hold in this setting.
Soon after, Benoist and Quint~\cite{benoist-quint1}
 gave an ergodic theoretic proof of this result 
 and extended  it 
 to certain actions of discrete groups  on homogeneous spaces. They also derived the following equidistribution result: 
 \emph{for every  $(x,y)\notin \Q^2/\Z^2$,  the  random trajectories of $(x,y)$ determined by $\nu$  almost surely  equidistribute towards the Haar measure on $\R^2/\Z^2$. }

Finally, Brown and Rodriguez-Hertz~\cite{br}, 
building on  the work of   Eskin and Mirzakhani \cite{eskin-mirzakhani},
 managed to recast these measure rigidity results in terms of Pesin theory 
 to obtain a
 version of the stiffness theorem of~\cite{BFLM} for general $C^2$ diffeomorphisms of compact surfaces. 
We shall describe their results in due time; for the moment  we 
content ourselves with one illustrative consequence of~\cite{br}. 
Let  $\nu  = \sum\alpha_j\delta_{f_j}$ 
be a finitely supported probability measure on $\SL_2(\Z)$ generating a non-elementary subgroup. 
Consider perturbations $\set{f_{i, \varepsilon }}$ of the 
$f_i$  in the group $\Diff^2_{\vol}(\R^2/\Z^2)$ of $C^2$ diffeomorphisms of $\R^2/\Z^2$ preserving the Haar measure. 
Set $\nu_\e  = \sum\alpha_j\delta_{f_{j, \e}}$.
Then, \emph{for sufficiently small perturbations,  
any $\nu_\e$-stationary measure on $\R^2/\Z^2$ is invariant and is a combination 
 of the Haar measure   and measures supported on finite $\Gamma_{\nu_\e}$-orbits}.

\begin{vcourte}
In this paper, we  prove a stiffness theorem for groups of 
algebraic diffeomorphisms of real algebraic surfaces. 
The work of Brown and Rodriguez-Hertz is our 
main source of inspiration and   a key ingredient for some of our main results. 
\end{vcourte}

\begin{vlongue}
In this paper, we  obtain a new generalization of the stiffness theorem of~\cite{BFLM}, 
for algebraic diffeomorphisms of real algebraic surfaces. 
Before entering into specifics, let us emphasize that the article~\cite{br}, by Brown and Rodriguez-Hertz, is  our 
main source of inspiration and   a key ingredient for some of our main results.
\end{vlongue}

\subsection{Sample results: stiffness, classification, and rigidity} 
Let $X$ be a smooth 
complex projective surface, or more generally a compact K\"ahler surface. 
Denote by $\Aut(X)$ its group of holomorphic diffeomorphisms, referred to in this paper as   {\bf{automorphisms}}. When $X\subset \P^N(\C)$
is defined by polynomial equations with real coefficients, the complex conjugation induces an anti-holomorphic involution 
$s\colon X\to X$, whose fixed point set is the real part $X(\R)$ of $X$. 
We denote by $X_\R$ the surface $X$ viewed as an algebraic variety defined over $\R$, and by 
$\Aut(X_\R)$ the group of automorphisms defined over $\R$; $\Aut(X_\R)$ is 
the subgroup of $\Aut(X)$ centralizing~$s$. When  $X(\R)\neq \emptyset$, the elements of $\Aut(X_\R)$ are  the 
real-analytic diffeomorphisms of $X(\R)$ admitting a holomorphic extension to $X$. Note that in stark contrast with groups of smooth diffeomorphisms,  the groups $\Aut(X_\R)$ and $\Aut(X)$ are typically \emph{discrete and at most countable}.  

The group $\Aut(X)$
acts on the cohomology  $H^*(X;\Z)$. By definition, a subgroup $\Gamma\subset \Aut(X)$
is {\bf{non-elementary}} if  its image $\Gamma^*\subset \GL(H^*(X;\C))$ contains a non-Abelian 
free group; equivalently, $\Gamma^*$ is not virtually Abelian.
When $\Gamma$ is non-elementary, there exists a pair $(f,g)\in \Gamma^2$ 
generating a free group of rank $2$ such that the topological entropy of 
every element  in that group is positive (see~Lemma~\ref{lem:non-elementary_free_groups}). 
Pentagon foldings provide examples for which $\Aut(X_\R)$ is non-elementary.

\subsubsection{Stiffness} As before, if  $\nu$ is a finitely supported probability measure on $\Aut(X)$, we denote by $\Gamma_\nu$ the subgroup 
generated by $\supp(\nu)$. 

\begin{mthm}\label{mthm:stiffness}
Let $X_\R$ be a real projective surface and $\nu$ be a finitely supported symmetric probability measure 
on $\Aut(X_\R)$. If $\Gamma_\nu$   preserves an area form on $X(\R)$, 
then every ergodic $\nu$-stationary measure $\mu$ on $X(\R)$ is either invariant or 
supported on a proper $\Gamma_\nu$-invariant  subvariety. 
In particular if there is no $\Gamma_\nu$-invariant algebraic curve, the random dynamical system $(X,\nu)$ is stiff. 
\end{mthm}

This theorem is mostly interesting when  $\Gamma_\nu$ is non-elementary 
and we  
focus on this case in the remainder of this introduction. 
Stationary measures supported on invariant curves are easily analysed 
(see~\S\ref{par:invariant_curves2}). 
Moreover, if $\Gamma_\nu$ is non-elementary, it is always possible to  contract all $\Gamma_\nu$-invariant curves, creating a complex
analytic surface $X_0$  with  finitely many singularities. Then  on $X_0(\R)$, stiffness  holds unconditionally. 

This result applies to many  interesting examples, because Abelian, K3, and Enriques surfaces, which concentrate 
most of the   dynamically interesting automorphisms  on compact  complex surfaces, admit 
 a canonical $\Aut(X)$-invariant $2$-form. 
In particular, it applies to 
pentagon foldings. Note also that linear Anosov maps 
on $\R^2/\Z^2$  fall into this category, so  Theorem \ref{mthm:stiffness} contains 
the stiffness  statement  of \cite{BFLM}  
in the two-dimensional case. 

\subsubsection{Invariant measures} Once  stiffness is established, the next step is to \emph{classify invariant measures}. A  \textbf{parabolic} automorphism of a compact Kähler 
surface is an automorphism $g$ such that the norm of $(g^n)^*$ on $H^{2}(X;\R)$ grows quadratically (as $\alpha n^2$ for some $\alpha>0$); such an 
automorphism preserves automatically a genus $1$ fibration on $X$ (see e.g.~\cite{invariant}). An example is given by the composition of the foldings $\sigma_i$ and $\sigma_{i+1}$ of two adjacent vertices in the space of pentagons. 
When 
$\Gamma_\nu$ contains a parabolic  automorphism, $\Gamma_\nu$-invariant measures 
are classified in~\cite{cantat_groupes, invariant}. 
A nice consequence
is that for a non-elementary group of $\Aut(X_\R)$
 containing parabolic elements and preserving an area form,  any
invariant ergodic measure is  either atomic, or concentrated on a $\Gamma_\nu$-invariant algebraic curve, or is the restriction of the 
area form on some open subset of $X(\R)$ bounded by a piecewise smooth curve.

For random pentagon foldings, these results give a complete answer to the equidistribution problem raised in 
\S \ref{par:RDS_intro}. Indeed, assume for simplicity
that the group generated by the five involutions $\sigma_i$ of $\Pent^0(\ell_0, \ldots, \ell_4)$ does not preserve 
any proper Zariski closed set, and that $\mathrm{Pent}^0(\ell_0, \ldots, \ell_4)$ is connected. Then 
the stiffness and classification theorems imply that the only stationary measure is the canonical 
area form. Therefore by Breiman's law of large 
numbers, for every initial pentagon 
$P\in \Pent^0(\ell_0, \ldots, \ell_4)$ and   almost every 
sequence  $(m_j)\in \{0, \ldots, 4\}^\N$, the random 
sequence 
$P_n=(\sigma_{m_{n-1}}\circ \cdots \circ \sigma_{m_0})(P)$ equidistributes with respect to 
the area form.  Thus, quantities like the asymptotic average of the diameter  
are given by explicit integrals of semi-algebraic functions, independently of the starting pentagon $P$. 

Another widely studied example 
is the family of \textbf{Wehler surfaces}. 
These are the smooth surfaces $X\subset \P^1\times \P^1\times \P^1$  
defined by an equation of degree $(2,2,2)$. Then for each index $i\in \{1,2,3\}$, 
the projection $\pi_{i}\colon X\to \P^1\times \P^1$ which ``forgets the variable $x_i$'' has degree~$2$; thus, there is an 
involution $\sigma_i$ of $X$ that permutes the  two points in the generic fiber of~$\pi_i$.
 
\begin{mcor}
Let $X_\R\subset \P^1\times \P^1\times \P^1$ be a real Wehler surface such that $X(\R)$ is non empty. If $X_\R$ is
generic, then: 
\begin{enumerate}[\em (1)]
\item the surface $X$ is a K3 surface and there is a unique (up to choosing an orientation of $X(\R)$) algebraic $2$-form $\vol_{X_\R}$ on $X(\R)$ such that $\int_{X(\R)}\vol_{X_{\R}}=1$;
\item the group $\Aut(X_\R)$   is generated by the three involutions $\sigma_i$ and coincides with $\Aut(X)$; furthermore it preserves the probability measure defined by $\vol_{X_\R}$;
\item if $\nu$ is finitely supported and $\Gamma_\nu$ has finite index in $\Aut(X_\R)$ then $(X(\R), \nu)$ is stiff; moreover  
the only 
$\nu$-stationary measures on $X(\R)$ are convex combinations of the probability measures defined by $\vol_{X_\R}$ on the connected components of $X(\R)$. 
\end{enumerate}
\end{mcor}

Here by generic we mean that the equation of $X$ belongs to  the complement of at most countably 
many hypersurfaces in the set of polynomial equations of degree $(2,2,2)$ (see \S\ref{par:Wehler} for details). 
This result follows from Theorem \ref{mthm:stiffness},
Proposition \ref{pro:Wehler_generic}, 
Corollary~B of~\cite{invariant}, and the 
generic    non-existence of finite orbits  
established in 
\cite{finite_orbits}. 
If we do not assume $X$ to be generic but assume only that $X$ does not contain any fiber of the three projections $\pi_i$, then the set of stationary measures supported in $X(\R)$ is a finite dimensional simplex (see \cite{finite_orbits}). 

\begin{vlongue}
The techniques of~\cite{cantat_groupes, invariant} do not apply in the absence of parabolic automorphisms. 
Here, we establish the following   measure rigidity result, which may be compared to Rudolph's theorem on the 
$\times 2\times 3$ conjecture, but in a non-commutative and non-linear context (see~\cite{rudolph}). 
\end{vlongue}
\begin{vcourte}
The techniques of~\cite{cantat_groupes, invariant} do not apply in the absence of parabolic automorphisms. 
Here, we establish the following rigidity result.
\end{vcourte}

\begin{mthm}\label{mthm:rigidity} Let $X_\R$ be a real projective surface. Let $\Gamma$ be a non-elementary subgroup of $\Aut(X_\R)$. If $\mu$ is a $\Gamma$-invariant probability measure on $X(\R)$ 
and if $\mu$ is ergodic and of positive entropy for some    $f\in \Gamma$, then $\mu$ is absolutely continuous
with respect to any area form on~$X(\R)$.
\end{mthm}

In particular if $\Gamma$ is a group of area preserving automorphisms, then up to normalization
$\mu$ will be the restriction of the area form on some $\Gamma$-invariant set. \textbf{Kummer examples} are a generalization of 
linear Anosov diffeomorphisms of tori to other projective surfaces  (see \cite{cantat-dupont, Cantat-Zeghib}). 
When $\Gamma$ contains a real Kummer example, we can derive an exact analogue of the classification of invariant measures of \cite{BFLM}, that is the assumption ``{\emph{$\mu$ has positive entropy}}'' can be replaced by ``{\emph{$\mu$ has no atoms}}''  (Theorem~\ref{thm:rigidity_kummer}).  
We also obtain a version of Theorem \ref{mthm:rigidity}   for 
polynomial automorphisms of the affine plane ${\mathbb{A}}^2_\R$(see Theorem~\ref{thm:rigidity_henon}).

\subsection{Some ingredients of the proofs}
The proofs of Theorems~\ref{mthm:stiffness} and \ref{mthm:rigidity} rely on  the   deep results of Brown and Rodriguez-Hertz
\cite{br}.   
To be more precise, recall that an ergodic stationary measure $\mu$ on $X$ admits  a pair
of Lyapunov exponents $\lambda^+(\mu)\geq \lambda^-(\mu)$, and that $\mu$ is  called {\bf{hyperbolic}} if 
$\lambda^+(\mu)> 0 > \lambda^-(\mu)$. In this case
   the (random) 
Oseledets theorem shows that for $\mu$-almost every $x$ and $\nu^\N$-almost every $\omega = (f_j)_{j\in \N}$ in $\Aut(X)^\N$, 
there exists a stable  direction $E^s_\omega (x)\subset 
T_xX_\R$. 
In \cite{br}, stiffness is established 
under the condition that 
$E^s_\omega(x)\subset T_xX_\R$   
depends non-trivially on the random itinerary $\omega = (f_j)_{j\in \N}$, or 
equivalently that stable directions  do not induce a   measurable $\Gamma_\nu$-invariant line field. 
One of our main contributions is to take care of this possibility in our setting:
for this we study  the dynamics on the {\emph{complex}} surface $X$.  

\begin{mthm}\label{mthm:alternative_stable}
Let $X$ be a complex  projective surface and $\nu$ be a finitely supported probability measure on $\Aut(X)$. If $\Gamma_\nu$ is non-elementary, then any hyperbolic ergodic $\nu$-stationary 
measure $\mu$ on $X$ satisfies the following alternative: 
\begin{enumerate}[\em (a)]
\item  either $\mu$ is invariant, and its fiber entropy $h_\mu(X;\nu)$ vanishes;
\item or $\mu$ is supported on a $\Gamma_\nu$-invariant algebraic curve;
\item or the field of Oseledets stable directions of $\mu$ is not $\Gamma_\nu$-invariant; in other words,
it genuinely depends on the    itinerary $\omega=(f_j)_{j\geq 0}\in \Aut(X)^\N$.
\end{enumerate} 
\end{mthm}

As opposed to Theorems~\ref{mthm:stiffness} 
and \ref{mthm:rigidity}, this result applies to the dynamics on the \emph{complex} manifold $X$, without assuming
the existence of an invariant volume form or an invariant real structure.
\begin{vlongue}
Understanding this somewhat technical result 
requires a substantial amount of material from the smooth ergodic theory of random dynamical systems, 
which will be introduced in due time. 
\end{vlongue}
When $\mu$ is not invariant, nor supported by a proper Zariski closed subset,
Assertion~(c) precisely says that the 
condition on stable directions used in \cite{br} is satisfied. This 
 is our   key input  towards Theorems~\ref{mthm:stiffness} and~\ref{mthm:rigidity}. 
The arguments leading to    Theorem \ref{mthm:alternative_stable}  
involve   an interesting blend of Hodge theory, pluripotential analysis, and Pesin theory. 
They rely on the following  well-known 
principle in higher dimensional holomorphic dynamics. 
If $\mu$ is 
ergodic and hyperbolic, 
almost every point $(\omega,x)$ provides a 
stable manifold $W^s_\omega(x)$ biholomorphic to $\C$. Then, according to  a 
construction going back to   Ahlfors  and Nevanlinna, to  any entire curve 
$\phi:\C \to X$ is   associated  a (family of) closed
positive  $(1,1)$-current(s)   describing the asymptotic distribution of  $\phi(\C)$ in $X$, 
hence also a  (family of)  cohomology class(es) in $H^2(X, \R)$. These classes 
link 
the stable manifolds of $\mu$ 
with the  action of $\Gamma_\nu$ on $H^2(X;\R)$,
which itself can be analyzed by combining tools from  complex algebraic geometry with Furstenberg's theory of random products of matrices.  

\begin{mthm}\label{mthm:currents}
Let $X$ be a complex projective surface. Let $\nu$ be a finitely supported probability measure on $\Aut(X)$ 
such that $\Gamma_\nu$ is non-elementary.  Let $\kappa_0$ be a fixed K\"ahler form on $X$.
\begin{enumerate}[\em (1)]
\item  If $\kappa$ is any K\"ahler form on $X$, then for $\nu^\N$-almost every 
$\omega:=(f_j)_{j\geq 0}\in \Aut(X)^\N$ the limit  
\[
T_{\omega}^s:=\lim_{n\to +\infty} \frac{1}{\int_X \kappa_0 \wedge (f_n\circ \cdots \circ f_0)^*\kappa } (f_n \circ \cdots \circ f_0)^*\kappa 
\]
exists as a  closed positive  $(1,1)$-current. Moreover this current $T_\omega^s$ does not depend on $\kappa$ and has H\"older continuous potentials. 

\item If the $\nu$-stationary measure $\mu$ is  ergodic, hyperbolic (or   more generally if  $\lambda^-(\mu)< 0 \leq  \lambda^+(\mu)$) and  not 
supported on a $\Gamma_\nu$-invariant proper Zariski closed set, then   for $\mu$-almost every $x$ and $\nu^\N$-almost every $\omega$, the only  Ahlfors-Nevanlinna current of mass $1$ 
(with respect to $\kappa_0$) 
associated to the stable manifold $W^s_\omega(x)$ coincides with $T_{\omega}^s$.
\end{enumerate}
\end{mthm}

 The right setting for such a statement is that of a compact K\"ahler surface, but we show 
in \S\ref{subs:X-is-projective} that any compact  K\"ahler surface supporting 
a non-elementary group of 
automorphisms is projective (see Appendix~\ref{par:appendix_non_kahler}
for the non-K\"ahler case). The algebraicity of $X$  is, in fact, a crucial technical ingredient in the proof of assertion~(2), because we use 
techniques of laminar currents which are available only on projective surfaces.  
Theorem~\ref{mthm:currents} enters 
 the proof of Theorem~\ref{mthm:alternative_stable} as follows: since 
$\Gamma_\nu$ is non-elementary, 
 Furstenberg's description of the  random action  on $H^2(X, \R)$ implies that 
  the cohomology class $[T_{\omega}^s]$ depends 
non-trivially on $\omega$; therefore for $\mu$-almost every $x$,   $W^s_\omega(x)$ also depends non-trivially on $\omega$. 
Then, taking   advantage of the complex structure again, we show in 
 Section~\ref{sec:No_Invariant_Line_Fields}, 
 that $E^s_\omega(x)$ depends non-trivially on $\omega$ as well.


\begin{rem}
Beyond finitely supported measures, Theorem~\ref{mthm:stiffness}, \ref{mthm:rigidity}, \ref{mthm:alternative_stable}, and~\ref{mthm:currents} hold under optimal moment conditions on $\nu$ (this adds several technicalities, notably in Sections~\ref{sec:furstenberg} and~\ref{sec:currents}). 
\end{rem}

\subsection{Organization of the article} 
Let $X$ be a compact K\"ahler surface  
and  $\nu$ be a probability measure on $\Aut(X)$.
\begin{enumerate}[ --]
\item In Section~\ref{par:Hodge_Minkowski} we describe the action of $\Aut(X)$ on  $H^*(X;\Z)$, in 
particular on  
$H^{1,1}(X;\R)$. The Hodge index theorem 
endows it with a Minkowski structure, which is essential in our understanding of the dynamics
of  $\Gamma_\nu$ on the cohomology. This section~\ref{par:Hodge_Minkowski} prepares the ground for the analysis of 
random products of matrices done in Section~\ref{sec:furstenberg}  (and it is also used in~\cite{invariant, finite_orbits}). A delicate point to keep in mind is that   the 
action of a non-elementary subgroup  of $\Aut(X)$ on  $H^{1,1}(X;\R)$ may be reducible.
\item Section~\ref{par:Examples_Classification} describes several classes of examples, including pentagon foldings and Wehler's surfaces.  It is also shown there that a compact  K\"ahler surface with a non-elementary group of automorphims is
necessarily projective (see~Theorem~\ref{thm:X-is-projective2} in \S\ref{subs:X-is-projective}).  
\item After a short Section~\ref{sec:Glossary_I}
introducting the vocabulary of 
random products of diffeomorphisms,  Furstenberg's theory of random products of matrices  is applied in
Section~\ref{sec:furstenberg} to the study of the action 
on $H^{1,1}(X;\R)$. 
This, combined with the theory of closed positive currents, leads to the proof of the first assertion of 
Theorem~\ref{mthm:currents} in Section~\ref{sec:currents} (see Corollary~\ref{cor:Tsomega} and Theorems~\ref{thm:Tsomega_continuous} and \ref{thm:holder}).  
The continuity of the potentials of the currents 
$T^s_\omega$, which plays a key role in 
Section \ref{sec:nevanlinna}, relies on a recent result of Gou\"ezel and Karlsson \cite{Gouezel-Karlsson}.
\item Pesin theory enters into   play in Section~\ref{sec:Glossary_II}, in which the basics of the smooth ergodic theory of random 
dynamical systems are described in some detail for complex surfaces. 
This is used in Section~\ref{sec:nevanlinna} to connect the 
stable manifolds to the currents $T_\omega^s$, using techniques of laminar currents (Theorem~\ref{thm:nevanlinna} gives the second part of Theorem~\ref{mthm:currents}). 
\item Theorem \ref{mthm:alternative_stable} is proven in Section~\ref{sec:No_Invariant_Line_Fields} by combining ideas of \cite{br} with Theorem~\ref{mthm:currents} and  an elementary fact from local complex geometry inspired by a lemma from 
\cite{bls}. 
\item Theorem~\ref{mthm:stiffness} is finally established in Section~\ref{sec:stiffness}.
 When $\Gamma_\nu$ is non-elementary (Theorem~\ref{thm:stiffness_real})
it follows rather directly from \cite{br}, Theorem \ref{mthm:alternative_stable}, and the  
\emph{invariance principle} of Crauel~\cite{crauel}. 
Elementary groups are handled separately by using the classification of automorphism groups of compact K\"ahler surfaces
(see Theorems~\ref{thm:stiffness_elementary_groups} and Proposition~\ref{pro:stiffness_elementary2}). Note that the symmetry of $\nu$ is used only in the elementary case. 
\item  Section~\ref{sec:rigidity}, as well as~\cite{invariant}, are devoted to the classification of invariant measures. We prove Theorem~\ref{mthm:rigidity} in the more precise form of 
Theorem~\ref{thm:rigidity}, as well as several  related results. 

\begin{vlongue}
This relies on a measure rigidity theorem of  \cite{br}, together 
with ideas similar to the ones involved in the proof of Theorem~\ref{mthm:alternative_stable}. 
\end{vlongue}
  \end{enumerate}

\subsection{Further comments}
\begin{enumerate}[ --]
\item This article is part of a series of papers dedicated to the dynamics of  groups of  automorphisms
of compact K\"ahler surfaces, notably K3 and Enriques surfaces. 
The article \cite{invariant} focuses on the classification of 
invariant measures in presence of parabolic elements, and 
in \cite{hyperbolic} we study the hyperbolicity of such measures.   
\begin{vcourte}
In \cite{finite_orbits} the existence of finite orbits is analyzed by using
 tools from algebraic and arithmetic dynamics.
\end{vcourte}
\begin{vlongue}
In \cite{finite_orbits} we study the existence of finite orbits for non-elementary group actions; tools 
from arithmetic dynamics are used to study the case where 
$X$ and its automorphisms are defined over a number field. 
\end{vlongue}
\item After the first   version of this paper and  \cite{finite_orbits} were released, Filip and Tosatti \cite{Filip-Tosatti:Heights} gave  
an alternate approach of some of the results of Section~\ref{sec:currents}. 
\item It is natural to wonder what remains of our results in the real-analytic category. 
As explained above, the proofs of Theorems~\ref{mthm:currents} and~\ref{mthm:alternative_stable} rely on
a global complex geometric argument (via Ahlfors-Nevanlinna currents and cohomology groups) 
to show that stable manifolds depend on random itineraries, 
and on local properties of complex analytic disks to go from stable manifolds to stable directions
(cf. Section~\ref{sec:No_Invariant_Line_Fields}). 
It is clear that the global arguments do not 
carry over to groups of  real-analytic diffeomorphisms of closed real surfaces. 
To be  more  explicit, consider the following fact: 
{\emph{if $f$ is  an automorphism of a complex projective surface with positive entropy, and if $W^s_f(x)$ and $W^u_f(x')$ 
are Zariski dense stable and unstable manifolds of saddle periodic points, then 
$W^s_f(x)\cap W^u_f(x')$ is non-empty}}. 
This can be derived from the same global strategy, namely Ahlfors-Nevanlinna currents, 
their laminarity, and the Hodge index theorem (see~\cite[Thm. 6.2]{Cantat:Milnor}). 
On the other hand, such a statement does \emph{not} hold for real analytic 
diffeomorphisms of closed surfaces.  Regarding the results of Section~\ref{sec:No_Invariant_Line_Fields}, while some of the results of \S~\ref{subs:generic_intersection} might 
persist in the real-analytic category, 
the key Lemma~\ref{lem:BLS_intersection} does not (see Remark~\ref{rem:BLS_intersection}). 
\item While not directly covered by this article, the character variety of the once punctured torus (or the four times punctured sphere) should be amenable to the same strategy (see~\cite{Cantat:BHPS, Goldman:1988, Goldman:2003}, \cite{Previte-Xia:Pacific, Previte-Xia:TAMS}, and \cite{Xiao-Chuan_Liu, Chung} for a short selection of papers on the subject).
\item In a forthcoming work, we intend 
to extend the results   of Brown and Rodriguez-Hertz to the complex setting;  
with Theorem~\ref{mthm:alternative_stable} at hand, this 
would extend Theorem~\ref{mthm:stiffness} from the real to the complex case.

 \end{enumerate}

\subsection{Conventions} 
Throughout the paper $C$ stands for a ``constant" which may change from line to line,   independently of 
some asymptotic quantity that should be clear from the context (typically an integer $n$ corresponding to the number of iterations of a dynamical system).  
We write $a\lesssim b$ if $a\leq Cb$ and $a\asymp b$ if $a\lesssim b \lesssim a$. 
Complex manifolds are considered to be connected, so from now on ``complex manifold'' stands for ``connected 
complex manifold''.  For  a random dynamical system on a disconnected complex manifold, there is a finite index sugbroup 
$\Gamma'$ 
of $\Gamma_\nu$ fixing 
each connected component, and an induced measure $\nu'$ on $\Gamma'$ with properties
qualitatively similar to those of $\nu$ (see \S\ref{subs:inducing}), so the problem is reduced to the connected case. 
 
 \subsection{Acknowledgments}
We are grateful to S\'ebastien Gou\"ezel,  François Ledrappier, and  Fran\-çois Maucourant for 
interesting discussions and insightful comments. The first named author was
partially supported by a grant from the French Academy of Sciences (Del Duca foundation), and the second named author
by a grant from the Institut Universitaire de France.

\section{Hodge index theorem and Minkowski spaces}\label{par:Hodge_Minkowski}

In this section we 
define the notion of a non-elementary group of automorphisms of a compact K\"ahler surface $X$. 
We study the action of such a group on the cohomology of $X$, 
and in particular the question of (ir)reducibilty. We refer to  Appendix~\ref{par:appendix_non_kahler} for a 
discussion of the non-K\"ahler case.

\subsection{Cohomology}

\subsubsection{Hodge decomposition}\label{par:Hodge_decomposition} Denote by $H^*(X;R)$ the cohomology of $X$ with coefficients
in the ring $R$; we shall use $R=\Z$, $\Q$, $\R$ or $\C$.  
The group $\Aut(X)$ acts on $H^*(X;\C)$, preserving the image of $H^*(X;\Z)$; 
$\Aut(X)^*$ will denote the image of $\Aut(X)$ in $\GL(H^2(X;\C))$. The Hodge decomposition 
\begin{equation}
H^{k}(X;\C)=\bigoplus_{p+q=k} H^{p,q}(X;\C)
\end{equation}
is $\Aut(X)$-invariant. On 
$H^{0,0}(X;\C)$ and $H^{2,2}(X;\C)$, $\Aut(X)$ acts trivially. Throughout the paper we denote by $[\alpha]$ 
the cohomology class of a closed differential form (or current) $\alpha$. 
 
The intersection form on $H^2(X;\Z)$ will be denoted by $\langle\cdot \, \vert\, \cdot \rangle$; the self-intersection $\langle a \vert a \rangle$ of a class $a$
will also be denoted by $a^2$ for simplicity. This intersection form  is $\Aut(X)$-invariant. By the Hodge index theorem, it is positive definite on the real 
part of $H^{2,0}(X;\C)\oplus H^{0,2}(X;\C)$ and it 
is  non-degenerate and  of signature $(1,h^{1,1}(X)-1)$ on $H^{1,1}(X;\R)$. 
\begin{vcourte}
Thus, we get:
\end{vcourte}

\begin{lem}\label{lem:unitary_on _H20} 
The restriction of $\Aut(X)^*$ to the subspace $H^{2,0}(X;\C)$ (resp. $H^{0,2}(X;\C)$) is contained in a compact
subgroup of $\GL(H^{2,0}(X;\C))$ (resp.  $\GL(H^{0,2}(X;\C))$).
\end{lem}

\begin{vlongue}
\begin{proof} This follows from the fact that $\langle \cdot \vert \cdot \rangle$ is positive definite on  the real 
part of $H^{2,0}(X;\C)\oplus H^{0,2}(X;\C)$. An equivalent way to describe this argument it to identify $H^{2,0}(X;\C)$ with the 
space of holomorphic $2$-forms on $X$. Then, there is a natural, $\Aut(X)$-invariant, hermitian form 
on this space: given two holomorphic $2$-forms $\Omega_1$ and $\Omega_2$, the hermitian
product is the integral
\begin{equation}
\int_X\Omega_1\wedge\overline{\Omega_2}.
\end{equation}
Thus, the image of $\Aut(X)$ in $\GL(H^{2,0}(X;\C))$ is relatively compact.\end{proof}
\end{vlongue}

The \textbf{N\'eron-Severi group} $\NS(X;\Z)$ is, by definition, the discrete subgroup of  $H^{1,1}(X;\R)$ defined by 
 $\NS(X;\Z) =H^{1,1}(X;\R)\cap H^{2}(X;\Z)$;
more precisely, it is  the intersection of $H^{1,1}(X;\R)$ with the image of $H^{2}(X;\Z)$ in 
$H^{2}(X;\R)$, i.e. with the torsion free part of the Abelian group $H^{2}(X;\Z)$.
The Lefschetz theorem on $(1,1)$-classes identifies  
$\NS(X;\Z)$ with the subgroup of $H^{1,1}(X;\R)$ given by Chern classes of line bundles on $X$. 
The N\'eron-Severi group  is $\Aut(X)$-invariant, as well as  $\NS(X;R):=\NS(X;\Z)\otimes_\Z R$ for $R=\Q$, $\R$, or $\C$. 
 The dimension of  $\NS(X;\R)$ is the \textbf{Picard number} $\rho(X)$. 

\subsubsection{Norm of $f^*$} Let $\abs{ \cdot }$ be any norm on the vector space $H^*(X;\C)$. If $L$ is a linear
transformation of $H^*(X;\C)$ we denote by $\norm{L}$ the associated operator nom and if 
  $W\subset H^*(X;\C)$ is an $L$-invariant subspace of 
$H^*(X;\C)$, we denote by $\norm{ L}_W$ the operator norm of $L\rest{ W}$. 

If $u$ is an element of $H^{1,0}(X;\C)$, then $u\wedge {\overline{u}}$ is an element of $H^{1,1}(X;\R)$
such that $\abs{ u }^2 \leq C \abs{ u\wedge {\overline{u}}} $ for some constant $C$ that
depends only on the choice of norm on the cohomology; in particular, the norm of $f^*$ on $H^{1,0}(X;\C)$ is controlled
by the norm of $f^*$ on $H^{1,1}(X;\C)$. Using complex conjugation, the same results hold on 
$H^{0,1}(X;\C)$; by Poincar\'e duality we also control $\norm{ f^*}_{H^{p,q}(X;\C)}$ 
for $p+q>2$. Together with Lemma~\ref{lem:unitary_on _H20}, we obtain:

\begin{lem}\label{lem:cohomological_norm_estimates} 
Let $X$ be a compact  K\"ahler surface. There exists a constant $C_0>1$ such that 
\[
C_0^{-1} \norm{ f^* }_{H^*(X;\C)} \leq \norm{ f^*}_{H^{1,1}(X;\R)} \leq \norm{ f^*}_{H^*(X;\C)}
\]
for every automorphism $f\in \Aut(X)$. 
\end{lem}

\subsection{The K\"ahler, nef,  and pseudo-effective cones}\label{par:cones_definition} 
(See \cite{Boucksom:ENS, Lazarsfeld:Book1} for details on the notions introduced in this section.)

Let $\Kah(X)\subset H^{1,1}(X;\R)$ be the {\bf{K\"ahler cone}},
i.e. the cone of classes of K\"ahler forms. 
Its closure $\Kahbar(X)$ is a salient, closed, convex cone, and
\begin{equation} 
\Kah(X)\subset \Kahbar(X)\subset\{ v\in H^{1,1}(X;\R)\; ; \;  \langle v \,\vert\, v \rangle \geq 0\}.
\end{equation} 
The intersection $\NS(X;\R)\cap \Kah(X)$ is the {\bf{ample cone}} $\Amp(X)$, while $\NS(X;\R)\cap \Kahbar(X)$ is the {\bf{nef cone}} $\Nef(X)$.
They are all  invariant under the action of $\Aut(X)$ on $H^{1,1}(X;\R)$.
We shall also say that the elements of $\Kahbar(X)$ are nef classes, but the notation $\Nef(X)$ will be reserved for $\NS(X;\R)\cap \Kahbar(X)$.
The set of classes of closed positive currents is the {\bf{pseudo-effective cone}} $\Psef(X)$.  This
cone is an $\Aut(X)$-invariant, salient, closed, convex cone. It is dual to $\Kahbar(X)$ 
for the intersection form (see~\cite[Lem. 4.1]{Boucksom:ENS}): 
\begin{equation}\label{eq:cone_duality}
\Kahbar(X)=\{ u \in H^{1,1}(X;R)\; ; \; \langle u \,\vert\, v\rangle\geq 0 \quad \forall v \in \Psef(X)\}
\end{equation}
and vice-versa.  

We fix once and for all a reference K\"ahler form $\kappa_0$ with $[\kappa_0]^2= \int \kappa_0\wedge \kappa_0 =1$. Then 
we define the \textbf{mass} of a pseudo-effective class $a$ by $\M(a) = \langle a\,\vert\, [\kappa_0] \rangle$, or equivalently the 
mass of a closed positive current $T$ by $\M(T)  = \int T\wedge \kappa_0$; we may also extend this definition to any class, pseudo-effective or not (but then $\M(a)= \langle a\,\vert\, [\kappa_0] \rangle$ may be negative). The compactness of the set of closed positive currents of mass 1   implies that, for any norm $\abs{\cdot}$ on $H^{1,1}(X, \R)$, there exists a constant $C$ such that 
\begin{equation}\label{eq:comparison_norm_mass}
\forall a\in \Psef(X), \;  \ C\inv\abs{a}\leq \M(a)\leq  C \abs{a}.
\end{equation}

If $v$ is an element of $\Psef(X)$ and $v^2\geq 0$, then by the Hodge index theorem we know that $\langle u \,\vert\, v\rangle\geq 0$
for every class $u\in H^{1,1}(X;\R)$ such that $u^2\geq 0$ and $\langle u\,\vert\, [\kappa_0]\rangle\geq 0$ (see Equation~\eqref{eq:reverse_CS}). So, in Equation~\eqref{eq:cone_duality}, the
most important constraints come from the classes $v\in \Psef(X)$ with $v^2<0$. If $v$ is such a class, its Zariski decomposition 
expresses $v$ as a sum $v=p(v)+n(v)$ with the following properties (see~\cite{Boucksom:ENS}): 
\begin{enumerate}
\item this decomposition is orthogonal: $\langle p(v) \,\vert\, n(v)\rangle = 0$;
\item $p(v)$ is a nef class, i.e. $p(v)\in \Kahbar(X)$;
\item $n(v)$ is negative: it is a sum $n(v)=\sum_i a_i [D_i]$ with positive coefficients $a_i\in \R_+^*$ 
of classes of irreducible curves $D_i\subset X$ such that the Gram matrix $(\langle D_i\,\vert\, D_j\rangle)$
is negative definite. 
\end{enumerate}

\begin{pro}\label{pro:extremal_rays}
If a ray $\R_+ v$ of the cone $\Psef(X)$ is extremal, then either $v^2\geq 0$ or $\R_+ v=\R_+[D]$
for some irreducible curve $D$ such that $D^2<0$. 
The cone $\Psef(X)$ contains 
at most countably many extremal rays $\R_+v$ with $v^2<0$. 

Let $u$ be an isotropic element of $\Kahbar(X)$. If $\R_+ u$ is not an 
extremal ray  of $\Psef(X)$, then $u$ is proportional to an integral class $u'\in \NS(X; \Z)$.
\end{pro}

\begin{proof}
If $\R_+ v$ is extremal, the Zariski decomposition $v=p(v)+n(v)$ involves only one term. If $v=p(v)$ 
then $v^2\geq 0$. Otherwise   $v=n(v)$ and by extremality $n(v)=a[D]$ for some irreducible curve $D$
with $D^2<0$. The countability assertion follows, because  $\NS(X;\Z)$ is countable.
For the last assertion, multiply $u$ by $\langle u \vert [\kappa_0]\rangle^{-1}$ to assume $\langle u \vert [\kappa_0]\rangle=1$ and 
 write $u$ as a convex combination $u=\int v \, d\alpha(v)$, where $\alpha$ is a probability measure 
on $\Psef(X)$ such that $\alpha$-almost every $v$ satisfies 
\begin{itemize}
\item $\langle v \vert [\kappa_0]\rangle =1$,
\item $\R_+ v$ is extremal in $\Psef(X)$ and does not contain $u$. 
\end{itemize} 
Since $u$ is nef, $\langle u\,\vert\, v \rangle \geq 0$ 
for each $v$; and $u$ being isotropic, we get $v\in u^\perp\setminus \R u$ for $\alpha$-almost every $v$. By the Hodge
 index theorem, $v^2<0$ almost surely. Now, the first assertion of this proposition implies that $v\in \R_+[D_v]$ for some irreducible curve 
$D_v\subset X$ with negative self-intersection;  there are only countably many classes of that type, thus $\alpha$ is purely atomic, and 
$u$ belongs to  $\Vect([D_v]; \alpha(v)>0)$, a subspace of $\NS(X;\R)$ defined over $\Q$. On  this subspace,  
$q_X$ is semi-negative, and by the Hodge index theorem its kernel is $\R u$. Since $\Vect([D_v]; \alpha(v)>0)$ and $q_X$ are defined over $\Q$, we deduce 
that $u$ is proportional to an integral class. 
\end{proof}

\subsection{Non-elementary subgroups of $\Aut(X)$} 
When $X$ is a compact K\"ahler surface, 
the  action of $\Aut(X)$ on $H^{1,1}(X, \R)$ is  subject to several constraints:   the   Hodge index theorem implies that it must 
preserve a Minkowski structure 
and  in addition it preserves the  lattice given by the Neron-Severi group. 
In this section we review the first consequences of these constraints. 
\subsubsection{Isometries of Minkowski spaces} 
Consider the Minkowski space $\R^{m+1}$, endowed with its quadratic form $q$ of signature $(1,m)$
defined by 
\begin{equation}
q(x)=x_0^2-\sum_{i=1}^mx_i^2.
\end{equation} 
The corresponding bilinear form will be denoted $\langle \cdot \vert \cdot \rangle$. For future reference, note the following reverse Schwarz inequality:
\begin{equation}\label{eq:reverse_CS}
\text{ if } \;  q(x)\geq 0 \text{ and } q(x')\geq 0 \; \text{ then } \;  \langle{x\,\vert\, x'}\rangle \geq q(x)^{1/2} q(x')^{1/2} \end{equation}
with equality if and only if $x$ and $x'$ are collinear. We say that a subspace $W\subset \R^{m+1}$ is of {\bf{Minkowski type}}
if the restriction $q_{\vert W}$ is non-degenerate and of signature $(1,\dim(W)-1)$.

In this section, we review some well-known facts concerning  isometries 
of $\R^{1,m}=(\R^{m+1},q)$ (see e.g. \cite{Ratcliffe, kapovich, franchi-lejan} for details).
We  denote by $\abs{\cdot}$ the Euclidean norm on $\R^{m+1}$, and by $\P\colon \R^{m+1}\setminus\{ 0\}\to \P( \R^{m+1})$
the projection on the projective space $\P( \R^{m+1})=\P^m(\R)$.

The hyperboloid $\set{x \; ; \;  q(x)=1}$ has two components, and we denote by 
  $\O^+_{1,m}(\R)$ the subgroup of the orthogonal group $\O_{1,m}(\R)$ that preserves 
the  component $\mathcal Q =\{ q(x)=1\; ; \; x_0 >0\}$.  Endowed with the distance 
$d_\Hyp(x,y)= \cosh\inv \langle x \,\vert\, y\rangle $, $\mathcal Q$ is  
 a model of the real hyperbolic space $\Hyp^m$
of dimension $m$. The boundary at infinity of $\Hyp^m$ will be identified with 
 $\partial\P(\mathcal Q)\subset \P(\R^{m+1})$ and will be denoted by $\fr\Hyp^m$. It is the set of isotropic lines of $q$. 

Any isometry $\gamma$ of $\Hyp^m$ is induced by an element of $\O^+_{1,m}(\R)$, and extends continuously to $\fr\Hyp^m$: its action 
on $\fr\Hyp^m$ is given by its linear projective action on $\P(\R^{m+1})$. Isometries are classified in three types, 
according to their fixed point set in  $\Hyp^m \cup \fr\Hyp^m$:
\begin{itemize}
\item $\gamma$ is {\bf{elliptic}} if $\gamma$ has a fixed point in $\Hyp^m$;
\item $\gamma$ is {\bf{parabolic}} if $\gamma$ has  no  fixed point in $\Hyp^m$ and a unique  fixed point    in $\fr \Hyp^m$;
\item  $\gamma$ is {\bf{loxodromic}}  if $\gamma$ has no  fixed point in $\Hyp^m$ and  exactly    two  fixed points in $\fr \Hyp^m$.
\end{itemize}
A subgroup $\Gamma$ of $\O^+_{1,m}(\R)$ is {\bf{non-elementary}} if it does not preserve any finite subset of 
$\Hyp^m\cup \partial\Hyp^m$.  Equivalently $\Gamma$ is non-elementary 
if and only if it contains two loxodromic elements with disjoint fixed point sets.

\begin{vcourte}
The group $\O^+_{1,m}(\R)$ admits a \textbf{Cartan} or  \textbf{KAK decomposition}.
To state it, denote by $e_0=(1,0, \ldots, 0)$ the first vector of the canonical basis of $\R^{m+1}$; this vector is 
an element of $\Hyp^m$, and its stabilizer $\mathrm{Stab}(e_0)$ in $\O^+_{1,m}(\R)$ is a maximal compact subgroup, 
isomorphic to $\O_{m-1}(\R)$.

\begin{lem}[See \S I.5 of \cite{franchi-lejan}]\label{lem:KAK}
Every $\gamma\in\O^+_{1,m}(\R) $ can be written (non-uniquely)
as $\gamma= k_1 a k_2$, where $k_i\in \mathrm{Stab}(\bfe_0)$ and $a$ is a matrix of the form
\[
\begin{pmatrix} \cosh r & \sinh r & 0\\ \sinh r & \cosh r & 0 \\ 0 & 0 & \id_{m-1} \end{pmatrix}
\]
with $r=d_\Hyp(e_0, \gamma e_0)$.
\end{lem}

As an  immediate corollary, we get: 

\begin{cor}\label{cor:KAK}
If $\norm{\cdot}$ denotes the operator norm associated to the euclidean norm in $\R^{m+1}$,  then
$\norm {\gamma}  = \norm{a}$, where $\gamma =k_1 a k_2$ is any Cartan decomposition of $\gamma$. In particular $ \norm {\gamma} = \norm{\gamma\inv}$ and 
\[
  \norm{\gamma}  \asymp \cosh  d_\Hyp(e_0, \gamma(e_0)) \asymp \abs{\gamma e_0}.
\]
Furthermore for every $e\in \Hyp^m$ and any $\gamma\in\O^+_{1,m}(\R) $
\[
\norm{\gamma}  \asymp \cosh  d_\Hyp(e, \gamma(e)), 
\]
where the implied constant depends only on the base point $e$.
\end{cor}
\end{vcourte}

\begin{vlongue}
The group $\O^+_{1,m}(\R)$ admits a \textbf{Cartan} or  \textbf{KAK decomposition} (see~\cite[\S I.5]{franchi-lejan}).
To state it, denote by $e_0=(1,0, \ldots, 0)$ the first vector of the canonical basis of $\R^{m+1}$; this vector is 
an element of $\Hyp^m$, and its stabilizer $\mathrm{Stab}(e_0)$ in $\O^+_{1,m}(\R)$ is a maximal compact subgroup, 
isomorphic to $\O_{m-1}(\R)$.

\begin{lem} \label{lem:KAK}
Every $\gamma\in\O^+_{1,m}(\R) $ can be written (non-uniquely)
as $\gamma= k_1 a k_2$, where $k_i\in \mathrm{Stab}(\bfe_0)$ and $a$ is a matrix of the form
\[
\begin{pmatrix} \cosh r & \sinh r & 0\\ \sinh r & \cosh r & 0 \\ 0 & 0 & \id_{m-1} \end{pmatrix}
\]
with $r=d_\Hyp(e_0, \gamma e_0)$.
\end{lem}
 
\begin{proof} Note that $K:=\mathrm{Stab}(e_0)$ acts transitively on the set of 
hyperbolic geodesics through~$e_0$.  Denote by $L$ the hyperbolic 
geodesic $\Hyp^m\cap \Vect(e_0, e_1)$, where $e_1=(0,1, 0, \ldots, 0)$ is the second element of the canonical basis of $\R^{m+1}$. 
If $\gamma(e_0) = e_0$ then $\gamma$ belongs to $K$ and we are done. Otherwise
choose $k_1,k_2\in K$ such that $k_1\inv(\gamma(e_0))\in L$, $k_2(\gamma\inv(e_0))\in L$, and $e_0$ lies in between 
$k_2(\gamma\inv(e_0))$  and $k_1\inv(\gamma(e_0))$; then $e_0$ is in fact the middle point of $[k_2(\gamma\inv(e_0)), k_1\inv(\gamma(e_0))]$ because 
$d_\Hyp(e_0, \gamma(e_0)) = d_\Hyp(e_0, \gamma\inv(e_0))>0$. The isometry
$a:=k_1\inv \gamma k_2\inv$ maps $k_2(\gamma\inv(e_0))\in L$ to $e_0$ and $e_0$ to $k_1\inv(\gamma(e_0))\in L$. It follows that $a$ is a 
hyperbolic translation along $L$ of translation length $d_\Hyp (e_0, k_1\inv(\gamma(e_0))  = d_\Hyp (e_0,  \gamma(e_0))$. To conclude, change 
$a$ into $a\circ k\inv$ and $k_2$ into $k\circ k_2$ where $k$ is the element of $K$ that preserves $e_1$ and acts like 
$a$ 
on the orthogonal complement of  $\Vect(e_0, e_1)$. \end{proof}

\begin{cor}\label{cor:KAK}
If $\norm{\cdot}$ denotes the operator norm associated to the euclidean norm in $\R^{m+1}$,  then
$\norm {\gamma}  = \norm{a}$, where $\gamma =k_1 a k_2$ is any Cartan decomposition of $\gamma$. In particular $ \norm {\gamma} = \norm{\gamma\inv}$ and 
\[
  \norm{\gamma}  \asymp \cosh  d_\Hyp(e_0, \gamma(e_0)) \asymp \abs{\gamma e_0}.
\]
Furthermore for every $e\in \Hyp^m$ and any $\gamma\in\O^+_{1,m}(\R) $
\[
\norm{\gamma}  \asymp \cosh  d_\Hyp(e, \gamma(e)), 
\]
where the implied constant depends only on the base point $e$.
\end{cor}

This is an  immediate corollary of the previous lemma. 
\end{vlongue}

\subsubsection{Irreducibility}  
A non-elementary subgroup of 
$\O^+_{1,m}(\R)$  
does not need to
 act irreducibly on $\R^{m+1}$. Proposition~\ref{pro:invariant_cohomological_decomposition}, below, 
clarifies the possible situations. 

\begin{lem}\label{lem:restriction_isom}
Let $\Gamma$ be a non-elementary subgroup of $\O^+_{1,m}(\R)$ (resp. $\gamma$ be an element 
of $\O^+_{1,m}(\R)$). Let $W$ be a subspace of $\R^{1,m}$. 
\begin{enumerate}[\em (1)]
\item If $W$ is $\Gamma$-invariant, then either $(W, q\rest{ W})$ is a Minkowski space and
$\Gamma\rest{ W}$ is non-elementary, or $q\rest{ W}$ is negative definite and $\Gamma\rest{ W}$ is 
contained in a compact subgroup of $\GL(W)$.
\item If $W$ is $\gamma$-invariant and contains a vector $w$ with $q(w)>0$, then $\gamma\rest{ W}$ 
has the same type (elliptic, parabolic, or loxodromic) as $\gamma$; in particular, $W$ contains the 
$\gamma$-invariant isotropic lines if $\gamma$ is parabolic or loxodromic.
\end{enumerate} 
\end{lem}

\begin{proof}
The restriction $q\rest{W}$  is either a Minkowski form or is negative definite. 
Indeed, it  cannot be positive definite, because $W$ would then be a $\Gamma$-invariant line
intersecting the hyperbolic space $\Hyp^m$ in a fixed point; and it cannot be degenerate, since 
otherwise its kernel would give a $\Gamma$-invariant point on $\partial \Hyp^m$. 
If $q\rest{ W}$ is a Minkowski form and $\Gamma\rest{ W}$ is elementary, then $\Gamma$  
preserves a finite subset of  
$(\Hyp^m\cup\partial\Hyp^m) \cap W$ and $\Gamma$ itself is elementary. This proves the first assertion. 
The proof of the second one is similar.
\end{proof}

Let $\Gamma$ be a non-elementary subgroup of $\O^+_{1,m}(\R)$.
Let $\Zar(\Gamma)\subset \O_{1,m}(\R)$ be the Zariski closure of $\Gamma$, and \begin{equation}
G=\Zar(\Gamma)^\mathrm{irr}
\end{equation}
the neutral component of $\Zar(\Gamma)$, for 
the Zariski topology. Note 
that the Lie group $G(\R)$ is not necessarily connected for the euclidean topology.

\begin{vcourte} 
\begin{lem}[see \cite{cantat-gao-habegger-xie}, \S 4.1]\label{lem:finite-index}
The group  $\Gamma\cap G(\R)$ has finite index in $\Gamma$. 
If $\Gamma_0$ is a finite index subgroup of $\Gamma$, then $\Zar(\Gamma_0)^\mathrm{irr}=G$.
\end{lem}
\end{vcourte}

\begin{vlongue} 
\begin{lem}\label{lem:finite-index}
The group  $\Gamma\cap G(\R)$ has finite index in $\Gamma$. 
If $\Gamma_0$ is a finite index subgroup of $\Gamma$, then $\Zar(\Gamma_0)^\mathrm{irr}=G$.
\end{lem}

\begin{proof} 
The index of $G$ in $\Zar(\Gamma)$ is equal to the number $\ell$ of
 irreducible components of the algebraic variety $\Zar(\Gamma)$, 
and the index of $\Gamma\cap G(\R)$ in $\Gamma$ is at most $ \ell$. 
Now, let $\Gamma_0$ be a finite index subgroup of $\Gamma$. Then,  $\Gamma_0\cap G(\R)$ 
has finite index in $\Gamma\cap G(\R)$, and we can fix a finite subset
$\{\alpha_1, \dots , \alpha_k\}\subset \Gamma\cap G(\R)$ such that
$\Gamma\cap G(\R)=\bigcup_j\alpha_j(\Gamma_0\cap G(\R))$.
So 
\begin{equation}
\Zar(\Gamma\cap G(\R)) \subset 
\bigcup_j\alpha_j\Zar(\Gamma_0\cap G(\R))\subset G(\R).
\end{equation}
Because $\Gamma\cap G(\R)$
is Zariski dense in the irreducible group $G$ we find $G =
\Zar(\Gamma_0\cap G(\R))$. So $G\subset
\Zar(\Gamma_0)$ and the Lemma follows as
$G =  \Zar(\Gamma)^\mathrm{irr}$. 
\end{proof}
\end{vlongue}

\begin{pro}\label{pro:invariant_cohomological_decomposition}
Let $\Gamma\subset \O^+_{1,m}(\R)$ be non-elementary.
\begin{enumerate}[\em (1)]
\item The  representation of $\Gamma\cap G(\R)$ (resp. of $G(\R)$) on $\R^{1,m}$ splits as a
direct sum of irreducible representations, with exactly one irreducible factor of Minkowski type: 
\[
\R^{1,m}=V_+\oplus V_0;
\]
here $V_+$ is of Minkowski type, and $V_0$ is an orthogonal
sum of irreducible representations $V_{0,j}$ on which the quadratic form $q$ is negative definite. 
\item The restriction $G\rest{ V_+}$ coincides with $\SO(V_+;q\rest{ V_+})$.
\item The subspaces $V_+$ and $V_0$ are $\Gamma$-invariant, and the  representation of 
$\Gamma$ on $V_+$ is strongly irreducible.
\end{enumerate}
\end{pro}

\begin{proof} A group $\Gamma$ is non-elementary if and only if any of its finite index subgroups 
is non-elementary. So, we can apply Lemma~\ref{lem:restriction_isom} to $\Gamma\cap G(\R)$:
if $W\subset \R^{1,m}$ is a non-trivial $(\Gamma\cap G(\R))$-invariant subspace, $q\rest{ W}$ is non-degenerate.
As a consequence, $\R^{1,m}$ is the direct sum $W\oplus W^\perp$, where $W^\perp$
is the orthogonal complement of $W$ with respect to $q$. 
This implies that the representation of $\Gamma\cap G(\R)$ on $\R^{1,m}$ splits as a
direct sum of irreducible representations, with exactly one irreducible factor of Minkowski type, 
as asserted in   (1).

The group $G$ preserves this decomposition, and by Proposition 1 of \cite{Benoist-Harpe}, the 
restriction $G\rest{ V_+}$ coincides with $\SO(V_+;q\rest{ V_+})$; this group is isomorphic to 
the almost simple group $\SO_{1,k}(\R)$, with $1+k=\dim(V_+)$. This proves the second assertion.

Since $G$ is normalized by $\Gamma$, we see that for any $\gamma\in \Gamma$, 
$\gamma V^+$ is a $G$-invariant subspace of the same dimension as $V^+$ and 
on which $q$ is of Minkowski type. 
 Hence   $V_+$, as well as    its orthogonal 
complement $V_0$ are $\Gamma$-invariant. By Lemma~\ref{lem:finite-index}, 
the action of $\Gamma$ on $V_+$ is strongly irreducible; indeed, 
if a finite index subgroup $\Gamma_0$ in $\Gamma$ preserves a non-trivial subspace of 
$V_+$ then, by Zariski density of $\Gamma_0\cap G(\R)$ in $G(\R)$, 
this subspace must be $V_+$ itself. On $V_0$, $\Gamma$ permutes the irreducible factors 
$V_{0,j}$. 
\end{proof}
  
Now, set $V=\R^{1,m}$  and assume that there is a lattice $V_\Z\subset V$ such that 
\begin{itemize}
\item[(i)] $V_\Z$ is $\Gamma$-invariant;
\item[(ii)] the quadratic form $q$ is an integral quadratic form on $V_\Z$.
\end{itemize}
In other words, there is a basis of $V$ with respect to which $q$ and the elements 
of $\Gamma$ are given by matrices with integer coefficients. In particular, $V$ has
a natural $\Q$-structure, with $V(\Q)=V_\Z\otimes_\Z \Q$. This situation naturally arises for the action
of automorphisms of compact K\"ahler surfaces on $\NS(X;\R)$. The next lemma will be useful in \cite{finite_orbits}.  
 
\begin{lem}\label{lem:decomposition_V+V0_rational}
If $\Gamma$ contains a parabolic element, the decomposition $V_+\oplus V_0$ is defined 
over~$\Q$, $\Gamma\rest{ V_0}$ is a finite group, and $G$ 
is the subgroup $\SO(V_+;q)\times  \{\id_{V_0}\}$ of $\O(V;q)$.
\end{lem}

\begin{proof}
If $\gamma\in \Gamma$ is parabolic, it fixes pointwise a unique isotropic line, therefore this line is defined over $\Q$. In addition it must
be contained in $V_+$ because $(\gamma^n(u))_{n\geq 0}$ converges to the boundary point determined by this line for every $u\in \Hyp^m$. 
So, $V_+$ contains at least one non-zero element of $V_\Z$. Since the action of
$\Gamma$ on $V_+$ is irreducible, the orbit of this vector generates $V_+$ and is contained in $V_\Z$, so $V_+$
is defined over $\Q$. Its orthogonal complement $V_0$ is also defined over $\Q$, because $q$ itself is defined over $\Q$. 
As a consequence,  $\Gamma\rest{V_0}$ preserves the  lattice $V_0\cap V_\Z$ and the negative definite
form $q\rest{ V_0}$; hence, it is finite. Thus $G\rest{ V_0}$ is trivial and the last assertion follows
from the above mentioned equality $G\rest{ V_+}=\SO(V_+;q\rest{ V_+})$.
\end{proof}

\begin{eg} \label{eg:square_free} The purpose of this example is to show that 
the existence of a parabolic element in $\Gamma$ is indeed necessary   
in Lemma~\ref{lem:decomposition_V+V0_rational}, even for a group of automorphisms of a K3 surface. 

Let $a$ be a positive square free integer, for instance $a=7$ or $15$. Let 
$\alpha$ be the positive square root $\sqrt{a}$,  $K$ be the quadratic field $\Q(\alpha)$,
and $\eta$ be the unique non-trivial automorphism of $K$, sending $\alpha$ to its
conjugate ${\overline{\alpha}}:=\eta(\alpha)=-\sqrt{a}$. We view $\eta$ as a second embedding 
of $K$ in $\C$. Let ${\mathcal O}_K$ 
be the ring of integers of $K$.

Let $\ell$ be an integer $\geq 2$.
Consider the quadratic form  in $\ell+1$ variables defined by 
\begin{equation}
q_\ell (x_0, x_1, \ldots, x_\ell) = \alpha x_0^2- x_1^2- \cdots - x_\ell^2.
\end{equation}
It is non-degenerate and its signature is $(1,\ell)$.
The orthogonal group $\O(q_\ell; {\mathcal O}_K)$ is a lattice in the real algebraic group $\O(q_\ell, \R)$. 
The conjugate quadratic form ${\overline{q_\ell}} = {\overline{\alpha}} x_0^2- x_1^2- \cdots - x_\ell^2$ 
is negative definite. 

Embed ${\mathcal O}_K^{\ell + 1}$ into $\R^{2\ell +2}$ by the map  $(x_i)\mapsto (x_i, \eta(x_i))$, 
to get a lattice $\Lambda \subset \R^{2\ell +2}$ and consider the quadratic form $Q_\ell:=q_\ell \oplus {\overline{q_\ell}}$.
Then embed $\O(q_\ell; {\mathcal O}_K)$ into $\O(Q_\ell;\R)$ by the homomorphism 
$A\in \O(q_\ell, {\mathcal O}_K)\mapsto A\oplus \eta(A)$; we denote its image by $\Gamma_\ell^*\subset \O(Q_\ell;\R)$.
It  is shown in ~\cite{Morris:Book-Lattices}, Chapter~6.4, that 
\begin{itemize}
\item $Q_\ell$ is defined over $\Z$ with respect to $\Lambda$, 
\item  $\Gamma_\ell^*\subset \O(Q_\ell; \Z)$ (with respect to this integral structure), 
\item  the group $G=\Zar(\Gamma_\ell^*)^\mathrm{irr}$ 
coincides with $\SO(q_\ell;\R)\times \SO^0({\overline{q_\ell}};\R)$ (and the group $\eta(\O(q_\ell; {\mathcal O}_K))$ is 
dense in the compact group $\O({\overline{q_\ell}};\R)$). 
\end{itemize}

Now, assume $2\leq \ell \leq 4$, so that $2\ell + 2 \leq 10$, and change $Q_\ell$ into $4Q_\ell$: 
it is an even quadratic form on the lattice $\Lambda \simeq \Z^{2\ell+2}$. According to 
\cite[Corollary~2.9]{Morrison:1984}, there is a complex projective K3 surface $X$ for which $(\NS(X;\Z), q_X)$ 
is isometric to $(\Lambda, 4Q_\ell)$. On such a surface, the self-intersection of every curve is
divisible by $4$ and consequently there is no $(-2)$-curve. So, by the Torelli theorem for K3 surfaces (see~\cite{BHPVDV}), 
$\Aut(X)^*_{\vert \NS(X;\Z)}$ has finite index in $\O(4Q_\ell;\Z)$. 

Since $\O(4Q_\ell;\Z)=\O(Q_\ell;\Z)$ we can view $\Gamma_\ell^*$ as a subgroup of $\O(4Q_\ell;\Z)$. Set $\Gamma^*=\Aut(X)^*\cap \Gamma_\ell^*$ 
and let $\Gamma$ denote its pre-image in $\Aut(X)$. Then, $\Gamma$ is a subgroup of $\Aut(X)$ for which 
the decomposition $\NS(X;\R)_+\oplus \NS(X;\R)_0$ is non-trivial (here, both have
dimension $\ell+1$) while the representation is irreducible over $\Q$. 
\end{eg}

\subsubsection{The hyperbolic space $\Hyp_X$}\label{par:hyp_X}
Let $X$ be a compact K\"ahler surface. 
By the Hodge index theorem, the intersection form on $H^{1,1}(X, \R)$ has signature $(1,h^{1,1}(X)-1)$. 
The hyperboloid $$\set{u\in H^{1,1}(X, \R), \ \langle u \,\vert\, u\rangle}=1$$ has two connected components, 
one of which intersecting the K\"ahler cone.  The 
hyperbolic space  $\Hyp_X$ is by definition   this connected component, which is thus a  model of 
 the   hyperbolic space of dimension $h^{1,1}(X)-1$. We denote by $d_\Hyp$ the hyperbolic distance, which is defined as before
 by $\cosh(d_\Hyp(u, v)) = \langle u \,\vert\, v\rangle$. From Lemma \ref{lem:cohomological_norm_estimates} and Corollary 
 \ref{cor:KAK} we see that if 
 $\norm{\cdot}$ is any norm on $H^*(X,\C)$, then 
 $\norm{f^*}\asymp \norm{(f^*)\inv} \asymp \langle [\kappa_0] \,\vert\, f^*[\kappa_0]\rangle$ (here $\kappa_0$ is the fixed K\"ahler form introduced in Section~\ref{par:cones_definition}).

According to the classification of  isometries of hyperbolic spaces, there are three types of automorphisms: \textbf{elliptic}, \textbf{parabolic} 
and \textbf{loxodromic}. An important fact for us is that the type of isometry is related to 
the dynamics on $X$; for instance, every parabolic automorphism preserves a genus~$1$ fibration, every loxodromic automorphism
has positive topological entropy (see~\cite{Cantat:Milnor} for  more details).   
A subgroup $\Gamma$ of $\Aut(X)$ is  
said to be {\bf{non-elementary}} if its action on $\Hyp_X$ is non-elementary. 
\begin{vcourte}
As we shall see below, the existence of such a subgroup forces $X$ to be projective; for expository reasons,  the proof of this result is postponed to \S\ref{subsub:X-is-projective}, Theorem~\ref{thm:X-is-projective2}.
\end{vcourte} 
\begin{vlongue}
As we shall see below, the existence of such a subgroup forces $X$ to be projective:
 
\begin{thm}\label{thm:X-is-projective1}
If $X$ is a compact K\"ahler surface such that $\Aut(X)$ is non-elementary, then $X$ is projective. 
\end{thm}

For expository reasons,  the proof of this result is postponed to \S\ref{subsub:X-is-projective}, Theorem~\ref{thm:X-is-projective2}.
\end{vlongue} 
 
\subsubsection{Automorphisms and N\'eron-Severi groups}\label{par:dec_h11_Gamma} 
Let $X$ be a compact K\"ahler surface and $\Gamma$ be a non-elementary 
subgroup of $\Aut(X)$. 
Let $\Gamma^*_{p,q}$ be the image of $\Gamma$ in $\GL(H^{p,q}(X;\C))$, and $\Gamma^*$ be its image in 
$\GL(H^2(X;\C))$.
If we combine Proposition~\ref{pro:invariant_cohomological_decomposition}  together 
with Lemma~\ref{lem:unitary_on _H20} for $\Gamma^*_{1,1}$, we get an invariant decomposition
\begin{equation}
H^{1,1}(X;\R) = H^{1,1}(X;\R)_+\oplus H^{1,1}(X;\R)_0.
\end{equation}
Denote by  $H^2(X;\R)_0$ the direct sum of $H^{1,1}(X;\R)_0$ and  of the real part of $H^{2,0}(X;\C)\oplus H^{0,2}(X;\C)$; then
\begin{equation}
H^2(X;\R) = H^{1,1}(X;\R)_+\oplus H^2(X;\R)_0 
\end{equation}
and $\Gamma^*\rest{H^2(X;\R)_0}$ is contained in a compact group (see Lemma~\ref{lem:unitary_on _H20}).  
The N\'eron-Severi group is $\Gamma$-invariant, and since $X$ is projective 
it contains a vector with positive self-intersection.  Then 
 Proposition~\ref{pro:invariant_cohomological_decomposition} and Lemma~\ref{lem:restriction_isom} imply:

\begin{pro}\label{pro:NS_H11}
Let $X$ be a compact K\"ahler surface and  $\Gamma$ be a non-elementary subgroup of $\Aut(X)$. 
Then  $H^{1,1}(X;\R)_+=\NS(X;\R)_+$ is a Minkowski space, and the action of $\Gamma$ on this space is non-elementary
and strongly irreducible. 
\end{pro}

Since non-elementary groups of  isometries of $\Hyp^m$ occur only for $m\geq 2$, we get: 
 
\begin{cor}\label{cor:pic_number_3}
Under the assumptions of Proposition \ref{pro:NS_H11}, 
the Picard number $\rho(X)$ is   greater than or equal to $3$. 
If   equality holds then $\NS(X;\R)_+=\NS(X;\R)$ and the
action of $\Gamma$ on $\NS(X;\R)$ is strongly irreducible. 
\end{cor}

From now on we set: 
\begin{equation}\label{eq:piGamma_nu-def}
\Pi_\Gamma:=H^{1,1}(X;\R)_+=\NS(X;\R)_+.
\end{equation}
This is a Minkowski space on which $\Gamma$  acts strongly irreducibly; the intersection form is negative 
definite on the orthogonal complement 
\begin{equation}
\Pi_\Gamma^\perp\subset H^{1,1}(X;\R).
\end{equation} 
Moreover by Proposition~\ref{pro:invariant_cohomological_decomposition}.(2)  the group $G=\Zar(\Gamma)^\mathrm{irr}$ satisfies 
$G(\R)\rest{\Pi_\Gamma}=\SO(\Pi_\Gamma)$. If $\Gamma$ contains a parabolic element, then 
$\Pi_\Gamma$ is rational  with respect to the integral structures 
of $\NS(X;\Z)$ and $H^2(X;\Z)$, 
and 
 $G(\R)=\SO(\Pi_\Gamma)\times \{\id_{\Pi_\Gamma^\perp}\}$ (see Lemma~\ref{lem:decomposition_V+V0_rational}).
 
\subsubsection{Invariant algebraic curves I}\label{par:Invariant_curves_I}

Assume that $\Gamma$ is non-elementary and 
let $C\subset X$ be an irreducible algebraic curve with a finite $\Gamma$-orbit. 
Then the action of $\Gamma$ on  $\Vect_\Z\set{f^*[C]; f\in \Gamma}\subset \NS(X; \Z)$ factors through a finite group. From Propositions~\ref{pro:invariant_cohomological_decomposition} and~\ref{pro:NS_H11} 
we deduce that
the intersection form is negative definite on $\Vect_\Z(\Gamma\cdot [C])$, thus 
$\Vect_\R(\Gamma\cdot [C])$ is one of the irreducible factors of $\NS(X, \R)_0$. 
This argument, together with  Grauert's contraction theorem, leads to the following result (we refer to~\cite{Cantat:Milnor, Kawaguchi:AJM} for a proof; the result holds more generally for subgroups containing a loxodromic element):

\begin{lem}\label{lem:periodic_curves} Let $X$ be a compact Kähler surface and  
$\Gamma$ be a non-elementary group of automorphisms on $X$. 
 Then, there are at most finitely many 
$\Gamma$-periodic irreducible curves. The intersection form is negative definite on the subspace of $\NS(X; \Z)$ generated by 
the classes of these curves. There is a compact complex analytic surface $X_0$ and a  $\Gamma$-equivariant
bimeromorphic morphism $X\to X_0$ that contracts these curves and is an isomorphism in their complement. 
\end{lem}
The next result follows from \cite{diller-jackson-sommese}.

\begin{pro} \label{pro:diller-jackson-sommese}
Let $X$ be a compact Kähler surface and 
$\Gamma$ a non-elementary subgroup of $\Aut(X)$. Then any $\Gamma$-periodic curve has arithmetic genus 0 or 1. 
\end{pro}

Note if $C$ is $\Gamma$-periodic, this result  applies to $\widetilde C = \Gamma\cdot C$, which is invariant. 
Then,  the normalization of any irreducible component of $\widetilde C$ has genus 0 or 1, 
and the incidence graph of the components of $\widetilde C$ obeys certain restrictions (see~\cite[\S 4.1]{Cantat:Milnor} for details). 
If furthermore $X$ is a K3 or Enriques surface, each component is a smooth rational curve of self-intersection $-2$.

\subsubsection{The limit set}\label{par:limit_set_I}
Let $\Gamma\subset \Aut(X)$ be non-elementary. The {\bf{limit set}} of $\Gamma$ 
is the closed subset $\Lim(\Gamma)\subset \partial \Hyp_X \subset \P\lrpar{H^{1,1}(X;\R)}$ defined
by one of the following equivalent assertions:
\begin{enumerate}
\item[(a)] $\Lim(\Gamma)$ is the smallest, non-empty, closed, and $\Gamma$-invariant subset of $\P(\overline{\Hyp_X})$;
\item[(b)] 
\begin{vcourte}
$\Lim(\Gamma)\subset \fr \Hyp_X$ is the closure of the set of fixed points of loxodromic elements of $\Gamma$
in $\fr\Hyp_X$;
\end{vcourte}
\begin{vlongue}
$\Lim(\Gamma)\subset \fr \Hyp_X$ is the closure of the set of fixed points of loxodromic elements of $\Gamma$
in $\fr\Hyp_X$ (these fixed points correspond to isotropic lines on which the loxodromic isometry act as a dilation or contraction);
\end{vlongue}
\item[(c)] $\Lim(\Gamma)$ is the accumulation set of any $\Gamma$-orbit $\Gamma(\P(v))\subset \P(H^{1,1}(X;\R))$, for any $v\notin \Pi_\Gamma^\perp$.
\end{enumerate}
We refer to \cite{kapovich, Ratcliffe}  for a study of such limit sets. From the second  
characterization we get:
\begin{lem}\label{lem:limitset_in_NS}
The limit set $\Lim(\Gamma)$ of a non-elementary group is contained in $\P(\Pi_\Gamma)\cap \partial\Hyp_X$.
\end{lem}

  From the third characterization, $\Lim(\Gamma)$ is contained in 
 the closure of $\Gamma(\P([\kappa]))$ for every K\"ahler form $\kappa$ on $X$. Since $X$ must be
projective, we can chose $[\kappa]$ in $\NS(X;\Z)$. As a consequence, 
$\Lim(\Gamma)$ is contained in $\Nef(X)$: 

\begin{lem}\label{lem:limitset_in_nef}
Let $X$ be a compact K\"ahler surface. If $\Gamma$ is a non-elementary subgroup of 
$\Aut(X)$ its limit set satisfies $\Lim(\Gamma)\subset \P(\Nef(X)) \subset \P(\NS(X;\R))$.
\end{lem}

\subsection{Parabolic automorphisms}\label{par:parabolic_basics} 
We collect  a few basic facts on parabolic
automorphisms: they will be used in the next section to describe explicit examples, and then in Section~\ref{sec:stiffness}.

Let $f$ be a parabolic automorphism of a compact K\"ahler surface.  
Then $f^*$ preserves a unique point on $\fr \Hyp_X$, and 
$f$ preserves a unique genus $1$ fibration 
$\pi_f\colon X\to B$ onto some Riemann surface $B$. The fixed point of $f^*$ on $\fr \Hyp_X$ is
given by the class $[F]$ of any fiber of $\pi_f$ (see~\cite{Cantat:Milnor}). The fibers of $\pi_f$ 
are the elements of the linear system $\vert F\vert$, $\pi_f$ is uniquely determined by $[F]$, and 
 if $g$ is another automorphism of $X$ that preserves a smooth fiber of $\pi_f$ (resp. the point $\P[F]\in \P \NS(X;\R)$), 
then $g$ preserves the fibration and is either elliptic or parabolic. 

\begin{lem}\label{lem:charac_parabolic_on_K3}
 Let $X$ be a K3 or Enriques surface, and $\pi\colon X\to B$ be a genus $1$ fibration.
 If $g\in\Aut(X)$ maps some fiber $F$ of $\pi$ to a fiber of $\pi$, then $g$ preserves the fibration and
either $g$ is parabolic or it is periodic of order $\leq 66$.
\end{lem}
\begin{proof}  
Since $g$ maps $F$ to some fiber $F'$,  it maps the complete linear system $\vert F\vert$ to $\vert F'\vert$, 
but both linear systems are made of the fibers of $\pi$. So $g$ preserves the fibration and is not loxodromic.  
If $g$ is not parabolic it is elliptic, and its action on cohomology has finite order since it 
preserves $H^2(X, \Z)$.
 On a K3 or Enriques 
surface every holomorphic vector field vanishes identically, so $\Aut(X)^0$ is trivial and the kernel of the homomorphism $ \Aut(X)\ni f\mapsto f^*$ is finite (see~\cite[Theorem 2.6]{Cantat:Milnor}); as a consequence, any elliptic automorphism has finite order.   The upper
bound on the order of $g$ was obtained in~\cite{Keum:2016}.
\end{proof}

\begin{pro}\label{pro:parabolic_infinite}
Let $X$ be a compact K\"ahler surface and let $f$ be a parabolic automorphism of $X$, preserving the genus $1$ fibration $\tau\colon X\to B$.
Consider the group $\Aut(X;\tau):=\{ g\in \Aut(X)\; ; \; \exists g_B\in \Aut(B), \; \tau\circ g = g_B\circ \tau\}$, and assume that the 
image of the homomorphism $g\in \Aut(X;\tau)\to g_B\in \Aut(B)$ is infinite. Then, $X$ is a torus. 
\end{pro}

This result directly follows from the proof of Proposition~3.6 in~\cite{cantat-favre}.
In particular the automorphism $f_B\in \Aut(B)$ such that $\pi_f\circ f = f_B\circ \pi_f$ has 
finite order when $X$ is a K3, an Enriques,  or a rational surface. 
\begin{vlongue}
The dynamics of these automorphisms is described in Section~\ref{par:parabolic_automorphisms}.\romain{attention}\romain{lemme pas cité bougé dans vlongue?}

\begin{lem}\label{lem:pairs_of_twists} 
If $\Gamma$ is  a subgroup of $\Aut(X)$ containing a parabolic automorphism $g$, then  $\Gamma$ is non-elementary 
if and only if it contains another parabolic automorphism
 $h$ such that the invariant fibrations $\pi_g$ and $\pi_h$ are distinct.
Then, the tangency locus of the two fibrations is either empty or a curve, and there are positive integers 
$m$, $n$ such that $g^m$ and $h^n$ generate a free group of rank $2$.
\end{lem}

\begin{proof} Let $F$ be a fiber of $\pi_g$. If $\Gamma$ is non-elementary, there is an element 
$f$ in $\Gamma$ that does not fix $[F]$; in particular $f$ does not preserve $\pi_g$. 
Then, $h:=f^{-1}\circ g\circ f$ is another parabolic automorphism with a distinct invariant fibration, namely $\pi_h=\pi_g\circ f$.
Being distinct, $\pi_g$ and $\pi_h$ have a tangency locus of codimension $\geq 1$. 

Conversely, if $\Gamma$ contains two parabolic automorphisms with distinct fixed point in $\fr \Hyp_X$, 
then the ping-pong lemma proves that there are powers $m$, $n\geq 1$ such that $\langle g^m, h^n\rangle$ 
is a free group of rank $2$; in particular, $\Gamma$ is non-elementary. (See~\cite{Cantat:Milnor} for more precise results.)
\end{proof}
\end{vlongue}

\section{Examples and classification}\label{par:Examples_Classification} 

This section may be skipped in a first reading. 
 It describes a few examples, and proves that a compact K\"ahler surface $X$ is projective when its   automorphism group  is non-elementary.

\subsection{Wehler surfaces (see~\cite{Cantat-Oguiso, Reschke, Wang:1995, Wehler})}\label{par:Wehler_surfaces_I}\label{par:Wehler} 

Consider the variety $M=\P^1\times \P^1\times \P^1$ and let $\pi_1$, 
$\pi_2$, and $\pi_3$ be the projections on the first, second, and third factor: $\pi_i(z_1,z_2,z_3)=z_i$.
Denote by $L_i$ the line bundle $\pi_i^*({\mathcal{O}} (1))$ and set
\begin{equation}
L =  L_1^2\otimes L_2^2\otimes L_3^2 = \pi_1^*({\mathcal{O}} (2))\otimes \pi_2^*({\mathcal{O}} (2))\otimes \pi_3^*({\mathcal{O}}(2)). 
\end{equation}
Since $K_{\P^1} = \mathcal {O}(-2)$, 
this line bundle $L$ is the dual of the canonical bundle $K_M$. By definition,  
$\vert L\vert\simeq \P(H^0(M, L))$ is the linear system of surfaces $X\subset M$  given by
the zeroes  of global sections $P\in H^0(M, L)$.  Using affine coordinates $(x_1,x_2,x_3)$ on $M=\P^1\times \P^1\times \P^1$,
such a surface is defined by a polynomial equation $P(x_1,x_2,x_3)=0$ whose degree with respect to each variable is $\leq 2$ (see~\cite{Cantat:Acta, McMullen:Crelle} for explicit examples). These surfaces will be referred to as 
 {\bf{Wehler surfaces}} or {\textbf{(2,2,2)-surfaces}}; modulo $\Aut(M)$, they form a family of dimension $17$.
 
 Fix $k\in \{1, 2, 3\}$ and denote by $i < j$ the other indices. If we project $X$ to $\P^1\times \P^1$ by  $\pi_{ij}=(\pi_i,\pi_j)$, 
we get a $2$ to $1$ cover (the generic fiber is made of two points, but some fibers may be rational curves). As soon as $X$ is smooth 
the involution $\sigma_k$ that permutes the two points in each (general) fiber of $\pi_{ij}$ is
an involutive automorphism of $X$; indeed $X$ is a K3 surface and any birational self-map of such a surface is an automorphism. 

\begin{pro}\label{pro:Wehler_surface_Pic}
There is a countable union of proper Zariski closed subsets $(W_i)_{i\geq 0}$ in  $\vert L\vert$ such that
\begin{enumerate}[\em (1)]
\item if $X$ is an element of $\vert L\vert\setminus W_0$, then $X$  is a smooth K3 surface and 
$X$ does not contain any fiber of the projections $\pi_{ij}$;
\item if $X$ is an element of $\vert L\vert\setminus (\bigcup_i W_i)$, the restriction morphism $\Pic(M)\to \Pic(X)$ 
is   surjective. In particular its Picard number is $\rho(X) = 3$. 
\end{enumerate}
\end{pro}

\begin{vcourte}
See \cite{Cantat-Oguiso} for the proof of this proposition, as well as that of Lemma \ref{lem:wehler_no_fiber_free}.
\end{vcourte}
From  the second assertion, we deduce that for a very general $X$, 
 $\Pic(X)$ is isomorphic to $\Pic(M)$: it is the free Abelian 
group of rank $3$, generated by the classes 
\begin{equation}
c_i:=[(L_i)_{\vert X}].
\end{equation} 
The elements of $\vert (L_i)_{\vert X} \vert $ are the curves of $X$ given by the equations $z_i=\alpha$ for some $\alpha \in \P^1$. The arithmetic
genus of these curves is equal to $1$: in other words the projection $(\pi_i)_{\vert X}\colon X\to \P^1$ is a genus $1$ fibration. Moreover, for a general choice of $X$ in $\vert L \vert$,  $(\pi_i)_{\vert X}$ has $24$ singular fibers of type ${\sf{I}}_1$, i.e. isomorphic to a rational curve with exactly one simple double point. 
The intersection form is given by $c_i^2=0$ and $\langle c_i\vert c_j \rangle= 2$ if $i\neq j$, so that its matrix is given by 
\begin{equation}\label{eq:intersection_matrix}
\left(\begin{array}{ccc} 0 & 2 & 2\\
2 & 0 & 2\\
2 & 2 & 0\end{array}\right).
\end{equation}

\begin{vlongue}
\begin{proof}[Proof of Proposition~\ref{pro:Wehler_surface_Pic}]  
By Bertini's theorem, $X$ is smooth as soon as it is in the complement of some proper Zariski closed subset $W_0\subset \vert L \vert$. 
Now, let us assume that $X$ is smooth. The adjunction formula implies that the canonical bundle $K_X$ is trivial. 
From the hyperplane section theorem of Lefschetz \cite{Milnor:Morse_Theory}, we know that $X$ is 
simply connected. So, $X$ is a K3 surface (see~\cite{BHPVDV}). Write the equation of $X$ as $A(x_1,x_2)x_3^2+B(x_1,x_2)x_3+C(x_1,x_2)=0$. 
Then, $X$ contains a fiber $\pi_{12}^{-1}(a_1,a_2)$ if and only if the three curves given by $A=0$, $B=0$, and $C=0$ contain the 
point $(a_1,a_2)$. This imposes a non-trivial algebraic condition on $X$; hence, enlarging $W_0$, the first assertion is satisfied. 

For the second assertion, we apply a general form of the Noether-Lesfchetz theorem \cite[Th\'eor\`eme 15.33]{Voisin:BookHodge}.
We know that $L$ is very ample, that $H^{2,0}(X)$ is isomorphic to $\C$. Indeed  $X$ is a K3 surface, and   $H^{2,0}(X)$ is contained in the vanishing 
cohomology since $X$ may degenerate on six copies of $\P^1\times \P^1$ (taking the equation $(x_1^2-1)(x_2^2-1)(x_3^2-1)=0$).
So, the Noether-Lefschetz theorem says precisely that the restriction morphism is surjective for a very general choice of 
$X\in \vert L\vert$. \end{proof}
\end{vlongue}

\begin{lem}\label{lem:wehler_no_fiber_free}
Assume that $X$ does not contain any fiber of the projection $\pi_{ij}$. Then, the involution $\sigma_k^*$ preserves the 
subspace $\Z c_1\oplus \Z c_2\oplus \Z c_3$ of $\NS(X;\Z)$ and 
\[
\sigma_k^*c_i=c_i,\;  \sigma_k^*c_j=c_j,\;  \sigma_k^*c_k=-c_k+2c_i+2c_j.
\]
Equivalently, the action of $\sigma_k^*$ on $\Vect_\R (c_1, c_2, c_3)$ preserves the classes $c_i$ and 
$c_j$ and acts as a  
reflection with respect to the hyperplane $\Vect(c_i, c_j)\subset \NS(X;\R)$.
In other words,  
setting $u_k=(c_1+c_2+c_3)-2c_k$, $\sigma_k(v)=v+\frac{1}{2} \langle v\vert  u_k\rangle u_k$ for all $v$ in $\Z c_1\oplus \Z c_2\oplus \Z c_3$.
\end{lem}

\begin{vlongue}
\begin{proof}
Since $\sigma_k$ preserves $\pi_{ij}$ it preserves the fibers of $\pi_i$ and $\pi_j$, hence $\sigma_k^*$ fixes
$c_i$ and $c_j$. Now, consider a fiber $C=\{ z_k=w\}\subset X$ of $\pi_k$. Then, $\sigma_k(C)\cup C=\pi_{ij}^{-1}(\pi_{ij}(C))$ because there is no curve in the 
fibers of $\pi_{ij}$. On the other hand, $\pi_{ij}(C)\subset \P^1\times \P^1$ is a (2,2)-curve so it is  rationally equivalent to the union of 
two vertical and two horizontal projective lines. This gives $\sigma_k^*c_k=-c_k+2c_i+2c_j$. 
\end{proof}
\end{vlongue}

Combining this lemma with the previous proposition, we see that a very general  Wehler surface
has Picard number $3$, $\Hyp_X$ has dimension 2, $\NS(X;\Z)=\Vect_\Z(c_1,c_2,c_3)$ 
and the matrices of the $\sigma_i^*$  in the basis $(c_i)$ are 
\begin{equation}\label{eq:matrices_sigmai_wehler}
\sigma_1^*= \left(\begin{array}{ccc} -1 & 0 & 0\\
2 & 1 & 0\\
2 & 0 & 1\end{array}\right), \;\,
\sigma_2^*= \left(\begin{array}{ccc} 1 & 2 & 0\\
0 & -1 & 0\\
0 & 2 & 1\end{array}\right), \;\, 
\sigma_3^*= \left(\begin{array}{ccc} 1 & 0 & 2\\
0 & 1 & 2 \\
0 & 0 & -1\end{array}\right).
\end{equation}

\begin{pro}\label{pro:Wehler_generic}
If $X$ is a very general  Wehler surface  then: 
\begin{enumerate}[\em (1)]
\item $X$ is a smooth K3 surface with Picard number $3$;
\item $\Aut(X)$ is equal to $\langle \sigma_1,   \sigma_2,   \sigma_3  \rangle$, it is a free product of three copies of $\Z/2\Z$, and 
$\Aut(X)^*$ is a finite index subgroup in the group of integral isometries of $\NS(X;\Z)$;
\item $\Aut(X)^*$ acts strongly irreducibly on $\NS(X;\R)$;
\item $\Aut(X)$ does not preserve any algebraic curve $D\subset X$; 
\item the limit set of $\Aut(X)^*$ is equal to $\partial \Hyp_X$;
\item the compositions $\sigma_i\circ \sigma_j$ and $\sigma_i\circ \sigma_j\circ \sigma_k$ are respectively parabolic 
and loxodromic for every triple $(i,j,k)$ with $\set{ i, j, k }=\set{1,2,3}$.
\end{enumerate}
\end{pro}


\begin{proof}
The first three assertions follow from Proposition~\ref{pro:Wehler_surface_Pic},  
\cite[\S 1.5]{Cantat:Acta} and 
\cite[Thm 3.6]{Cantat-Oguiso}.
For the fourth one, note that any invariant curve $D$ would yield a non-trivial fixed point $[D]$ in $\NS(X;\Z)$, contradicting assertion (3). 
The fifth one follows from the second because the limit set of a lattice in $\Isom(\NS(X;\R))$ is always equal to $\partial \Hyp_X$.
To prove the last assertion, it suffices to compute the corresponding product of matrices given in Equation~\eqref{eq:matrices_sigmai_wehler} (see~\cite{Cantat:Acta}).\end{proof}

\begin{rem}
In~\cite{Baragar:2016}, Baragar gives examples of smooth surfaces $X\in \vert L\vert$ for which $\rho(X)\geq 4$ and the limit
set of $\Aut(X)^*$ in $\partial \Hyp_X$ is a genuine fractal set. 
\end{rem}

\subsection{Pentagons}\label{par:pentagons}
The dynamics  on the space  of pentagons   with given side lengths, introduced in \S\ref{par:pentagons_intro},
shares important similarities with the dynamics on Wehler surfaces. A pentagon with side lengths $\ell_0, \ldots, \ell_4$ modulo translations of 
the plane is the same as the data of a 5-tuple of vectors $(v_i)_{i=0, \ldots , 4}$ in $\R^2$ (identified with $\C$)
of respective length $\ell_i$   such that $\sum_i v_i = 0$. Write $v_i   = \ell_i t_i$ with $\abs{t_i} = 1$. Then the action of $\SO_2(\R)$ can be identified to the diagonal multiplicative action of $\U_1=\{\alpha \in \C\; ; \; \abs{\alpha}=1\}$ on the $t_i$:
\begin{equation}
 \alpha\cdot (t_0, \ldots, t_4)=(\alpha t_0, \ldots \alpha t_4).
\end{equation} 
Now, following Darboux \cite{darboux}, we consider the surface $X$ in $\P^4_\C$  defined by the equations
\begin{equation}\label{eq:equation_pentagons}
\begin{cases} 
\ell_0z_0+ \ell_1 z_1+\ell_2z_2+\ell_3z_3+\ell_4 z_4 = 0\\
\ell_0/z_0+ \ell_1/ z_1+\ell_2/z_2+\ell_3/z_3+\ell_4/ z_4 = 0
\end{cases}
\end{equation}
where $[z_0:\ldots :z_4]$ is some fixed choice of homogeneous coordinates, and the second 
equation must be multiplied by $z_0z_1z_2z_3z_4$ to obtain a homogeneous equation of degree $4$.

\begin{rem}\label{rem:hessian} This surface is isomorphic to the Hessian of a cubic surface (see~\cite[\S 9]{Dolgachev:CAG}). More precisely, consider a cubic surface
$S\subset \P^3_\C$ whose equation  $F$ can be written in Sylvester's pentahedral form that is, 
as a sum $F=\sum_{i=0}^4 \lambda_i F_i^3$ for some complex numbers $\lambda_i$ and linear forms $F_i$ with $\sum_{i=0}^4 F_i=0$.
By definition, its Hessian surface $H_F$ is defined by $\det(\partial_i\partial_j F)=0$. Then, using the linear forms
$F_i$ to embed $H_F$ in $\P^4_\C$,  we obtain the surface defined by the pair of
equations $\sum_{i=0}^4 z_i=0$ and $\sum_{i=0}^4\frac{1}{\lambda_iz_i}=0$. Thus, $H_F$ is our surface $X$, 
for $\ell_i^2=\lambda_i$.
We refer to \cite{Dolgachev-Keum, Dardanelli-vanGeemen, Dolgachev:Salem, Rosenberg} for an introduction to 
these surfaces and their birational transformations.\end{rem}

\begin{vlongue}
For completeness, we prove some of its basic properties.
\end{vlongue}

\begin{lem}\label{lem:pentagon_smoothness}
Let $\ell=(\ell_0, \ldots, \ell_4)$ be an element of $(\C^*)^5$.
The surface $X\subset \P^4_\C$ defined by the system~\eqref{eq:equation_pentagons} has $10$ singularities 
at the points $q_{ij}$ determined by the system of equations 
$\ell_i z_i+\ell_j z_j=0$, $z_k=z_l=z_m=0$ with $i<j$ and $\{i,j,k,l,m\}=\{0,1,2,3,4\}$. 
In the complement of these ten isolated singularities, $X$ is smooth if and only if 
\begin{equation}\label{eq:pentagon_smoothness}
\sum_{i=0}^4 \e_i \ell_i \neq 0\quad \forall \e_i \in \set{\pm 1}.
\end{equation} 
\end{lem}

\begin{vcourte} 
The proof is elementary and left to the reader.
\end{vcourte}
\begin{vlongue}
\begin{proof}
We first look for singularities in the complement of the hyperplanes $z_i=0$, and work in the chart $z_0 = 1$. Then $z_4 = - (\ell_0+ \ell_1z_1+ \ell_2z_2 + \ell_3z_4)/\ell_4$ and we replace in the
second equation of \eqref{eq:equation_pentagons} to obtain an affine equation of $X$ in this chart, namely:
\begin{equation}\label{eq:equation_pentagons_affine}
\frac{\ell_1}{z_1} + \frac{\ell_2}{z_2} + \frac{\ell_3}{z_3} - \frac{\ell_4^2}{\ell_0+ \ell_1z_1+ \ell_2z_2+ \ell_3z_3} + \ell_0 = 0. 
\end{equation}
Singularities are determined by  the system of equations
$z_1^2 =z_2^2=z_3^2 = \ell_4^{-2}(\ell_0+ \ell_1z_1+ \ell_2z_2 + \ell_3z_3)^2$. So, by symmetry, at a singularity where none of the coordinates vanishes we must have 
$z_i=\varepsilon _i z$ for some $\varepsilon _i=\pm 1$ and a common factor $z\neq 0$; this is precisely Condition~\eqref{eq:pentagon_smoothness}.

Looking for singularities with one coordinate equal to $0$, say $z_1=0$ in the chart $z_0=1$, we obtain the system 
of equations 
\begin{equation}\label{eq:equation_pentagons_singularities}
\begin{cases} 
0= (\ell_0 z_2z_3+\ell_3 z_2+\ell_2z_3)(\ell_0 +\ell_2 z_2+\ell_3 z_3)+(\ell_1^2-\ell_4^2)z_2z_3\\
0=\ell_1 z_3(\ell_0+2\ell_2 z_2+\ell_3z_3)\\
0= \ell_1 z_2 (\ell_0+\ell_2z_2+2\ell_3 z_3)
\end{cases}
\end{equation}
together with $\ell_0+\ell_2 z_2+\ell_3z_3+\ell_4z_4=0$ and $\ell_1 z_2z_3z_4=0$ (in particular, $z_2$, $z_3$ or $z_4$
must vanish).
The solutions of this system are given by $z_1=z_2=z_3=0$, which gives the point $q_{04}=[\ell_4:0:0:0:-\ell_0]$,
or $z_1=z_2=0$ and $\ell_0+\ell_3z_3=0$, which corresponds to $q_{03}=[\ell_3:0:0:-\ell_0:0]$, or $z_1=z_3=0$
which gives $q_{02}$, or $z_1=z_4=0$ but then either $z_2=0$ or $z_3=0$ and we end up again with $q_{02}$ and $q_{03}$.
The result follows by symmetry.
\end{proof}
\end{vlongue}

\begin{lem}
If $\ell\in (\C^*)^5$ satisfies Condition~\eqref{eq:pentagon_smoothness}, then the ten singularities are simple nodes (Morse singularities) and 
the surface $X$ is a (singular) K3 surface: a minimal resolution $\hat X$  of $X$ is a K3 surface, which is  
 obtained by blowing-up its ten nodes, thereby creating ten rational $(-2)$-curves.
\end{lem}

\begin{proof}
Working in the chart $z_0=1$ and replacing $z_4$ by $-(\ell_0+\ell_1z_1+\ell_2z_2+\ell_3z_3)/\ell_4$, 
the quadratic term 
of the equation of $X$ at the singularity $(z_1,z_2,z_3)=(0,0,0)$ is $(-\ell_0/\ell_4)Q$, where 
\begin{equation}\label{eq:pentagons_quadratic_Q}
Q(z_1,z_2,z_3)=\ell_1z_2z_3+\ell_2z_1z_3+\ell_3z_1z_2
\end{equation}
is a non-degenerate quadratic form (its determinant is $2\ell_1\ell_2\ell_3\neq 0$). So locally $X$ is holomorphically 
equivalent to the quadratic cone $\{Q=0\}$, hence to a quotient singularity $(\C^2,0)/\eta$ with $\eta(x,y)=(-x,-y)$. 
The minimal resolution of such a singularity is obtained by a simple blow-up of the ambient space, the exceptional 
divisor being a $(-2)$-curve in the smooth surface~$\hat{X}$.
The adjunction formula shows that there is a holomorphic $2$-form $\Omega_X$ on the regular part of $X$; locally, $\Omega_X$
lifts to an $\eta$-invariant form $\Omega_X'$ on $\C^2\setminus \{ 0\}$, which by Hartogs extends at the origin
to a non-vanishing $2$-form. To recover $\hat{X}$, one can first blow-up $\C^2$ at the origin and then take the quotient 
by (the lift of) $\eta$: a simple calculation shows that $\Omega_X'$ determines a non-vanishing $2$-form on $\hat{X}$.
After such a surgery is done at the ten nodes, $\hat{X}$ is a smooth surface with a non-vanishing section of $K_{\hat{X}}$; 
since it contains at least ten rational curves, it can not be an Abelian surface, so it must be a K3 surface. 
\end{proof}

\begin{rem}
Let $L_{ij}$ be the line defined by the  equations 
$z_i=0$, $z_j=0$, $\ell_0z_0+\cdots +\ell_4z_4=0$; each of these ten lines is contained in $X$, each of them contains
$3$ singularities of $X$ (namely $q_{kl}$, $q_{lm}$, $q_{km}$ with obvious notations), and each singularity is contained in 
three of these lines. If one projects them on a plane, the ten lines $L_{ij}$ form a Desargues configuration (see~\cite{Dolgachev:Salem, Dolgachev-Keum}).
\end{rem}

All this works for any choice of complex numbers $\ell_i\neq 0$. Now, since the $\ell_i$ are real, $X$ is endowed with two real structures.
First, one can consider the complex conjugation $c\colon [z_i]\mapsto [\overline{z_i}]$ on $\P^4(\C)$ and restrict it to $X$: this gives a first
antiholomorphic involution $c_X$. Another one is given by $s_X\colon [z_i]\mapsto [1/\overline{z_i}]$.
To be more precise, consider first, the quartic birational involution $J\in \Bir(\P^4_\C)$ defined by $J ([z_i]) =[1/z_i]$; 
$J$ preserves $X$, it determines a birational transformation $J_X\in \Bir(X)$, and on ${\hat{X}}$ it becomes an automorphism because every birational transformation of a K3 surface is regular. 
Thus, $s_X=J_X\circ c_X$ determines a second antiholomorphic involution $s_{{\hat{X}}}$ of ${\hat{X}}$. In what follows, we denote by $(X,s_X)$ this real structure (even if it would be better to study it on ${\hat{X}}$); 
its real part is the fixed point set of $s_X$, i.e. the set of points in $X(\C)$ with coordinates of modulus $1$: the real part does not contain any of the singularities of $X$, this is why we prefer to stay in $X$ rather than lift everything to ${\hat{X}}$. Thus, {\emph{with the 
real structure defined by $s_X$, the real part of $X$  coincides with $\mathrm{Pent}^0(\ell_0, \ldots , \ell_4)$}} if $(\ell_i)\in (\R_+^*)^5$.

\begin{rem}\label{rem:topology}
When $\ell_i>0$ for all indices $i\in \{0,\ldots, 4\}$,
a complete description of the possible homeomorphism types for  the real locus (in the smooth and singular cases) is given in 
\cite{Curtis-Steiner}: {\emph{in the smooth case, it is an  orientable surface of genus $g = 0, \ldots, 4$ or  the union of two tori}}. 
\end{rem} 

\begin{rem}\label{rem:pentagon_enriques}
The involution $J$ preserves $X$ and the two real structures $(X,c_X)$ and $(X,s_X)$. 
It lifts to a fixed point free involution ${\hat{J_X}}$ on $\hat{X}$, and $\hat{X}/{\hat{J_X}}$ is an Enriques surface. 
On pentagons, $J$ corresponds to the symmetry $(x,y)\in \R^2\mapsto (x,-y)$ that reverses orientation. Thus we see that 
the space of pentagons modulo affine isometries is an Enriques surface. 
When $X$  acquires an eleventh singularity which is fixed by $J_X$, then $\hat{X}/{\hat{J_X}}$ becomes a Coble surface: 
see~\cite[\S 5]{Dolgachev:Salem}
for nice explicit examples. This happens 
 for instance when all lengths are $1$, except one which is equal to  $2$ (this corresponds to $t=1/4$ in~\cite[\S 5.2]{Dolgachev:Salem}).
 \end{rem}

Finally, let us express the folding transformations in coordinates.
Given $i\neq j$ in $\set{0, \ldots, 4}$ (consecutive or not) we define an involution $(t_i, t_j)\mapsto (t_i', t_j')$ preserving the vector $\ell_i t_i+ \ell_j t_j$
by taking the  symmetric of  $t_i$ and $t_j$ 
with respect to the line directed by $\ell_i t_i+ \ell_j t_j$. In coordinates, 
$t'_k  = u / t_k$ for some $u$ of modulus 1, and equating   $\ell_i t_i+ \ell_j t_j =  \ell_i t'_i+ \ell_j t'_j$ one obtains 
\begin{equation}
(t_i',t_j') = \lrpar{\frac{u}{t_i}, \frac{u}{t_j}} \text{, with }  u  = \frac{\ell_i t_i + \ell_j t_j}{\ell_i t_i\inv + \ell_j t_j\inv}.
\end{equation}
Observe  that these computations also make sense when the $\ell_i$ are complex numbers, or when we
replace the $t_i$ by the complex numbers $z_i$.
This defines a birational involution $\sigma_{ij}: X\dasharrow X$, 
\begin{equation}
\sigma_{ij}[z_0: \ldots :z_4]=[z'_0:\ldots :z_4']
\end{equation}
with $z'_k=z_k$ if $k\neq i,j$, $z'_i=vz_j$, and $z_j'=vz_i$ with $v= (\ell_i z_i+\ell_jz_j)/(\ell_i z_j+\ell_j z_i)$. Again, since every birational self-map of a K3 surface is an automorphism, these involutions $\sigma_{ij}$ are elements of $\Aut({\hat{X}})$
that commute with the antiholomorphic involution $s_{{\hat{X}}}$; hence, they generate a subgroup of $\Aut({\hat{X}}; s_{{\hat{X}}})$.
Thus we have constructed a family of projective surfaces ${\hat{X}}$, depending on a parameter $\ell \in \P^4(\C)$, endowed with a 
group of automorphisms generated by involutions. Note that  this group can be elementary: for instance when the five lengths are all equal the group is finite
because in that case $(z_i', z_j') = (z_j, z_i)$. When $j=i+1$ modulo $5$, $\sigma_{ij}$ corresponds to the folding transformation described in the
introduction.

\begin{rem}\label{rem:sigma_geometric_pentagon}
Pick a singular point $q_{ij}$, and project $X$ from that point onto a plane, say the plane $\{z_i=0\}$ in the
hyperplane $P=\{\ell_0 z_0+ \cdots +\ell_4z_4=0\}$. One gets a $2$ to $1$ cover $X\to \P^2_\C$, ramified along a sextic curve (this curve is the union of two cubics, 
see~\cite{Rosenberg}). 
The involution $\sigma_{ij}$ permutes the points in the fibers of this $2$ to $1$ cover: if $x$ is 
a point of $X$, the line joining $q_{ij}$ and $x$ intersects $X$ in the third point $\sigma_{ij}(x)$.
The singularity $q_{ij}$ is an indeterminacy point, mapped by $\sigma_{ij}$ to the opposite line $L_{ij}$.
\end{rem}

\begin{pro}\label{pro:pentagons}
For a  general parameter  $\ell\in \P^4(\C)$:
\begin{enumerate}[\em (1)]
\item $X$ is a K3 surface with ten nodes, with two real structures $c_X$ and $s_X$ when $\ell\in \P^4(\R)$;
\item if $i$, $j=i+1$, $k=i+2$ are distinct consecutive indices (modulo $5$), then $\sigma_{ij}\circ\sigma_{jk}$ is a parabolic transformation on ${\hat{X}}$;
\item if $i$, $j$, $k$, and $l$ are four distinct indices (modulo $5$), then $\sigma_{ij}$ commutes to $\sigma_{kl}$.
\item the group $\Gamma$ generated by the   involutions  $\sigma_{ij}$  is a non-elementary 
subgroup of $\Aut({\hat{X}}; s_{\hat{X}})$ that does not preserve any algebraic curve. 
\end{enumerate}
\end{pro}

In~\cite{Dolgachev:Salem}, Dolgachev computes the action of $\sigma_{ij}$ on $\NS(\hat{X})$. This contains a proof of this 
proposition. He also describes, up to finite index, the Coxeter group generated by the $\sigma_{ij}$. The automorphism groups of 
$\hat{X}$ and of the Enriques surface ${\hat{X}}/{\hat{J_X}}$ are described in~\cite{Dolgachev-Keum} and \cite{Shimada}.

\begin{vlongue}
\begin{proof}
We already established   Assertion~(1) in the previous lemmas. For Assertion~(2), 
denote by $l,m$ the indices for which $\set{i,j,k,l,m} = \set{0, \ldots, 4}$, and consider the linear projection 
$\pi_{lm}\colon \P^5(\C)\dasharrow\P^1(\C)$ defined by 
$[z_0:\ldots :z_4]\mapsto [z_l:z_m]$. The fibers of $\pi_{lm}$  are the hyperplanes containing the plane
$\{z_l=z_m=0\}$, which intersects $X$ on the line $L_{l m}$. This line is a common component of the
pencil of curves cut out by the fibers of $\pi_{l m}$ on $X$, and the mobile part of this pencil determines 
a fibration $\pi_{l m\vert X}\colon X\to \P^1$ whose fibers are the plane cubics
\begin{equation}\label{eq:cubic_pilm}
(\ell_l z_l+\ell_m z_m)(\ell_m z_l+\ell_l z_m)z_i z_j z_k=z_l z_m (\ell_i z_jz_k+\ell_j z_iz_k+\ell_k z_iz_j)(\ell_i z_i+\ell_j z_j+\ell_k z_k),
\end{equation}
with $[z_l:z_m]$ fixed. The general member of this fibration is a smooth cubic, hence a curve of genus $1$. 

Then $\sigma_{ij}$ and $\sigma_{jk}$ preserve $\pi_{l m\vert X}$, and along the general fiber of 
$\pi_{l m\vert X}$ each of them is described by Remark~\ref{rem:sigma_geometric_pentagon}; for instance, $\sigma_{ij}(x)$
is the third  point of intersection of the cubic with the line $(q_{ij}, x)$. Thus, writing such a  cubic as $\C/\Lambda_{[z_l:z_m]}$, 
$\sigma_{ij}$ acts as $z\mapsto -z+b_{ij}$, for some $b_{ij}\in \C/\Lambda_{[z_l:z_m]}$ that depends on $[z_l:z_m]$ and the parameter $\ell$; it has four fixed points on the
cubic curve, which are the points of intersection of the cubic~\eqref{eq:cubic_pilm} with the hyperplanes $z_i=z_j$ and $z_i=-z_j$; equivalently, the
line $(q_{ij},x)$ is tangent to the cubic at these four points.

By Lemma~\ref{lem:charac_parabolic_on_K3}, either $\sigma_{ij}\circ\sigma_{jk}$
is of order $\leq 66$ (in fact of order $\leq 12$ because it preserves $\pi_{l m\vert X}$ fiber-wise), or it is parabolic. 
Due to this bound on the order, and the fact that there do exist pentagons for which $\sigma_{ij}\circ\sigma_{jk}$ is of infinite order
(indeed, this reduces to the corresponding fact for quadrilaterals, see the example below),
  $\sigma_{ij}\circ\sigma_{jk}$ is parabolic for general $\ell$.  
  
\begin{eg}
Take $\ell=1$ and $m=2$, and normalize our pentagons to assume that $t_0=1$, which means that the first vertices
are $a_0=(0,0)$ and $a_1=(\ell_0,0)$; in homogeneous coordinates this corresponds to the normalization 
$[1:z_1:z_2:z_3:z_4]$ with $z_i= t_i$. Now, the pentagon in a fiber of $\pi_{12\vert X}$ have three 
fixed vertices, namely $a_0$, $a_1$ and $a_2$. The remaining vertices $a_3$ and $a_4$ 
move on the circles centered at $a_2$ and $a_0$ and of respective radii $\ell_2$ and $\ell_4$, with the 
constraint $a_3a_4=\ell_3$. The circles are two conics, the fiber is a $2$ to $1$ cover of each of these two conics,
and  the automorphisms $\sigma_{23}$ and $\sigma_{34}$ preserve these fibers. Forgetting the vertex $a_1$, 
and looking at the quadrilateral $(a_0, a_2, a_3, a_4)$, one recovers the involutions described in \cite{benoist-hulin}.
The fixed points of $\sigma_{23}$ correspond to configurations with tangent circles, i.e. $a_3$ on the segment $[a_2,a_4]$.
\end{eg}

Assertion~(3) follows directly from the fact that $\sigma_{ij}$ changes the coordinates $z_i$ and $z_j$ but keeps the other three fixed. 

Finally, for a general parameter $\ell$, $\Gamma$ contains two such parabolics  associated to distinct fibrations 
$\pi_{lm}$ and $\pi_{l'm'}$ so it is non-elementary (see Lemma~\ref{lem:pairs_of_twists}). In addition $\Gamma$ does not preserve any curve 
in $\hat{X}$. Indeed, let $E\subset {\hat{X}}$ be a $\Gamma$-periodic irreducible curve, and denote by $F$ its image in 
$\P^4_\C$ under the projection ${\hat{X}}\to X$. If $F$ is  a point, it is one of the singularities $q_{ij}$, and changing $E$ into its image
under (the lift of) $\sigma_{ij}$ the curve $F$ becomes the line $L_{ij}$. So, we may assume that $F$ is an irreducible curve. 
Now, the orbit of $F$ is periodic under the action of the parabolic automorphisms $g_i=\sigma_{ij}\circ\sigma_{jk}$ with $k=j+1$ and $j=i+1$.
Since the invariant curves of a parabolic automorphisms are contained in the fibers of its invariant fibration, we deduce that $F$ is 
contained in the fibers of each of the projections $\pi_{lm}$; this is obviously impossible.\end{proof}
\end{vlongue}

\subsection{Enriques surfaces (see~\cite{Cossec-Dolgachev:book, Dolgachev:Kyoto})}\label{par:Enriques}
Enriques surfaces are quotients of K3 surfaces by fixed point free involutions. 
According to Horikawa and Kond\={o} (\cite{Horikawa:Enriques1,Horikawa:Enriques2, Kondo:1994}),
the moduli space ${\mathcal{M}}_E$ of complex Enriques surfaces is a rational quasi-projective variety of dimension 10.
An Enriques surface $X$ is nodal if it contains a smooth rational curve; such rational curves have self-intersection
$-2$, and are called nodal-curves or $(-2)$-curves. Nodal Enriques surfaces form a hypersurface in ${\mathcal{M}}_E$.

For any Enriques surface $X$, the lattice $(\NS(X;\Z), q_X)$ is isomorphic to the orthogonal direct sum $E_{10}=U \operp  E_8(-1)$, 
(\footnote{Here, $U$ is the standard $2$-dimensional Minkowski lattice, $(\Z^2, x_1x_2)$, and $E_8$ is the root lattice given by the 
corresponding Dynkin diagram; so $E_8(-1)$ is negative definite, and $E_{10}$ has signature $(1,9)$ (see \cite[Chap. II]{Cossec-Dolgachev:book}). Also, recall that in this  paper
$\NS(X;\Z)$ denotes the 
 torsion free part of the N\'eron-Severi group, which is sometimes  denoted by ${\mathsf{Num}}(X;\Z)$ in the literature on 
Enriques surfaces.}). 
Let $W_X\subset \O(\NS(X;\Z))$ be the subgroup generated by reflexions about classes $u$ such that  $u^2=-2$, 
 and   $W_X(2)$ be the subgroup of $W_X$  acting trivially on $\NS(X;\Z)$ modulo $2$. 
Both $W_X$ and $W_X(2)$ have  finite index in $\O(\NS(X;\Z))$. 
The following result is due independently to Nikulin and Barth and Peters (see \cite{Dolgachev:Kyoto} for details and references).

\begin{thm} 
If $X$ is an  Enriques surface which is not nodal, the homomorphism 
$ \Aut(X)\ni f \mapsto f^*\in \GL(H^2(X, \Z))$ is injective, 
and its image 
satisfies
$W_X(2) \subset \Aut(X)^* \subset W_X$.
\end{thm}

In particular, for any unnodal Enriques surface, $\Aut(X)$ is non-elementary, contains parabolic elements, and acts irreducibly on 
$\NS(X;\R)$; thus, it does
not preserve any curve. 

\subsection{Examples on rational surfaces: Coble and Blanc}\label{par:Coble-Blanc}

Closely related to Enriques surfaces are the examples of Coble, obtained by blowing up
the ten nodes of a general rational sextic curve $C_0\subset \P^2$. The result is a rational surface 
$X$ with a large group of automorphisms. To be precise, consider the canonical class 
$K_X\subset \NS(X;\Z)$; its orthogonal complement $K_X^\perp$ is a lattice of dimension $10$, isomorphic
to $E_{10}$, and we define $W_X(2)$ exactly in the same way as for Enriques surfaces. 
Then, $\Aut(X)^*$ preserves the decomposition $K_X\oplus K_X^\perp$, and $\Aut(X)^*$ contains $W_X(2)$ when $X$ does not contain any smooth rational curve of self-intersection $-2$ (see~\cite{Cantat-Dolgachev}, Theorem~3.5).
Also, Coble surfaces may be thought of as degeneracies 
of Enriques surfaces: an interesting difference is that $[K_X]$ is non trivial; in particular, $\NS(X;\Z)_0$ is always 
non-trivial, for any $\Gamma\subset \Aut(X)$. There is a holomorphic  section of $-2K_X$
vanishing exactly along the strict transform $C\subset X$ of the rational sextic curve $C_0$; this means that there is
a meromorphic section $\Omega_X=\xi(x,y) (dx\wedge dy)^2$ of $K_X^{\otimes 2}$ that does not vanish and has
a simple pole along $C$. Thus, the formula
\begin{equation}
\vol_X(U)=\int_U
 \abs{\xi(x,y)} dx\wedge dy\wedge d{\overline{x}}\wedge d{\overline{y}} =\int_U  \abs{\xi(x,y)} ({\mathsf{i}}dx\wedge  d{\overline{x}}) \wedge ({\mathsf{i}} dy\wedge d{\overline{y}})
\end{equation}
determines a finite measure(\footnote{if locally $C=\{ x=0\}$  then $\xi(x,y)=\eta(x,y)/x$ where $\eta$ is regular; thus, $\abs{\xi} =\abs{\eta} \abs{x}^{-1}$ is locally integrable because $\frac{1}{r^\alpha}$ is integrable with respect to $rdrd\theta$ when $\alpha<2$}) $\vol_X$ $={\text{``}}\,\Omega_X^{1/2}\wedge {\overline{\Omega_X^{1/2}}}\,{\text{''}}$, 
which we may assume to be a probability after multiplying $\Omega_X$ by some
adequate constant;
\begin{vcourte}
this measure is $\Aut(X)$-invariant (because $\vol_X$ is uniquely determined by the complex structure).
\end{vcourte}
\begin{vlongue}
this measure is $\Aut(X)$-invariant (because $\vol_X$ is uniquely determined by the complex structure; see also Remark~\ref{rem:pm1}  below).
\end{vlongue}

Another family of examples has been described by Blanc in \cite{Blanc:Michigan}. One starts with a smooth cubic curve 
$C_0\subset \P^2$. If $q_1$ is a point of $C_0$, there is a unique birational involution $s_1$ of $\P^2$ that fixes $C_0$ pointwise
and preserves the pencil of lines through $q_1$. The indeterminacy points of $s_1$ are $q_1$ and the four tangency 
points of $C_0$ with this pencil (one of them may be ``infinitely near $q_1$’’ and in that case it corresponds to the tangent direction of $C_0$ at $q_1$); thus the indeterminacies of $s_1$ are resolved by blowing-up points of $C_0$ (or points of its strict transform). After such a sequence of blow-ups $s_1$ 
becomes an automorphism of a rational surface $X_1$ that fixes pointwise the strict transform of $C_0$. So, if we blow-up other points of this
curve, $s_1$ lifts to an automorphism of the new surface. In particular, we can start with a finite number of points $q_i\in C_0$, $i=1, \ldots, k$, and resolve simultaneously the indeterminacies of the involutions $s_i$ determined by the $q_i$. The result is a surface $X$, 
with a subgroup $\Gamma:=\langle s_1, \ldots, s_k\rangle$ of $\Aut(X)$.  Blanc proves that (1) there are no relations between these involutions, 
that is,  $\Gamma$ is a free product 
$\langle s_1, \ldots, s_k\rangle \simeq \bigast_{i=1}^{k} \Z/2\Z$, (2) the composition of two distinct involutions $s_i\circ s_j$ is parabolic, and (3) the composition of three distinct involutions is loxodromic. 
There is a meromorphic section $\Omega_X$ of $K_X$ with a simple pole along the strict transform of $C_0$, but 
the form $\vol_X:=\Omega_X\wedge{\overline{\Omega_X}}$ is  not integrable.

\begin{vlongue}
\begin{rem}\label{rem:pm1}
If $\Gamma\subset \Aut(X)$ is generated by involutions and there is a meromorphic form $\Omega$ such that $f^*\Omega=\xi(f)\Omega$ 
for every $f\in \Gamma$, then $\xi(f)=\pm 1$: 
this is the case for Blanc's examples or general Coble surfaces, since $W_X(2)$ is also 
generated by involutions (see \cite{Dolgachev:Kyoto}).  
\end{rem}
\end{vlongue}

\subsection{Real forms}\label{par:Examples-Real-Forms} 
For  each of the examples described in Sections~\ref{par:Wehler} to 
  \ref{par:Coble-Blanc}, we may ask for 
the existence of an additional real structure on $X$, and look at the group of automorphisms $\Aut(X_\R)$ that preserve the
real structure (automorphisms commuting with the anti-holomorphic involution describing the real structure). 
Note that if $X$ is a smooth projective variety with a real structure, then $X(\R)$ is either empty or a compact, smooth, and totally real surface in $X$. 

If $X$ is a Wehler surface defined by a polynomial equation $P(x_1,x_2,x_3)$ with real coefficients the $\sigma_i$ 
are automatically defined over~$\R$. If $X$ is a Blanc surface for which $C_0$ is defined over $\R$ and the points $q_i$ 
are chosen in $C_0(\R)$, then again $\langle s_1, \ldots, s_k\rangle\subset \Aut(X_\R)$. Real Enriques and Coble surfaces
provide also many examples for which $\Aut(X_\R)$ is non-elementary (see~\cite{Degtyarev-Itenberg-Kharlamov:LNM}).

\subsection{Surfaces admitting non-elementary groups of automorphisms} \label{subs:X-is-projective}%
\begin{vcourte}
All surfaces in the previous examples are projective. This is a general fact, which we prove in this paragraph: 
we rely on the Kodaira-Enriques classification  to describe compact K\"ahler surfaces which support  a non-elementary group
of automorphisms and prove Theorem \ref{thm:X-is-projective2}.
\end{vcourte}
\begin{vlongue}
The surfaces in the previous examples are all projective. This is a general fact, which we prove in this paragraph: we rely on the Kodaira-Enriques classification  to describe compact K\"ahler surfaces which support  a non-elementary group
of automorphisms and prove Theorem \ref{thm:X-is-projective1}.
\end{vlongue}

\subsubsection{Minimal models}\label{par:minimal_models} 
We refer to Theorem~10.1 of  \cite{Cantat:Milnor} for the following result:
\begin{thm}\label{thm:existence_loxodromic}
If $X$ is a compact K\"ahler surface with a loxodromic automorphism, then
\begin{itemize}
\item either $X$ is a rational surface, and there is a birational morphism $\pi\colon X\to \P^2_\C$;
\item or the Kodaira dimension of $X$ is equal to $0$, and there is an $\Aut(X)$-equivariant 
bimeromorphic morphism $\pi\colon X\to X_0$ such that $X_0$ is a compact torus, a K3 surface, 
or an Enriques surface. 
\end{itemize} 
In particular, $h^{2,0}(X)$ equals $0$ or $1$. 
\end{thm}

\begin{rem}\label{rem:volume_form} If $X$ is a torus or K3 surface, there is a holomorphic $2$-form $\Omega_X$ on $X$ that does not
vanish and satisfies $\int_X\Omega_X\wedge {\overline{\Omega_X}} = 1$. It is unique up
to multiplication by a complex number of modulus $1$. A consequence of utmost importance to us is that  the volume form 
\begin{equation}
\Omega_X\wedge {\overline{\Omega_X}}
\end{equation}
is $\Aut(X)$-invariant.
Furthermore for every $f$ we can write  $f^*\Omega_X=J(f)\Omega_X$, where the Jacobian
$f\in \Aut(X)\mapsto J(f)\in {\mathbb{U}}_1$ is a unitary character on the group $\Aut(X)$. Since 
$H^{2,0}(X;\C)$ is generated by $[\Omega_X]$, we obtain
\begin{vcourte}
$f^*w=J(f)w \quad \forall w \in H^{2,0}(X;\C)$.
\end{vcourte}
\begin{vlongue}
\begin{equation}
f^*w=J(f)w \quad \forall w \in H^{2,0}(X;\C).
\end{equation} 
\end{vlongue}
 If $Y$ is 
an Enriques surface, and $X\to Y$ is its universal cover, then $X$ is a K3 surface: the volume
form $\Omega_X\wedge {\overline{\Omega_X}}$ is invariant under the group of deck transformations, 
and determines an $\Aut(Y)$-invariant volume form on $Y$. So, if $X$ is not rational, the dynamics of
$\Aut(X)$ is conservative: it preserves a {\bf{canonical volume form}} which is uniquely determined by the complex
structure of $X$. 
\end{rem}

It follows from Theorem~\ref{thm:existence_loxodromic}   that, in most   
cases, $\Aut(X)$ is countable (see \cite[Rmk 3.3]{Cantat:Milnor}).

\begin{pro}\label{prop:discrete}
Let $X$ be a compact K\"ahler surface.
If $\Aut(X)$ contains a loxodromic element, then the kernel of the homomorphism $\Aut(X)\to \Aut(X)^*\subset \GL(\NS(X;\Z))$ is 
finite unless $X$ is a torus.  So, if $\Aut(X)$ is non-elementary, then $\Aut(X)$ is discrete or $X$ is a torus.  
\end{pro}

\subsubsection{Projectivity} \label{subsub:X-is-projective}
\begin{mthm}\label{thm:X-is-projective2} 
Let $X$ be a compact K\"ahler surface and $\Gamma$ be a non-elementary 
subgroup of $\Aut(X)$.  Then $X$ is   projective, and is birationally equivalent to a rational surface, an Abelian surface, 
a K3 surface, or an Enriques surface. 
\end{mthm}

From the discussion in \S\S \ref{par:Wehler}--\ref{par:Coble-Blanc} we see that 
there exist examples with a non-elementary group of automorphisms for each of these four classes
of surfaces. Theorem~\ref{thm:X-is-projective2} is a direct consequence of Theorem \ref{thm:existence_loxodromic}
and the following lemmas. 

\begin{lem}\label{lem:proj_root} Let $f$ be a loxodromic automorphism of a compact K\"ahler surface $X$. 
The following properties are equivalent:
\begin{enumerate}
\item on $H^{2,0}(X;\C)$, $f^*$ acts by multiplication by a root of unity;
\item $X$ is projective.
\end{enumerate}
\end{lem}

\begin{vlongue}
If $X$ supports a loxodromic automorphism, then $\dim(H^{2,0}(X;\C))\leq 1$; and 
with   notation as in  Remark~\ref{rem:volume_form}, the first assertion is equivalent to 
\begin{enumerate}
\item[{\em (1')}] {\emph{either $H^{2,0}(X;\C)=0$ or $J(f)$ is a root of unity}}. 
\end{enumerate}
\end{vlongue}

\begin{proof}[Proof of Lemma~\ref{lem:proj_root}]
The characteristic polynomial $\chi_f$ of $f^*\colon H^2(X;\Z)\to H^2(X;\Z)$ is  a monic polynomial with 
integer coefficients. Since $f$ is loxodromic, $f^*$ has a real eigenvalue $\lambda(f)>1$. 
Besides $\lambda(f)$ and $\lambda(f)\inv$, all other roots 
 of  $\chi_f$ have modulus $1$, so 
$\lambda(f)$ is a reciprocal quadratic integer or a Salem number
 (see \S~2.4.3 of~\cite{Cantat:Milnor} for more details).  Thus, the decomposition of $\chi_f$ into irreducible factors can be written as
\begin{equation}
\chi_f(t)=S_f(t) \times R_f(t) = S_f(t) \times \prod_{i=1}^m C_{f,i}(t) 
\end{equation}
where $S_f$ is a Salem polynomial or a reciprocal quadratic polynomial, and 
the $C_{f,i}$ are cyclotomic polynomials. 
In particular if $\xi$ is an eigenvalue  of $f^*$ and a root of unity, we see that 
 $\xi$ is a root of $R_f(t)$ but not of $S_f(t)$.
 
The subspace $H^{2,0}(\C)\subset H^2(X;\C)$ is $f^*$-invariant and, by Lemma~\ref{lem:unitary_on _H20}, all eigenvalues of $f^*$ on that subspace 
have modulus $1$; if an eigenvalue of $f^*\rest{H^{2,0}(X;\C)}$ is not a root of unity, then it is a root of $S_f$.

Assume that all eigenvalues of $f^*$ on $H^{2,0}(X;\C)$ are roots of unity. 
Then $\Ker(S_f(f^*))\subset H^2(X;\R)$ is a $f^*$-invariant 
subspace of $H^{1,1}(X;\R)$. This subspace is defined over $\Q$ and
is of Minkowski type; in particular, it contains integral classes of
positive self-intersection, and by the  Kodaira embedding
theorem, $X$ is projective.  Conversely, assume that $X$ is projective. The N\'eron-Severi group $\NS(X;\Q)\subset H^{1,1}(X;\R)$ is $f^*$-invariant and 
contains vectors of positive self-intersection, so by Proposition \ref{pro:invariant_cohomological_decomposition} it contains all 
isotropic lines associated to loxodromic automorphisms. 
Now any $f^*$ invariant subspace defined over $\Q$ and containing the eigenspace associated to 
$\lambda(f^*)$  contains $\Ker(S_f(f^*))$, so we deduce that  $\Ker(S_f(f^*))\subset \NS(X;\Q)$.  In particular, $\Ker(S_f(f^*))$ does not intersect $H^{2,0}(X;\C)$, which is invariant, 
 and we conclude that all eigenvalues 
 of $f^*$ on  $H^{2,0}(X;\C)$ are roots of unity. 
\end{proof}

\begin{lem}\label{lem:virtually_cyclic_non_projective}
Let $X$ be a  compact K\"ahler surface. If $X$ is not projective, then $\Aut(X)^*$ is virtually Abelian and if it contains a loxodromic element it is virtually cyclic.
\end{lem}

\begin{proof} 
Assume that $\Aut(X)^*$ is not virtually Abelian, or that it contains a loxodromic element without being virtually cyclic. 
According to Theorem~3.2 of \cite{Cantat:Milnor}, $\Aut(X)^*$  contains a non-Abelian free group $\Gamma$ such that all 
elements of $\Gamma\setminus \{\id\}$ are loxodromic; 
from Theorem~\ref{thm:existence_loxodromic},  either $h^{2,0}(X)=0$ or $X$ 
is the blow-up of a torus or a K3 surface. 
In the first case, $H^2(X;\R)=H^{1,1}(X;\R)$ so, by the Hodge index theorem, $H^{1,1}(X;\R)$ contains an integral class with positive self-intersection; then, the Kodaira 
embedding theorem shows that $X$ is projective.  
In the second case, by uniqueness of the minimal model, 
the morphism $X\to X_0$ onto the minimal model of $X$ is $\Aut(X)$-equivariant, so we can assume
that $X=X_0$ is minimal and  $h^{2,0}(X)=1$. 
Consider the homomorphism $J\colon \Aut(X)\to {\mathbb{U}}_1$, as in Remark~\ref{rem:volume_form}. 
Since ${\mathbb{U}}_1$ is Abelian  $\ker(J\rest{\Gamma})$ contains loxodromic elements: indeed if $f,g\in \Gamma$ 
and $f\neq g$ then 
$[f,g] = f g f\inv g\inv$ is loxodromic and  $J([f,g]) =1$. From 
Lemma~\ref{lem:proj_root} we deduce that $X$ is projective. 
\end{proof}

\section{Glossary of random dynamics, I}\label{sec:Glossary_I}

We now initiate the random iteration by introducing a probability measure on $\Aut(X)$. In this section we introduce a first set 
of   ideas from   the theory of  random dynamical systems, as well as some notation that will be used throughout the paper. 

\subsection{Random holomorphic dynamical systems}\label{par:Random_holomorphic_dynamical_systems}
Let  $X$ be a compact K\"ahler surface,  such that $\Aut(X)$ is non-elementary. 
Note that $\Aut(X)$ is locally compact for the topology of uniform convergence --in many interesting cases it is actually discrete (see Proposition \ref{prop:discrete})-- 
so it admits a natural Borel structure. We fix some Riemannian structure on $X$, for instance the one induced by the  K\"ahler form
$\kappa_0$. 
For $f\in \Aut(X)$, we denote by $\norm{ f}_{{{C}}^1}$ the maximum of $\norm{ Df_x}$ where the norm of 
$Df_x\colon T_xM\to T_{f(x)}M$ is computed with respect to this Riemannian metric. 

We consider  a probability measure $\nu$ on $\Aut(X)$  
satisfying the \textbf{moment condition} (or integrability condition)
\begin{equation}\label{eq:moment}
\int \lrpar{\log\norm{f}_{C^1(X)}+ \log\norm{f\inv}_{C^1(X)}} \, d\nu (f) < + \infty. 
\end{equation}
\begin{vcourte}
The finiteness of the integral in \eqref{eq:moment} does not depend on our choice of Riemannian metric. 
When the support of  $\nu$ is finite, the integrability~\eqref{eq:moment}, 
as well as stronger moment conditions which will appear later (see Conditions~\eqref{eq:exponential_moment} and~\eqref{eq:exponential_moment_cohomology}), are obviously satisfied. 
\end{vcourte}
\begin{vlongue}
The norm 
$\norm{\,\cdot\,}_{C^1(X)}$ is relative to our choice of Riemannian metric, 
but the finiteness of the integral in \eqref{eq:moment} does not depend on this choice. 
In many interesting situations the support of  $\nu$ will be finite, in which case the integrability~\eqref{eq:moment}, 
as well as stronger moment conditions which will appear later (see Conditions~\eqref{eq:exponential_moment} and~\eqref{eq:exponential_moment_cohomology}), are obviously satisfied. 
\end{vlongue}

\begin{vcourte}
\begin{lem}\label{lem:momentC1_to_Ck} The measure $\nu$ satisfies the moment condition \eqref{eq:moment} if and only if it satisfies
the higher moment conditions
\[
\int \lrpar{\log\norm{f}_{C^k(X)}+ \log\norm{f\inv}_{C^k(X)}} \, d\nu (f) <\infty,
\]
for all $k\geq 1$.
\end{lem}
\end{vcourte}
\begin{vlongue}
\begin{lem}\label{lem:momentC1_to_Ck}  The measure $\nu$ satisfies the moment condition \eqref{eq:moment} if and only if it satisfies
the higher moment conditions
\begin{equation}\label{eq:moment_Ck}
\int \lrpar{\log\norm{f}_{C^k(X)}+ \log\norm{f\inv}_{C^k(X)}} \, d\nu (f) <\infty,
\end{equation}
for all $k\geq 1$.
\end{lem}

Here the $C^k$ norm is relative to the expression of $f$ in a system of charts (we don't need to be precise here because only the finiteness in \eqref{eq:moment_Ck} matters). 
\end{vlongue}
This lemma follows from the Cauchy estimates.  In particular, if $\nu$ satisfies~\eqref{eq:moment}, then it satisfies a similar moment condition for the $C^2$ norm, a property required to apply Pesin's theory. 

Given $\nu$, we shall consider independent, identically distributed sequences $(f_n)_{n\geq 0}$
of random automorphisms of $X$ with distribution $\nu$, 
and study the dynamics of random compositions of the form $f_{n-1}\!\circ \cdots \circ \! f_0$. 
The data $(X, \nu)$ will 
be referred to as a {\bf{random holomorphic dynamical system}} on $X$. 
Many properties of  $(X, \nu)$
  depend on the properties of the subgroup
  \begin{equation}
  \Gamma = \Gamma_\nu:=
 \langle \supp(\nu)\rangle
 \end{equation} 
 generated by (the support of) $\nu$ in $\Aut(X)$. 
 If in addition $\Gamma_\nu$ is non-elementary, we say that 
$(X, \nu)$ is \textbf{non-elementary}.
 
\subsection{Invariant and stationary measures}
Let $G$ be a topological group and $\nu$ be a probability measure on $G$.
Consider a measurable action of $G$ on some measurable space $(M, \mathcal A)$. Every $f\in G$ determines 
a push-forward operator $\mu\mapsto f_\varstar\mu$, acting on positive (resp. probability) 
measures $\mu$ on $(M, \mathcal A)$. 
By definition, a probability measure $\mu$ on $(M, \mathcal A)$ is  {\bf{$\nu$-stationary}} if
\begin{equation} 
\int f_\varstar\mu \,  d\nu(f) = \mu,
\end{equation} 
and it is \textbf{$\nu$-almost surely invariant} if $f_\varstar\mu = \mu$ for $\nu$-almost every $f$. 
Let us stress that we only deal with probability measures in this definition;
slightly abusing terminology, most often we drop the mention to $\nu$ and the mention
that $\mu$ is a probability. 
A stationary  measure is  \textbf{ergodic} 
if it is an extremal point of the closed convex set of stationary measures (see~\cite[\S 2.1.3]{benoist-quint_book}). 

If $\mu$ is almost surely  invariant then it is stationary but the converse is generally false. 
If $M$ is compact, the action $G\times M\to M$ is continuous, and $\sA$ is the Borel $\sigma$-algebra, the Kakutani fixed point theorem implies 
the existence of at least one stationary measure. 
\begin{vcourte}
On the other hand the existence of an invariant measure is a very restrictive property (see Sections~\ref{par:intro_stiffness} and \ref{par:furstenberg_measure}). 
\end{vcourte}
\begin{vlongue}
On the other hand the existence of an invariant measure is a very restrictive property. 
For instance, proximal, strongly irreducible linear actions on 
projective spaces have no (almost surely) invariant probability measure (see Sections~\ref{par:intro_stiffness} and \ref{par:furstenberg_measure}). 
\end{vlongue}
Following Furstenberg \cite{furstenberg_stiffness} we say that 
an action is \textbf{stiff} (or $\nu$-stiff) if any $\nu$-stationary measure is $\nu$-almost surely invariant. 

We shall consider several measurable actions of  $\Aut(X)$:  its tautological 
 action  on $X$, but also its action on the projectivized tangent bundle $\P(TX)$, on cohomology groups of $X$ and their 
 projectivizations, on spaces of currents, etc. In all cases, $M$ will be a locally compact space and $\mathcal A$ its  Borel $
 \sigma$-algebra, which will be denoted by $\mathcal B(M)$.  

\begin{rem}\label{rem:discrete} 
Since $X$ is compact and the action $\Aut(X)\times X\to X$ is continuous,  a probability measure $\mu$ on $(X,\mathcal B(X))$  is $\nu$-almost surely invariant if and only if it is invariant under the action of the closure of $\Gamma_\nu$ in $\Aut(X)$; this follows from the dominated convergence theorem. 
\end{rem}


\subsection{Random compositions}\label{par:random_compositions1}
Set $\Omega = \Aut(X)^\N$, endowed with its product topology.
The associated Borel $\sigma$-algebra coincides with the product $\sigma$-algebra, 
and it is generated by cylinders 
(see \S~\ref{par:definition_skew_products}).  
We endow $\Omega$ with the product measure $\nu^\N$. Choosing  a random element in $\Omega$ with respect to $\nu^\N$ 
is equivalent to  choosing an independent and identically distributed random sequence of automorphisms in $\Aut(X)$ with   distribution 
$\nu$. For $\omega\in \Omega$, we let $f_\omega = f_0$ and denote by $f^n_\omega$   
the left composition of the $n$ first terms of $\omega$, that is 
\begin{equation}
f^n_\omega= f_{n-1}\circ \cdots \circ f_0
\end{equation}
for $n>0$.
By definition $f^0_\omega  = \mathrm{id}$.
Let us record for future reference the following consequence  of the Borel-Cantelli lemma. 
\begin{vcourte}
We denote by $\sigma\colon \Omega\to \Omega$ the unilateral shift.
\end{vcourte}
\begin{vlongue}
We denote by $\sigma\colon \Omega\to \Omega$ the unilateral shift, i.e. the continuous transformation 
defined by $\sigma(f_0, f_1, \ldots)=\sigma(f_1, f_2, \ldots)$.
\end{vlongue}
\begin{lem}\label{lem:borel-cantelli_moment}
If $(X, \nu)$ is a random dynamical system satisfying the moment condition \eqref{eq:moment}, then 
for $\nu^\N$-almost every sequence $\omega=(f_n)\in \Omega$,
\[
 \unsur{n} \lrpar{\log{\norm{f_n}_{ {C}^1}} +  \log{\norm{f_n\inv}_{ {C}^1}}}\tendvers 0. 
\]
\end{lem}

\begin{vlongue}
\begin{rem}\label{rem:general_RDS}
We are \textbf{not} considering the most general version of random holomorphic 
dynamical systems: one might  consider compositions  $f_{\vartheta^{n-1}(\xi)} \circ \cdots \circ f_{\vartheta(\xi)}\circ f_{\xi}$ where 
$\vartheta:\Sigma\to \Sigma$ is some  measure preserving transformation of a probability space and 
$\Sigma\ni\xi\mapsto f_\xi\in \Aut(X)$ is  measurable. The methods developed below do not apply to this more general setting.
\end{rem}
\end{vlongue}

\section{Furstenberg theory in $H^{1,1}(X;\R)$}\label{sec:furstenberg}
  
Consider  a non-elementary  random holomorphic dynamical system $(X,\nu)$ on a compact K\"ahler surface, satisfying the 
moment condition \eqref{eq:moment}. The main purpose of this section is to analyze the linear action 
of $(X, \nu)$ on $H^{1,1}(X, \R)$ by way of the   theory of random products of matrices.  Basic references for this
subject are the books by Bougerol and Lacroix~\cite{bougerol-lacroix} and by Benoist and Quint~\cite{benoist-quint_book}.
  
\subsection{Moments and cohomology}\label{subs:moments_cohomology}
\begin{vlongue}
We start with a general discussion on the dilatation of cohomology classes under smooth transformations. 
\end{vlongue}
Let $M$ be a compact connected manifold of dimension $m$, endowed with some 
Riemannian metric $\g$. If $f\colon M\to M$ is a smooth map, 
$\norm{f}_{C^1}$ denotes
the maximum norm of its tangent action, 
computed with respect to $\g$ (see Section~\ref{par:Random_holomorphic_dynamical_systems}). 
Thus, $f$ is a Lipschitz map with $\Lip(f)=\norm{f}_{C^1}$
for the distance determined by $\g$; in particular $\norm{f}_{C^1}\geq 1$ whenever $f$ is onto.
Fix a norm $\abs{ \cdot }_{H^k}$ on each cohomology group $H^k(M;\R)$, for 
$0\leq k\leq m$. 

\begin{lem} There is a constant $C>0$, that depends only on $M$, $\g$, and
the norms $\abs{  \cdot }_{H^k}$, such that 
$
\abs{ f^*[\alpha]}_{H^k} \leq C^k \Lip(f)^k \abs{ [\alpha]}_{H^k}
$ 
for every class $[\alpha]\in H^k(M;\R)$ and every map $f\colon M\to M$ of class $C^1$. 
In other
words, the operator norm $\norm{ f^*}_{H^k}$ is controlled by the Lipschitz constant:
\[ 
\norm{ f^*}_{H^k}\leq C^k \Lip(f)^k \leq  C^k \norm{ f}_{C^1}^k.
\]
\end{lem}
 
\begin{proof}  
Pick a basis of the homology group $H_k(M;\R)\simeq H^k(M;\R)^*$ given by smoothly immersed, compact, $k$-dimensional manifolds $\iota_i\colon N_i\to M$, 
and a basis of $H^k(M;\R)$ given by smooth $k$-forms $\alpha_j$. Then, the integral $\int_{N_i} \iota_i^*(f^*\alpha_j)$ is bounded from 
above by $C^k\norm{f}_{C^1}^k$ for some constant $C$, because 
\begin{equation}
\vert (f^*\alpha_j)_x(v_1, \ldots, v_k)\vert = \vert \alpha_j(f_*v_1, \ldots, f_* v_k) \vert \leq c_j \norm{f}_{C^1}^k \prod_{\ell=1}^k\abs{v_\ell}_{\g}
\end{equation}
for every point $x\in M$ and every $k$-tuple of tangent vectors 
$v_\ell\in T_xM$; here, $c_j$ is the supremum of the norm of the multilinear map $(\alpha_j)_x$ over $x\in M$. 
\end{proof} 
 
If $\nu$ is a probability measure on  $\Diff(M)$ satisfying the moment condition \eqref{eq:moment}, then   
\begin{equation} \label{eq:moment_Hk}
\forall 1\leq k \leq m, \quad \int_{\Diff(M)}   \log\lrpar{\norm{ f^*}_{H^k}} +\log \lrpar{\norm{(f^{-1})^*}_{H^k}} \, d\nu(f)<+\infty.
\end{equation}
If we specialize this to automorphisms of compact K\"ahler surfaces
we get 
\begin{equation} \label{eq:moment_H1_both_sides}
\int_{\Aut(X)}  \log\lrpar{\norm{ f^*}_{H^{1, 1}}} +\log \lrpar{\norm{(f^{-1})^*}_{H^{1, 1}}}  \; d\nu(f)<+\infty, 
\end{equation}
which is actually equivalent to~\eqref{eq:moment_Hk} by Lemma \ref{lem:cohomological_norm_estimates}.
We saw in \S\ref{par:hyp_X} that $\norm{f^*}_{H^{1, 1}} \asymp \norm{(f\inv)^*}_{H^{1, 1}}$, 
so this last condition is in turn  equivalent to 
\begin{equation} \label{eq:moment_H1_final}
\int_{\Aut(X)}  \log\lrpar{\norm{ f^*}_{H^{1, 1}}}  \; d\nu(f)<+\infty.
\end{equation}

\subsection{Cohomological Lyapunov exponent} \label{subs:furstenberg}
From now on we denote by $\abs{\cdot}$ a norm on $H^{1,1}(X, \R)$ and by $\norm{\cdot}$ the associated operator norm. 
The linear action induced by the random dynamical system $(X, \nu)$ on $H^{1,1}(X, \R)$ defines a random product of matrices.  
Since   the moment condition~\eqref{eq:moment_H1_final} is satisfied, we can define  the 
\textbf{upper Lyapunov exponent}
$\lambda_{H^{1,1}}$ (or $\lambda_{H^{1,1}}(\nu)$)  by  
\begin{align}\label{eq:def_LE}
\lambda_{H^{1,1}} &= \lim_{n\to +\infty} \frac{1}{n}\int \log(\norm{ (f^n_\omega)^*}) d\nu^\N(\omega)\\
& =  \lim_{n\to +\infty} \frac{1}{n}\log\norm{ (f^n_\omega)^*}\label{eq:def_LE2}
\end{align}
where the second equality holds almost surely, i.e. for $\nu^\N$-almost every $\omega\in \Omega$.
This  convergence  follows from Kingman's subadditive ergodic theorem, since 
$\norm{\cdot}$ being an operator norm, $(\omega, n)\mapsto \log(\norm{ (f^n_\omega)^*})$ defines a subadditive cocycle (see~\cite[Thm 4.28]{benoist-quint_book} or~\cite[Thm I.4.1]{bougerol-lacroix}). 
Note that $(f_\omega^n)^* = f_{0}^*\circ \cdots \circ f_{n-1}^*$, so we are dealing 
with right compositions instead of the usual 
left composition. However since $f_{0}^*\circ \cdots \circ f_{n-1}^*$ has the same 
distribution as $f_{n-1}^*\circ \cdots \circ f_{0}^*$, the Lyapunov exponent in \eqref{eq:def_LE} corresponds to the usual definition of the upper Lyapunov exponent of the  random product of matrices. We refer to~\cite{bougerol-lacroix, Ledrappier:SaintFlour}
for the definition and main properties of the subsequent Lyapunov exponents (see also~\cite[\S 10.5]{benoist-quint_book}).

\begin{pro}
Let  $(X, \nu)$ be  a non-elementary holomorphic dynamical system on a compact K\"ahler surface, satisfying 
the moment condition~\eqref{eq:moment}, or more generally~\eqref{eq:moment_H1_final}. 
Then the cohomological   Lyapunov exponent $\lambda_{H^{1,1}}$ is positive and
the other Lyapunov exponents of  the linear  
  action  on $H^{1,1}(X, \R)$ are $-\lambda_{H^{1,1}}$, with multiplicity $1$, and $0$, with multiplicity $h^{1,1}(X)-2$. 
\end{pro}

\begin{proof}
 Consider the $\Gamma_\nu$-invariant decomposition 
$\Pi_{\Gamma_\nu}\oplus \Pi_{\Gamma_\nu}^\perp$ given by Proposition \ref{pro:NS_H11} and Equation~\eqref{eq:piGamma_nu-def}.
Since  the intersection form is negative definite on $\Pi_{\Gamma_\nu}^\perp$, the group $\Gamma_\nu^*\rest{\Pi_{\Gamma_\nu}^\perp}$ is bounded and 
all Lyapunov exponents of $\Gamma_\nu^*\rest{\Pi_{\Gamma_\nu}^\perp}$ vanish.
The 
linear action of $\Gamma_\nu$ on $\Pi_{\Gamma_\nu}$ is strongly irreducible and non-elementary, hence not relatively compact. 
Therefore Furstenberg's theorem asserts that 
$\lambda_{H^{1,1}}>0$ (see e.g.~\cite[Thm III.6.3]{bougerol-lacroix} or~\cite[Cor 4.32]{benoist-quint_book}), and 
the remaining properties of the Lyapunov spectrum on $\Pi_{\Gamma_\nu}$  follow from
 the KAK decomposition in $\O^+_{1,m}(\R)$, with $1+m=\dim(\Pi_{\Gamma_\nu})$ (see Lemma~\ref{lem:KAK}). \end{proof}
 
\begin{lem}\label{lem:growth}
If $a\in H^{1,1}(X;\R)$ satisfies $a^2>0$, for instance if $a$ is a K\"ahler class,  then
\[
\lim_{n\to +\infty}\frac{1}{n}\log \abs{ (f_\omega^n)^*a} =\lambda_{H^{1,1}}
\]
for $\nu^\N$-almost every $\omega$.
\end{lem}

\begin{proof}
Corollary \ref{cor:KAK} implies that if $a\in \Hyp_X$ then for every $f\in \Aut(X)$, 
$\abs{f^*a}\asymp \norm{f^*}$, where the implied constants depend only on $a$. Thus the result follows from Equation~\eqref{eq:def_LE2}. 
\end{proof}

\begin{rem}\label{rem:order_composition}
It is natural to expect that  Lemma \ref{lem:growth} holds for any $a\in \Pi_\Gamma\setminus \set{0}$; 
this is true under the more stringent moment assumption \eqref{eq:exponential_moment}  (see the proof 
of Proposition~\ref{pro:gouezel} below).
\end{rem}

If the order of compositions is reversed (which is less  natural from the point of view of iterated pull-backs), then Lemma \ref{lem:growth} indeed holds for any $a$ in $\Pi_{\Gamma_\nu}$ (see~\cite[Cor. III.3.4.i]{bougerol-lacroix}):

\begin{lem}\label{lem:growth2}
For any $a\in \Pi_{\Gamma_\nu}\setminus\set{0}$ and for $\nu^\N$-almost every $\omega=(f_n)_{n\geq 0} \in \Omega$ we have
$$
\lim_{n\to +\infty}\frac{1}{n}\log \abs{  f_n^*\cdots f_1^* a} =\lambda_{H^{1,1}}. 
$$
\end{lem}

\subsection{The measure $\mu_\partial$}\label{par:furstenberg_measure} 

By Furstenberg's theory the linear projective action of the random  dynamical system
$(X, \nu)$ on $\P\Pi_{\Gamma_\nu} \subset \P H^{1,1}(X;\R)$ admits a unique stationary measure $\mu_{\P\Pi_{\Gamma_\nu}}$; this measure does
not charge any proper  projective subspace of $\P\Pi_{\Gamma_\nu}$.     
Recall that the mass of a class $a$ is defined by $\M(a)=\langle a \vert [\kappa_0]\rangle$ (see \S~\ref{par:cones_definition}). 

\begin{lem}\label{lem:definition_e}
For $\nu^\N$-almost every $\omega$, there exists a unique nef class $e(\omega)$  such that $\M(e(\omega))=1$ and
\begin{equation}\label{eq:definition_e}
 \frac{1}{\M((f_\omega^n)^*a)} (f_\omega^n)^*a \underset{n\to\infty}\longrightarrow e(\omega)
\end{equation}
for any  pseudo-effective 
class $a$ with $a^2>0$ (in particular for any K\"ahler class). 
In addition, the class $e(\omega)$ is almost surely isotropic and  $\P(e(\omega))$ is a 
point  of the limit set $\Lim(\Gamma_\nu)\subset\partial\Hyp_X$. 
\end{lem}

Before starting the proof, note that $\Gamma_\nu^*\rest{\Pi_{\Gamma_\nu}}$ is proximal, in the sense of~\cite[\S 4.1]{benoist-quint_book}; 
equivalently, $\Gamma_\nu^*\rest{\Pi_{\Gamma_\nu}}$ is contracting, in the sense of~\cite[Def III.1.3]{bougerol-lacroix}.
In other words, there are sequences of elements $g_n\in \Gamma_\nu$ such that $\norm{g_n^*}^{-1}g_n^*\rest{\Pi_{\Gamma_\nu}}$ converges to a matrix of
rank $1$: for instance 
one can take $g_n=f^n$, where $f\in \Gamma_\nu$ is any loxodromic automorphism.

\begin{proof}
For $f\in \Aut(X)$,   we use 
the   notation   $\underline f^*$ for its action on    $\P H^{1,1}(X;\R) $. 
Since the action of $\Gamma_\nu$ on $\Pi_{\Gamma_\nu}$ is strongly irreducible and proximal, 
its projective action satisfies the following contraction property (see~\cite[Thm III.3.1]{bougerol-lacroix}): 
there is a measurable map $\omega\in \Omega\mapsto  \underline{e}(\omega)\in \P\Pi_{\Gamma_\nu}$ such that 
for almost every $\omega$, any cluster value $L(\omega)$ of 
\begin{equation}
\unsur{\norm {f_0^*\cdots   f_n^* }}f_0^*\cdots   f_n^*
\end{equation}
in $\End(\Pi_{\Gamma_\nu})$ is an endomorphism of rank $1$ whose range is equal to $\R\underline{e}(\omega)$.
 
Let $e(\omega)$ be the unique vector of mass $1$ in the line $\R\underline{e}(\omega)$. If $a\in \Pi_{\Gamma_\nu}$ satisfies $a^2>0$ and $\M(a)>0$, 
then any cluster value of $\M((f_\omega^n)^*a)^{-1}(f_\omega^n)^*a$ 
must coincide with $e(\omega)$ because by Corollary \ref{cor:KAK}
the mass $\M((f_\omega^n)^*a)$ is comparable to the norm $\norm {f_0^*\cdots   f_n^* }$. Thus, the convergence \eqref{eq:definition_e} is
satisfied. Furthermore  $e(\omega)$ is nef, because we can apply this convergence to a nef class $a$ and $\Aut(X)$ preserves the nef cone. Also,  $e(\omega)$ belongs to  $\Lim(\Gamma_\nu)$, hence it is isotropic.  
Now, let $a$ and $a'$ be two classes of $\Hyp_X$ with $a\in \Pi_{\Gamma_\nu}$. Since the hyperbolic distance between 
$ (f_\omega^n)^*(a)$ and $(f_\omega^n)^*(a')$ remains constant and  the convergence 
\eqref{eq:definition_e} holds for~$a$, it also holds for $a'$.
This concludes the proof, for every class with positive self-intersection is proportional to a unique class in $\Hyp_X$. 
\end{proof}
\begin{vlongue}
\begin{rem}
As in Remark \ref{rem:order_composition}, under the exponential moment condition \eqref{eq:exponential_moment},
the convergence in Equation~\eqref{eq:definition_e} holds for any $a\in \Pi_\Gamma\setminus \set{0}$ and almost every $\omega\in \Omega$; 
to be precise, $\frac{1}{\M((f_\omega^n)^*a)} (f_\omega^n)^*a$ converges towards $e(\omega)$ or its opposite.
Then, we actually get the convergence for any $a\in H^{1,1}(X;\R)\setminus \Pi_\Gamma^\perp$
(write $a = a_++a_0  $  and use  that   $\Gamma_\nu$ acts by isometries on $\Pi_\Gamma^\perp$)
\end{rem}
\end{vlongue}

Here is a summary of the properties of  the stationary measure $\mu_{\P\Pi_{\Gamma_\nu}}$; from now on, we view it as a measure on 
$\P H^{1,1}(X;\R)$ and rename it as $\mu_\fr$ because it is supported on $\fr \Hyp_X$.

\begin{thm}\label{thm:def_stationary} The probability  measure defined on  $\P H^{1,1}(X;\R)$ by
\begin{equation}\label{eq:representation_mufr} 
\mu_{\partial}=\int \delta_{\P(e(\omega))}\, d\nu^\N(\omega)
\end{equation}
is  $\nu$-stationary and  ergodic. 
It is the unique stationary measure on $\P H^{1,1}(X;\R) $ such that  
$\mu_\partial(\P (\Pi_{\Gamma_\nu}^\perp))=0$. The measure $\mu_\partial$ has no atoms  and is supported on   $\Lim(\Gamma_\nu)$; 
in particular, if  $\Lambda'\subset \Lim(\Gamma_\nu)$ is such that $\mu_\partial(\Lambda')>0$ then $\Lambda'$ is 
uncountable. 

The top Lyapunov exponent satisfies  the so-called Furstenberg formula:  
\begin{align}
\label{eq:Furstenberg_formula1}\lambda_{H^{1,1}} & =  \int \log\left( \frac{\abs{ f^*{\tilde{u}}}}{\abs{ {\tilde{u}}}}\right)\, d\nu(f)\, d\mu_\partial(u), 
\end{align}
where $\tilde u\in H^{1,1}(X, \R)\setminus\set{0}$ denotes any lift of $u\in \Lim(\Gamma_\nu) \subset \P H^{1,1}(X, \R)$. 
\end{thm}

\begin{proof}
The  ergodicity of $\mu_\fr=\mu_{\P\Pi_{\Gamma_\nu}}$ as well as its representation \eqref{eq:representation_mufr} 
follow from the properties of the action 
of $\Gamma_\nu$ on $\P (\Pi_\Gamma)$ (see \cite[Chap. III]{bougerol-lacroix}). Also, we know that $\lambda_{H^{1,1}}$ is equal 
to the top Lyapunov exponent of the restriction of the action to   $\P (\Pi_{\Gamma_\nu})$, so the 
formula  \eqref{eq:Furstenberg_formula1} follows from   the strongly irreducible case (see~\cite[Cor III.3.4]{bougerol-lacroix}). 

Let now $\mu$ be a stationary measure on $\P H^{1,1}(X;\R)$ such that $\mu(\P \Pi_{\Gamma_\nu}^\perp)=0$.  
A martingale convergence argument shows that
$(\underline f_\omega^n)^*\mu$ converges to some measure $\mu_\omega$  for almost every
$\omega$ (see \cite[Lem. II.2.1]{bougerol-lacroix}). 
Since $\Gamma_\nu$ preserves the decomposition $\Pi_{\Gamma_\nu} \oplus \Pi_{\Gamma_\nu}^\perp$ and 
$\norm{ (f_\omega^n)^*}$ tends to infinity while $\parallel (f_\omega^n)^*\rest{\Pi_{\Gamma_\nu}^\perp} \parallel$ 
stays uniformly bounded, 
we get that $(f_\omega^n)^*u$ converges to $\P \Pi_{\Gamma_\nu}$ for $\mu$-almost every $u$ and $\nu^\N$-almost every $\omega$; thus $\mu_\omega$ is almost surely supported on $\P \Pi_{\Gamma_\nu}$. Since by stationarity
 $\mu  = \int \mu_\omega d\nu^\N(\omega)$ we conclude that 
 $\mu$ gives full mass to $\P (\Pi_{\Gamma_\nu})$, hence $\mu = \mu_\fr$. 
%
\end{proof}

\begin{rem}
If $\supp(\nu)$ generates $\Gamma_\nu$ as a semi-group, then $\supp(\mu_\partial) = \Lim(\Gamma_\nu)$, 
otherwise the inclusion can be strict: take   a Schottky group $\Gamma = \langle f,g \rangle \subset \mathsf{PSL}(2, \R)$ and 
$\nu = (\delta_f+\delta_g)/2$.
\end{rem}

\begin{rem}\label{rem:lambda_with_log_mass}
Since $\Lim(\Gamma_\nu)\subset \Psef(X)$, for every $u\in \Lim(\Gamma_\nu)$ there exists a unique 
$\tilde u$ such that $\P \tilde u  = u$ and $\langle \tilde{u}\, \vert \, [\kappa_0]\rangle= \M (\tilde u) = 1$. Then the following 
formula holds: 
\begin{align}
\label{eq:Furstenberg_formula2}
\lambda_{H^{1,1}} =   \int \log   \lrpar{\M(f^*{\tilde{u}})} \, d\nu(f) \, d\mu_\partial(u) =  \int \log\left( \frac{\M(f^*{\tilde{u}})}{\M({\tilde{u}})}\right) \, d\nu(f) \, d\mu_\partial(u).    
\end{align}
Indeed set $r(w)=\M(w)/\abs{ w}$. On the limit set this function 
satisfies $1/C\leq r({\tilde{u}})\leq C$, where $C$ is the positive constant from Equation~\eqref{eq:comparison_norm_mass}. 
Then, for all $m\geq 1$,the stationarity 
of  $\mu_\partial$ implies   
\[
\int \log\left( \frac{r( f^*{\tilde{u}})}{r({\tilde{u}})} \right) \, d\nu (f) \, d\mu_\partial(u) = \\
\int \log\left( \frac{r( f_m^* \cdots f_0^*{\tilde{u}})}{r(f_{m-1}^* \cdots f_0^*{\tilde{u}})}\right)\, d\nu (f_{m})\cdots d\nu (f_0)\, d\mu_\partial(u).
\]
Summing from $m=0$ to $n-1$, 
 telescoping the sum, and   dividing by $n$ gives 
\[ 
 \int \log\left( \frac{r( f^*{\tilde{u}})}{r({\tilde{u}})} \right) \, d\nu (f) \, d\mu_\partial(u) =
 \frac1n \int \log\left( \frac{r( f_{n-1}^* \cdots f_0^*{\tilde{u}})}{r({\tilde{u}})}\right)\, d\nu (f_{n-1})\cdots d\nu (f_0)\, d\mu_\partial(u).
\] 
Finally since   $1/C\leq r \leq C$, the right hand side  
tends to zero as $n\to\infty$. Hence 
   the integral of $\log(r\circ f^*/r)$ 
vanishes, and \eqref{eq:Furstenberg_formula2} follows from Furstenberg's formula.\qed
\end{rem}

\begin{pro}\label{pro:extremality}
The  point $\P(e(\omega))$ is $\nu^\N$-almost surely   extremal in  $\P(\Kahbar(X))$ and in 
$\P(\Psef(X))$. 
\end{pro}

\begin{proof}
The class $e(\omega)$ almost surely belongs to $\Kahbar(X)$ and to the isotropic cone. By the Hodge index theorem --more precisely, by the case of equality in the reverse Schwarz Inequality~\eqref{eq:reverse_CS}-- 
$e(\omega)$ cannot be a non-trivial convex combination of classes with non-negative intersection and mass $1$;  
so $\P(e(\omega))$ is  an extremal point of the convex set $\P(\Kahbar(X))\subset \P H^{1,1}(X;\R)$.  
  
From Proposition~\ref{pro:extremal_rays}, there are at most countably many points $\P(u)$ in $\P(\Kahbar(X))$ such that $u^2=0$ and $\P(u)$ is not extremal in $\P(\Psef(X))$. Therefore the second assertion follows from the fact that $\mu_\fr$ is atomless. 
\end{proof}

\subsection{Some estimates for random products of matrices}
The   
aim of this section is to establish some technical facts 
which will play a crucial role in our study of the
closed positive currents $T^s_\omega$ in Section~\ref{sec:currents}. The key results are Theorem~\ref{thm:GK_good_times}
and Lemma~\ref{lem:variant_furstenberg}.

\subsubsection{Sequences of good times}
We  now describe a theorem of Gou\"ezel and Karlsson, specialized to
 our specific context. Fix a base point $e_0$ in 
the hyperbolic space $\Hyp_X$, for instance $e_0=[\kappa_0]$ with $\kappa_0$ a fixed K\"ahler form (as in Section~\ref{par:cones_definition}). 
Consider the two functions of $(n,\omega)\in \N\times \Omega$ defined by 
\begin{equation}\label{eq:two_cocycles}
T(n,\omega)= d_\Hyp(e_0, (f_\omega^n)^*e_0), \quad N(n,\omega)=\log \norm{ (f_\omega^n)^*}.
\end{equation}
They satisfy the subadditive cocycle property
\begin{equation}
a(n+m, \omega)\leq a(n,\omega)+a(m,\sigma^n(\omega)),
\end{equation}
where $\sigma$ is the unilateral shift on $\Omega$ (see~\S~\ref{par:random_compositions1}).
Let $a(n,\omega)$ be such a subadditive cocycle;  if $a(1,\omega)$ is integrable the  asymptotic average is defined 
to be the limit  \begin{equation}
A=\lim_{n\to +\infty} \frac{1}{n}\int a(n,\omega)\, d\nu^\N(\omega);
\end{equation}
it exists in $[-\infty,+\infty)$, and we say it is finite if $A\neq -\infty$.
The functions $T$ and $N$ are examples of ergodic subadditive cocycles and from 
Theorem~\ref{thm:def_stationary}, Remark~\ref{rem:lambda_with_log_mass}, and Corollary~\ref{cor:KAK}, we deduce that
the asymptotic average of each of these cocycles is equal to $ \lambda_{H^{1,1}}$.

Following \cite{Gouezel-Karlsson}, we  say that   $a(n,\omega)$ is 
{\bf{tight along the sequence of positive integers $(n_i)$}} if there is a sequence of real numbers $(\delta_\ell) = (\delta_\ell(\omega))_{\ell\geq 0}$ 
such that 
\begin{itemize}
\item[(i)] $\delta_\ell$ converges to $0$ as $\ell$ goes to $+\infty$;
\item[(ii)] for every $i$, and for every $0\leq \ell \leq n_i$, 
$$
\abs{ a(n_i,\omega)-a(n_i-\ell,\sigma^\ell(\omega))-A\ell} \leq \ell \delta_\ell;
$$
\item[(iii)] for every $i$ and for every $0\leq \ell \leq n_i$ 
$$
a(n_i,\omega)-a(n_i-\ell,\omega)\geq (A-\delta_\ell)\ell.
$$
\end{itemize}

\begin{thm}[Gou\"ezel and Karlsson \cite{Gouezel-Karlsson}]\label{thm:GK_good_times} Let $a(n,\omega)$ be an ergodic subadditive cocycle, with 
a finite asymptotic average $A$. Then, for almost every $\omega$, the cocycle is tight along 
a  subsequence $(n_i(\omega))$ of positive upper density. 
\end{thm}

Recall that the     (asymptotic) {{upper density}} of a subset $S$ of $\N$ is
the non-negative number defined by
$ {\overline{\mathrm{dens}}}(S)=\limsup_{k\to +\infty} \left(\frac{1}{k}\vert S\cap [0,k-1]\vert \right)$. 
A sequence $(n_i)_{i\geq 0}$ is said to have positive upper density if the
set  of  its values $S=\{ n_i\; ; \; i\geq 0\}$ satisfies ${\overline{\mathrm{dens}}}(S) >0$.

\begin{proof} Let us explain how this result follows from \cite{Gouezel-Karlsson}. 
First, fix a small positive real number $\rho>0$, and apply 
Theorem~1.1 and Remark 1.2 of \cite{Gouezel-Karlsson} to get a set $\Omega_\rho$ of measure $1-\rho$ such that the first
two properties (i) and (ii) are satisfied for every $\omega\in \Omega_\rho$ with respect to a sequence $(\delta_\ell)$ that
does not depend on $\omega$, and for a sequence of times $(n_i(\omega))$ of upper density $\geq 1-\rho$.
To get (iii), we apply Lemma 2.3 of \cite{Gouezel-Karlsson} to the sub-additive cocycle $a(n,\omega)$ (not to the cocycle $b(n,\omega)= a(n, \sigma^{-n}(\omega))$
as done in \cite{Gouezel-Karlsson}). For every $\e >0$, there is 
a subset $\Omega'_\e\subset \Omega$ and a sequence $(\delta'_\ell)_{\ell \geq 0}$ such that
\begin{itemize}
\item[(a)] $\nu^\N(\Omega'_\e)>1-\e$, and $\delta'_\ell$ converges towards $0$ as $\ell$ goes to $+\infty$;
\item[(b)] for every $\omega\in \Omega'_\e$, there is a set of bad times $B(\omega)\subset \N$ such that for every $k\geq 0$
$\abs{ B(\omega)\cap [0,k-1]}\leq \e k$,  and  
for every $n\notin B(\omega)$ and every $0\leq \ell \leq n$,
$$ 
a(n,\omega)-a(n-\ell,\omega)\geq (A-\delta'_\ell) \ell.
$$ 
\end{itemize}
If $\omega$ belongs to $\Omega_\rho\cap \Omega'_\varepsilon $, the 
set of indices $i$ for which $n_i(\omega)\notin B(\omega)$ is infinite. More precisely, 
the set $S(\omega)=\{ n_j(\omega) \; ; \; n_j(\omega)\notin B(\omega)\}$ has asymptotic
upper density $\geq 1-\rho - \varepsilon $. Along this subsequence, the three properties (i), (ii), and (iii)
are satisfied. Since this holds for all $\omega \in \Omega'_\varepsilon \cap \Omega_\rho$ and
the measure of this set is $\geq 1-\rho-\varepsilon $, this holds for $\nu^\N$-almost every $\omega$. 
\end{proof}

\begin{cor}\label{cor:GK_good_times}
For $\nu^\N$-almost every $\omega\in \Aut(X)^\N$, there is an increasing sequence of integers $(n_i(\omega))$ going to $+\infty$
and a real number $A(\omega)$ such that 
\[
\sum_{j=0}^{n_i(\omega)} \frac{\big\|{ \big(f_\omega^{j}\big)^*}\big\|}{\big\|{ \big(f_\omega^{n_i(\omega)}\big)^*}\big\|} \leq A(\omega) \; {\text{ and }} \; 
\sum_{j=0}^{n_i(\omega)} \frac{\big\|{ \big(f_{\sigma^j(\omega)}^{n_i(\omega)-j}\big)^*}\big\|}{\big\|{ \big(f_\omega^{n_i(\omega)}\big)^*}\big\|} \leq A(\omega)
\]
for all indices $i\geq 0$. 
\end{cor}

\begin{proof} Apply Theorem~\ref{thm:GK_good_times} to the subadditive cocyle $N(n,\omega)$ and 
note that 
\begin{equation}
\sum_{j=0}^{n_i(\omega)}  \frac{\big\|{ \big(f_\omega^{j}\big)^*}\big\|}{\big\|{ \big(f_\omega^{n_i(\omega)}\big)^*}\big\|}   = 
\sum_{\ell=0}^{n_i(\omega)}  \frac{\big\|{ \big(f_\omega^{n_i- \ell}\big)^*}\big\|}{\big\|{ \big(f_\omega^{n_i}\big)^*}\big\|}  = 
\sum_{\ell=0}^{n_i(\omega)}  
\frac{e^{N(n_i-\ell,\omega) }}{e^{ N(n_i, \omega)} }\leq 
\sum_{\ell=0}^{n_i(\omega)}    e^{-\ell (\lambda_{H^{1,1}}-\delta_\ell)}
\end{equation}
which is bounded as  $n_i(\omega)\to \infty$. The second estimate is similar. 
\end{proof}

\subsubsection{A mass estimate for pull-backs}
Assume that $(X, \nu)$ is non-elementary and satisfies the condition~\eqref{eq:moment}.
Recall from Lemma~\ref{lem:growth2} that $\M((f_\omega^n)^* a)\inv (f_\omega^n)^*a$ converges to the
pseudo-effective class $e(\omega)$ for almost every $\omega$ and every K\"ahler class
$a$. Thus, on a set of total $\nu^\N$-measure, 
this convergence holds for all $\sigma^k(\omega)$, $k\geq 0$. Since $\M(e(\omega)) = 1$, we obtain
\begin{equation}\label{eq:equivariance_e}
f_0^* e(\sigma\omega) = \M(f_0^*e(\sigma\omega)) e(\omega);
\end{equation}
more generally, for every $k\geq 1$,
 \begin{equation}
(f_\omega^k)^* e(\sigma^k\omega) = \M((f_\omega^k)^*e(\sigma^k\omega)) e(\omega).
\end{equation}

\begin{vcourte}
 \begin{lem}\label{lem:variant_furstenberg}
For $\nu^\N$-almost every $\omega$,  we have  
$
\frac1{n} \log \M((f_\omega^n)^*e(\sigma^n\omega))   \underset{n\to\infty}{\longrightarrow}\lambda_{H^{1,1}}.
$
\end{lem}
\end{vcourte}
\begin{vlongue}
 \begin{lem}\label{lem:variant_furstenberg}
For $\nu^\N$-almost every $\omega$,  we have  
\[
\frac1{n} \log \M((f_\omega^n)^*e(\sigma^n\omega))   \underset{n\to\infty}{\longrightarrow}\lambda_{H^{1,1}}.
\]
\end{lem}
\end{vlongue}

This does \emph{not} follow from Lemma~\ref{lem:growth} because $e(\sigma^n\omega)$ depends on $n$. %

 \begin{proof} 
\begin{vcourte}
For almost every $\omega$, for every $k\geq 1$,  and for  every K\"ahler class $a$,  $e(\sigma^k\omega)$ is the limit of 
$\M(f_k^*\cdots f_{n-1}^* a)^{-1} f_k^*\cdots f_{n-1}^* a$ as $n$ goes to $+\infty$. So
\begin{equation}\label{eq:zetakf}
f_0^*\cdots f_{k-1}^* e(\sigma^k(\omega))   = \left(\lim_{n\to\infty} \frac{\M( f_0^*\cdots f_{n-1}^*a)}{\M( f_k^*\cdots f_{n-1}^*a)}\right) e(\omega)
=: \zeta(k, \omega) e(\omega)
\end{equation}
where $\zeta(k, \omega)$ is both equal to $\M((f_\omega^k)^* e(\sigma^k(\omega)))$ and to the  limit 
\begin{equation}\label{eq:zetakf2}
\zeta(k, \omega) = \lim_{n\to\infty} \frac{\M( f_0^*\cdots f_{n-1}^*a)}{\M( f_k^*\cdots f_{n-1}^*a)}= \lim_{n\to\infty} \frac{\M( (f_\omega^n)^*a)}{\M( (f_{\sigma^k(\omega)}^{n-k})^*a)} .
\end{equation}
We want to show that,   $\nu^\N$-almost surely, 
$ (  1/k)\log \zeta(k,\omega)  $   converges to $\lambda_{H^{1,1}}$. 
\end{vcourte}
\begin{vlongue}
For almost every $\omega$, for every $k\geq 1$,  and for  every K\"ahler class $a$, we have
\begin{equation}
  e(\sigma^k\omega) = \lim_{n\to\infty} \frac{f_k^*\cdots f_{n-1}^* a}{\M(f_k^*\cdots f_{n-1}^* a)}.
  \end{equation}
So
\begin{equation}\label{eq:zetakf}
f_0^*\cdots f_{k-1}^* e(\sigma^k(\omega))   = \left(\lim_{n\to\infty} \frac{\M( f_0^*\cdots f_{n-1}^*a)}{\M( f_k^*\cdots f_{n-1}^*a)}\right) e(\omega)
=: \zeta(k, \omega) e(\omega)
\end{equation}
where $\zeta(k, \omega)$ is both equal to $\M((f_\omega^k)^* e(\sigma^k(\omega)))$ and to the  limit 
\begin{equation}\label{eq:zetakf2}
\zeta(k, \omega) = \lim_{n\to\infty} \frac{\M( f_0^*\cdots f_{n-1}^*a)}{\M( f_k^*\cdots f_{n-1}^*a)}= \lim_{n\to\infty} \frac{\M( (f_\omega^n)^*a)}{\M( (f_{\sigma^k(\omega)}^{n-k})^*a)} .
\end{equation}
We want to show that,   $\nu^\N$-almost surely, 
$ (  1/k)\log \zeta(k,\omega)  $   converges to $\lambda_{H^{1,1}}$. 
\end{vlongue}

Before starting the proof, note that $\zeta$ is a multiplicative cocycle: 
$ \zeta(k,\omega)= \prod_{\ell = 1}^k \zeta (1, \sigma^\ell\omega)$; in particular, $\log \zeta(k,\omega) $
is equal to the Birkhoff sum $\sum_{\ell=1}^k \log \zeta (1, \sigma^\ell\omega)$. Since 
\begin{equation} 
C\inv\norm{(f_0\inv)^*}_{H^{1,1}} \leq \M(f_0^*e(\sigma(\omega))) \leq C\norm{f_0^*}_{H^{1,1}} ,
\end{equation}
our moment condition shows 
that $\log(\zeta(1,\omega))$ is integrable. So, by the ergodic theorem of Birkhoff, $\lim_{k} \frac1{k} \log \zeta(k,\omega)$
exists $\nu^\N$-almost surely. 

Pick a sequence $(n_i)$ of good times for $\omega$, as in Theorem~\ref{thm:GK_good_times}. If we compute
the limit in Equation~\eqref{eq:zetakf2} along the subsequence $(n_i)$ we see that $\zeta(k,\omega) \geq C \exp((\lambda_{H^{1,1}}-\delta(k))k)$
for some constant $C>0$, and some sequence $\delta(k)$ converging to $0$ as $k$ goes to $+\infty$. This gives 
\begin{equation}\label{eq:ineq_zeta_lambda1}
 \limsup_{k\to +\infty} \frac1{k} \log \zeta(k,\omega)  \geq \lambda_{H^{1,1}}.
 \end{equation}
Now, consider the linear cocycle $\Upsilon: \Omega \times H^{1,1}(X, \R) \to  \Omega\times H^{1,1}(X, \R)$
defined by 
\begin{equation}
\Upsilon (\omega, u) = (\sigma(\omega), (f_\omega^1)_* u)
\end{equation}
and let $\mathbb P \Upsilon$ be the associated projective cocycle on $ \Omega\times \P H^{1,1}(X, \R)$. 
 The Lyapunov exponents of $\Upsilon$ 
 are $\pm\lambda_{H^{1,1}}$, each with multiplicity $1$, and $0$, with multiplicity $h^{1,1}(X)-2$.  
Since $\P((f_\omega^1)^*e(\sigma(\omega))) = \P(e(\omega))$, the measurable 
section $\set{(\omega, \P(e(\omega)))\; ; \;   \omega\in \Omega}$ is $\P\Upsilon$-invariant. 
Therefore, by ergodicity of $\sigma$ with respect to $\nu^\N$,   $m= \int \delta_{\P(e(\omega))} \, d\nu^{\N}(\omega)$ defines an invariant and ergodic measure 
for $\mathbb P \Upsilon$. It follows from the invariance of the decomposition into characteristic subspaces in 
Oseledets' theorem that  $e(\omega)$ is contained in a given  characteristic subspace  of the cocycle $\Upsilon$; thus, if 
$\lambda$ denotes the Lyapunov exponent  of $\Upsilon$ in that characteristic subspace, we get (as in Remark~\ref{rem:lambda_with_log_mass}) that
\begin{vcourte}
\begin{align*}
\lambda = \int \log \frac{\abs{(f_\omega^1)_* u}}{\abs{u}} \; dm(\omega, u) 
 & = \int \log  \frac{\M((f_\omega^1)_*(  e(\omega))}{\M(  e(\omega))}\, d\nu^{\N} (\omega)  = \int \log \zeta(1, \omega)\inv  \,  d\nu^{\N} (\omega) 
\end{align*}  
\end{vcourte}
\begin{vlongue}
\begin{align}
\lambda = \int \log \frac{\abs{(f_\omega^1)_* u}}{\abs{u}} \; dm(\omega, u) 
 & = \int \log  \frac{\M((f_\omega^1)_*(  e(\omega))}{\M(  e(\omega))}\, d\nu^{\N} (\omega)\\
&  = \int \log \zeta(1, \omega)\inv  \,  d\nu^{\N} (\omega) 
\end{align}  
\end{vlongue}
(see Ledrappier \cite[\S 1.5]{Ledrappier:SaintFlour}).
Birkhoff's  ergodic theorem    implies that $\lim_{k} \frac1{k} \log \zeta(k,\omega)  = -\lambda$, with 
 $ \lambda\in  \set{\pm\lambda_{H^{1,1}}, 0}$,  therefore 
     the Inequality~\eqref{eq:ineq_zeta_lambda1} concludes the proof. 
 \end{proof}

\subsubsection{Exponential moments}\label{subsub:exponential_moments} 
The result of this section will only be used in Theorem \ref{thm:holder} so this paragraph may be skipped on a first reading.  
Consider the exponential moment condition 
\begin{equation}\label{eq:exponential_moment}
\exists \tau>0, \ \int  \lrpar{{\norm{ f}_{ C^1}} + {\norm{f^{-1}}_{C^1}}}^\tau  \; d\nu(f)<+\infty.
\end{equation}
As in Section~\ref{subs:moments_cohomology}, this upper bound implies the cohomological moment condition
\begin{equation}\label{eq:exponential_moment_cohomology}
\exists \tau>0, \ \int  \lrpar{{\norm{ f^*}_{H^{1, 1}}} + {\norm{(f^{-1})^*}_{H^{1, 1}}}}^\tau  \; d\nu(f)<+\infty.
\end{equation}

\begin{pro}\label{pro:gouezel}
Assume that $\nu$ satisfies the Condition~\eqref{eq:exponential_moment}. Let 
$D\colon \Aut(X)\to \R_+$ be a measurable function such that $\int  {  D(f)^{\tau'}}d\nu(f)<\infty$ for some $\tau'>0$. 
Then, there is a measurable function $B\colon \Omega\to \R_+$ satisfying
\[  
\int  \log^+(B(\omega)) \,  d\nu^\N(\omega)<\infty  
\]
such that for $\nu^\N$-almost every $\omega = (f_n)$ and every $n\geq 0$
\[ 
 \sum_{j=1}^{n-1} D(f_{j-1})
  \frac{\big\| {f_j^*\cdots f_{n-1}^*} \big\| }{\norm{f_0^*\cdots f_{n-1}^*}}  \leq B(\omega), \;   \text{ and } \; 
  \sum_{j=1}^{n-1} D(f_j)\frac{\big\| {f_0^*\cdots f_{j-1}^*} \big\| }{\norm{f_0^*\cdots f_{n-1}^*}}   \leq B(\omega).
\] 
\end{pro}

This is a refined version of Corollary \ref{cor:GK_good_times}.  

\begin{proof}  
We are grateful to  S\'ebastien Gou\"ezel for explaining this argument to us. We temporarily use the notation $\P(\cdot)$  for probability with respect to $\nu^n$ or $\nu^\N$ (so, here, $\P$ does not denote projectivisation). 

\smallskip

\noindent{\bf{First Estimate.--}} We start with the first estimate: $ \sum_{j=1}^{n-1} D(f_{j-1})
  \frac{\big\| {f_j^*\cdots f_{n-1}^*} \big\| }{\norm{f_0^*\cdots f_{n-1}^*}}  \leq B(\omega)$.

{\bf{Step 1.--}} For every $0<\e<\lambda_{H^{1,1}}$ there exists 
constants $c, C>0$ such that 
\begin{equation}
\P\lrpar{\abs{(f_\omega^n)^*b}\leq e^{\e n}} \leq Ce^{-cn}.
\end{equation}
for every $b\in \Pi_\Gamma$ with $\abs{ b}=1$.
This large deviation result, which is uniform in $n$ and $b$, follows from condition~\eqref{eq:exponential_moment_cohomology} (see for instance~\cite[\S V.6]{bougerol-lacroix}, and~\cite[\S 12]{benoist-quint_book}).

{\bf{Step 2.--}} Let us prove that
\begin{equation}\label{eq:proba_estimate}
\P\lrpar{\frac{\big\| {f_j^*\cdots f_{n-1}^*} \big\| }{\norm{f_0^*\cdots f_{n-1}^*}} > e^{-\e j} }\leq Ce^{-cj}.
\end{equation}
 
For this,  fix $f_j, \ldots, f_{n-1}$. Then, there is a point $a\in \Pi_\Gamma$  with $\abs{a}=1$ such  that  
$\norm{f_j^*\cdots f_{n-1}^*} = \abs{ f_j^*\cdots f_{n-1}^*a }$. Hence, if $\norm{f_0^*\cdots f_{n-1}^*} < \big\| {f_j^*\cdots f_{n-1}^*} \big\| e^{\e j}$, 
we infer that 
\begin{equation}
\abs{f_0^*\cdots f_{n-1}^*a}  < \big\| {f_j^*\cdots f_{n-1}^*} \big\| e^{\e j} = \big|{f_j^*\cdots f_{n-1}^*a} \big| e^{\e j}.
\end{equation}
Thus, if we set 
\begin{equation}
b = \unsur{\big|{f_j^*\cdots f_{n-1}^*a}\big| } f_j^*\cdots f_{n-1}^*a ,
\end{equation} 
we obtain that  $\abs{f_0^*\cdots f_{j-1}^*b} < e^{\e j}$; this happens with (conditional) 
probability $\leq Ce^{-cj}$ (relative to $\nu^{* j}$), for the uniform 
 constants given in Step~1. Averaging over  $f_j, \ldots, f_{n-1}$, we get the result.

{\bf{Step 3.--}} The moment condition satisfied by $D$ and Markov's inequality imply $\P(D>K)\leq C_1 K^{-\tau'}$ for some constant $C_1>0$.
Fix $\e\in \R_+^*$ small with respect to $\lambda_{H^{1,1}}$ and $\tau'$. Then, on a set $\Omega(\e,J)$ 
of measure  
\begin{equation}
\nu^\N(\Omega(\e,J))\geq 1-C_2(e^{-(\e \tau'/2) J}+e^{-\e c J}),
\end{equation} 
for some $C_2=C_2(\e)>0$, we have both $D(f_{j-1})\leq e^{\e j/2}$ and $\frac{\norm{f_j^*\cdots f_{n-1}^*}}{\norm{f_0^*\cdots f_{n-1}^*}} \leq e^{-\e j}$ for all $j\geq J$. 
For $\omega=(f_n)$ in $\Omega(\e,J)$, we get
  \begin{align}\label{eq:sum_gouezel}
 \sum_{j=1}^{n-1} D(f_{j-1})\frac{\big\| {f_j^*\cdots f_{n-1}^*} \big\| }{\norm{f_0^*\cdots f_{n-1}^*}}  
 &\leq  \sum_{j=1}^J D(f_{j-1}) \frac{\big\| {f_j^*\cdots f_{n-1}^*} \big\| }{\norm{f_0^*\cdots f_{n-1}^*}} +  \sum_{j=J+1}^{n-1} e^{-\e j/2}\\
 \notag&\leq   \sum_{j=1}^J D(f_{j-1}) \norm{(f_{j-1}\inv)^* \cdots  (f_0\inv)^*} + C_3 \\ 
 & \notag= C_3+ \sum_{j=0}^{J-1} \norm{f_0^*}\cdots \norm{f_j^*} D(f_j).
 \end{align}
The moment condition~\eqref{eq:exponential_moment} gives $\P(\norm{f^*}>K)\leq C_4K^{-\tau}$ and
as already noticed, we also have $\P(D(f)>K)\leq C_1 K^{-\tau'}$. So, with $\eta=\min(\tau, \tau')$, 
there is a set of probability at least
$1- C_5JK^{-\eta}$ on which
\begin{equation} 
\sum_{j=0}^{J-1}  D(f_j) \norm{f_0^*}\cdots \norm{f_j^*} \leq C_6J K^{J+2}.
\end{equation}
 Taking $K = J^{3/\eta}$, we have $JK^{-\eta}=J^{-2}$, and we obtain 
\begin{equation}\label{eq:J}
\P\lrpar{ \sum_{j=0}^{J-1} D(f_j)\norm{f_0^*}\cdots \norm{f_j^*}  > J^{1+3 (J + 2)/\eta}  }\leq C_7J^{-2}.
\end{equation}
Also, note that $J^{1+(3 J + 6)/\eta} \leq \exp\lrpar{CJ^{3/2}}$.

By the Borel-Cantelli lemma, the sum in \eqref{eq:sum_gouezel} is almost surely bounded  
by some constant $B(\omega)$
 which satisfies  $\P\lrpar{\log B > J^{3/2}}\leq CJ^{-2}$; in particular $\ee\lrpar{\log^+ B}<\infty$. 
 
\smallskip

{\bf{Second Estimate.--}} To obtain the second estimate of Proposition~\ref{pro:gouezel}, we apply the above proof 
to the reversed random dynamical system, induced by $\check{\nu}:f\mapsto \nu(f\inv)$. Indeed, the core of the 
argument is the inequality \eqref{eq:sum_gouezel} which is not sensitive to the order of compositions. 
\end{proof}


\section{Limit currents}\label{sec:currents}

Our goal in this section is to prove the counterpart of  the convergence \eqref{eq:definition_e} 
at the level of   closed positive currents on $X$. 
Throughout this section we fix a non-elementary random holomorphic dynamical system   
$(X, \nu)$ satisfying the moment condition \eqref{eq:moment}, so that all results of \S \ref{sec:furstenberg} apply. 
We refer the reader to \cite{Guedj-Zeriahi:Book} (in particular Chapter 8) for basics on pluripotential theory on compact K\"ahler manifolds (see also \cite{Demailly:book}).  

\subsection{Potentials and cohomology classes of positive closed currents}\label{par:Notations_kappa_i_potentials}
 
Let us fix once and for all a family of K\"ahler forms $(\kappa_i)_{1\leq i\leq h^{1,1}(X)}$ such that  $[\kappa_i]^2 = 1$ and 
the $[\kappa_i]$ form a basis of  $H^{1,1}(X;\R)$; in addition we require 
that the $\kappa_i$ satisfy  
\begin{equation}\label{eq:beta}
\kappa_0= \beta \sum_i\kappa_i
\end{equation}
for some $\beta >0$, where $\kappa_0$ is the K\"ahler form chosen in Section~\ref{par:cones_definition}  (note that necessarily $\beta <1$). 
We also fix  a smooth   volume form $\vol_X$ on $X$, normalized by $\int_X \vol=1$. 
On tori, K3 and Enriques surfaces, we choose $\vol_X$ to be the canonical $\Aut(X)$-invariant volume form (see Remark~\ref{rem:volume_form}). 
It is convenient to assume in all cases that $\vol_X$ is also the volume form associated with the K\"ahler metric $\kappa_0$ (up to 
scaling). 
\begin{vlongue}
On tori, K3 and Enriques surfaces this implies that $\kappa_0$ is the unique Ricci-flat K\"ahler metric in its K\"ahler class; its 
existence is guaranteed by Yau's theorem (see \cite{Filip-Tosatti} for the interest of such a choice in holomorphic dynamics).  
\end{vlongue} 

Unless otherwise specified, the currents we shall consider will be of type $(1,1)$.
The action of a current $T$ on a test form $\varphi$ will 
be denoted by $\langle T, \varphi\rangle$ or $\int T\wedge \varphi$. If $T$ is closed, we denote its cohomology class by $[T]$; so, if $\varphi$
is a closed form, $\langle T, \varphi\rangle=\langle [T] \, \vert\,   [\varphi]\rangle$.
By definition the {\bf{mass}} of a current is the quantity $\M(T) = \int T\wedge \kappa_0$; so $\M(T) = \langle [T]\vert [\kappa_0]\rangle$ when $T$ is closed. 

\subsubsection{Normalized potentials}\label{subs:normalized_potentials}
If $a$ is an element of $H^{1,1}(X;\R)$, we denote by $(c_i(a))_{1\leq i \leq h^{1,1}(X)}$ 
its coordinates
in the basis $([\kappa_i])$, so that  $a =\sum_i c_i(a)[\kappa_i]$. Then, we set
\begin{equation}
\Theta(a)=\sum_i c_i(a) \kappa_i.
\end{equation}
Likewise, given a closed $(1,1)$-form  $\alpha$ or a closed current of bidegree $(1,1)$, we set $c_i(\alpha)=c_i([\alpha])$ and
$\Theta(\alpha)=\Theta([\alpha])$; hence, $[\Theta(\alpha)]=[\alpha]$. 
It is worth keeping in mind that  some coefficients $c_i(\alpha)$ can be negative and
  $\Theta(\alpha)$ need not be semi-positive, even if $\alpha$ is a K\"ahler form.
If $T$ is a closed positive current of bidegree $(1,1)$ on $X$ we define its 
\textbf{normalized potential} to be the unique 
 function  $u_T\in L^1(X)$ such that 
\begin{equation}\label{eq:def_uT}
T=\Theta(T)+dd^c(u_T) \; \text{ and } \; \int_X u_T \; \vol=0
\end{equation}
 (see  \cite[\S 8.1]{Guedj-Zeriahi:Book}). 
The function $u_T$ is locally given as the difference $v-w$ of a psh potential $v$ of $T$ and a smooth potential $w$ 
of $\Theta(T)$. 

\begin{lem}\label{lem:uniform_Ak_psh} 
There is a  constant $A>0$ such that the following properties are satisfied for every 
closed positive current $T$ of mass $1$ 
\begin{enumerate}[{\em (1)}]
\item $-A\leq c_i(T) \leq A$ for all $1\leq i\leq h^{1,1}(X)$, and $-A\kappa_0 \leq \Theta(T)\leq A\kappa_0$.
\item  the function $u_T$ is $(A\kappa_0)$-psh: $dd^c(u_T)+A\kappa_0$ is a positive current.
\end{enumerate}
\end{lem}

\begin{proof} Since the coefficients $T\mapsto c_i(T)$ are continuous functions on the space of currents and closed positive currents
of mass $1$  form a compact set $K$, the functions $\abs{c_i}$ are bounded by some uniform constant $A'$ on $K$. 
Setting $A=A'\beta\inv$, with $\beta$ as in Equation~\eqref{eq:beta}, we get $-A\kappa_0 \leq \Theta(T)\leq A\kappa_0$ for all $T\in K$.  
Then  $dd^cu_T=T-\Theta(T)\geq   -A \kappa_0$ and (2) follows. \end{proof}

\begin{cor}\label{cor:compact} 
The set of potentials $\set{ u_T \; \vert \; T \, {\text{ is a closed positive current of mass}} \, 1\ \text{on } X}$ is  a compact subset of $L^{1}(X;\vol)$.
\end{cor}

\begin{proof}
Since this is a set of $(A\kappa_0)$-psh functions which are normalized with respect to a smooth 
volume form,  the result follows from Proposition 8.5 and Remark 8.6 in \cite{Guedj-Zeriahi:Book}. 
\end{proof}

\begin{rem}\label{rem:normalization}
Another usual normalization  imposes the condition $\sup_{x\in X} u_T(x)=0$; by compactness this would only change 
$u_T$ by some uniformly bounded constant. However since many of our dynamical examples preserve a natural volume form 
it is more convenient for us to normalize as in \eqref{eq:def_uT}. 
\end{rem}

\subsubsection{The diameter of a pseudo-effective class}
For a class $a\in \Psef(X)$ we define  
\begin{equation}
\Cur(a)=\{ T\; ; \; T\, {\text{is a closed positive current with}}\; [T]=a\},
\end{equation}
 This is a compact convex subset of the space of currents.
If $S$ and $T$
are two elements of $\Cur(a)$, then $\Theta(S)=\Theta(T)=\Theta(a)$ and $T-S=dd^c(u_T-u_S)$. We set 
\begin{equation}
\dist(S,T)=\int_X\vert u_S-u_T\vert\,  \vol \; . 
\end{equation}
This is a distance that metrizes the weak topology on $\Cur(a)$: this follows for instance from the fact that 
by Corollary \ref{cor:compact}  $(\Cur(a), \dist)$ is compact. 
By definition, the \textbf{diameter} of  $a$ is 
\begin{equation}
\diam(a)=\diam(\Cur(a))= \sup\{\dist(S,T)\; ; \; S, \, T \, {\text{in}}\; \Cur(a)\},
\end{equation}
If $a\in \Psef(X)$, then $\diam(a)$  is a non-negative real number which is finite by Corollary~\ref{cor:compact}. 
If $\Cur(a)= \emptyset$, we set $\diam(a)=-\infty$.  Note that $\diam$ is 
homogeneous of degree~$1$: $\diam(t a)=t\diam(a)$ for every $a\in \Psef(X)$ and $t>0$. 

\begin{vlongue}  
\begin{eg}
Let $\pi\colon X\to B$ be a fibration of genus $1$. Let $a$ be the cohomology class of any fiber $X_w=\pi\inv(w)$, $w\in B$. 
Then, to every probability measure $\mu_B$ on $B$ corresponds  a closed positive current $T_{\mu_B}\in \Cur(a)$, defined 
by $\langle T_{\mu_B}, \varphi\rangle = \int_B \int_{X_w} \varphi d\mu_B(w)$, and any closed positive current in $\Cur(a)$ is
of this form. In this case $\diam(a)>0$. 
Now, assume that $f$ is a loxodromic automorphism of $X$, and denote by $\theta_f$ the unique $(1,1)$-class
of mass $1$ that satisfies $f^*\theta_f=\lambda_f \theta_f$, where $\lambda_f$ is the spectral radius of $f^*\in \GL(H^{1,1}(X;\R))$; 
then $\Cur(\theta_f)$ is represented by a unique closed positive current $T_f^+$ and $\diam(\theta_f)=0$. For generic Wehler surfaces, these two types of
classes, given by eigenvectors of loxodromic automorphisms and classes of genus $1$ fibrations, are dense in the boundary of 
$\Hyp_X\cap \NS(X;\R)$  (see~\cite{Cantat:Milnor}).
\end{eg}
\end{vlongue}
\begin{lem}\label{lem:diam_measurable}
On $\Psef(X)$, $a\mapsto \diam(a)$ is upper semi-continuous, hence measurable. 
\end{lem}

\begin{proof}
Let $(a_n)$  be a sequence of  pseudo-effective classes converging to $a$.
For every $n$    we choose a pair 
of currents $(S_n,T_n)$ in $\Cur(a_n)^2$ such that $\dist(S_n,T_n)\geq \diam(a_n)-1/n$. The masses of
$S_n$ and $T_n$ are uniformly bounded because they depend only on $a_n$. By 
Corollary~\ref{cor:compact}, we can extract a subsequence such that   $S_n$ and $T_n$
converge towards closed positive currents $S$, $T\in \Cur(a)$, and   $u_{S_n}$ 
and $u_{T_n}$ converge towards their respective  potentials $u_S$ and $u_T$ in $L^1(X,\vol)$. Then, 
$\dist(S,T)=\int_X\vert u_S-u_T\vert \vol =\lim_n \dist(S_n,T_n)$, which  shows that 
$ \diam(a)\geq \limsup_{n}\left( \diam(a_n)\right)$.
\end{proof}
 
\subsection{Action of $\Aut(X)$}

\subsubsection{A volume estimate}
Let $X$ be a compact, complex manifold, and let $\vol$ be a $C^0$-volume form on $X$
with $\vol(X)=1$. If $f$ is an automorphism of $X$,  let $\Jac(f)\colon X\to \R$ denote its Jacobian determinant
with respect to the volume form $\vol$: $f^*\vol= \Jac(f)\vol$. The following lemma is a variation 
on well-known ideas in holomorphic dynamics (see for instance~\cite{Guedj:Fourier}). 

\begin{lem}\label{lem:volume-estimate} Let $\kappa$ be a hermitian form on $X$.
Let $h$ be a $\kappa$-psh function on $X$ such that $\int_X h\, \vol =0$, and let $f$ be an automorphism of $X$. Then, 
\[
\int_X \abs{ h\circ f} \, \vol \leq C\log(C\norm{ \Jac(f^{-1}) }_\infty) 
\]
for some positive constant $C$ that depends on $(X,\kappa)$ but neither on $f$ nor on $h$. 
\end{lem}

\begin{proof}
We first observe that there is a constant $c>0$ such that $\vol\{\vert h\vert\geq t\}\leq c \exp(-t/c)$; this follows 
from Lemma 8.10 and Theorem 8.11 in \cite{Guedj-Zeriahi:Book}, together with Chebychev's inequality 
(see Remark \ref{rem:normalization} for the normalization).  
Then, we get 
\begin{eqnarray}
\int_X \vert h\circ f\vert \, \vol & = & \int_0^\infty \vol\{ \vert h\circ f\vert \geq t\} dt \\
&\notag = & \int_0^\infty \vol(f^{-1}\{ \vert h\vert \geq t\}) dt \\
&\notag \leq & \int_0^s \vol(X) dt + \norm{ \Jac(f^{-1}) }_\infty \int_s^\infty c\exp(-t/c)dt \\
& \label{eq:vol_lastline} \leq & s \, \vol(X)+ \norm{ \Jac(f^{-1}) }_\infty  c^2\exp(-s/c)
\end{eqnarray}
where the inequality in the third line follows from the change of variable formula. 
Now, we minimize   \eqref{eq:vol_lastline} by choosing 
$s=c\log (c\norm{ \Jac(f^{-1}) }_\infty/\vol(X))$
and we infer that   
\begin{equation}
\int_X \vert h\circ f\vert \, \vol\leq c\, \vol(X)\left(1+\log\left( \frac{c\norm{ \Jac(f^{-1}) }_\infty}{\vol(X)}\right)\right).
\end{equation}
Since the total volume is invariant, $\norm{ \Jac(f) }_\infty\geq 1$, and the  asserted estimate follows. 
\end{proof}

\subsubsection{Equivariance}
Let us come back to the study of $(X,\nu)$. 
If $f$ is an automorphism of $X$, then $f^*\Cur(a)=\Cur(f^*(a))$ for every class $a\in H^{1,1}(X, \R)$. 
If $a\in \Psef(X)$  and $T\in \Cur(a)$, then $T=\Theta(a)+dd^c(u_T)$ and  
\begin{equation}\label{eq:f*T}
f^*T=f^*\Theta(a)+dd^c(u_T\circ f)= \Theta(f^*a) + dd^c(u_{f^*\Theta(a)}  + u_T\circ f).
\end{equation}
This shows that the normalized potential of $f^*T$ is given by 
\begin{equation}\label{eq:ufT_uT}
u_{f^*T}= u_{f^*\Theta(a)} + u_T\circ f +E(f,T)
\end{equation}
where $E(f,T)\in \R$ is the constant for which the integral of $u_{f^*T}$ vanishes; since $u_{f^*\Theta(a)}$ has mean $0$, we get
\begin{equation}\label{eq:formula-E(f,T)}
E(f,T)=-\int_X \lrpar{u_{f^*\Theta(a)} + u_T\circ f} \, \vol=-\int_X  u_T\circ f \, \vol.
\end{equation}

\begin{rem}  
If $\vol$ is $f$-invariant, for instance if it is the canonical volume on a K3 or Enriques surface, 
then $E(f,T)=0$, which simplifies a little bit the analysis of the potentials below.
\end{rem}

\begin{lem}\label{lem:EfT}
On the set of closed positive currents of mass $1$, the function $(f,T)\mapsto E(f,T)$ satisfies  
$$
\vert E(f,T)\vert    \leq C \log\left( C\norm{ \Jac(f^{-1})}_\infty \right)
$$
where the implied positive constant  $C$   depends neither  on $f$ nor on  $T$.
\end{lem}

\begin{proof} 
From Lemma~\ref{lem:uniform_Ak_psh}, the potentials $u_T$ are uniformly $(A\kappa_0)$-psh, so the conclusion follows from Equation~\eqref{eq:formula-E(f,T)} and Lemma~\ref{lem:volume-estimate}.
\end{proof}
 
\begin{lem}\label{lem:diam_f*a} 
There exists a constant $C$ such that if $a$ is any pseudo-effective class  
of mass $1$, 
and $f$ is any automorphism of $X$, then
$$
\diam(f^*a)\leq   C\log\left( C\norm{\Jac (f^{-1})}_\infty\right).
$$
\end{lem}

 \begin{proof}
 Indeed, if $S$ and $T $ belong to $\Cur(a)$, by Equation~\eqref{eq:ufT_uT}  we have  
$ u_{f^*T} - u_{f^*S}   = (u_T-u_S)\circ f + E(f,T)- E(f,S)
$, so 
 \begin{equation}
 \dist(f^*T,f^*S) \leq \int\abs{ u_T\circ f } \, \vol + \int\abs{ u_S\circ f } \, \vol + \abs{E(f,T)}+\abs{E(f,S)};
 \end{equation}
 and the result follows from Lemmas \ref{lem:volume-estimate} and \ref{lem:EfT}, since 
 $u_S$ and $u_T$ are uniformly $(A\kappa_0)$-psh. 
 \end{proof}

\subsubsection{An estimate for canonical potentials}

 \begin{lem}\label{lem:utheta}
 For any K\"ahler form $\kappa $ on $X$ 
there exists  a positive constant $C(\kappa)$ such that for every $f\in \mathrm{Aut}(X)$,   
$$
\norm{ u_{f^*\kappa}}_{C^1} \leq C(\kappa) \norm{ f}_{C^1}^2  .
$$
In addition   $C(\kappa)\leq C'\norm{\kappa}_\infty$, where $\norm{\kappa}_\infty$ is  the sup norm of the coefficients of $\kappa$ in 
a system of coordinate charts, and $C'$ depends only on $X$ (and the choice of these coordinate charts). 
\end{lem}


Recall the choice of K\"ahler forms $(\kappa_i)$ from \S~\ref{par:Notations_kappa_i_potentials} and the definition of 
$\Theta(\cdot)$ from \S~\ref{subs:normalized_potentials}.

\begin{cor}\label{cor:utheta}
If $\kappa= \sum_i  c_i\kappa_i$ in Lemma \ref{lem:utheta},  then the constant $C(\kappa)$ satisfies 
$C(\kappa)\leq C''\M(\kappa)$. 
Likewise,  $\norm{u_{f^*\Theta(a)} }_{C^1}\leq C'''\M(a) \norm{f}_{C^1}^2$ for all  $a\in \Psef(X)$.
\end{cor}

Indeed   $C(\kappa)\leq C'  \norm{\kappa}_\infty \leq C'' \sum_i \abs{c_i}$ and  
$u_{f^*\Theta(a)} = \sum c_i(a) u_{f^*\kappa_i}$.

 \begin{proof}[Proof of Lemma~\ref{lem:utheta}] 
By definition 
 $f^*\kappa   - \Theta(f^*\kappa) = dd^c \lrpar{u_{f^*\kappa}}$. The desired estimate  will be obtained 
 by constructing a solution $\phi$ to the equation 
 \begin{equation}\label{eq:ddcphi}
 dd^c \phi = f^*\kappa   - \Theta(f^*\kappa)
 \end{equation} which satisfies 
 $\norm{ \phi}_{C^1} \leq C \norm{ f}_{C^1}^2  $. Then, since $u_{f^*\kappa}$ and $\phi$ 
 differ by a constant and $u_{f^*\kappa}$ is known to vanish at some point, it follows that $u_{f^*\kappa}$ satisfies the same estimate. 
To construct the potential $\phi$, we follow the method of  Dinh and Sibony \cite[Prop. 2.1]{dinh-sibony_jams} which is itself 
based on \cite{bost-gillet-soule}  (we keep the notation from 
\cite{dinh-sibony_jams}). Let $\alpha $ be a  closed $(2,2)$-form on $X\times X$ 
which is cohomologous to the diagonal   $\Delta$.
In \cite{bost-gillet-soule}, Bost, Gillet and Soul\'e  construct an explicit   $(1,1)$-form $K$ on $X\times X$ such 
that $dd^cK = [\Delta]-\alpha$; they refer to it as the ``Green current''. 
It is $C^\infty$ outside the diagonal, and along 
$\Delta$, it satisfies the estimates  
\begin{equation}\label{eq:K}
K(x,y)  = O\lrpar{ \frac{\log \abs{x-y}}{\abs{x-y}^2}}\;  \text{ and } \; \nabla K(x,y)  = O\lrpar{ \frac{\log \abs{x-y}}{\abs{x-y}^3}}
\end{equation} 
(here we mean that these estimates hold for the  coefficients of $K$ and $\nabla K$  in local coordinates).
These estimates are    easily deduced from the explicit 
expression of $K$ as $\pi_*(\widehat \varphi \eta - \beta)$ given in the proof of Proposition~2.1 of \cite{dinh-sibony_jams}, where 
$\pi:\widehat{X\times X}\to X\times X$ is the blow-up of the diagonal, $\eta$ and $\beta$ are smooth (1,1) forms on 
$\widehat{X\times X}$ and $\widehat \varphi $ is a function with logarithmic singularities along the proper transform 
of $\Delta$ in $X\times X$. 
It is shown in \cite[Prop. 2.1]{dinh-sibony_jams} that a solution to Equation~\eqref{eq:ddcphi} is given by 
\begin{equation}
\phi(x) = \int_{y\in X} K(x,y) \wedge \lrpar{f^*\kappa (y) -  \Theta(f^*\kappa)(y) }
\end{equation}
(in the notation of  \cite{dinh-sibony_jams}, $f^*\kappa$ and $\Theta(f^*\kappa)$ correspond
to $\Omega^+$ and $\Omega^-$ respectively). 
The coefficients of the smooth $(1,1)$-forms 
$f^*\kappa$   and $ \Theta(f^*\kappa)$ have their uniform norms bounded by  $C \norm{ f}_{C^1}^2$, where 
$C = C(\kappa) \leq C'\norm{\kappa}_\infty$. 
The first estimate in \eqref{eq:K} 
implies that the coefficients of $K$ belong to $L^{p}_{\mathrm{loc}}$ for $p<2$, so it follows from the H\"older inequality that 
$\norm{\phi}_{C^0}\leq C''\norm{\kappa}_\infty \norm{ f}_{C^1}^2 $  (for some constant $C''$ depending only on $X$). 
A similar estimate for $ \nabla \phi$ is obtained from derivation under the integral sign 
and the fact that $ \nabla K \in L^{p}_{\mathrm{loc}}$ for $p<4/3$. This concludes the proof.
\end{proof}

\subsection{Convergence  and extremality}   

\begin{thm}\label{thm:uniq+extremal}
Let $(X,\nu)$ be a non-elementary random holomorphic dynamical system on a compact K\"ahler surface $X$,
satisfying the moment condition \eqref{eq:moment}. Then for 
  $\mu_\fr$-almost every point $\underline a\in \Lim(\Gamma)$, the following properties hold:
\begin{enumerate}[{\em (1)}]
\item there is a unique nef and isotropic class $ {a} \in H^{1,1}(X;\R)$ of mass $1$ with $\P( {a})=\underline a$;
\item the convex set $\Cur({a})$ is a singleton $\{ T_{ {a}}\}$; 
\item the class $\underline a$ is an extremal point of $\P(\Kahbar(X))$ and of $\P(\Psef(X))$; 
\item the current $T_{{a}}$ is   extremal   in the convex set of closed positive currents of mass $1$.
\end{enumerate}
\end{thm}

Combining this result with Lemma \ref{lem:definition_e}  and Equation \eqref{eq:representation_mufr}   we obtain 
the first and second assertions of the following corollary; the third assertion follows from the first one and the equivariance relation~\eqref{eq:equivariance_e}.

\begin{cor}\label{cor:Tsomega}
The following properties are satisfied for $\nu^\N$-almost every $\omega$:
\begin{enumerate}[\em (1)]
\item there exists a unique closed positive  current $T_\omega^s$ in the cohomology class $e(\omega)$;  
\item for every K\"ahler form $\kappa$,
$$
  \unsur{\M\lrpar{ (f_\omega^n)^* \kappa }} (f_\omega^n)^* \kappa \tendvers T_\omega^s. 
$$
\item the currents $T^s_\omega$ satisfy the equivariance property
$$
  (f_\omega)^*T_{\sigma(\omega)}^s=\frac{\M((f_\omega)^*T_{\sigma(\omega)}^s)}{\M(T_\omega^s)}T_\omega^s= {\M((f_\omega)^*T_{\sigma(\omega)}^s)}T_\omega^s.
$$
\end{enumerate}
\end{cor}

\begin{proof}[Proof of Theorem \ref{thm:uniq+extremal}]
The first and third properties were already established, respectively in 
Lemma~\ref{lem:limitset_in_NS} and~\ref{lem:limitset_in_nef} and Proposition \ref{pro:extremality}.  
Property (4) follows from (2) and (3). It remains to prove (2). 
For this, we denote by $\underline{f}^*$ the projective action of $f^*$  on $\P H^{1,1}(X;\R)$.
For $\underline a\in \Lim(\Gamma)$, let us  set ${\mathrm{diam}}\lrpar{\underline a}=\diam(a)$, 
where $a$ is the unique pseudo-effective class of mass~$1$ such that  $\P( {a})=\underline a$; 
this defines a measurable function on $\Lim(\Gamma)$, by Lemma \ref{lem:diam_measurable}. 
Our purpose is to show that ${\mathrm{diam}}\lrpar{\underline a} =0$ for $\mu_\fr$-almost every $\underline{a}$.
The stationarity of $\mu_\fr$ reads 
\begin{equation}
\int {\mathrm{diam}}\lrpar{\underline {a}} \, d\mu_\partial\lrpar{\underline a}   = 
 \int\!\! \int   {\mathrm{diam}}\lrpar{
{\underline{f}}^*\lrpar{\underline {a}}} \, d\nu(f) d\mu_\partial\lrpar{\underline a} 
\end{equation} 
and iterating  
this relation gives
\begin{equation}
\int {\mathrm{diam}}\lrpar{\underline {a}} \, d\mu_\partial\lrpar{\underline a} 
= \int {\mathrm{diam}}\lrpar{
\underline{f}_n^*\cdots  \underline{f}_1^*\lrpar{\underline {a}}} \, d\nu(f_1)\cdots d\nu(f_n) d\mu_\partial\lrpar{\underline a} 
\end{equation} 
(notice the order of compositions chosen here). Since the  diameter is upper-semicontinuous it is 
 uniformly bounded on $\Lim(\Gamma)$.  So, if we prove that  
 \begin{equation}\label{eq:lim_diameters}
\lim_{n\to +\infty} {\mathrm{diam}}\big({
 \underline{f}_n^*\cdots  \underline{f}_1^*\lrpar{\underline {a}}}\big) = 0
\end{equation} 
 for $\nu^\N$-almost every $(f_n)$ and every  $\underline{a}$, then we can apply the dominated convergence theorem to infer that 
 ${\mathrm{diam}}\lrpar{\underline a} =0$ $\mu_\fr$-almost surely.  To derive the convergence~\eqref{eq:lim_diameters},
note that
\begin{equation}\label{eq:diam_jac} 
{\mathrm{diam}}\lrpar{
\underline{f}_n^*\cdots  \underline{f}_1^*\lrpar{\underline {a}}} = \frac{\diam\lrpar{f_n^*\cdots f_1^* a}}
{\M\lrpar{{f_n^*\cdots f_1^* a}}} 
\end{equation}
because $\diam$ is homogeneous.
Applying Lemma \ref{lem:diam_f*a}  and the multiplicativity of the Jacobian we  get that 
\begin{equation}
{\mathrm{diam}}\lrpar{\underline{f}_n^*\cdots  \underline{f}_1^*\lrpar{\underline {a}}}
 \leq \frac{C\log \lrpar{C \norm{\jac(f_1\circ \cdots \circ f_n)\inv}_\infty}}{\M\lrpar{{f_n^*\cdots f_1^* a}}} \leq  C \frac{ \sum_{i=0}^{n-1} \log{\norm{f_i\inv}_{C^1}}}{\M(f_n^*\cdots f_1^* a)}. 
\end{equation}
We  conclude with two remarks. Firstly, the moment condition~\eqref{eq:moment} implies that the sequence $\unsur{n} \sum_{i=0}^{n-1} \log{\norm{f_i\inv}_{C^1}}$ is almost surely bounded. Secondly,  Lemma \ref{lem:growth2} shows that
 $\M(f_n^*\cdots f_1^* a)$ goes exponentially fast to infinity for $\nu^\N$-almost every $\omega=(f_n)$ (this is where the order of compositions matters). 
 Thus ${\mathrm{diam}}\big({
\underline{f}_n^*\cdots  \underline{f}_1^*\lrpar{\underline {a}}}\big) \to 0$ almost surely, and we are done. 
\end{proof}

\begin{rem}\label{rem:measurability}
The uniqueness of $T_a$ in its cohomology class implies that $T_a$ depends measurably on $a$.
 Indeed there is a set  $E\subset \Lim(\Gamma)$ of full 
measure along which the map $\underline a\mapsto T_a$ 
is continuous (recall that the space $\Cur_1(X)$ of positive closed currents of mass 1 on $X$ is a 
compact metrizable space). This implies that $\underline a\mapsto T_a$
is a measurable map from $\Lim(\Gamma)$, endowed with 
the $\mu_\fr$-completion of the Borel $\sigma$-algebra,   to 
$\Cur_1(X)$, endowed with its Borel $\sigma$-algebra. 
\end{rem}

\subsection{Continuous potentials}

We now study the limit currents $T^s_\omega$ introduced in Corollary \ref{cor:Tsomega}. 

\begin{thm} \label{thm:Tsomega_continuous}
Let $(X,\nu)$ be a non-elementary random holomorphic dynamical system on a compact K\"ahler surface $X$,
satisfying the moment condition~\eqref{eq:moment}. Then for $\nu^\N$-a.e. $\omega$ the current 
$T^s_\omega$ has continuous potentials. 
\end{thm}

 
\begin{lem}\label{lem:bounded_potentials_subsequence}
Let $\kappa$ be any K\"ahler form on $X$.  For $\nu^\N$-almost every $\omega$, there exists an increasing   sequence of integers 
$(n_i)_{i\geq 0}=(n_i(\omega))$ 
such that 
\begin{enumerate}[\em (1)]
\item the potentials of the pull-back currents
$\M((f_\omega^{n_i})^*\kappa)^{-1}u_{\lrpar{f_\omega^{n_i}}^*\kappa}$ are uniformly bounded;
\item the same holds for the push-forward  currents
$\M((f_\omega^{n_i})_*\kappa)^{-1}u_{(f_\omega^{n_i})_*\kappa}$. 
\end{enumerate}
If the exponential moment condition \eqref{eq:exponential_moment} holds, these assertions hold for all 
$n$ (i.e. extracting a subsequence $(n_i)$ is not necessary); 
in addition  the function $\omega\mapsto \log^+\norm{u_{T^s_\omega}}_\infty $ is $\nu^\N$-integrable. 
 \end{lem}


\begin{proof}[Proof of the Lemma] Recall the  notation $\omega=(f_n)_{n\geq 0}$. First, 
\begin{align}
 f_{n-1}^*\kappa & =   f_{n-1}^*\Theta(\kappa)+dd^c\left( u_\kappa\rond f_{n-1}\right) \\
&\notag =  \Theta(f_{n-1}^*\kappa)+dd^c\left( u_{f_{n-1}^*\Theta(\kappa)}+u_\kappa\rond f_{n-1} \right)
\end{align}
(For the moment, we do not introduce the constants $E(f_n; \kappa)$ in the computation). We obtain 
\begin{align*}
f_{n-2}^*f_{n-1}^*\kappa  &= f_{n-2}^*\Theta(f_{n-1}^*\kappa)+dd^c\left( u_{f_{n-1}^*\Theta(\kappa)}\rond f_{n-2}+u_\kappa\rond (f_{n-1}\rond f_{n-2}) \right)  \\
&=  \Theta(f_{n-2}^*f_{n-1}^*\kappa)+dd^c\left( u_{ f_{n-2}^*\Theta(f_{n-1}^*\kappa)}+u_{f_{n-1}^*\Theta(\kappa)}  \rond f_{n-2}+u_\kappa\rond (f_{n-1}\rond f_{n-2}) \right)  .
\end{align*}
Setting $G_{j,k}=f_{k-1}\!\circ \cdots \circ \! f_j$, for $ j\leq k-1$, (so in particular $G_{0,j} = f_\omega^j$ for all $j\geq 1$) 
and $G_{j,j}=\id_X$,  we get 
\begin{align}
(f_\omega^n)^*\kappa &= \Theta((f_\omega^n)^*\kappa) + dd^c\left( u_\kappa\rond f_\omega^n+ \sum_{j=0}^{n-1} u_{f_j^*\Theta(G_{j+1,n}^*\kappa)}\circ G_{0,j} \right).
\end{align}
Let $u_n$ denote the function in the parenthesis. We want to estimate the sup-norm $\norm{ u_n}_\infty$. 
Lemma \ref{lem:utheta} and Corollary \ref{cor:utheta}  provide successively the following upper bounds
\begin{align}
\norm{ u_{f_j^*\Theta(G_{j+1,n}^*\kappa)}}_\infty    
 \leq C \norm{ f_j }_{C^1}^2\M(G_{j+1,n}^*\kappa) 
  \leq  C  \M(\kappa) \norm{ f_j}_{C^1}^2 \norm{ G_{j+1,n}^*}, 
\end{align}
\begin{equation} \label{eq:sum_potential}
 \norm{ \unsur{\M((f_\omega^{n})^*\kappa)}u_{n}}_\infty 
  \leq \frac{\norm{ u_\kappa}_\infty}{\M((f_\omega^{n})^*\kappa)}+ C\M(\kappa) 
  \sum_{j=0}^{n-1}\norm{ f_j}_{C^1}^2
  \frac{\norm{ G_{j+1,n}^*}}{\M((f_\omega^{n})^*\kappa)}.
\end{equation}
To estimate this sum  we apply  Theorem \ref{thm:GK_good_times}  to the subadditive cocycle 
$N(n,\omega)=\log \norm{ (f_\omega^n)^*}$, as we did for
Corollary~\ref{cor:GK_good_times}:  there exists a sequence $(\delta_j)$ 
of positive  numbers converging to $0$, an increasing
 sequence $n_i=n_i(\omega)$ of integers,  and a constant $C'(\omega)$ 
 such that 
\begin{equation}
\frac{\big\|{ G_{j+1,n_i}^*}\big\|}{\M((f_\omega^{n_i})^*\kappa)}\asymp \frac{\big\|{ f_{j+1}^*  \cdots   f_{n_i-1}^*}\big\|}{\norm{ f_{0}^*  \cdots   f_{n_i-1}^*}}\leq C' \exp(-(\lambda_1-\delta_j)j)
\end{equation}
for all $i\geq 1$ and all $0\leq j\leq n_i$.
 Fix any real number $\e$ with $0< \e < \lambda_1$. 
Then from Lemma~\ref{lem:borel-cantelli_moment}, 
we know that,  for almost every $\omega$, there is a constant $C''(\omega)$ 
such that $\norm{ f_j}_{C^1}^2\leq C''\exp(\e j)$.
So from \eqref{eq:sum_potential} we get 
\begin{equation}  
 \norm{ \unsur{\M((f_\omega^{n_i})^*\kappa)}u_{n_i}}_\infty    \leq \frac{\norm{ u_\kappa}_\infty}{\M((f_\omega^{n_i})^*\kappa)}+ 
  C'''(\omega)  \M(\kappa)\sum_{j=0}^{n_i-1}\exp(-(\lambda_1-\e-\delta(j))j)  
\end{equation}
This inequality shows that  $\norm{ \M((f_\omega^{n_i})^*\kappa)^{-1}u_{n_i}}_\infty $ 
is uniformly bounded. 

Now,  note that $u_{(f_\omega^n)^*\kappa}=u_n+E_n$ with $E_n   = - \int u_n \vol$. 
Since $\norm{ \M((f_\omega^{n_i})^*\kappa)^{-1}u_{n_i}}_\infty $ is uniformly bounded, 
so is $\M((f_\omega^{n_i})^*\kappa)^{-1} E_{n_i}$, and the first assertion of the lemma is established. 

The second assertion is proved exactly in the same way, except that the  expressions
 of the form ${f_j^*\Theta(G_{j+1,n}^*\kappa)}$ must be replaced 
by ${(f_{n-j}\inv)^*\Theta( (f_0\inv\rond \cdots \rond f_{n-j-1}\inv)^*\kappa)}$; then we use
the second estimate in Corollary~\ref{cor:GK_good_times}, and 
the fact that for every $f\in \Aut(X)$, $\norm{f^*}\asymp \norm{(f\inv)^*}$.

If  the exponential moment condition \eqref{eq:exponential_moment} holds, we follow the same 
argument and apply Proposition 
\ref{pro:gouezel} -- instead of Theorem \ref{thm:GK_good_times} --
 to~\eqref{eq:sum_potential}, with $D(f) = \norm{f}_{C^1} ^2$. \end{proof}

\begin{proof}[Proof of Theorem~\ref{thm:Tsomega_continuous}] 
First, we prove that {\emph{the normalized potential $u_{T_\omega^s}$ is bounded, for $\nu^\N$-almost every $\omega$}}.
To see this, recall that $\M((f_\omega^n)^*\kappa)^{-1} (f_\omega^n)^*\kappa$ converges to $T_\omega^s$ as $n\to \infty$. 
From Lemma~\ref{lem:bounded_potentials_subsequence}, we know that the normalized potentials 
$\M((f_\omega^{n})^*\kappa)^{-1}u_{\lrpar{f_\omega^{n}}^*\kappa}$ of the currents $\M((f_\omega^n)^*\kappa)^{-1} \lrpar{f_\omega^n}^*\kappa$ are uniformly bounded along some subsequence $n_i=n_i(\omega)$. These potentials are 
$A\kappa_0$-psh functions on $X$ so, by compactness, they converge to  $u_{T_\omega^s}$ in $L^1(X;\vol)$. Thus, 
$u_{T_\omega^s}$ is essentially bounded. We conclude that  $u_{T_\omega^s}$ is bounded because quasi-plurisubhar\-monic functions are  
upper semi-continuous and have a value (in $\R\cup\{-\infty\}$) at every point.

Now,  we show  that {\emph{$u_{T_\omega^s}$ is continuous}}. 
Here, the argument is similar to the one used to prove Theorem \ref{thm:uniq+extremal}.
If $T$ is a positive closed current with bounded potential on $X$, we define 
\begin{equation}
{\mathrm{Jump}}(T)=\max_{x\in X} \left(\limsup_{y\to x} u_{T}(y)- \liminf_{y\to x} u_{T}(y)\right).  
\end{equation}
Then  $0\leq {\mathrm{Jump}}(T)\leq 2\norm{ u_{T}}_\infty$, and 
$\mathrm{Jump}(T)=0$ if and only if $u_{T}$ is continuous. In addition 
$\mathrm{Jump}(f^*T) = \mathrm{Jump}(T)$ for every $f\in \Aut(X)$
because $f^*T=\Theta(f^*a) + dd^c(u_{f^*\Theta(a)}  + u_T\circ f)$ and  $u_{f^*\Theta([T])}$ is continuous (see Equation~\eqref{eq:f*T}). From the equivariance relation 
\begin{equation}\label{eq:equivariance_Tsn}
T^s_\omega = \unsur{\M\lrpar{{\lrpar{f_\omega^n }^* T^s_{\sigma^n\omega}}}} T^s_{\sigma^n\omega},
\end{equation}
which follows from the third assertion of Corollary~\ref{cor:Tsomega}, we get  
\begin{equation}
{\mathrm{Jump}}\lrpar{T^s_\omega}=  \unsur{\M\lrpar{{\lrpar{f_\omega^n }^* T^s_{\sigma^n\omega}}}} {\mathrm{Jump}}\lrpar{T^s_{\sigma^n\omega}}. 
\end{equation}
Remark \ref{rem:measurability} says that $\omega\mapsto T^s_\omega$ is measurable; hence, $\omega\mapsto u_{T^s_\omega}$ is measurable. If $C$ is  large enough,  the first step of the proof gives a subset $\Omega_C\subset \Omega$ such that $\nu(\Omega_C)>0$ 
and $\norm{ u_{T^s_\omega}}_\infty\leq C$   for all $\omega\in \Omega_C$. 
By ergodicity of the shift, $\sigma^n\omega\in \Omega_C$ for almost every $\omega$ and infinitely many $n$; for such an $n$,  $\big\Vert u_{T^s_{\sigma^n\omega}} \big\Vert_\infty\leq C$ and ${\mathrm{Jump}}\lrpar{T^s_{\sigma^n\omega}}\leq 2C$. 
By Lemma \ref{lem:variant_furstenberg}, 
$\M\lrpar{{\lrpar{f_\omega^n }^* T^s_{\sigma^n\omega}}}$ goes to infinity almost surely.
So,  ${\mathrm{Jump}}\lrpar{T^s_\omega}=0$, and the proof is complete.\end{proof}


\begin{thm}\label{thm:holder}
Let $(X,\nu)$ be a non-elementary random holomorphic dynamical system on a compact K\"ahler surface $X$,
satisfying the exponential moment condition \eqref{eq:exponential_moment}. Then there exists $\theta>0$ such that 
for $\nu^\N$-almost every $\omega$ the potential  $u_{T^s_\omega}$ is H\"older  continuous of exponent~$\theta$. 
\end{thm}

\begin{vlongue}
The proof is a variation on the following well-known fact, applied to $u  = u_{T_\omega^s}$: let $u_n$ be a sequence of continuous functions converging uniformly to $u:M\to \R$ on some metric space $M$. If $\norm{u_n-u}_\infty\leq A^n$ and $\Lip(u_n)\leq B^n$
with $A<1<B$, then $u$ is a H\"older continuous function for the exponent $\alpha = -\log(A)/(\log(B)-\log(A))$. 
\end{vlongue}

\begin{proof}
The initial computations are similar (but not identical)  to those used to reach
Lemma~\ref{lem:bounded_potentials_subsequence}. 
Keeping the notation $G_{j,n}=f_{n-1}\circ \cdots \circ f_j$, a descending 
induction starting from 
\begin{equation}
f_{n-1}^* T_{\sigma^n \omega}^s   = \Theta(f_{n-1}^*T_{\sigma^n \omega}^s)
+dd^c\lrpar{u_{ f_{n-1}^* \Theta(T_{\sigma \omega}^s)} + u_{T_{\sigma^n \omega}^s} \circ f_{n-1}}
\end{equation}
yields
 \begin{align}
 (f_\omega^n)^*T_{\sigma^n \omega}^s &= 
 \label{eq:enorme}  \Theta\lrpar{(f_\omega^n)^*T_{\sigma^n \omega}^s}  +  dd^c\lrpar{\sum_{j=0}^{n-1} u_{f_j^*\Theta(G_{j+1, n}^* T_{\sigma^n \omega}^s)} \circ f_\omega^{j} 
 +  u_{T_{\sigma^n \omega}^s} \circ f_\omega^{n}
  }. 
\end{align}
Thus, there is a constant of normalization $E=E(\omega; n)$ such that 
\begin{equation} \label{eq:uTs_final}
u_{T^s_\omega} = \unsur{\M((f_\omega^n)^*( T_{\sigma^n\omega}^s))} \lrpar{\sum_{j=0}^{n-1} u_{f_j^*\Theta(G_{j+1, n}^* T_{\sigma^n \omega}^s)} \circ f_\omega^{j} 
 +  u_{T_{\sigma^n \omega}^s} \circ f_\omega^{n}
  }+ E.
\end{equation} 
Note that the additional term $E$ does not affect the modulus of continuity of $u_{T^s_\omega}$.
Since $\Lip(f_j)\leq \norm{f_j}_{C^1}$ for all $j$,  Lemma~\ref{lem:utheta}  and Corollary~\ref{cor:utheta} imply 
$\Lip (u_{f_j^*\Theta(a)})\leq  C  \norm{f_j}_{C^1}^2 \M(a)$ for every class $a\in \Psef(X)$; hence 
\begin{align}\Lip\lrpar{u_{f_j^*\Theta(G_{j+1, n}^* T_{\sigma^n \omega}^s)}}& \leq C  \norm{f_j}_{C^1}^2 \M(G_{j+1, n}^* T_{\sigma^n \omega}^s) \leq 
C  \norm{f_j}_{C^1}^2 \norm{G_{j+1, n}^* } \\&\leq C  \norm{f_j}_{C^1}^2 \prod_{\ell=j+1}^{n-1} \norm{f_\ell^*}_{H^{1, 1}}
\leq C   \prod_{\ell=j}^{n-1} \norm{f_\ell}^2_{C^1} .
\end{align}
Finally, since $1\leq \Lip(f_j)$ for every $0\leq j\leq n-1$, we obtain
\begin{align}
\Lip\lrpar{u_{f_j^*\Theta(G_{j+1, n}^* T_{\sigma^n \omega}^s)}\circ f_\omega^j }   
& \leq \Lip\lrpar{u_{f_j^*\Theta(G_{j+1, n}^* T_{\sigma^n \omega}^s)}} \prod_{\ell=0}^{j-1} \Lip (f_\ell)  \leq  C \prod_{\ell=0}^{n-1} \norm{f_\ell}_{C^1}^2. 
 \end{align}
Denoting  the modulus of continuity by $\moco(u,r) = \sup_{d(x,x')\leq r} \abs{u(x)  - u(x')}$, 
we infer from Equation~\eqref{eq:uTs_final} that 
\begin{equation}
\moco({u_{T^s_\omega}, r} )\leq    \unsur{\M\lrpar{(f_\omega^n)^*( T_{\sigma^n\omega}^s)}}
 \lrpar{  C n  \prod_{\ell=0}^{n-1} \norm{f_\ell}_{C^1}^2 \cdot r+ \big\|{u_{T_{\sigma^n \omega}^s}}\big\|_\infty}. 
\end{equation}
To ease notation set $\lambda = \lambda_{H^{1,1}}$. Fix a small $\e>0$. 
By Lemma \ref{lem:variant_furstenberg},   
for almost every $\omega$ there exists $C  = C_\e(\omega)$ such that 
$\M\lrpar{(f_\omega^n)^*( T_{\sigma^n\omega}^s)}^{-1} \leq C e^{-n(\lambda-\e)}$ for every $n$. 
Fix $M$ larger than but close to $\exp\lrpar{\ee \lrpar{\log\norm{f}_{C^1}}}$.
Applied to the $\nu^\N$-integrable function $\omega=(f_n)\mapsto \log\norm{f_0}_{C^1}$, the 
Birkhoff ergodic theorem  gives 
\begin{equation}
\prod_{\ell=0}^{n-1} \norm{f_\ell}_{C^1}^2 \leq CM^n  
\;  \text{ as  well as } \; 
n\prod_{\ell=0}^{n-1} \norm{f_\ell}_{C^1}^2 \leq CM^n
\end{equation} 
for some $C=C_M(\omega)$ (increase $M$ to deduce the  second inequality from the first). Thus,
\begin{equation}
\moco( {u_{T^s_\omega}, r}) \leq  C_1   e^{-n(\lambda -\e)} \lrpar{M^n r+ 
\big\|{u_{T_{\sigma^n \omega}^s}}\big\|_\infty}
\end{equation}
for some $C_1>0$.
By Lemma~\ref{lem:bounded_potentials_subsequence}, $\omega\mapsto \log^+\norm{u_{T^s_\omega}}_\infty$ 
is integrable,  so for almost every $\omega$ there exists $C_2= C_\e(\omega)$ 
such that    $\big\|{u_{T_{\sigma^n \omega}^s}}\big\|_\infty \leq C_2  e^{\e n}$ holds for all $n$,
 and we infer that 
\begin{equation}
\moco({u_{T^s_\omega}, r}) \leq  C_3  e^{-n  (\lambda- \e)} (M^nr+e^{\e n})  = 
C_3  e^{-n (\lambda- 2\e)}\lrpar{(Me^{-\e})^nr+1}.
\end{equation}
Choosing $n$ so that 
   $r \asymp  (Me^{-\e})^{-n}$ we get $\moco({u_{T^s_\omega}, r}) \leq  C_4  r^\theta$
with $\theta = \frac{\lambda-2\e}{ \log M+\e}$ and 
the  proof of the  theorem  is complete.
\end{proof}

\section{Glossary of random dynamics, II}\label{sec:Glossary_II}

 In this section we consider a random holomorphic dynamical system $(X, \nu)$ on a compact K\"ahler surface, 
satisfying the moment  condition \eqref{eq:moment}. 
Our goal is to collect a number of facts from the ergodic theory of random dynamical systems, including the construction of 
associated skew products, fibered entropy and Lyapunov exponents of stationary measures, stable and unstable manifolds, 
and various measurable partitions.  Here the group $\Gamma_\nu$ may a priori be elementary; 
also, the compactness assumption on   $X$ can be dropped  in most of these results 
(in this case  \eqref{eq:moment} should be strengthened to a $C^2$-moment condition).
Since some subsequent arguments   rely on the work \cite{br} of Brown and Rodriguez-Hertz,  
we have tried to make notation consistent with  that paper as much as possible.  

\subsection{Skew products and stationary measures associated to $(X, \nu)$}\label{par:definition_skew_products}
Define:
\begin{itemize}
\item  $\Omega = \Aut(X)^\N$,  whose  elements are denoted by $\omega =  (f_n)_{n\geq 0}$. 
On $\Omega$, the one-sided  shift is denoted by $\sigma\colon \Omega \to \Omega$. 
\item   $\Sigma = \Aut(X)^\Z$,  whose  elements are denoted by $\xi =  (f_n)_{n\in \Z}$. On $\Sigma$, the two-sided shift 
 is denoted by $\vartheta\colon \Sigma\to \Sigma$. 
\item $\X =  \Sigma  \times X$ and $\X_+  = \Omega  \times X$, whose  elements are denoted by 
$\cx = (\xi, x)$ and  $\cx = (\omega, x)$ respectively. The natural projections are  denoted by
$\pi_\Sigma:\X\to \Sigma$ (resp. $\pi_\Omega:\X_+\to \Omega$) and $\pi_X:\X \to X$ (resp. $\pi_X:\X_+\to X$, using the same notation).
\end{itemize}
Recall that the product $\sigma$-algebra on $\Omega$ (resp. $\Sigma$) is generated by {\bf{cylinders}} (\footnote{Cylinders are   
products $C=\prod C_j$ of Borel sets, all of which are equal to $\Aut(X)$ except finitely many of them. For simplicity, we denote a cylinder by $C=\prod_{j=0}^N C_j$ if $C_k=\Aut(X)$ for $\vert k\vert >N$.}), and that it 
coincides with the Borel $\sigma$-algebra $\mathcal B(\Omega)$ (resp. $\mathcal B(\Sigma)$) (see \cite[Lem. 6.4.2]{bogachev_book}).

\subsubsection{Skew products} For $\omega\in \Omega$ and $n\geq 1$,  $f^n_\omega$   
is the left composition  $f^n_\omega= f_{n-1}\circ \cdots \circ f_0$;  in 
particular, $f^1_\omega=f_0$ (see \S~\ref{par:random_compositions1}).  For  $n=0$, 
we set $f^0_\omega  = \mathrm{id}$. This is consistent with the notation used in the previous sections. 
The same notation $f_\xi^n$ is used for $\xi\in \Sigma$ and $n\geq 0$. When
  $n<0$, we set 
 $f^n_\xi=(f_{n})\inv \circ\cdots \circ (f_{-1})\inv$. With this definition the cocycle formula 
 $f^{n+m}_\xi  =  f^n_{\vartheta^m\xi}\circ f^m_\xi$ holds for all $(m,n) \in \Z^2$ and $\xi\in \Sigma$. 
By definition, the skew products induced by the random dynamical system $(X, \nu)$ 
are the  transformations $ \F_+\colon  \X_+  \to \X_+$ and $\F \colon  \X   \to \X $ defined by 
\begin{vcourte}
 \begin{equation}
F_+\colon (\omega, x) \longmapsto (\sigma\omega, f^1_\omega(x)) \quad \quad {\text{and}}  \quad \quad
 \F \colon \; (\xi, x) \longmapsto (\vartheta\xi, f^1_\xi(x)). 
\end{equation} 
\end{vcourte}
\begin{vlongue}
 \begin{align}
F_+\colon (\omega, x) &\longmapsto (\sigma\omega, f^1_\omega(x)) \\
 \F \colon \; (\xi, x) &\longmapsto (\vartheta\xi, f^1_\xi(x)). 
\end{align} 
\end{vlongue}
If $\varpi\colon \X \to \X_+$ denotes the natural projection,  
then $\varpi\circ \F = \F_+ \circ \varpi$.
Note that $F$ is invertible, with $F^{-1}(\cx)=(\vartheta^{-1}\xi, f^{-1}_{\vartheta^{-1}\xi}(x))$, but $F_+$ is not;
 indeed 
{\emph{$(\X, \F)$ is the natural extension of $(\X_+, \F_+)$}}.

\begin{lem} \label{lem:stationary}
The   measure $\mu$ on $X$ is  stationary if and only if 
 the product measure \[\m_+:= \nu^\N\times \mu\] on $\X_+$ is invariant under $\F_+$. 
\end{lem}

\begin{vlongue}
\begin{proof}[Proof of Lemma~\ref{lem:stationary}]
The invariance of $\m_+$ is equivalent to the equality
\begin{equation} \label{eq:stationary}
\m_+(\F_+\inv (C\times A))  = \m_+ (C\times A)  = \left( \prod_{j=0}^N \nu(C_j) \right) \cdot  \mu(A),
\end{equation} 
for all cylinders  $C = C_0\times \cdots \times C_N$ in $\Omega$ and 
Borel sets $A\subset X$.
By definition 
\begin{equation}
\F_+\inv (C\times A)  = \set{(\omega, x)\in \Omega\times X\; ;  \ f_{N+1}\in  C_{N} , \ldots , f_1\in C_{0}, \romain{typos ici}
f_0(x)\in A}, 
\end{equation}
so  clearly it is enough to check \eqref{eq:stationary} for $N=1$.\romain{en fait on fait plutôt $N=-1$ ici, i.e.$C = \Omega$}  Now by Fubini's theorem 
\begin{align} \label{eq:fubini}
(\nu\times \mu)\lrpar{\set{ (f_0, x) \; ; \  f_0(x)\in A }} 
&\notag= \int\int \mathbf{1}_{f_0\inv (A)}(x) \, d\nu(f_0)\, d\mu(x)  \\ &= \int\int \mu(f_0\inv (A))\, d\nu(f_0) 
\end{align} 
and the result follows. 
\end{proof}
\end{vlongue}

A stationary measure is said
{\bf{ergodic}} if it is an extremal point in the convex set  of stationary measures; hence, $\mu$ is ergodic if and only if
$\m_+$ is $\F_+$-ergodic. Actually  $\mu$ is ergodic if and only if every $\nu$-almost surely invariant measurable  subset $A\subset X$  (that is a measurable subset such that for $\nu$-almost every $f$, $\mu(A\Delta f\inv(A)) =0$) 
has measure $\mu(A)=0$ or 1. This is by no means 
 obvious since 
$\F_+$-invariant sets have no reason to be of product type. This statement is part of
 the so-called \textbf{random ergodic theorem} (see Propositions 1.8 and 1.9 in \cite{benoist-quint_book}). 





\begin{pro}\label{pro:m} 
There exists a unique $\F$-invariant probability measure $\m$ on $\X$ projecting on $\m_+$ under the natural 
projection $\X\to \X_+$. Moreover,
\begin{enumerate}[\em (1)]
\item the measure $\m$ is equal to the weak-$\star$ limit 
\[
\m  = \lim_{n\to\infty} (\F^n)_\varstar \lrpar{\nu^\Z\times \mu}.
\]

\item the projections $(\pi_\Sigma)_*\m$ and $(\pi_X)_*\m$ are respectively equal to $\nu^\Z$ and $\mu$;

\item the equality $\m = \nu^\Z\times \mu $ holds if and only if 
$\mu$ is   $f$-invariant for $\nu$-almost every $f$;

\item $(\X, \F, \m)$ is ergodic if and only if $(\X_+, \F_+, \m_+)$ is. 
\end{enumerate}
\end{pro}

The existence and uniqueness of $\m$, as well as the characterization of its ergodicity, follow from the fact that 
$(\X, \F)$ is the natural extension of $(\X_+, \F_+)$ (see~\cite[\S 1.2]{kifer} for a detailed explanation). 
\begin{vcourte}
See \cite[\S I.1]{liu-qian} for the proof of Assertions (1), (2) and (3).
\end{vcourte}
\begin{vlongue}
\begin{proof}[Proof of (1), (2), (3)] Let us prove directely that the limit in (1) does exist, and show that this limit $\m$ satisfies (2) and (3). 
Since $\varpi_\varstar\lrpar{\nu^\Z\times \mu}  =  \nu^\N\times \mu = \m_+$ and
$\varpi\circ \F = \F_+ \circ \varpi$, the $\F_+$-invariance of  $ \m_+$ gives
$\varpi_\varstar(\F^n)_\varstar \lrpar{\nu^\Z\times \mu} = \m_+$ for every $n\in \Z$. 
So if we prove that the limit $\lim_{n\to\infty} (\F^n)_\varstar \lrpar{\nu^\Z\times \mu} $ exists, then this limit $\m$
will be an $\F$-invariant probability measure 
projecting on $\m_+$ under $\varpi$; hence it will coincide with the invariant measure $\m$. 
  
To prove  this convergence, we consider a cylinder 
$C = \prod_{j=-N}^N C_j$ in $\Sigma$ and a Borel set $A\subset X$, and we
show that $(\nu^\Z\times \mu )(\F^{-n}(C\times A))$ 
stabilizes for $n>N$. Arguing as in Lemma~\ref{lem:stationary}, we see that 
the set $\F^{-n}(C\times A)$ is equal to the set of points $\cx = (\xi, x)$ satisfying the constraints $(\theta^n\xi)_j\in C_j$ for $-N\leq j\leq N$ and 
$x\in (f_\xi^n)^{-1}(A)$; for $n>N$, these constraints are independent, \romain{remarque Johan: non}
and $\lrpar{\nu^\Z\times \mu}(\F^{-n}(C\times A))$ 
is equal to 
\begin{align}
\nu^\Z(\theta^{-n}(C))\times  \lrpar{ \nu^n\times\mu}\lrpar{
\set{ (f_0, \ldots , f_{n-1}, x) \; ;  \; \ f_{n-1}\circ \cdots  \circ f_0 (x)\in A }}.
\end{align}
Then the invariance of $\nu^\Z$ under the shift and the the stationarity of $\mu$ give (see Equation~\eqref{eq:fubini})
\begin{align}
\lrpar{\nu^\Z\times \mu}(\F^{-n}(C\times A)) &=  \nu^\Z(C) \times \int \mu\lrpar{f_0\inv\circ \cdots \circ f_{n-1}\inv A} \nu(f_0) \cdots \nu(f_{n-1}) \\
&= \notag\nu^\Z(C) \times \mu(A). 
\end{align}\romain{non il y a un problème: ça stabilise mais pas à $\nu^\Z(C) \times \mu(A)$ sinon on aurait 
$\sm  = \nu^\Z\times \mu$!  Le truc c'est que la partie négative du cylindre sélectionne un itinéraire précis (contrairement à la partie positive qui prend toutes les branches}
This proves Assertions~(1) and~(2).
For Assertion~(3) it will be  enough for us to consider the case  where $\Gamma$ is discrete. By Assertion~(1) we see that 
$\nu^\Z\times \mu$ is $F$-invariant if and only if $\m = \nu^\Z\times \mu$.
Now assume  $\m = \nu^\Z\times \mu $ and let us show that $\mu$ is $\Gamma_\nu$-invariant. The reverse implication is similar.
Fix $f_0 \in \supp(\nu)$  and consider the cylinder $C = C_0  = \set{f_0}$ (in $0^{\mathrm{th}}$ position). If $A\subset X$ is a Borel subset we have 
\begin{equation}\label{eq:CtimesA}
\lrpar{\nu^\Z\times \mu}(\F(C\times A)) = \lrpar{\nu^\Z\times \mu}(C\times A) = \nu\lrpar{C_0} \times \mu(A).
\end{equation} On the other hand 
$\F(C\times A) 
= \vartheta (C) \times f_0(A)$ so  the left hand side of \eqref{eq:CtimesA} is equal to $ \nu\lrpar{C_0} \times \mu(f_0(A))$. 
Thus,  $\mu(f_0(A)) = \mu(A)$, which proves that $\mu$ is $\Gamma_\nu$-invariant.  
\end{proof}
\end{vlongue}

\subsubsection{Past, future, and partitions}\label{par:past_future_partitions}
Let $\mathcal F$ denote the $\sigma$-algebra on $\cX$ obtained by taking the $\m$-completion of 
$\mathcal{B}(\Sigma)\otimes \mathcal{B}(X)$. 
It will often be important to detect objects depending only on the ``future'' or  on the ``past''. To formalize this, we define 
two $\sigma$-algebras on $\Sigma$: 
\begin{itemize}
\item $\hat{ \mathcal{F}}^+$ is the $\nu^\Z$-completion of the $\sigma$-algebra generated by the cylinders $C = \prod_{j=0}^N C_j$.
\item $\hat{\mathcal {F}}^-$ is the $\nu^\Z$-completion of the $\sigma$-algebra generated by the cylinders $C = \prod_{j=-N}^{-1} C_j$.  
\end{itemize}
To formulate it differently, we define {\bf{local stable and unstable sets}}  for the 
shift $\vartheta$:  
\begin{equation}\label{eq:definition_sigma_s_u}
\Sigma^s_\loc(\xi) = \set{\eta \in \Sigma\, ;  \ \forall i\geq 0, \ \eta_i = \xi_i} \; \text{ and } \; 
\Sigma^u_\loc(\xi) = \set{\eta \in \Sigma\, ;  \ \forall i< 0, \ \eta_i = \xi_i}. 
\end{equation}
Then a   subset of 
 $ \Sigma$ is $\hat{\mathcal {F}}^+$-measurable (resp. $\hat{\mathcal {F}}^-$ measurable) if, up to a set of zero 
$\nu^\Z$-measure, it is Borel and  saturated by local stable sets $\Sigma^s_\loc(\xi)$ (resp. unstable sets $\Sigma^u_\loc(\xi)$). 
The $\sigma$-algebra $\mathcal{F}^+$  on $\X$ will be the 
$\m$-completion of $\hat{ \mathcal{F}}^+\otimes \mathcal B(X)$.
An $\mathcal{F}^+$-measurable object should be understood
as  ``depending only on the future'', thus it  makes  sense on $\X$ and on $\X_+$. 
Actually $\mathcal{F}^+$ coincides with the completion of the pull-back of $\mathcal{B}(\X_+)$ under 
$\varpi:\X\to \X_+$. 
The $\sigma$-algebra $\mathcal F^-$ of ``objects depending only on the past'' is defined analogously. 
Consider the partition into the subsets $\mathcal F^-(\cx):=\Sigma^u_\loc(\xi)\times \set{x}$
(each of them can be naturally identified to~$\Omega$).
Then, modulo $\m$-negligible sets, the elements of  $\mathcal F^-$ are saturated by this partition.

For   $\xi\in \Sigma$ we set $X_\xi = \set{\xi}\times X =  \pi_\Sigma\inv(\xi)$, which can be  naturally identified with $X$ via 
$\pi_X$.  The disintegration of the probability measure $\m$ with respect to the partition into fibers of $\pi_\Sigma$ gives rise to 
a family of conditional probabilities $\m_\xi$ such that $\m=\int \m_\xi \, d\nu^\Z(\xi)$, because $(\pi_\Sigma)_*\m=\nu^\Z$.

\begin{lem}\label{lem:conditionals}
The conditional measure $\m_\xi$ on $X_\xi$ satisfies $\nu^\Z$-almost surely  
\begin{equation}\notag
\m_\xi = \lim_{n\to+\infty} (f_{-1}\circ \cdots \circ  f_{-n})_\varstar \mu = \lim_{n\to+\infty} (f^{n}_{\vartheta^{-n}\xi})_\varstar \mu.
\end{equation}
In particular, the family of measures $\xi\mapsto \m_\xi$ is $\mathcal{F}^-$-measurable.  
\end{lem}


\begin{vlongue}
\begin{proof} It follows from the martingale convergence theorem that the limit 
\begin{equation}
\tilde\mu_\xi:= \lim_{n\to+\infty} (f_{-1}\circ \cdots \circ  f_{-n})_\varstar \mu
\end{equation} 
exists almost surely
  (see e.g. \cite[\S 2.5]{benoist-quint_book} or \cite[\S II.2]{bougerol-lacroix}).  Now 
$\F^n$ maps $X_{\vartheta^{-n}\xi}$ to $X_\xi$ and  $F^n\rest {X_{\vartheta^{-n}\xi}} = f_{-1}\circ \cdots \circ  f_{-n}$, so 
\begin{equation}
\lrpar{(\F^n)_\varstar (\nu^\Z\times \mu)}(\; \cdot\; \vert X_\xi)
= (f_{-1}\circ \cdots \circ  f_{-n})_\varstar \mu. 
\end{equation} 
Identify $\tilde\mu_\xi$ with a measure 
on $X_\xi$. For every continuous function $\phi$ on $\X$ the dominated convergence theorem gives
\begin{align}
\lrpar{ (\F^n)_\varstar (\nu^\Z\times \mu)} (\varphi)
& = \int \lrpar{\int_{X_\xi}\varphi(x) \;  d(f_{-1}\circ \cdots \circ  f_{-n})_\varstar \mu(x)} d\nu^\Z(\xi) \\
& \tendvers 
\int \lrpar{\int_{X_\xi}\varphi(x) \; d\tilde\mu_\xi(x)}d\nu^\Z(\xi).
\end{align}
But $\lrpar{ (\F^n)_\varstar (\nu^\Z\times \mu)} (\varphi)$ converges to $\m(\varphi)$, and the marginal of $\m$ with respect to the projection $\pi_\Sigma\colon \X\to \Sigma$ is $\nu^\Z$, so we get the result. \end{proof}
\end{vlongue}

\begin{vcourte}
Indeed, the convergence is a consequence of  the martingale convergence theorem (see   \cite[\S 2.5]{benoist-quint_book} for details) 
and the  second assertion easily follows.
\end{vcourte}

 Since $\xi \mapsto \m_\xi$ is  $\mathcal F^-$-measurable, the conditional measures of $\m$ 
on the atoms ${\mathcal F^-}(\cx)=\Sigma^u_\loc(\xi)\times \set{x}$ of the partition generating $\mathcal F^-$ are induced by 
the lifts of the conditionals of $\nu^\Z$ on the  $\Sigma^u_\loc(\xi)$, via the natural projection $\pi_\Sigma:\X\to \Sigma $. In addition we can simultaneously identify $\Sigma^u_\loc(\xi)$  to $\Omega$   
and    $\nu^\Z(\; \cdot\; \vert \;   \Sigma^u_\loc)$   to $\nu^\N$. In this way  we get 
\begin{equation}\label{eq:conditional}
\m(\; \cdot\; \vert \;   \mathcal F^-(\cx))  \; =  \;  {\nu^\Z(\; \cdot\; \vert \;  \Sigma^u_\loc(\xi))}\times \delta_x 
\; \simeq  \; \nu^\N
\end{equation}
for $\m$-almost every $\cx = (\xi, x)\in \X$. This corresponds to Equation~(9)  in \cite{br}. By \cite[Prop. 4.6]{br}, this implies 
that  $\mathcal F^+\cap \mathcal F^-$ is 
 equivalent, modulo $\m$-negligible sets, to $\set{\emptyset, \Sigma}\otimes \mathcal B(X)$.
 
\subsection{Lyapunov exponents} \label{subs:lyapunov}
Let $\mu$ be a stationary measure for $(X, \nu)$; assume that $\mu$ (or equivalently $\m$ or $\m_+$) is ergodic.
The 
upper and lower Lyapunov exponents $\lambda^+\geq \lambda^-$ are respectively defined by the almost sure  limits 
\begin{equation}
\lambda^{+}   = \lim_{n\to\infty} \unsur{n} \log \norm{D_xf_\omega ^n} \; \text{ and } \; \lambda^{-}   =  
\lim_{n\to\infty} \unsur{n} \log \norm{\lrpar{D_xf_\omega^{n}}\inv }\inv;
\end{equation}
the existence of these limits is guaranteed by Kingman's subadditive ergodic theorem,
thanks to the moment condition~\eqref{eq:moment}, and the convergence also holds on average.
Let us now apply  the Oseledets theorem successively to the tangent  cocycle defined by the fiber dynamics 
$(\X_+, F_+, \m_+)$, and then to the cocycle associated to $(\X, F, \m)$. 

\subsubsection{The non-invertible setting} 
Define the  tangent bundles  $T\X_+ := \Omega \times TX$  and $T\X  := \Sigma  \times TX$, and
denote by $DF$ and $DF_+$ the natural tangent maps, that is $D_{(\xi, x)}F: \set{\xi}\times T_xX\to 
\set{\vartheta \xi}\times T_{f_\xi(x)}X$ is induced by $D_xf^1_\xi$:
\begin{equation}
D_{(\xi, x)}F(v)=D_xf_\xi^1(v) \quad (\forall v \in T_xX_\xi=T_xX)
\end{equation}
For the non-invertible dynamics on 
$\X_+$, the Oseledets theorem gives:  
for $\m_+$-almost every $(\omega, x)$, there exists 
a non-trivial complex subspace  $V^-(\omega, x)$ 
of $\set{\omega}\times T_xX$ such that 
\begin{align}
\forall v\in  V^-(\omega, x)\setminus\{0\}, \; \;  & \lim_{n\to + \infty}  \unsur{n}\log \norm{D_x f^n_\omega (v)}  = \lambda^{-} \\
\forall  v\notin V^-(\omega,x), \; \;   & \lim_{n\to + \infty}  \unsur{n}\log \norm{D_x  f^n_\omega (v)}  = \lambda^{+}. 
\end{align}
The field of subspaces $V^-$ is measurable and almost surely invariant. 
Two cases can occur: either $\lambda^{-}<\lambda^{+}$ and $V^-(\omega, x)$ is almost surely a complex line, or 
$\lambda^{-} = \lambda^{+}$ and $V^-(\omega, x)=\set{\omega}\times T_xX$. 
 
\subsubsection{The invertible setting} 
For the dynamical system $\F\colon \X\to \X$, the statement is: 
\begin{itemize}
\item if $\lambda^{-} = \lambda^{+}$ then for $\m$-almost every $\cx = (\xi,x)$, for  every non-zero $v\in  T_xX_\xi\simeq T_xX$, 
\begin{equation}
\lim_{n\to\pm \infty}  \unsur{n}\log \norm{D_x  f^n_\xi (v)}  = \lambda^{-};
\end{equation}
\item if $\lambda^{-} < \lambda^{+}$ then for $\m$-almost every $\cx$ there exists a 
decomposition  $ T_xX_\xi=E^-(\xi, x) \oplus  E^+(\xi, x)$ such that for $\star\in \set{-, +}$ and
 every $v\in E ^\star(\xi, x)\setminus \set{0}$,  
\begin{equation}
\lim_{n\to \pm\infty} \unsur{n}\log \norm{D_xf^n_\xi(v)}  = \lambda^\star.
\end{equation}
Furthermore the line fields $E^{\pm}$ are measurable and  invariant, and $\log\abs{\angle (E^-, E^+)}$ is integrable
(here, the ``angle'' $\angle(E^-(\cx),E^+(\cx))$ is the distance between the two lines $E^-(\cx)$ and $E^+(\cx)$ in $\P(T_{\cx}\X)$).  
\end{itemize}

\subsubsection{Hyperbolicity}\label{par:hyperbolicity_of_mu}
It can happen that $\lambda^{-}$ and $\lambda^{+}$ have the same sign. 
If $\lambda^-$ and $\lambda^+$ are both negative,  the conditional measures $\m_\xi$ are atomic: 
this can be shown by adapting   a classical Pesin-theoretic argument  (see e.g. \cite[Cor. S.5.2]{katok-hasselblatt})
to the fibered dynamics of $F$ on $\cX$ 
 (see \cite[Prop. 2]{lejan} for a  direct
  proof and an example where the  $\m_\xi$ have several atoms). 
Such random dynamical systems are called \textbf{proximal}. 
For instance, generic  random products of automorphisms of 
$\P^2(\C)$, that is of matrices in $\PGL(3, \C)$, are proximal; 
in such examples the stationary measure is   not invariant. Other examples are given by  contracting iterated 
function systems. 
 
When  $\lambda^+$ and $\lambda^-$ are both non-negative, 
we have the so-called  \textbf{invariance principle}:
  
\begin{thm}\label{thm:invariance_principle}
Let $(X, \nu)$ be a random holomorphic dynamical system satisfying the integrability condition 
\eqref{eq:moment}, and let $\mu$ be an ergodic stationary  measure.  If $\lambda^+(\mu) \geq \lambda^-(\mu)\geq 0$ 
then $\mu$ is almost surely invariant. 
\end{thm}

This result was proven by Crauel, 
building on ideas  of Ledrappier  \serge{Ici j'ai juste enlevé la référence interne à l'ancienne section 11}
(see Theorem 5.1, Corollary 5.3 and  Remark 5.6 in \cite{crauel}, 
and also Avila-Viana~\cite[Thm B]{avila-viana}). 
 
\begin{rem}\label{rem:invariance_principle_strict}
{\textit{If   $\lambda^-$ and $\lambda^+$ are both positive then $\mu$ is atomic}}.
Indeed, since $\mu$ is almost surely invariant we get $\m = \nu^\Z\times \mu$. Reversing time, the
Lyapunov exponents of $\m$ become negative, so
as explained above  the measures $\m_\xi$ are atomic. By invariance $\m_\xi = \mu$, so $\mu$ is atomic too.  
\end{rem} 

By definition, $\mu$ is {\bf hyperbolic} if $\lambda^{-} <0<\lambda^{+}$. In this case
we rather use the conventional superscripts $s/u$ instead of $-/+$ for stable and unstable objects. 
We also have $E^s = V^s$ in this case (and more generally when $\lambda^{-}<\lambda^{+}$); 
so, it follows that the complex line field $E^s$  on $T\X$ 
is $\mathcal F^+$-measurable. Conversely the unstable line field $E^u$ is $\mathcal F^-$-measurable. 

\subsection{Invariant volume forms} 
Let us  start with a well-known result. 

\begin{lem}\label{lem:sum_exponents}
Let $(X, \nu)$ be a random holomorphic dynamical system satisfying the integrability condition 
\eqref{eq:moment}, and $\mu$ be an ergodic   stationary probability measure.  Then 
$$
\lambda^{-}+\lambda^{+} = \int \log\abs{\jac f(x)} d\mu(x)d\nu(f),
$$
where $\jac$ denotes the Jacobian determinant relative to any smooth volume form on $X$.
\end{lem}

We omit the proof, since this result is a corollary of Proposition  \ref{pro:sum_exponents} below.
When $X$ is an Abelian, or K3, or Enriques surface, Remark~\ref{rem:volume_form} provides an $\Aut(X)$-invariant volume form on $X$.
Thus, we obtain: 

\begin{cor}\label{cor:sum_exponentsK3}
Assume that $X$ is an Abelian, or K3, or Enriques surface. Let $\nu$ be a probability measure on $\Aut(X)$  
satisfying the integrability condition \eqref{eq:moment}, and $\mu$ be an ergodic $\nu$-stationary measure. 
Then $\lambda^{-}+\lambda^{+}=0$. 
\end{cor}

 Let $\eta$ be a non-trivial meromorphic $2$-form on the surface $X$. There is a cocycle $\jac_\eta$, with values in  the multiplicative group $\Mer(X)^\times$ of non-zero meromorphic functions, such that 
\begin{equation}
f^*\eta=\jac_\eta(f)\eta
\end{equation}
for every $f\in \Aut(X)$. 
We say that $\eta$ is {\bf{almost invariant}}  
if $\abs{\jac_\eta(f)(x)}=1$ for every $x\in X$ and $\nu$-almost every $f\in \Aut(X)$ (in particular $\jac_\eta(f)$ is a constant).
We refer to  \S\ref{par:Coble-Blanc} for examples with an invariant meromorphic 2-form.

\begin{pro}\label{pro:sum_exponents}
Let $(X, \nu)$ be a random holomorphic dynamical system satisfying the integrability condition 
\eqref{eq:moment}, and $\mu$ be an ergodic stationary measure.  
Let $\eta$ be a non-trivial meromorphic $2$-form on $X$ such that 
\begin{enumerate}[\em (i)]
\item $\displaystyle \int \log^+\vert \jac_\eta(f)(x)\vert d\mu(x) d\nu(f)<+\infty$;
\item $\mu$ gives zero mass to the set of  zeroes and poles of $\eta$.
\end{enumerate}
Then 
\[
\lambda^{-}+\lambda^{+} = \int \log(\abs{\jac_\eta f (x)}^2) d\mu(x)d\nu(f);
\]
in particular $\lambda^{-} + \lambda^{+} = 0$ if $\eta$  is almost invariant. 
\end{pro}


\begin{proof}   
Fix a trivialization of the tangent bundle $TX$, given by a measurable family of linear isomorphisms 
$L(x)\colon T_xX\to \C^2$ such that (a) $\det(L(x))=1$ and (b) $1/C\leq \norm{L(x)}+\norm{L(x)^{-1}}\leq C$, for some constant $C>1$; here, 
the determinant is relative to the volume form $\vol$ on $X$ and the standard volume form $dz_1\wedge dz_2$ on $\C^2$, and  the 
norm is with respect to the K\"ahler metric $(\kappa_0)_x$ 
on $T_xX$ and the standard euclidean metric on $\C^2$.  For
$(\xi, x)\in \X$ and $n\geq 0$, the differential $D_xf^n_\xi$ is expressed in this trivialization as a
matrix  $A^{(n)}(\xi, x)=L(f^n_\xi(x)) \circ D_xf^n_\xi\circ L(x)^{-1}$. Let $\chi^{-}_n(\xi, x)\leq \chi^{+}_n(\xi, x)$  be the  singular values of   
$A^{(n)}(\xi, x)$. Then $\m$-almost surely, $\unsur{n}\log \chi^{\pm}_n(\xi, x) \to \lambda^{\pm}$ as $n\to +\infty$. 

The form $\eta\wedge{\overline{\eta}}$ can be written $\eta\wedge {\overline{\eta}}=\varphi(x) \vol$ for some
function  $\varphi\colon X\to [0,+\infty]$. Locally, one can write $\eta=h(x) dx_1\wedge dx_2$ where $(x_1,x_2)$ 
are local holomorphic coordinates  and $h$ is a meromorphic function; then $\varphi(x)\vol= \abs{h(x)}^2 dx_1\wedge dx_2\wedge d{\overline{x_1}}\wedge d\overline{x_2}$. The jacobian $\jac_\eta$ satisfies 
\begin{equation}
\vert \jac_\eta(f)(x)\vert^2=\frac{\varphi(f(x))}{\varphi(x)}\jac_\vol(f)(x)
\end{equation}
for every $f\in \Aut(X)$ and $x\in X$.
Using $\det(L(x))=1$, we get 
\begin{equation}\label{eq:determinant}
\det (A^{(n)}(\xi, x)) = \jac_\vol(f^n_\xi)(x),
\end{equation}
and then  
\begin{equation}\label{eq:jacobian}
\unsur{n}\log \chi^{-}_n(\xi, x) +  \unsur{n}\log \chi^{+}_n(\xi, x)  =  \frac{2}{n}\log \abs{\jac_\eta f^n_\xi(x)} - \unsur{n}\log(\varphi(f^n_\xi(x))/\varphi(x)).
\end{equation}
By the Oseledets theorem, the left hand side of \eqref{eq:jacobian} converges almost surely to $\lambda^{-}+\lambda^{+}$.
Since the Jacobian $\jac_\eta$ is multiplicative along orbits, i.e. $\jac_\eta f^n_\xi(x)  = \prod_{k=0}^{n-1} 
\jac_\eta f_{\vartheta^k\xi}(f^k_\xi x)$, the integrability condition and the ergodic theorem imply that, almost surely, 
\begin{align}
 \lim_{n\to\infty} \frac{1}{n} \log\abs{\jac_\eta f_\xi^n(x)}   
&=  \int \log\abs{\jac_\eta f_\xi^1 (x)} d\m(\xi, x) \\
&\notag =  \int \log\abs{\jac_\eta f_\omega^1 (x)} d\m_+(\omega, x)  \\
&\notag  =  \int \log\abs{\jac_\eta f(x)} d\mu(x)d\nu(f).
\end{align} 
Let ${\mathrm{div}}(\eta)$ be the set of zeroes and poles of $\eta$. Since $\mu$ is ergodic and does not charge ${\mathrm{div}}(\eta)$, we deduce that for $\m$-almost every $(\xi, x)$, 
there is a sequence $(n_j)$  such that $f^{n_j}_\xi(x)$ stays at  
positive distance from ${\mathrm{div}}(\eta)$; along such a sequence, $\log\vert \varphi(f^{n_j}_\xi(x))/\varphi(x)\vert$ 
stays bounded, and the right hand side of \eqref{eq:jacobian} tends to 
$2\int \log\abs{\jac_\eta f(x)} d\mu(x)d\nu(f)$. This concludes the proof. 
\end{proof}



\subsection{Intermezzo: local complex geometry}\label{subs:local}
Recall that $X$ is endowed with a Riemannian structure, hence a distance, induced by the K\"ahler metric $\kappa_0$. 
For $x\in X$, we denote by $\euc_x$ the translation-invariant Hermitian metric on $T_xX$ (which is considered here
as a manifold in its own right) associated to   the 
Riemannian structure induced by  $(\kappa_0)_x$. 
Given  any orthonormal basis $(e_1,e_2)$ of $T_xX$ for this metric, we obtain 
a linear isometric isomorphism from $T_xX$ to $\C^2$, endowed respectively with $\euc_x$ and the standard euclidean metric; 
we shall implicitly use such identifications in what follows.  

We denote by $\disk(z;r)$ the disk of radius $r$ around $z$ in $\C$, and set $\disk(r)=\disk(0;r)$.

\subsubsection{Hausdorff and $C^1$-convergence}\label{par:Hausdorff_C1}
Let $U\subset \C$ be a domain. If $\gamma\colon U\to X$ is a holomorphic curve, we can lift it canonically to a curve $\gamma^{(1)}\colon U\to TX$
by setting $\gamma^{(1)}(z)=(\gamma(z),\gamma'(z))\in T_{\gamma(z)}X$, where $\gamma'(z)$ denotes the velocity of $\gamma$ at $z$. 
The   
Kähler form $\kappa_0$ induces a Riemannian metric and therefore a distance $\dist_{TX}$ on $TX$. We say that
two parametrized curves $\gamma_1$ and $\gamma_2$ are $\delta$-close in the $C^1$-topology if  $\dist_{TX}(\gamma_1^{(1)}(z),\gamma_2^{(1)}(z))\leq \delta$ uniformly on $U$. This implies that $\gamma_1(U)$ and $\gamma_2(U)$ are $\delta$-close  in the Hausdorff sense, 
but the converse does not hold (take $U=\disk(1)$, $\gamma_1(z)=(z,0)$, and $\gamma_2(z)=(z^k, \e z^\ell)$ with $k$ and $\ell$ large while $\e$ is small). 

\subsubsection{Good charts}\label{par:good_charts} Let $R_0$ be the injectivity radius of $\kappa_0$.
We fix once and for all a family of charts $\hexp_x\colon U_x\subset T_xX\to X$ with the following properties 
(for some constant $C_0$): 
\begin{enumerate}[(i)]
\item $\hexp_x(0)=x$ and $(D\hexp_x)_0=\id_{T_xX}$; 
\item $\hexp_x$ is a holomorphic diffeomorphism from its domain of definition $U_x$ to an open subset $V_x$ contained in 
the ball of radius $R_0$ around $x$; 
\item on $U_x$, the Riemannian metrics $\euc_x$ and $\hexp_x^*\kappa_0$ satisfy $C_0\inv \leq \euc_x / \hexp_x^*\kappa_0 \leq C_0$;
\item the family of maps $\hexp_x$ depends continuously on $x$.
\end{enumerate}
With $r_0\leq R_0/(\sqrt{2}C_0)$, we can add: 
\begin{enumerate}[(i)]
\item[(v)] for every orthonormal basis $(e_1,e_2)$ of $T_xX$, the bidisk $\disk(r_0)e_1+ \disk(r_0)e_2$ is contained
in $U_x$; in particular, the ball of radius $r_0$ centered at the origin for $\euc_x$ is contained in $U_x$.
\end{enumerate}
To make assertion~(iv) more precise, fix a continuous family of orthonormal basis $(e_1(x),e_2(x))$ on some open set $V$ 
of $X$:  Assertion~(iv) means that, if we compose $\hexp_x$ with the linear isomorphism $(z_1,z_2)\in \C^2\mapsto z_1e_1(x)+z_2e_2(x)\in T_xX$ 
we obtain a continuous family of maps. If needed, we can also add the following property (see \cite[pp. 107-109]{griffiths-harris}): 
\begin{enumerate}[(i)]
\item[(iii')] $\euc_x$ osculates $\hexp_x^*\kappa_0$ up to order 2 at $x$.
\end{enumerate}

\subsubsection{Families of disks}\label{par:disks_size}
A holomorphic disk $\Delta\subset X$ containing $x$ is said to be a disk {\bf{of size (at least)  $r$ at $x$}} (resp. {\bf{of size exactly $r$ at $x$}}),  for some $r < r_0$,
if there is an orthonormal basis $(e_1,e_2)$ of $T_xX$ such that $\hexp_x^{-1}(\Delta)$ contains  (resp. is) the graph 
$\{ ze_1+\varphi(z)e_2\; ; \; z\in \disk(r)\}$ for some holomorphic map $\varphi\colon \disk(r)\to \disk(r)$. 
By the Koebe distortion theorem  if $\Delta$ has size $r$ at $x$, then its geometric characteristics around $x$ at scale smaller than 
$r/2$, say,  are  comparable to that of a flat disk. 
An alternative definition for the concept of disks of size $\geq r$ could be that $\Delta$ contains the image of an injective holomorphic map $\gamma\colon \disk(r)\to X$
such that  $\gamma(\fr \disk(r))\subset X\setminus B_{X}(x;r)$ and $\norm{\gamma'}\leq D$, for some fixed   constant $D$. 
Then, if $\Delta$ contains a disk of 
size $r$ for one of these definitions, it contains a disk of size $\e_0 r$ for the other definition, for some uniform $\e_0>0$;
in particular, there is a  constant $C$ depending only on $(X, \kappa_0)$ such that 
a disk of size $r$ at $x$ contains an embedded submanifold of $B_X(x; Cr)$.  



Let $(x_n)$ be a sequence converging to $x$ in $X$, and let $r$ be smaller than the radius $r_0$ introduced in Assertion~{(v)}, \S~\ref{par:good_charts}.
Let $\Delta_n$ be a family of disks of size at least 
$r$ at $x_n$ and $\Delta$ be a disk of size at least $r$ at $x$. We 
say that $\Delta_n$ {\bf{converges towards}} 
$\Delta$ {\emph{as a sequence of disks of size $r$}}, 
if there is an orthonormal basis $(e_1,e_2)$ of $T_xX$ for $\euc_x$ such that 
\begin{enumerate}[(i)]
\item $\hexp_x^{-1}(\Delta)$ contains the graph $\{z e_1+\varphi(z)e_2; z\in \disk(r)\}$ for some holomorphic function $\varphi\colon \disk(r)\to \disk(r)$; 
\item for every $s<r$, if $n$ is large enough, the disk $\hexp_x^{-1}(\Delta_n)$ contains the graph  
$\{ze_1+\varphi_n(z)e_2; z\in \disk(s)\}$ of a holomorphic function $\varphi_n\colon \disk(s)\to \disk(r)$;
\item for every $\e>0$, we have  $\abs{\varphi(z)-\varphi_n(z)}<\e$ on $\disk(s)$  if $n$ is large enough.
\end{enumerate}
By the Cauchy estimates, the 
convergence then holds in the $C^1$-topology (see~\S~\ref{par:Hausdorff_C1}). It follows
from the usual compactness criteria for holomorphic functions 
 that the space of disks of size $r$ on $X$ is compact (for the topology induced by the Hausdorff topology in $X$).
 Likewise, if  a sequence  of disks of size $r$ converges in the Hausdorff sense, then it also 
converges in the $C^1$ sense, at least as disks of size $s<r$, because two holomorphic functions $\varphi$ and $\psi$ from 
$\disk(r)$ to $\disk(r)$ whose graphs are $\e$-close are also $\e (r-s)^{-1}$-close in the $C^1$-topology.

It may also   be the case that  
the $\Delta_n$ are contained in different fibers $X_{\xi_n}$ of $\X$. By definition, we say that the sequence $\Delta_n$ converges 
to $\Delta\subset X_\xi$ if $\xi_n$ converges to $\xi$ and the projections of $\Delta_n$ converge to $\Delta$ in $X$. 

\subsubsection{Entire curves}  An {\bf{entire curve}} in $X$ is, by definition, a holomorphic map $\psi\colon \C\to X$. The curve is 
immersed if its velocity $\psi'$ does not vanish. Our main examples of immersed curves will, in fact, be injective and immersed entire curves. 
If $\psi_1$ and $\psi_2$  are two immersed entire curves with the same image, 
there exists a holomorphic diffeomorphism of $\C$, i.e. a non-constant affine map $A\colon z\mapsto az+b$, such 
that $\psi_2=\psi_1\circ A$. 
If $\psi$ is an immersed entire 
curve and $\abs{\psi'}\geq \eta$ on $\disk(z_0,s)$, its image contains a disk of size $Cs$ at $\psi(z_0)$, for some 
$C>0$ that depends only on $\eta$ and $\kappa_0$. 

\subsection{Stable and unstable manifolds}\label{subs:pesin}
By Lemma~\ref{lem:momentC1_to_Ck}, Condition~\eqref{eq:moment} implies similar moment conditions for higher derivatives, so Pesin's theory applies. 
The following proposition summarizes the main properties of Pesin local stable and unstable manifolds. Recall that a function $h$ is 
\textbf{$\e$-slowly varying}, relatively to some dynamical system $g$, if $e^{-\e}\leq h(g(x))/h(x)\leq e^\e$ for every $x$.  
We view the stable manifold of 
$\cx= (\xi,x)$ as contained in $X_\xi$; it can also be viewed as a subset  of $X$: whether we consider one or the other point of view should be clear 
from the context. 
If $\cx=(\xi,x)$ and $\cy=(\xi,y)$ are points of the same fiber $X_\xi$, we denote by $\dist_X(\cx,\cy)$ the Riemannian distance between $x$ and $y$ computed in~$X$.

\begin{pro}\label{pro:pesin} Let  $(X, \nu)$ be  a random holomorphic dynamical system, and $\mu$ be an ergodic and hyperbolic stationary measure. 
Then, for every $\delta>0$, there exists measurable 
positive $\delta$-slowly varying functions $r$ and $C$ on $\X$ (depending on $\delta$) 
and, for $\m$-almost every $\cx=(\xi,x) \in \X$, local stable and unstable manifolds 
 $ W^{s}_{r(\cx)}(\cx)$ and $W^{u}_{r(\cx)}(\cx)$ in $X_\xi$ 
 such that $\m$-almost surely:
\begin{enumerate}[{\em (1)}]
\item $W^{s}_{r(\cx)}(\cx)$ and $W^u_{r(\cx)}(\cx)$ are holomorphic disks of size at least $ 2 r(\cx)$ at $\cx$ 
 respectively tangent to $E^s(\cx)$ and $E^u(\cx)$;
\item for every $\mathscr{y}\in W^{s}_{r(\cx)}(\cx)$ and every $n\geq 0$,  
$$\dist_X(F^n(\cx), F^n(\mathscr{y})) \leq C(\cx) \exp ((\lambda^- +\delta) n);$$ likewise for every $\mathscr{y}\in W^{u}_{r(\cx)}(\cx)$ and every $n\geq 0$
$$\dist_X(F^{-n}(\cx), F^{-n}(\mathscr{y})) \leq C(\cx) \exp (-(\lambda^+ -\delta) n);$$
\item  
$F( W^{s}_{r(\cx)}(\cx)) \subset W^{s}_{r(F(\cx))}(F(\cx)) $ and $F\inv( W^{u}_{r(F(\cx))}(F(\cx)) )\subset W^{u}_{r(\cx)}(\cx) $.
\end{enumerate}
\end{pro}
  
By Lusin's theorem, for every $\e>0$  we can select 
a compact subset $\mathcal R_\e\subset \X$ with $\m(\mathcal R_\e)>1-\e$, on which 
$r(\cx)$ and $C(\cx)$ can be replaced by uniform constants
(respectively denoted by $r$ and $C$) and the following additional property holds:
\begin{enumerate}[{\em (1)}]
\item[(4)] \emph{on $\mathcal R_\e$ the  local stable and unstable manifolds $W^{s/u}_r(\cx)$ 
vary continuously for the  $C^1$-topology  (in the sense of \S~\ref{par:Hausdorff_C1} and~\ref{par:disks_size}).}
\end{enumerate}
The subsets $\mathcal R_\e$ are usually called {\bf{Pesin sets}}, or regular sets.
We  also denote  the local stable or unstable manifolds by $W^{s/u}_\loc(\cx)$, or by $W^{s/u}_r(\cx)$
when $\cx$ is in a Pesin set on which $r(\cdot)\geq r$. On several occasions we will have to deal with measurability issues for $W^{s/u}_\loc(\cx)$ as a function of $\cx$: this will be done 
by exhausting $\X$ by Pesin sets  and using their continuity on $\mathcal R_\e$.
 
The global stable and unstable manifolds of $\cx$ are respectively defined by the following increasing unions:
\begin{equation}\label{eq:union}
W^s(\cx) = \bigcup_{n\geq 0} F^{-n} \lrpar{W^s_{r(\cx)}(F^n(\cx))} \; \text{ and } \; W^u(\cx) = \bigcup_{n\geq 0} F^{n} \lrpar{W^u_{r(\cx)}(F^{-n}(\cx))}.
\end{equation}
In particular, they are injectively immersed holomorphic curves in $X_\xi$. Pesin theory shows that: 
\begin{align}
W^s(\cx)    & = \set{(\xi,y) \in X_\xi \; ; \ \limsup_{n\to\infty}\unsur{n}\log \dist_X(F^n(\xi,y), F^n(\xi,x)) <0}  \\    W^u (\cx)    &  = \set{ (\xi, y) \in X_\xi\; ; \ \limsup_{n\to-\infty}\unsur{\abs{n}}\log \dist_X(F^n(\xi,y), F^n(\xi,x)) <0}.
\end{align} 

\begin{pro}
Under the assumptions of Proposition \ref{pro:pesin},
  $W^s(\cx)$ and $W^u(\cx)$ are biholomorphic to $\C$ for $\m$-almost every $\cx$. 
\end{pro}

More precisely, $W^s(\cx)$
is parametrized by an injectively immersed  entire curve    
$\psi^s_\cx: \C\to X$  such that $\psi^s_\cx(0) = x$ 
and this parametrization is  unique, up to an homothety 
$z\mapsto az$  of $\C$. Likewise, $W^u(\cx)$  
is parametrized by such an entire curve $\psi^u_\cx$.  
 
\begin{proof}  
By \eqref{eq:union} and Proposition~\ref{pro:pesin}.{(3)},   $W^s(\cx)$ is an increasing union of disks and is therefore a Riemann surface homeomorphic to $\R^2$; so, it is biholomorphic to $\C$ or $\disk$.
Let $A\subset \X$ be a set of  positive measure   on which $r\geq r_0$ and $C\leq C_0$. By Proposition \ref{pro:pesin}.{(2)}, 
there exists $n_0\in \N$ and  $m_0>0$ such that if $n\geq n_0$ and if
$\cx $ and  $F^{n}(\cx)$ belong to $ A$, then
$W^{s}_r(F^{n}( \xi,x))\setminus \lrpar{F^{n} W^s_{r}( \xi,x)}$ 
is an annulus of modulus $\geq m_0$. 
Now for $\m$-almost every $\cx\in \X$ there is  an  infinite sequence  $(k_j)$ such that 
$F^{k_j}(\cx) \in A$
and $k_{j+1} - k_j> n_0$.  For such an $\cx$, 
$W^s(\cx)\setminus W^s_r(\cx)$ contains an infinite  nested sequence of annuli of modulus at least $m_0$, namely the $F^{-k_{j+1}}(W^s_r(F^{k_{j+1}}(\cx))\setminus F^{k_{j+1}-k_j}(W^s_r(F^{k_j}(\cx))$. Thus, $W^s(\cx)$ is biholomorphic to  $\C$.
\end{proof}

If we are only interested in stable manifolds, there is a simplified version of Proposition 
\ref{pro:pesin} which takes place   on $X$: 
 
 \begin{pro}\label{pro:pesin_stable}
Let $(X, \nu)$ be a random holomorphic dynamical system and $\mu$ an ergodic stationary measure, whose Lyapunov exponents
satisfy $\lambda^{-} < 0 \leq \lambda^{+}$. Then for $\m_+$-almost every  $(\omega, x)$ the stable set 
\[
W^s (\omega, x) = \set{y \in X\; ; \ \ \limsup_{n\to\infty}  \frac{1}{n} \log \dist_X(f_\omega^n(y), f_\omega^n(x)) < 0}
\]
is an injectively immersed entire curve in $X$.  
\end{pro}

Indeed, stable manifolds can be obtained from  a purely ``one-sided'' construction, that is, by considering only positive iterates (see \cite[Chap. III]{liu-qian}). This also shows that local stable manifolds in $\cX$ are 
$\mathcal F^+$-measurable, and may be viewed as living in $\cX_+$.


\subsection{Fibered entropy} Here we recall the definition of the \textbf{metric fibered entropy} 
of a stationary measure $\mu$ (see~\cite[\S 2.1]{kifer} or \cite[Chap. 0 and I]{liu-qian} for more details). 
If $\eta$ is a finite measurable partition of
$X$, its entropy relative to $\mu$ is $H_\mu(\eta) = -\sum_{C\in \eta} \mu(C) \log\mu(C)$. Then, we set
\begin{align}
h_{\mu}(X, \nu; \eta) & = \lim_{n\to\infty} \unsur{n}\int H_\mu\lrpar{\bigvee_{k=0}^{n-1} \lrpar{f^k_\xi}\inv(\eta) } d\nu^\N(\xi), \\
h_{\mu}(X, \nu) & = \sup\set{h_{\mu}(X, \nu; \eta) \; ; \;  \eta\text{ a finite measurable partition  of  }X}.
\end{align}
Actually 
$h_{\mu}(X, \nu; \eta)$ can be interpreted as a conditional (or fibered) entropy for the skew-products $F_+$   
on   $\X_+$  and $F$ on $\X$. Indeed,  the so-called Abramov-Rokhlin formula holds~\cite{bogenschutz-crauel}:
\begin{align}
h_{\mu}(X, \nu) & = h_{\nu^\N\times \mu}(F_+\vert \eta_\Omega) = h_{\m_+}(F_+) - h_{\nu^\N}(\sigma) \\
& = h_{\m}(F\vert \eta_\Sigma)  \;   =  \; h_\m(F)  - h_{\nu^\Z}(\vartheta),
\end{align}
where $\eta_\Omega$ (resp. $\eta_\Sigma$) denotes the partition into fibers of the 
first projection $\pi_\Omega\colon \X_+\to \Omega$ (resp.
 $\pi_\Sigma\colon \X \to \Sigma$)  and  in the second and fourth equalities we assume   
$h_{\nu^\N} (\sigma)= h_{\nu^\Z}(\vartheta)<\infty$. 
The next result is the fibered version of the \textbf{Margulis-Ruelle inequality}. 

\begin{pro}\label{pro:margulis-ruelle}
Let $(X, \nu)$ be a random holomorphic dynamical system satisfying the moment condition \eqref{eq:moment}
and $\mu$ be an ergodic stationary measure. If $h_{\mu}(X, \nu)>0$ then $\mu$ is hyperbolic and 
$\min(\lambda^+,  - \lambda^-) \geq \unsur{2} h_{\mu}(X, \nu)$.
\end{pro}
 
\begin{proof} 
See   \cite{bahnmuller-bogenschutz} or \cite[Chap. II]{liu-qian} for the  inequality 
$\lambda^+  \geq \unsur{2} h_{\mu}(X, \nu)$. For
$- \lambda^-\geq \unsur{2} h_{\mu}(X, \nu)$, we use the fact that 
$h_m(F\vert \eta_\Sigma) = h_m(F\inv\vert \eta_\Sigma)$ 
(see e.g. 
\cite[I.4.2]{liu-qian}) and apply the Margulis-Ruelle inequality to $F\inv$. Beware that there is a slightly delicate point here: 
$(F\inv, \m)$ is not associated to  a random dynamical system in our sense; fortunately,    
the statement of the Margulis-Ruelle inequality in \cite{bahnmuller-bogenschutz} (see also 
\cite[Appendix A]{liu-qian})   covers this situation. 
\end{proof}

\subsection{Unstable conditionals and entropy} \label{subs:conditionals} 
Assume $\mu$ is ergodic and hyperbolic. 
By definition, an \textbf{unstable Pesin partition} $\eta^u$ on $\X$
is a measurable partition of $(\X, \mathcal{F}, \mu)$ with the following properties: 
\begin{itemize}
\item $\eta$ is increasing: $F\inv \eta^u$ refines $\eta^u$; 
\item  for $\m$-almost every $\cx$, $\eta^u(\cx)$ is an open subset of 
$W^u(\cx)$ and  
\begin{equation}\label{eq:exhaustion}
\bigcup_{n\geq 0} F^n\lrpar{\eta^u (F^{-n} (\cx))} = W^u(\cx);
\end{equation}
\item $\eta^u$ is a generator, {i.e.} $\bigvee_{n=0}^\infty F^{-n}(\eta^u)$  coincides $\m$-almost surely with the partition into points.
\end{itemize}
Here, as usual, $\eta^u(\cx)$ denotes the atom of $\eta^u$ containing $\cx$, and  $F\inv\eta^u$ is the partition defined by $(F\inv\eta^u)(\cx)   = F\inv(\eta^u(F(\cx)))$.
The definition of a  \textbf{stable Pesin partition} $\eta^s$ is   similar.  
A neat proof of the existence of such a partition is given by Ledrappier and Strelcyn in
\cite{ledrappier-strelcyn}, which easily adapts to the random setting 
\begin{vcourte}
(see \cite[\S IV.2]{liu-qian}, and \cite{cantat-dujardin:vlongue}). 
\end{vcourte}
\begin{vlongue}
(see \cite[\S IV.2]{liu-qian}). 
\end{vlongue}

\begin{lem} \label{lem:stable_partition} 
There exists a stable (resp. unstable) Pesin partition whose atoms  are 
$\mathcal F^+$-measurable (resp. $\mathcal F^-$-measurable), that is,  saturated by local stable (resp. unstable) sets $\Sigma^s_\loc\times\{ x \}$ (resp. $\Sigma^u_\loc\times\{ x \}$).
\end{lem}

\begin{vlongue}
\begin{proof}
To justify the existence of such a partition, we briefly review the proof  of Ledrappier and Strelcyn  
\cite{ledrappier-strelcyn} and show that it can be rendered   $\mathcal F^+$-measurable. 
Let   $E$ be a set of positive measure in $\X$ such that (a) $\pi_X(E)$ is contained in a ball of radius 
$r_0$, (b) for every $\cx= (\xi,x)\in E$, and every $0<r\leq 2r_0$,  $W^s(\cx)$ 
contains a disk  of size  exactly  $r$ at $\cx$, denoted by $\Delta^s(\cx, r)$
 and  (c) for every $0<r\leq 2r_0$, $E\ni\cx\mapsto \Delta^s(\cx, r)$ is continuous for the $C^1$ topology. 
 Then for $0<r< r_0$ we define a measurable partition $\eta_r$ 
whose atoms are the $\Delta^s(\cx, r)$ for $x\in E$ as well as  $\X\setminus \bigcup_{\cx\in E} \Delta^s(\cx, r)$. 
Since stable manifolds are $\mathcal F^+$-measurable, we can further require that for every $\xi' \in \Sigma^s_\loc(\xi)$, 
  with  $\cx'  = (\xi', x)$, we have
$\Delta^s(\cx', r) = \Delta^s(\cx, r)$.
The argument of \cite{ledrappier-strelcyn} shows  
that for Lebesgue-almost every $r\in [0, r_0]$, the partition 
$\eta^s  = \bigvee_{n=0}^\infty F^{-n}(\eta_r)$ is a Pesin stable partition. Thus with $\cx$ and $\cx'$ as above we infer that 
 \begin{equation}
 \eta^s(\cx') =  \bigcap_{n\geq 0} F^{-n} \eta_r(F^n(\cx'))  = \bigcap_{n\geq 0} F^{-n} \eta_r(F^n(\cx)) =\eta^s(\cx)
 \end{equation}
 where the middle equality comes from the fact that $\vartheta^n\xi'\in \Sigma^s_\loc(\vartheta^n\xi)$, and we are done. 
 \end{proof}
\end{vlongue}
 
%

 The existence of   unstable partitions enables us to give a meaning to the \textbf{unstable conditionals} of $\m$. Indeed, first observe that if $\eta^u$ and $\zeta^u$ are two  unstable Pesin partitions, then $\m$-almost surely $\m(\cdot \vert \eta^u)$ and $\m(\cdot \vert
 \zeta^u)$ coincide up to a multiplicative factor on $\eta^u(\cx)\cap \zeta^u(\cx)$. Furthermore, there exists a sequence of unstable partitions 
 $\eta^u_n$ such that for almost every $\cx$, 
if $K$ is a compact subset of $W^u(\cx)$ for the  intrinsic topology (i.e. the topology induced by the 
biholomorphism $W^u(\cx)\simeq \C$) then $K\subset \eta_n^u(\cx)$ for sufficiently large $n$: indeed 
by \eqref{eq:exhaustion}, the sequence of partitions $F^{n}\eta^u$ does the job. Hence almost surely the conditional measure of 
$\m$ on $W^u(\cx)$  is well-defined up to scale; we define $\m^u_\cx$ by normalizing so that 
$\m^u_\cx(\eta^u(\cx)) = 1$. 

The next proposition is known as the (relative) \textbf{Rokhlin entropy formula}, stated here in our specific context. 
 
\begin{pro}\label{pro:rokhlin_formula}
Let $(X, \nu)$ be a random holomorphic dynamical system satisfying the moment condition~\eqref{eq:moment},
and $\mu$ be an ergodic and hyperbolic stationary measure. Let $\eta^u$ be an 
unstable Pesin partition. Then
\[
h_\mu(X, \nu) = H_\m(F\inv\eta^u\vert \, \eta^u) :=  \int \log J_{\eta^u}(\cx) d\m(\cx),
\]
where $J_{\eta^u}(\cx)$ is the ``Jacobian'' of $F$ relative to $\eta^u$, that is 
\[
J_{\eta^u}(\cx) = \m\left(F\inv \left(\eta^u(F(\cx))\right)\vert\, \eta^u(\cx)\right)\inv.
\]
 \end{pro}
 
\begin{proof}[Sketch of proof] The argument  is based on the following sequence of equalities, in which  $\eta_\Sigma $ is the partition into fibers of $\pi_\Sigma$, as before:
 \begin{align}
\notag h_\mu(X, \nu) &= h_\m (F\vert {\eta_\Sigma} ) = h_\m (F\inv \vert {\eta_\Sigma} )\\
 &= h_\m (F\inv \vert  \eta^u \vee {\eta_\Sigma} ) \label{eq:entropy_generator}\\
&\notag := H_\m(\eta^u\vert F\eta^u \vee  \eta_\Sigma) 
 =H_\m(\eta^u\vert F\eta^u)  = H_\m(F\inv\eta^u\vert \eta^u)
 \end{align}
The equalities in the first and last line follow 
from the general properties of conditional entropy: see \cite[Chap. 0]{liu-qian} for a presentation adapted to our context (note 
that the conditional entropy would be  denoted by $h^{\eta_\Sigma}_\m$ there) or 
Rokhlin \cite{rokhlin} for a thorough treatment. On the other hand the equality \eqref{eq:entropy_generator} is non-trivial. 
 If $\eta^u$ were  of the form $\bigvee_{n=0}^{+\infty} \eta$, where $\eta$ is a 2-sided generator with finite entropy, this equality 
 would follow from the general theory. For a Pesin unstable partition the result   was 
 established for diffeomorphisms in \cite[Cor 5.3]{ledrappier-young1} 
and adapted to random dynamical systems in \cite[Cor. VI.7.1]{liu-qian}. 
\end{proof}
 
 \begin{rem}\label{rem:stable_rokhlin}  
 It is customary to  present the Rokhlin entropy formula using unstable partitions, mostly 
 because entropy  is associated to expansion. Nonetheless,  a similar formula holds in the stable direction:
 $$h_\mu(X, \nu)  =  \int \log J_{\eta^s}(\cx) d\m(\cx)\text{ where }
 J_{\eta^s}(\cx) = \m\left(F  \left(\eta^s(F\inv(\cx))\right)\vert\, \eta^s(\cx)\right)\inv.$$
The proof is identical to that of Proposition~\ref{pro:rokhlin_formula}, applied to $F\inv$, 
with however the same caveat as in 
Proposition~\ref{pro:margulis-ruelle}: $(F\inv, \m)$ is not associated to 
 a random dynamical system in our sense. The only non-trivial point is to 
  check that the key equality~\eqref{eq:entropy_generator} holds in this case. Fortunately, 
 the main purpose of  \cite{bahnmuller-liu} is to explain how to 
adapt \cite[Chap. VI]{liu-qian}, hence the equality~\eqref{eq:entropy_generator}, 
 to a more general notion of ``random dynamical system'' which covers the case 
of $(F\inv, \m)$ (see in particular the last lines of  \cite[\S5]{bahnmuller-liu} for a short 
 discussion of the Rokhlin formula).
 \end{rem}

The following consequence of the Rokhlin  formula will play an important role in Section~\ref{sec:No_Invariant_Line_Fields}.  

\begin{cor} \label{cor:rokhlin_zero_entropy} 
Under the assumptions of the previous proposition, 
the following assertions are equivalent:
\begin{enumerate}[\em (a)]
\item $h_\mu(X, \nu) = 0$;
\item  $\m(\cdot \vert \eta^u(\cx)) = \delta_\cx$ for $\m$-almost every $\cx$;
\item   
$\m(\cdot \vert \eta^u(\cx))$ is atomic for $\m$-almost every $\cx$. 
\end{enumerate}
The same result holds for the stable Pesin partition $\eta^s$.  
\end{cor}

\begin{proof} In view of the 
definition of $J_{\eta^u}$, the entropy vanishes if and only if for $\m$-almost every $\cx$, $\m(\cdot \vert \eta^u(\cx))$ is carried by a single atom 
of the finer partition $F\inv \eta^u$. Now since  $H_\m(F\inv\eta^u\vert \, \eta^u) = \unsur{n}  H_\m(F^{-n}\eta^u\vert \, \eta^u)$, the same is true  
for $F^{-n}\eta^u$, and finally since $(F^{-n}\eta^u)$ is generating,  we conclude that (a)$\Leftrightarrow$(b). 
That (c) implies (a) follows from the same ideas but it is slightly more delicate, see \cite[\S 2.1-2.2]{viana-yang} for a 
clear exposition in the case of the iteration a single diffeomorphism, which readily adapts  to our setting.

The result for the stable Pesin partition $\eta^s$ follows by changing $F$ to $F\inv$ 
 (see however  Remark~\ref{rem:stable_rokhlin}).
\end{proof}

\begin{vlongue}
A further result is that  if the fiber entropy vanishes  there is a  set of full $\m$-measure   
which intersects any  global unstable leaf  in only one point.  
This was originally shown for individual diffeomorphisms  in \cite[Thm. B]{ledrappier-young1}. 
\end{vlongue}

\section{Stable manifolds and limit currents}\label{sec:nevanlinna}

Let as before $(X, \nu)$ be a non-elementary random holomorphic dynamical 
system on a compact K\"ahler (hence projective) 
surface, and assume $\mu$ is an ergodic stationary measure admitting exactly one negative 
Lyapunov exponent, as in Proposition \ref{pro:pesin_stable}.
 Our purpose in this section is to relate the   stable manifolds   $W^s(\omega, x)$  to the stable currents $T_\omega^s$ constructed in \S \ref{sec:currents}. 
According to Proposition~\ref{pro:pesin_stable},  the stable manifolds are parametrized by injective entire curves;  
the link between these curves and the stable currents will be given by the well-known 
 Ahlfors-Nevanlinna construction of positive closed currents associated to entire curves.
 
\subsection{Ahlfors-Nevanlinna currents}\label{par:Ahlfors-Nevanlinna-Construction} 
We denote by $\set{V}$ the integration current on a (possibly non-closed, or singular) curve $V$. 
Let $\phi:\C\to X$ be an entire curve. By definition, if $\alpha$ is a test 2-form, 
$\langle \phi_*\set{\disk(0,t)}, \alpha\rangle=  \int_{\disk(0,t)} \phi^*\alpha$, which accounts for possible multiplicities
coming from the lack of injectivity of $\phi$; $\phi_*\set{\disk(0,t)}=\set{\phi(\disk(0,t))}$ when $\phi$ is injective. 
Set 
\begin{equation}
A(R) =  \int_{\disk(0, R)}\phi^*\kappa_0 \;  \text{ and } \; T(R) = \int_0^R A(t) \frac{dt}{t}. 
\end{equation}
for $R>0$. When $\phi$ is an immersion, $A(R)$ is the area of $ \phi(\disk(0, R))$; in all cases,
$A(R)$ is the mass of  
$\phi_*\{(\disk(0, R)) \}$. 
  
\begin{pro}[see Brunella  {\cite[\S 1]{Brunella}}]\label{pro:brunella}
If $\phi:\C\to X$ is  a non-constant entire curve, there exist  sequences of radii 
$(R_n)$ increasing to infinity such that the sequence of currents 
\[
N(R_n)  = \unsur{T(R_n)} \int_0^{R_n} \phi_*\set{\disk(0,t)} \frac{dt}{t}
\]
 converges to a closed positive current $T$. 
If furthermore $\phi(\C)$ is Zariski dense, and $T$ is such a closed current, the class
$[T]\in H^{1,1}(X, \R)$ is nef. In particular $\langle [T] \, \vert\, [T] \rangle\geq 0$  and
 $\langle [T] \, \vert\, [C] \rangle\geq 0$ for every algebraic curve $C\subset X$. 
\end{pro}
Such limit currents $T$ will be referred to as \textbf{Ahlfors-Nevanlinna currents} associated to the entire curve $\phi\colon \C\to X$. 
If  $\phi(\C)$  is not  Zariski dense then the closure $\overline{\phi(\C)}$  (for the euclidean topology) is a (possibly singular) 
curve of genus $0$ or $1$; if $\phi$ is injective, then $\overline{\phi(\C)}$ is rational.

\subsection{Equidistribution of stable manifolds} If $\mu$ is hyperbolic,  or more generally if it admits exactly one 
 negative Lyapunov exponent, then, for $\m_+$-almost every $\cx=(\omega, x)\in \X_+$, the   stable manifold $W^s(\cx)$, which   
 is viewed here 
 as a subset of $X$ as in Proposition \ref{pro:pesin_stable}, is parametrized by an injectively immersed 
entire curve. Then we can  relate the Ahlfors-Nevanlinna currents to the limit currents $T^s_\omega$;
here are the three main results that will be proved in this section.  
 
\begin{thm}\label{thm:nevanlinna}
Let $(X, \nu)$ be a non-elementary random holomorphic dynamical 
system on a compact K\"ahler surface, satisfying \eqref{eq:moment}. Let $\mu$ be an ergodic stationary measure such that $\lambda^-(\mu)<0\leq \lambda^+(\mu)$. 
Then  exactly one of the following alternative holds.
\begin{enumerate}[{\em (a)}]
\item For $\m_+$-almost every $\cx$, the stable manifold $W^s(\cx)$
 is not Zariski dense. Then $\mu$ is supported on  a 
$\Gamma_\nu$-invariant curve  
$Y\subset X$ and for $\m_+$-almost every $\cx$, $W^s(\cx) \subset Y$. In addition   every component of $Y$ is 
a rational curve, and the intersection form is negative definite on the subspace of $H^{1,1}(X;\R)$ generated by the  classes of components of $Y$.
\item For $\m_+$-almost every $\cx$ the stable manifold $W^s(\cx)$
 is  Zariski dense and the only normalized Ahlfors-Nevanlinna  current associated to 
 $W^s(\cx)$ is $T^s_\omega$.  
\end{enumerate}
\end{thm}

\begin{cor}\label{cor:nevanlinna}
Under the assumptions of Theorem~\ref{thm:nevanlinna}, if in addition 
 $\mu$ is hyperbolic and non-atomic, then the Alternative~\emph{(b)} is equivalent to 
 \begin{enumerate}[{\em (b')}]
 \item[{\emph{ (b')}}] $\mu$ is not supported on a $\Gamma_\nu$-invariant curve. 
 \end{enumerate} 
 \end{cor}

\begin{cor}\label{cor:HD_stable}
Under the assumptions of Theorem~\ref{thm:nevanlinna},  assume furthermore that $\nu$ satisfies the exponential moment condition \eqref{eq:exponential_moment}. Then in Alternative~(b)   there exists $\theta>0$ such that for 
$\m_+$-almost every $\cx\in \X_+$ the Hausdorff dimension of $\overline{W^s(\cx)}$ equals  $2+\theta$.  \end{cor} 

\subsection{Proof of Theorem~\ref{thm:nevanlinna} and  its corollaries} 
  
We work under the assumptions of Theorem~\ref{thm:nevanlinna}.  

\begin{lem}\label{lem:mu_zariski}
If there exists a proper Zariski closed subset of $X$ with positive $\mu$-measure, then:
\begin{itemize}
\item either $\mu$  is the uniform counting measure on a finite orbit of $\Gamma_\nu$;
\item or $\mu$ has no atom and it is supported on a $\Gamma_\nu$-invariant algebraic curve, which is the 
$\Gamma_\nu$-orbit of an irreducible  algebraic curve.
\end{itemize}
\end{lem}
\begin{proof}
Consider the real number $\delta_{\mathrm{max}}^0(\mu)=\max_{x\in X} \mu\lrpar{\{x\}}$. If $\delta_{\mathrm{max}}^0(\mu)>0$, 
there is a non-empty finite set $F\subset X$ for which $\mu\lrpar{\{x\}} = \delta_{\mathrm{max}}^0(\mu)$. 
By stationarity, $F$ is $\Gamma_\nu$-invariant, and by ergodicity  $\mu$ is the uniform measure  on $F$.
Now, assume that  $\mu$ has no atom. Let $\delta_{\mathrm{max}}^1(\mu)$ be the maximum of $\mu(D)$
among all irreducible curves $D\subset X$. If $\mu(Z)>0$ for some proper Zariski closed subset $Z\subset X$, then
$\delta_{\mathrm{max}}^1(\mu)>0$. Since two distinct irreducible curves intersect in at
most finitely many points and $\mu$ has no atom, there are only finitely many irreducible curves $E$ such 
that $\mu(E)=\delta_{\mathrm{max}}^1(\mu)$. To conclude, we argue as in the zero dimensional case.
\end{proof}


If $V\subset X$ is a smooth curve, possibly with boundary, if $T$ is a  closed  positive $(1,1)$-current on $X$ with a continuous 
normalized potential $u_T$ (as in~\S~\ref{subs:normalized_potentials}), then  by definition 
\begin{equation}
\langle T\wedge \set{V}, \varphi\rangle  = \int_V \varphi \, \Theta(T) + \int_V \varphi \, dd^c(u_T\rest{V}),
\end{equation}
for every test function $\varphi$. Here is the key relation  between  stable manifolds and limit currents:

\begin{lem}\label{lem:stable_intersection_Ts}
For $\m_+$-almost every  $\cx = (\omega, x)$, if 
$\Delta$ is a disk contained in $W^s(\cx)$, then $T_\omega^s \wedge \{ \Delta\} = 0$. 
\end{lem}

\begin{proof}   
With no loss of generality we assume 
that the boundary of  the disk $\Delta$ in $W^s(\cx)\simeq \C$  is smooth. 
We consider points  $\cx=(\omega, x)\in \X_+$ which are generic in the following sense: 
they are regular from the point of view of Pesin's theory, and 
 $T^s_\omega $ satisfies the conclusions of \S \ref{sec:currents}. 

By Pesin's theory, for every $\e>0$, there is a set $A_\e\subset \mathbf N$ of density 
larger than  $1-\e$, such that for $n$ in $A_\e$, the local stable manifold  $W^s_{r}(F_+^n(\cx))$ 
is a disk of size $r=r(\e)$ at $f_\omega^n(x)$  
and $f_\omega^n(\Delta)$ is a disk contained in an   exponentially small 
neighborhood of $f_\omega^n(x)$.  We  have 
\begin{equation}\label{eq:TnsDelta}
\M(T_{\sigma^n\omega}^s\wedge \{f_\omega^n(\Delta)\}) 
= \int_{ W^s_{r} (F_+^n(\cx))}\! \mathbf{1}_{f_\omega^n(\Delta)} \, \Theta({T^s_{\sigma^n\omega}})
+  \int_{ W^s_{r}(F_+^n(\cx))}\!  \mathbf{1}_{f_\omega^n(\Delta)}
dd^c u_{T_{\sigma^n\omega}^s}.
\end{equation}
Since $\M(T_{\sigma^n\omega}^s)=1$,  Lemma \ref{lem:uniform_Ak_psh} shows that $\Theta({T^s_{\sigma^n\omega}}) $ is bounded by $A\kappa_0$; so 
the first integral on the right hand side of \eqref{eq:TnsDelta}  is bounded by a constant times the area of 
$f_\omega^n(\Delta)$, which is exponentially small.
By ergodicity, there exists $A'_\e \subset A_\e$ of density at least 
$1-2\e$ such that if $n\in A'_\e$,  $\| u_{T_{\sigma^n\omega}^s}\| _\infty$  is  bounded by 
some contant $D_\e>0$.  For such an $n$, let $\chi$ be a test function in $W^s_r (F_+^n(\cx))$ 
such that $\chi=1$ in $W^s_{r/2} (F_+^n(\cx))$, and vanishing near $\fr W^s_{r} (F_+^n(\cx))$. Note that since 
$W^s_r (F_+^n(\cx))$  is of size $r$, the $C^2$-norm of $\chi$
  depends only on $r$.  We write  
\begin{align}
\notag\int_{ W^s_{r}(F_+^n(\cx))}\!  \mathbf{1}_{f_\omega^n(\Delta)}
dd^c u_{T_{\sigma^n\omega}^s} 
&\leq \int_{ W^s_{r}(F_+^n(\cx))}\!  \chi dd^c u_{T_{\sigma^n\omega}^s} \\ 
&=  \int_{ W^s_{r}(F_+^n(\cx))}\!   
u_{T_{\sigma^n\omega}^s} dd^c\chi \\ \notag
& \leq C(r) \norm{\chi}_{C^2} \big\|{u_{T_{\sigma^n\omega}^s} } \big\|_\infty
\end{align}
where $C(r)$ bounds the area of $W^s_{r}(F_+^n(\cx))$; this last term is uniformly bounded because $n\in A'_\e$.
Thus we conclude that $\M(T_{\sigma^n\omega}^s\wedge\{f_\omega^n(\Delta)\})$ is bounded along such a subsequence. 
  
 On the other hand, the relation $(f_\omega^n)^* T_{\sigma^n\omega}^s =  \M((f_\omega^n)^*T_{\sigma^n\omega}^s)  T_{\omega}^s$ gives 
\begin{equation}
 T_{\sigma^n(\omega)}^s\wedge \{ f_\omega^n(\Delta)\} = \M\lrpar{(f_\omega^n)^*T_{\sigma^n(\omega)}^s}
  (f_\omega^n)_* (T_\omega^s\wedge \{ \Delta\}).
\end{equation} 
The mass $\M((f_\omega^n)_* (T_\omega^s\wedge \{ \Delta\}))$ is constant, equal to the mass of the
measure $T_\omega^s\wedge \{ \Delta\}$; so
\begin{equation}
\M\lrpar{T_{\sigma^n(\omega)}^s\wedge \{ f_\omega^n(\Delta)\} } = \M((f_\omega^n)^*T_{\sigma^n(\omega)}^s) \M(T_\omega^s\wedge \{\Delta\}).
\end{equation}
By Lemma \ref{lem:variant_furstenberg}, $\M((f_\omega^n)^*T_{\sigma^n(\omega)}^s) $ goes exponentially fast  to infinity. Since the left
hand side is bounded, this shows that $\M(T_\omega^s\wedge \{\Delta\})=0$, as desired.
\end{proof}

With Lemma~\ref{lem:periodic_curves}, the following statement takes care of the first alternative in Theorem \ref{thm:nevanlinna}. 

\begin{lem}\label{lem:stable_zariski_dense}
If there is a Borel 
subset $A\subset \X_+$ of positive measure  such that 
 for every $\cx \in A$, the stable manifold $W^s(\cx)$ is contained in an algebraic curve, 
then $\mu$ is supported on a $\Gamma_\nu$-invariant algebraic curve. 
In addition, for $\m_+$-almost every $\cx$,  $\overline{W^s(\cx)}$ is an irreducible rational curve of negative self-intersection. 
\end{lem}

\begin{proof}
For $\cx\in A$, let  $D(\cx)$ be the Zariski closure of $W^s(\cx)$. Discarding a set of measure zero if needed, 
$W^s(\cx)$ is biholomorphic to $\C$ so $D(\cx)$ is a (possibly singular) irreducible rational curve, 
and $D(\cx)\setminus W^s(\cx)$ is reduced to a point.
By Lemma 
\ref{lem:stable_intersection_Ts}, $T^s_\omega\wedge \set{\Delta} = 0$ for every disk $\Delta \subset W^s(\cx)$. Since  
$T^s_\omega$ has continuous potentials, $T^s_\omega\wedge \set{D(\cx)}$ gives no mass to points (see 
e.g. \cite[Lem. 10.13]{cantat-dupont} for the singular case). 
It follows that $T^s_\omega\wedge \set{D(\cx)}=0$, hence $\langle e(\omega) \, \vert \, [D(\cx)]\rangle= 0$. 

By the Hodge index theorem, either $[D(\cx)]^2<0$ or $[D(\cx)]$ is proportional to $e(\omega)$, however 
 this latter case would contradict the 
fact that $e(\omega)$ is $\nu^\N$-almost surely irrational (see Theorem~\ref{thm:def_stationary}; one could also use that $\Cur(e(\omega))$ is reduced to $T^s_\omega$). Thus, $[D(\cx)]^2<0$.

An irreducible curve with negative self-intersection is uniquely determined by its cohomology class; since $\NS(X; \Z)$ is countable,  there are only countably many irreducible curves $(D_k)_{k\in \N}$ with  negative self intersection.
Since $W^s_\loc(\cx)\subset D_k$ if and only if  $D(\cx) = D_k$, and 
since local stable manifolds vary continuously on the Pesin regular set $\mathcal R_\e$
 for every $\e>0$, we infer that 
$\set{\cx \in A\; ; \;  \ D(\cx) = D_k}$ is measurable for every $k$. Hence there exists an 
index $k$ such that $\m_+ \lrpar{ \set{ \cx\in A\; ;  \;  [D(\cx)]=[D_k]}}>0$.  Since $x$ belongs to $W^s_\loc(\cx)$, Fubini's theorem 
implies that $\mu(D_k)>0$, and  Lemma \ref{lem:mu_zariski} shows that  $\mu$ is 
supported on the $\Gamma_\nu$-orbit of 
$D_k$. 

Finally, this argument shows that the 
property   $W^s_\loc(\cx)\subset \bigcup_{k\in \N} D_k$, or equivalently that 
$W^s_\loc(\cx)$ is contained in a rational curve of negative self intersection, is invariant and measurable, 
so by ergodicity of $\m_+$ it is of full measure. The proof is complete.
\end{proof}

We are now ready to conclude the proof of Theorem~\ref{thm:nevanlinna}. 
Let $A$ be the set of Pesin regular points such that $W^s(\cx)$ is contained in an algebraic curve. 
From the proof 
of Lemma \ref{lem:stable_zariski_dense},  $\cx$ belongs to $A$ if and only if $W^s_\loc(\cx)$ 
 is contained in one of the countably many irreducible curves 
 $D_k\subset X$ of negative self-intersection. This  
   condition   determines a countable union of closed subsets in the Pesin sets ${\mathcal{R}}_\e$, hence
   $A$ is  Borel measurable.  
%
%
By Lemma \ref{lem:stable_zariski_dense}, if  $A$ has positive $\m_+$-measure 
 then   Alternative~{{(a)}} holds. So,  if (a) is not satisfied, $W^s(\cx)$ is almost surely Zariski dense. 
Pick such a generic $\cx$, which further satisfies the conclusion of Lemma~\ref{lem:stable_intersection_Ts},
and let $N$ be 
  an Ahlfors-Nevanlinna current  associated to $W^s(\cx)$. By Proposition 
\ref{pro:brunella}, $[N]$ is a nef class so $[N]^2 \geq 0$. Thus, if we are able to show that 
 $\langle [N] \, \vert \, [T^s_\omega] \rangle = 0$, 
we deduce from the Hodge index theorem and $\M(N)=1$ that 
$[N] =[T^s_\omega]  = e(\omega)$, hence $N = T^s_\omega$ by Theorem \ref{thm:uniq+extremal}. 
So, it   only remains to prove that $\langle [N] \, \vert \, [T^s_\omega] \rangle = 0$, or equivalently 
\begin{equation}\label{eq:NT=0}
N\wedge T^s_\omega = 0.
\end{equation} 
This is intuitively clear because $N$ is an Ahlfors-Nevanlinna current
associated to the entire curve $W^s(\cx)$ and  $T^s_\omega\wedge \set{\Delta}=0$ for every bounded  disk 
$\Delta\subset  W^s(\cx)$. However, there is a technical difficulty to derive \eqref{eq:NT=0} from $T^s_\omega\wedge \set{\Delta}=0$, even if $W^s(\cx)$ is an increasing union of such disks~$\Delta$. 

At least two methods were designed to deal with this situation: the first one uses the geometric intersection  
theory of laminar currents (see \cite{bls, isect}), and the second one was developed by Dinh and Sibony in the preprint 
version of \cite{dinh-sibony_jams} (details are published in~\cite[\S 10.4]{cantat-dupont}). Unfortunately these papers 
only deal with the case of currents of the form 
$\lim_{n} \unsur{A(R_n)}{\phi(\disk(0, R_n))}$, instead of the Ahlfors-Nevanlinna currents introduced in 
Section~\ref{par:Ahlfors-Nevanlinna-Construction}, which were designed to get the nef property 
stated in Proposition~\ref{pro:brunella}. So, we have to explain how to adapt the formalism of 
\cite{bls, isect} to the Ahlfors-Nevanlinna currents of Proposition \ref{pro:brunella}.

Following \cite{duval} we say that $T$ is an \textbf{Ahlfors current} if there exists a sequence $(\Delta_n)$ of \emph{unions}
of smoothly bounded holomorphic disks such that $\length(\fr \Delta_n)   = o\lrpar{\M(\Delta_n)}$  and 
$T$ is the limit as $n\to\infty$ of   the sequence of normalized integration currents
$\unsur{\M(\Delta_n)}\set{\Delta_n}$; here,  $\length(\fr \Delta_n)$ is by definition the sum of the lengths
  of the boundaries of the disks constituting $\Delta_n$,  lengths which are computed with respect to the Riemannian 
metric induced by $\kappa_0$.  We say furthermore that $T$ is an \textbf{injective Ahlfors current} if 
the disks constituting  $\Delta_n$ are disjoint or intersect along 
subsets with relative non-empty interior.
By discretizing the integral defining the currents $N(R_n)$ in Proposition~\eqref{pro:brunella} we see that any Ahlfors-Nevanlinna current 
is an injective Ahlfors current.  

{\bf{Strongly approximable}} laminar currents are a class of positive currents introduced in \cite{isect} with geometric properties 
which are well suited for geometric intersection theory. 
In a nutshell, a current $T$ is a strongly approximable laminar current if for every $r>0$, there exists a uniformly laminar 
current $T_r$ (non closed in general) made of disks of size $r$, and such that 
$\M(T-T_r)  = O(r^2)$.  This mass estimate is  crucial for the  geometric understanding  of 
wedge products of such currents.  
Since these notions have been studied in a number of papers, we refer to \cite{bls, isect, Cantat:Milnor} 
for definitions, the basic properties of these currents, and technical details. This presentation in terms of disks of size $r$ is from 
\cite[\S 4]{fatou}.
The next lemma is a mild generalization of  the methods of \cite[\S 7]{bls},  \cite[\S 4.3]{Cantat:Acta} and  \cite[\S 4]{isect}. 
For completeness we provide the details in 
  Appendix~\ref{par:appendix_ahlfors}.

\begin{lem}\label{lem:ahlfors_current}
Any injective Ahlfors current $T$ on a projective surface $X$ is a strongly approximable laminar current: if 
$T= \lim_n \unsur{\M(\Delta_n)}\set{\Delta_n}$ where the disks $\Delta_n$ 
have smooth boundaries and $\length(\fr \Delta_n)   = o\lrpar{\M(\Delta_n)}$,
one can construct a family of uniformly laminar currents $T_r$, 
whose constitutive disks are limits of pieces of the $\Delta_n$, and such that  if $S$ is any closed positive current with continuous potential on $X$, 
then $S\wedge T_r$ increases to $S\wedge T$ as $r$  decreases to $0$. 
 \end{lem}

With this lemma at hand, let us conclude the proof of Theorem~\ref{thm:nevanlinna}. Since $X$ is projective, we can apply the 
previous lemma to  any Ahlfors-Nevanlinna current $N$ associated to $W^s(\cx)$. 
In this way we get a family of currents $N_r$ such that  $N_{r}\wedge T^s_\omega$ increases to $N\wedge T^s_\omega$
as $r$  decreases to~$0$. 
On the other hand, by Lemma~\ref{lem:stable_intersection_Ts}, the intersection of $T^s_\omega$ with 
every disk contained 
in $W^s(\cx)$ vanishes, so again using the fact that $T^s_\omega$ has a continuous potential, 
we infer that if $\Delta$ is any disk subordinate to $N_r$, $T^s_\omega\wedge \set{\Delta} =0$. Hence 
$N_{r}\wedge T^s_\omega = 0$ for every $r>0$, and finally  $N\wedge   T^s_\omega=0$, as desired.  
\qed

\begin{proof}[Proof of Corollary \ref{cor:nevanlinna}] 
Since (b') and (a) are contradictory, (b') implies (b). Conversely assume that $\mu$ is hyperbolic, non atomic 
and supported on a $\Gamma_\nu$-invariant curve $C$. Since $\mu$ has no atom, it gives full mass to the regular set of 
$C$, hence $\Sigma\times T(\mathrm{Reg}(C))$ defines a $DF$-invariant bundle, and by the Oseledets theorem    the ergodic random dynamical system $(C, \nu, \mu)$ must either have a positive or a negative Lyapunov exponent. If this 
exponent were
positive then $\mu$ would be atomic, as observed in Section~\ref{par:hyperbolicity_of_mu}.
Hence, the Lyapunov exponent tangent to $C$ is negative
and $W^s(\cx)$ is contained in $C$ for $\m_+$-almost every $\cx$. So (b) implies (b').
\end{proof}

\begin{proof}[Proof of Corollary \ref{cor:HD_stable}]
Since $\nu$ satisfies an exponential moment condition, Theorem \ref{thm:holder} provides a $\theta>0$ such 
that $u_{T^s_\omega}$ is  H\"older continuous of exponent  $\theta$ for $\nu^\N$-almost every~$\omega$. This implies
that $T^s_\omega$ gives mass $0$ to sets of Hausdorff dimension $<2 + \theta$  (see \cite[Thm 1.7.3]{sibony}). 
Since for $\m_+$-almost every $x$, $\supp(T^s_\omega)\subset \overline{ W^s(\cx)}$, we infer that  
$\mathrm{HDim}\big({\overline{ W^s(\cx)}}\big)\geq 2+\theta$. 

To conclude the proof it is enough to show that $\cx\mapsto \mathrm{HDim}\big({\overline{ W^s(\cx)}}\big)$ 
is constant on a set of full $\m_+$-measure. Indeed,  $\cx\mapsto \mathrm{HDim}\big({\overline{ W^s(\cx)}}\big)$ 
defines an  $F_+$-invariant function, defined on the full measure set $\mathcal R$ 
of Pesin regular points. If we show that this function is measurable, then the result follows 
by ergodicity. This is a consequence of the following two facts:
\begin{enumerate}
\item  the assignment  
$\cx\mapsto   {\overline{ W^s(\cx)}}$ defines a Borel map from $\mathcal R$ to the space $\mathcal{K}(X)$ 
of compact subsets of $X$; 
\item the function $\mathcal{K}(X)\ni K \mapsto \mathrm{HDim}(K)$ is Borel (see \cite[Thm 2.1]{mattila-mauldin}).
\end{enumerate}
In both cases $\mathcal{K}(X)$ is endowed with the topology induced by the Hausdorff metric. 
For the first point, observe that $\mathcal R$ is the increasing union 
of the compact sets $\mathcal R_\e$ so it is Borel; then, on a Pesin set $\mathcal R_\e$, 
$\cx\mapsto   {\overline{ W^s_r(\cx)}}$   is continuous, so $\cx\mapsto  F^{-n} \big(\overline{ W^s_r(F^n(\cx))}\big)$ is continuous as well. 
Since $F^{-n} \big(\overline{ W^s_r(F^n(\cx))}\big)$ converges to  ${\overline{ W^s(\cx)}}$ in the Hausdorff topology, 
we infer that $\cx\mapsto   {\overline{ W^s (\cx)}}$  is a pointwise limit of continuous maps on $\mathcal R_\e$, hence Borel, 
and finally $\cx\mapsto   {\overline{ W^s (\cx)}}$ is Borel on $\mathcal R$, as claimed.   
\end{proof} 

\section{No invariant line fields}\label{sec:No_Invariant_Line_Fields}

As above, let  $(X, \nu)$ be a random holomorphic dynamical system satisfying the moment condition~\eqref{eq:moment}, 
and $\mu$ be an ergodic hyperbolic stationary measure. 
From \S\ref{subs:lyapunov} and \S\ref{subs:pesin}, the local stable manifolds  and stable Oseledets directions are 
$\mathcal F^+$-measurable; so, $E^s(\xi, x)$ is naturally identified to 
$E^s(\omega, x)$ under the  projection $(\xi, x) \in\X \mapsto (\omega,x)\in \X_+$, and  the same property holds for stable manifolds.  
Then,  $\m_+$-almost every $\cx \in \X_+$
has a Pesin stable manifold $W^s(\cx)$ (resp. direction $E^s(\cx)$). 
Let $V (\cx)=V(\omega,x)$ be such a measurable family of objects (stable manifolds, or stable directions, etc); we say that 
$V (\cx)$ is \textbf{non-random} if for  $\mu$-almost every $x$, $V(\omega,x)$ 
does not depend on 
$\omega$, that is,  there exists $V(x)$ such that $V({\omega},x) = V(x)$  for  $\nu^\N$-almost every $\omega$. 
If $V$ is not non-random,  
we say that $V$ \textbf{depends non-trivially on the itinerary}. 
 Since stable directions depend only on the future,  the random versus non-random dichotomy can be analyzed in $\X_+$ or in $\X$.
Our purpose in this section is to establish the following result.

\begin{thm}\label{thm:alternative_stable}
Let $(X, \nu)$ be a non-elementary  random holomorphic dynamical system on a compact K\"ahler surface 
satisfying the Condition \eqref{eq:moment}. 
Let $\mu$ be an ergodic and hyperbolic  stationary measure,  not supported on a $\Gamma_ \nu$-invariant curve.  Then the 
following alternative holds:
\begin{enumerate}[{\em (a)}]
\item
either the Oseledets stable directions depend non-trivially on the itinerary;
\item
or $\mu$ is $\nu$-almost surely invariant and  $h_\mu(X, \nu) = 0$. 
\end{enumerate}
\end{thm}

In fact, 
our stiffness theorems imply that often, $\mu$ is also invariant in case {(a)} (see \S \ref{sec:stiffness}). In {(b)}, the almost-sure invariance
implies that $\mu$ is in fact $\Gamma_\nu$-invariant (see Remark~\ref{rem:discrete}). 
It turns out that   {(a)} and  {(b)} are mutually exclusive. Indeed
the main argument of  \cite{br} 
(\footnote{This actually requires checking that the whole proof of  \cite{br}
can be reproduced   in our complex setting: we will come back to this issue  in a forthcoming paper. 
 Since we are just using this remark here in Corollary \ref{cor:positive_entropy} we take the liberty to anticipate on that research.})
 implies that the fiber entropy is positive if the  Oseledets stable directions depend non-trivially on the itinerary
 (see \cite[Rmk 12.3]{br}). So we get the following:

\begin{cor}\label{cor:positive_entropy}
Let $(X, \nu, \mu)$ be as in Theorem \ref{thm:alternative_stable}. If 
$\mu$ is not $\nu$-almost surely invariant, then its fiber entropy is positive.
\end{cor}

\begin{vlongue}
To motivate the following pages, let us
 give a heuristic explanation  for the fact that 
 $h_\mu(X,\nu)=0$ when the stable 
directions are non-random. Fix a stable Pesin partition $\eta^s$; according to Corollary~\ref{cor:rokhlin_zero_entropy},  we have to 
show that the conditional measures $\m(\cdot \vert \eta^s(\cx))$ are atomic. 
Since the stable directions are non-random, 
the stable manifolds $W^s_\loc(\xi,x)$ and $W^s_\loc(\xi’,x)$ are generically tangent at $x$. 
For simplicity, assume 
that they are tangent for $\mu$-almost all $x$ and {\textit{for all pairs $(\xi,\xi’)$}, and that 
\textit{$W^s_\loc(\xi,x)$ 
depends continuously on $(\xi,x)$}}. Take such a generic point $x$; if $\m(\cdot \vert \eta^s(\xi,x))$ is not atomic, there
is a sequence of generic points $x_j\in W^s_\loc(\xi,x)$ converging to $x$ in $X\simeq X_\xi$. Fix $\xi’\neq \xi$. 
Then by continuity $W^s_\loc(\xi’,x_j)$ converges towards $W^s_\loc(\xi’,x)$,  is disjoint 
from $W^s_\loc(\xi’,x)$, and is tangent to $W^s(\xi, x)$ at $x_j$. 
This contradicts the following local geometrical result: if $C$ and $D$ 
are local smooth irreducible curves through the origin in $\disk^2$, with an order of contact equal to $k$, 
and if $D_n\subset \disk^2$ is a sequence of curves such that $D_n\cap D=\emptyset$ 
but $D_n$ converges towards $D$ in $\disk^2$, then for $n$ sufficiently large, $D_n$ intersects $C$ transversally in $k$ points.  
\end{vlongue}

\subsection{Intersection multiplicities}
\begin{vlongue}
Let us start with some basics on intersection multiplicities for curves. 
\end{vlongue} 
If  $V_1$ and $V_2$ are germs of  curves at $0\in \C^2$, with an isolated intersection at 0,
 the \textbf{intersection multiplicity}  $\inter_0(V_1, V_2)$ is, by definition, 
the number of intersection points of $V_1$ and $V_2+u$ in $N$ for small generic $u\in \C^2$, where $N$ is a neighborhood 
of 0 such that $V_1\cap V_2\cap N = \set{0}$ (see \cite[\S 12]{chirka}).
It is a positive integer, and 
$\inter_0(V_1, V_2)=1$ if and only if $V_1$ and $V_2$ are transverse at 0. 
We extend this definition by setting $\inter_0(V_1, V_2) = 0$ 
if $V_1$ or $V_2$ does not contain 0 and  $\inter_0(V_1, V_2)= \infty$ if 0 is not an isolated point of $V_1\cap V_2$, that is locally
$V_1$ and $V_2$  share an  irreducible component. The intersection multiplicity  extends to analytic cycles (that is, formal integer combinations of analytic curves).

\begin{lem}\label{lem:upper_sc_inter} 
The multiplicity of intersection $\inter_0(\cdot, \cdot)$ is upper semi-continuous for the Hausdorff topology 
on analytic cycles. 
\end{lem}

In our situation we will only apply this result to holomorphic disks with multiplicity 1, in which case the topology is just the usual local Hausdorff topology.

\begin{proof} Assume $\inter_0(V_1, V_2)=k$ and $V_{1,n} \to V_1$ (resp. $V_{2,n} \to V_2$)
 as cycles; we have to show that $\limsup  \inter_0(V_{1,n}, V_{2,n}) \leq k$. If $k = \infty$ there is nothing to 
 prove. Otherwise, $\set{0}$ is isolated in $V_1\cap V_2$, so we can fix a neighborhood $U$ of $0$ such that 
 $V_1\cap V_2\cap U  =\set{0}$; then, the result follows from \cite[Prop 2 p.141]{chirka} (stability of proper intersections). 
\end{proof}

\subsection{Generic intersection multiplicity of stable manifolds}\label{subs:generic_intersection}
Recall from \S\ref{subs:pesin} that for $\m$-almost every $\cx = (\xi,x)\in \X$ there exists a local stable manifold $W^s_{r(\cx)}(\cx) \subset 
X_\xi\simeq X$, depending measurably on $\cx$; we might simply denote it by $W^s_\loc(\cx)$.  

Let us cover a subset of full measure in $\X$ by Pesin subsets $\mathcal R_{\e_n}$. Take a point $x\in X$, and consider
the set of points $((\xi, x), (\zeta, x))\in {\mathcal R_{\e_n}}\times {\mathcal R_{\e_m}}$, for some fixed pair of indices $(n,m)$; 
Lemma~\ref{lem:upper_sc_inter} shows that the 
intersection multiplicity $\inter_x \lrpar{W^s_\loc(\xi, x), W^s_\loc(\zeta, x)}$ is an 
 upper semi-continuous function of $((\xi, x), (\zeta, x))$
on that compact set. Thus, the intersection multiplicity $\inter_x \lrpar{W^s_\loc(\xi, x), W^s_\loc(\zeta, x)}$ 
is a measurable function of  $(\xi,\zeta)$.  
Recall that 
\begin{itemize}
\item the $\sigma$-algebra $\mathcal F^-$ on $\cX$ is generated, modulo $\m$-negligible sets, by 
the partition into subsets of the form
$\Sigma^u_\loc(\xi)\times \set{x}$ (see \S~\ref{par:definition_skew_products}, Equation~\eqref{eq:definition_sigma_s_u});
\item  $\xi \mapsto \m_\xi$ is  $\mathcal F^-$-measurable, {i.e} $\m_\xi = \m_\zeta$ almost surely when $\zeta\in \Sigma^u_\loc(\xi)$;
\item the conditional measures of $\m$ with respect to this partition satisfy  (see Equation~\eqref{eq:conditional})
\begin{equation}\label{eq:conditional_Part_8}
\m(\; \cdot\; \vert \;   \mathcal F^-(\cx)) =    {\nu^\Z(\; \cdot\; \vert \;  \Sigma^u_\loc(\xi))}\times \delta_x.
\end{equation}
\end{itemize}
 
 The next lemma can be seen  as a complex analytic version of   \cite[Lemma 9.9]{br}.  

\begin{lem}\label{lem:alternative_multiplicity}
Let $k\geq 1$ be an integer. Exactly one of the following assertions holds:  
\begin{enumerate}[\em (a)]
\item for $\m$-almost every $\cx=(\xi,x)$ and for $\m(\; \cdot\; \vert \;  \mathcal F^-(\xi,x))$-almost every $\eta$
\[
\inter_x\lrpar{W^s_\loc(\xi, x),  W^s_\loc(\eta, x)}\geq  k+1;
\]
\item  for $\m$-almost every $\cx=(\xi,x)$ and for $\m(\; \cdot\; \vert \;  \mathcal F^-(\xi,x))$-almost every $\eta$ 
\[ 
\inter_x\lrpar{W^s_\loc(\xi, x),  W^s_\loc(\eta, x)}\leq  k.
\]
\end{enumerate}
\end{lem}

\begin{proof} 
The   relation defined on $\X$ by  $(\xi, x)\simeq_{k} (\eta, y)$ if 
$x=y$ and $W^s_\loc(\xi, x)$ and $W^s_\loc(\eta, y)$ have order of contact at least $k+1$ at $x$ 
is an equivalence relation which defines a partition $\mathcal Q_k $  of $\X$. We shall see  below that $\mathcal Q_k $
is a measurable partition.
Since $F\colon \X\to \X$ acts by diffeomorphisms on the fibers $X$ of $\X$, we get that 
$F(\mathcal Q_k (\cx)) = \mathcal Q_k(F(\cx))$ for almost every $\cx\in \X$. Then, the proof of \cite[Lemma 9.9]{br}
applies
 verbatim to show that if 
\begin{equation}
\m\lrpar{\set{\cx \; ; \;  \m(\mathcal Q_k  (\cx)\vert \mathcal F^-(\cx))>0 }} >0,
\end{equation}
then 
 \begin{equation}
\m\lrpar{\set{\cx \; ; \;  \m(\mathcal Q_k (\cx)\vert \mathcal F^-(\cx))=1 }} =1.
\end{equation} 
This is exactly the desired statement. (This assertion says more than the mere ergodicity of $\m$, which  
 only implies that 
$\m\lrpar{\set{\cx; \ \m(\mathcal Q_k (\cx)\vert \mathcal F^-(\cx))>0 }} =1$.)  

It remains to explain why $\mathcal Q_k$ is a measurable partition. For this, we have to express the atoms 
of  $\mathcal Q_k$ as the fibers of a measurable map to a Lebesgue space. 
As for the measurability of the intersection multiplicity, we consider an exhaustion of $\X$ by countably many Pesin sets; then, it
is sufficient to work in restriction to some compact set $\mathcal K \subset \X$ on which local stable manifolds have uniform size and vary continuously. 
Taking a finite cover of $X$ by good charts (see~\S~\ref{par:good_charts}), and restricting $\mathcal K$ again 
to keep only those local stable manifolds which are graphs over some fixed direction, we can also assume that 
$\pi_X(\mathcal K)$ is contained in the image of a chart $\hexp_{x_0}\colon U_{x_0}\to V_{x_0}\subset X$ and 
there is an orthonormal basis $(e_1,e_2)$ such that for every $\cy\in \mathcal K$ the local stable 
manifold $\pi_X(W^s_\loc(\cy))$ is a graph $\{z e_1+\psi^s_\cy(z)e_2\}$ in this chart, for some holomorphic function $\psi^s_\cy$ on $\disk(r)$.
Now the map from $\mathcal K $ to $ \C^2\times \C^k$ defined by 
\begin{equation}
\cx\longmapsto \lrpar{\hexp_{x_0}^{-1}(\pi_X(\cx)), (\psi^s_\cx)'(0), \ldots , (\psi^s_\cx)^{(k)}(0)}
\end{equation}
is continuous.  Since the fibers of this map are precisely the (intersection with $\mathcal K$ of the) atoms of $\mathcal Q_k$, we are done. 
\end{proof}

The previous lemma is stated on $\X$ because its proof relies on the ergodic properties of~$F$. 
However, since stable manifolds depend only on the future, it admits the following more elementary formulation on $X$:
 
 \begin{cor}\label{cor:alternative_multiplicity}
 Let $k\geq 1$ be an integer. Exactly one of the following assertions holds: 
\begin{enumerate}[\em (a)]
\item for $\mu$-almost every $x\in X$ and $(\nu^\N) ^2$-almost every $(\omega, \omega')$, 
\[
\inter_x\lrpar{W^s_\loc(\omega, x),  W^s_\loc(\omega', x)}\geq k+1;
\]
\item or  for $\mu$-almost every $x\in X$ and $(\nu^\N) ^2$-almost every $(\omega, \omega')$, 
\[
\inter_x\lrpar{W^s_\loc(\omega, x),  W^s_\loc(\omega', x)}\leq k.
\]
 \end{enumerate}
 \end{cor}

 Combined with results from the previous sections,  this alternative leads to the existence of a  
finite order of contact $k_0$ between generic stable manifolds $W^s_\loc(\omega,x)$ and $W^s_\loc(\omega', x)$:

\begin{lem}\label{lem:generic_multiplicity}
There exists a   unique 
{\emph{finite}} integer $1\leq k_0 < +\infty$ such that for 
$\mu$-almost every $x\in X$ and $(\nu^\N) ^2$-almost every pair $(\omega, \omega')$, 
$
\inter_x\lrpar{W^s(\omega, x),  W^s(\omega', x)}= k_0.
$
\end{lem}

\begin{proof}
Fix a small   $\e>0$ and consider a compact set $\mathcal R_\e\subset \X_+$ with $\m_+(\mathcal R_\e)\geq 1-\e$, 
along which local stable manifolds have size at least $r(\e)$ and vary continuously.
Since by Theorem \ref{thm:nevanlinna} 
for $\m_+$-a.e. $\cx$, the only Nevanlinna current associated to $W^s(\cx)$ is $T^s_\omega$, we 
  can further assume that this property holds  for every $\cx\in \mathcal R_\e$. 
  Let $A\subset X$ be a  subset of full $\mu$-measure 
on which the alternative of Corollary \ref{cor:alternative_multiplicity} holds for every $k\geq 1$. 
In $\X_+$, consider  the measurable partition into fibers of the form $\Omega\times \set{x}$; it  corresponds to the partition $\mathcal F^-$ in 
Lemma \ref{lem:alternative_multiplicity}.  Then, the associated conditional measures 
$\m_+(\, \cdot \,  \vert \, \Omega\times \set{x})$ are naturally identified with $\nu^\N$. 
Fix  $x\in A$     such that  $\m_+(\mathcal R_\e\vert \Omega\times \set{x})>0$. 
Since $(X,\nu)$ is non-elementary, 
Theorems~\ref{thm:def_stationary} and~\ref{thm:uniq+extremal} provide pairs
$(\omega_1, \omega_2)$ in  $(\pi_\Omega(\mathcal R_\e))^2$
for which the currents   $T^s_{\omega_1}$ and $T^s_{\omega_2}$ are not cohomologous. 
By Theorem \ref{thm:nevanlinna} these currents describe respectively the asymptotic distribution of 
  $W^s(\omega_1, x)$  and  $W^s(\omega_2, x)$    so we infer that 
$W^s(\omega_1, x)\neq W^s(\omega_2, x)$ and by the analytic continuation  principle 
it follows that  $W_\loc^s(\omega_1, x)\neq W_\loc^s(\omega_2, x)$. 
Let $k_1<\infty$ 
be the intersection multiplicity of these manifolds at $x$. Since the intersection multiplicity is upper semi-continuous, we infer that for 
$\omega'_j\in \mathcal R_\e$ close to $\omega_j$, $j=1,2$, $\inter_x(W_\loc^s(\omega'_1, x), W_\loc^s(\omega'_2, x))\leq k_1$. Thus for $k=k_1$ 
we are in case (b) of the alternative of   Corollary \ref{cor:alternative_multiplicity}. 
Applying then Corollary \ref{cor:alternative_multiplicity} 
successively for $k=1, \ldots , k_1$, there is a first integer $k_0$ for which case (b) holds, and since 
(a) holds for $k_0-1$, we conclude that generically  $\inter_x\lrpar{W_\loc^s(\omega, x),  W^s_\loc(\omega', x)} = k_0$.
\end{proof}

\subsection{Transversal perturbations} 
The key ingredient  in the proof of Theorem \ref{thm:alternative_stable}  is the following 
basic geometric lemma, which is a quantitative refinement of  \cite[Lemma 6.4]{bls}. 
  
\begin{lem}\label{lem:BLS_intersection}
Let $k$ be a positive integer.  If $r$ and $c$ are positive real numbers, then there are two positive real numbers $\delta=\delta(k,r,c)$ 
and $\alpha=\alpha(k, r,c)$ with the
following property. Let $M_1$ and $M_2$ be two complex analytic curves in $\disk(r)\times \disk(r)\subset \C^2$ such that 
\begin{enumerate}[\em (i)]
\item $M_1$ and $M_2$ are graphs $\{(z,f_j(z))\; ; \; w\in \disk_r\}$ of holomorphic functions $f_j\colon\disk(r)\to \disk(r)$;
\item $M_1\cap M_2=\{(0,0)\}$,  and $\inter_{(0,0)}(M_1, M_2) = k$;
\item the $k$-th derivative satisfies $\abs{(f_1-f_2)^{(k)}(0)} \geq c$.
\end{enumerate}
If $M_3\subset \disk(r)\times \disk(r)$ is a complex curve that does not intersect $M_1$ but is $\delta$-close to $M_1$ in the $C^1$-topology , then 
$M_2$ and $M_3$ have exactly $k$ transverse 
intersection points in $\disk(\alpha r)\times \disk(\alpha r)$ (i.e. with multiplicity~$1$).
\end{lem}

\begin{proof}
Without loss of generality we may assume that   $\delta <1$.   

\smallskip

{\bf Step 1.-- }
We claim that there exists $\alpha_1 = \alpha_1 (k,r,c)$ such that 
for every $\alpha\leq \alpha_1$ and every  $z\in \disk (\alpha r)$ the following estimates hold: 
\begin{align}\label{eq:f12k}
\frac{1}{2 } \frac{\abs{(f_1-f_2)^{(k)}(0)}}{k!} \abs{z}^k & \leq \abs{f_1(z)- f_2(z)} \leq \frac{3}{2 }\frac{ \abs{(f_1-f_2)^{(k)}(0)} }{k!} \abs{z}^k \\
\frac{1}{2} \frac{\abs{(f_1-f_2)^{(k)}(0)} }{(k-1)!}\abs{z}^{k-1} & \leq \abs{ f_1'(z) - f_2'(z) } \leq \frac{3}{2}  \frac{\abs{(f_1-f_2)^{(k)}(0)} }{(k-1)!}\abs{z}^{k-1}. \label{eq:derivative_f_12}
\end{align}

Indeed put  $g=f_1-f_2=\sum_{m\geq k} g_m z^m$. Assumptions (i) and (iii) give$\abs{g(z)}\leq 2r$ on $\disk(r)$, 
and $ g^{(k)}(0)\neq 0$. By the Cauchy estimates, $\abs{g_n}\leq 2r^{1-n}$ for all $n\geq 0$.  
Then on $\disk(\alpha r)$ we get 
\begin{align} 
 \notag  \abs{g(z)-\frac{g^{(k)}(0)}{k!} z^k}   &\leq 2 r \left( \frac{\abs{z}}{r}\right)^{k+1} 
\left( 1-\frac{\abs{z}}{r}\right)^{-1} \leq 2r^{1-k} \frac{\alpha }{1-\alpha}   \abs{z}^k.
\end{align}  
There exists $\alpha_1(k,r,c)$ such that as soon as $\alpha \leq \alpha_1$, the right hand side 
 of this inequality is smaller than  
 $c\abs{z}^k/(2 k!)$; hence Estimate~\eqref{eq:f12k} follows. The same argument applies for 
 \eqref{eq:derivative_f_12} because  
\begin{align} 
\notag  \abs{g'(z)- \frac{g^{(k)}(0)}{(k-1)!}z^{k-1}}  &\leq 4 (k+1) \left( \frac{\abs{z}}{r}\right)^k \left( 1-\frac{\abs{z}}{r}\right)^{-2} 
\leq 4 (k+1) r^{1-k} \frac{\alpha}{ (1-\alpha)^{2}}\abs{z}^{k-1}.  
\end{align}

\smallskip

{\bf Step 2.--} For every $\alpha\leq \alpha_1$, if  $\delta < c (\alpha r)^k/ 2k!$, 
$M_2$ and $M_3$ have exactly $k$   
intersection points, counted with multiplicities,  in $\disk(\alpha r)\times \disk(\alpha r)$. 

Indeed, the intersection points of $M_3$ and $M_2$ correspond to the solutions of the equation $f_3=f_2$. 
To locate its roots, note that on the circle $\fr \disk(\alpha r)$, the Inequality~\eqref{eq:f12k} implies 
\begin{equation}\label{eq:rouche}
 \abs{f_1 - f_2} \geq \frac12 \frac{c}{k!} (\alpha r)^k. 
\end{equation}
Since $\abs{f_1-f_3}< \delta$, the choice $\delta < c (\alpha r)^k/ 2k!$ is tailored to assure that the hypothesis of the Rouch\'e theorem 
is satisfied  in $\disk (\alpha r)$; so, counted with multiplicities, there
are $k$ solutions to the equation $f_3=f_2$ on that disk. Furthermore by the Schwarz lemma 
$\abs{f_2}< \alpha r$ on $\disk(\alpha r)$ so the corresponding intersection points between $M_2$ and $M_3$ 
are contained in  $\disk(\alpha r)\times \disk(\alpha r)$.

If $k=1$ the proof is already complete at this stage, so from now on we assume $k\geq 2$. 

\smallskip

{\bf Step 3.--} Set $\delta_0=\abs{f_3(0)}$, and note that $\delta_0\leq \delta$. 
Then for every $\alpha \leq 1/2$, in $\disk (\alpha r)$  we have 
\begin{align}\label{eq:harnack}
\delta_0^{\frac{1+\alpha}{1-\alpha}}\leq  \abs{f_1(z)  - f_3(z)} &\leq \delta_0^{\frac{1-\alpha}{1+\alpha}} \\
 \label{eq:harnack_derivee}
\abs{f_1'(z)  - f_3'(z)} &\leq \unsur{\alpha r} \delta_0^{\frac{1-2\alpha}{1+2\alpha}}. 
\end{align}
For this, recall the Harnack inequality: for any negative harmonic function in $\disk $
\begin{equation} 
\frac{1- \abs{\zeta}}{1+\abs{\zeta}}\leq  \frac{u(\zeta)}{u(0)}\leq \frac{1+ \abs{\zeta}}{1-\abs{\zeta}}. 
\end{equation}
Since $f_1-f_3$ does not vanish and $\abs{f_1-f_3}\leq \delta<1$  in $\disk(r)$, the function 
$\log\abs{f_1-f_3}$ is harmonic and negative there. Thus for $\alpha \leq 1/2$, the Harnack inequality can be applied 
to $\zeta\mapsto (f_1-f_3)(r\zeta)$ in $\disk$: this  gives \eqref{eq:harnack}. Likewise, 
 we infer that 
\begin{equation}
\delta_0^{\frac{1+2\alpha}{1-2\alpha}} \leq \abs{f_1(z) - f_3(z)} \leq \delta_0^{\frac{1-2\alpha}{1+2\alpha}} 
\end{equation}
 in $\disk (2\alpha r)$,  and~\eqref{eq:harnack_derivee} follows from the Cauchy estimate 
$\norm{g'}_{\disk(\alpha r)} \leq (\alpha r)\inv 
\norm{g}_{\disk(2\alpha r)}$.

\smallskip

{\bf Step 4.--} We now conclude the proof.  Fix $\alpha  =  \alpha (k,r,c)$ such that $\alpha  \leq \alpha_1$ and
\begin{equation}
\beta(\alpha):=\frac{1-2\alpha }{1+2\alpha } -\frac{k-1}{k}\times  \frac{1+ \alpha }{1- \alpha } >0.
\end{equation} (This will be our final choice for $\alpha$.)
 Fix $\delta  <  c (\alpha  r)^k/ 2k!$ and consider a solution $z_0$ of the equation 
$f_2(z) = f_3(z)$ in $\disk(\alpha r)$ provided by Step 2. The transversality of $M_2$ and $M_3$ at 
$(z_0, f_2(z_0))$ is equivalent to $f_3'(z_0)\neq f_2'(z_0)$, so we only need 
\begin{equation}
\abs{(f_3-f_1)'(z_0)}< \abs{(f_2-f_1)'(z_0)}.
\end{equation}
Since $(f_1- f_3)(z_0)  = (f_1-f_2)(z_0)$, combining the right hand side of Inequality~\eqref{eq:f12k} and the left hand side of 
Inequality \ref{eq:harnack}, we get that 
\begin{equation}
\frac{3}{2}\frac{ \abs{(f_1-f_2)^{(k)}(0)} }{k!} \abs{z_0}^k \geq \delta_0^{\frac{1+\alpha}{1-\alpha}}.
\end{equation}
thus 
\begin{equation}
\abs{z_0} \geq \delta_0^{\unsur{k}  \frac{1+\alpha}{1-\alpha}} \lrpar{\frac{2k!}{3  }}^{\unsur{k}} \abs{(f_1-f_2)^{(k)}(0)}^{-\unsur{k}} .
\end{equation}
Hence by \eqref{eq:derivative_f_12} we get that 
\begin{align}
 \abs{(f_2-f_1)'(z_0)} &\geq \unsur{2(k-1)!}  \lrpar{\frac{2k!}{3}}^\frac{k-1}{k}  \delta_0^{\frac{k-1}{k}  \frac{1+\alpha}{1-\alpha}}
 \abs{(f_1-f_2)^{(k)}(0)}^{\unsur{k}} \\
  &\notag \geq  \unsur{2(k-1)!}\lrpar{\frac{2k!}{3}}^\frac{k-1}{k}  \delta_0^{\frac{k-1}{k}  \frac{1+\alpha}{1-\alpha}}  c^{\unsur{k}}. 
 \end{align}
 On the other hand by 
Estimate \eqref{eq:harnack_derivee}
\begin{equation}
\abs{(f_3-f_1)'(z_0)}  \leq \unsur{\alpha r} \delta_0^{\frac{1-2\alpha}{1+2\alpha}}
 \end{equation}
Since $\delta_0 \leq \delta$, we only need to impose one more constraint on $\delta$ (together with $\delta < c (\alpha r)^k/ 2k!$), namely
\begin{equation}
\delta^{\beta(\alpha)}<  \unsur{2(k-1)!}\lrpar{\frac{2k!}{3}}^\frac{k-1}{k}    c^{\unsur{k}} r \alpha,
\end{equation} 
to get the desired inequality $\abs{(f_3-f_1)'(z_0)} < \abs{(f_2-f_1)'(z_0)}$.
 \end{proof} 
 

\begin{rem}\label{rem:BLS_intersection} 
Lemma~\ref{lem:BLS_intersection} does not hold in the real analytic setting. Indeed, take an integer $n\equiv 1 \mod[4]$ and 
consider the $n$-th Chebychev polynomial $T_n$, 
defined by $T_n(\cos\theta) = \cos(n\theta)$; it satisfies $\abs{T_n}\leq 1$ on $[-1,1]$,  
$\abs{T_n'}\leq 2n$ on $[-1/2, 1/2]$, and   $T_n'(0)=n$. Then, set  $P_n(x)  = \frac{10}{n^2}T\lrpar{x-\frac{5}{n}} + \frac{25}{n^2}$.
This function satisfies $P_n'(5/n) = 10/n$, $P_n(5/n) = (5/n)^2$, and $15/n^2 \leq P_n\leq 35/n^2$ on $[-1, 1]$. Now, if $n$ is large, 
 $M_1 = \set{y=0}$, $M_2 =\set{y=x^2}$ and $M_3 =  \set{y=P_n(x)}$ are three smooth algebraic curves in  
$(-1,1)^2\subset \R^2$ such that $M_3$ is disjoint from $M_1$ but close to it in the $C^1$ topology, and $M_3$ is tangent to $M_2$ at 
$(5/n, 25/n^2)$. Similar arguments can be used to show 
 that the semi-continuity of Lemma~\ref{lem:upper_sc_inter} fails for real analytic curves (though Corollary~\ref{cor:alternative_multiplicity} may still be valid for real analytic random dynamical systems). \end{rem}

Let $\Delta_1$ and $\Delta_2$ be two disks of size $r$ at $x\in X$, which are tangent at $x$; let $e_1\in T_xX$
be a unit vector in $T_x\Delta_1=T_x\Delta_2$ and $e_2$ a unit vector orthogonal to $e_1$ for $\kappa_0$.
Then, in the chart $\hexp_x$, 
$\Delta_1$ and $\Delta_2$ are graphs $\{ze_1+\psi_i(z)e_2\}$ of 
holomorphic functions $\psi_i\colon \disk(r)\to \disk(r)$, $i=1$, $2$, such that $\psi_i(0)=0$ and $\psi_i'(0)=0$.
If $\inter_x(\Delta_1, \Delta_2) =  k$, then 
for $j = 1, \ldots , k-1$ one has $\psi^{(j)}_1(0) = \psi^{(j)}_2(0)$ and $\psi^{(k)}_1(0) \neq  \psi^{(k)}_2(0)$.
We define the \textbf{$k$-osculation} of $\Delta_1$ and  $\Delta_2$ at $x$ to be 
\begin{equation}
\osc_{k,x,r}(\Delta_1, \Delta_2) = \abs{\psi^{(k)}_1(0) - \psi^{(k)}_2(0)}. 
\end{equation}
If $s\leq r$ and we consider $\Delta_1$ and $\Delta_2$ as disks of size $s$, then $\osc_{k,x,s}(\Delta_1, \Delta_2) =\osc_{k,x,r}(\Delta_1, \Delta_2)$.
Thus, $\osc_{k,x,r}(\Delta_1, \Delta_2)$ does not depend on $r$,
so we may denote this
osculation number by $\osc_{k,x}(\Delta_1, \Delta_2)$. 
With this terminology,  Lemma~\ref{lem:BLS_intersection}  directly implies the following corollary.

\begin{cor}\label{cor:bls_quantitative}
Let $k$ be a positive integer, and $r$ and $c$ be   positive real numbers. Then, there are two positive 
real numbers $\delta$ and $\alpha$, depending on $(k,r,c)$, satisfying the following property. 
Let $\Delta_1$ and $\Delta_2$ be two holomorphic disks of size $r$ through $x$, such that $\inter_x(\Delta_1, \Delta_2)=k$
and $\osc_{k,x}(\Delta_1, \Delta_2))\geq c$.  Let $\Delta_3$ be a holomorphic disk of size $r$  such that 
$\Delta_3$ is $\delta$-close to $\Delta_1$ in the $C^1$-topology but  $\Delta_3\cap \Delta_1 = \emptyset$. Then 
 $\Delta_3$ intersects $\Delta_2$ transversely in exactly  $k$ points in $B_X(x, \alpha r)$. 
\end{cor}

The following   lemma follows directly    from the first step of the proof of Lemma \ref{lem:BLS_intersection}.  
  
  \begin{lem}\label{lem:unique_intersection}
Let $k$ be a positive integer, and $r$ and $c$ be   positive real numbers.
Then there exists a constant $\beta$ depending only on $(r,k,c)$ such that if 
  $\Delta_1$ and $\Delta_2$ are two holomorphic disks of size $r$ through $x$, such that $k=\inter_x(\Delta_1, \Delta_2)$
and $\osc_{k,x}(\Delta_1, \Delta_2))\geq c$, then $x$ is  the only point of intersection between 
$\Delta_1$ and $\Delta_2$ in the ball $B_X(x,\beta r)$. 
   \end{lem}

\subsection{Proof of Theorem  \ref{thm:alternative_stable}}
 
Before starting the proof, we record the following two   facts from elementary measure theory:
 
\begin{lem}\label{lem:trivial}
Let $(\Omega, \mathcal F, \pp)$ be a probability space, and $\delta\in (0,1)$. 
\begin{enumerate}[\em(1)]
\item If $\varphi$ is a measurable function with values in $[0, 1]$ and such that $\int \varphi \; d\pp\geq 1-\delta$, then 
\[
\pp\lrpar{\set{ x \; ; \;  \varphi(x)\geq 1-\sqrt\delta} }\geq 1-\sqrt\delta.
\]
\item  If $A_j$ is a sequence of measurable subsets   such that 
$\pp(A_j)\geq 1-\delta$ for every $j$, then $\pp(\limsup A_j)\geq 1-\delta$. 
\end{enumerate}
\end{lem}

Let us now prove Theorem  \ref{thm:alternative_stable}.
If the integer $k_0$  of Lemma \ref{lem:generic_multiplicity} is equal to 1, then Pesin stable manifolds 
corresponding to different itineraries at a $\mu$-generic point $x \in X$ are 
generically transverse; hence, we are in case {(a)} of the theorem --note that the conclusion is actually stronger than mere non-randomness. 
So, we now assume $k_0>1$ and we prove that $\mu$ is almost surely invariant and that its entropy is equal to zero. 

\smallskip

{\bf{Step 1.--}} First, we construct a subset $\mathcal G_\e$ of ``good points'' in $\X$.  

As described in Section~\ref{par:past_future_partitions}, the atoms of $\mathcal F^-$ are the sets
$\mathcal F^-(\cx)=\Sigma^u_\loc(\xi)\times \set{x}$ and the measures $\m(\, \cdot \, \vert \mathcal F^-(\cx))$  
can be naturally identified to $\nu^\N$  under the natural projections  
 $\mathcal F^-(\cx) \overset{\sim}\to  \Sigma^u_\loc(\xi) \overset{\sim}\to  \Omega$.
 For notational simplicity we denote these measures 
by $\m^{\mathcal F^-}_\cx$. 

For a small $\e>0$, let  $\mathcal R_\e\subset \X $ be a compact subset 
with $\m(\mathcal R_\e)> 1-\e$, along which local stable manifolds have size at least $2r(\e)$ and vary continuously.
Since $\int \m^{\mathcal F^-}_\cx (\mathcal R_\e) \, d\m(\cx) \geq 1-\e$, by Lemma \ref{lem:trivial}~(1)
we can select a compact subset $\mathcal R_\e'\subset \mathcal R_\e$ with 
$\m(\mathcal R'_\e)\geq 1-\sqrt \e$ such that for every $\cx\in \mathcal R'_\e$ 
one has $\m^{\mathcal F^-}_\cx (\mathcal R_\e) \geq 1-\sqrt \e$. 

By assumption,  
$
\inter_x(W^s_\loc(\cy_1), W^s_\loc(\cy_2))= k_0
$ 
for $\m$-almost every $\cx = (\xi, x)\in \mathcal R'_\e$ and for  $(\m^{\mathcal F^-}_\cx\otimes \m^{\mathcal F^-}_\cx)$-almost every
pair of points $(\cy_1, \cy_2)\in (\mathcal F^-(\cx)\cap   \mathcal R_\e)^2$. Then there exists  $\mathcal R''_\e \subset \mathcal R'_\e$
of measure at least $1-2\sqrt{\e}$ and a constant $c(\e)>0$  such that 
\begin{equation}
\osc_{k_0,x, r(\e)} (W^s_\loc(\cy_1), W^s_\loc(\cy_2)) \geq c(\e)
\end{equation}
for every $\cx=(\xi,x)\in \mathcal R''_\e$ and all pairs $(\cy_1,\cy_2)$ in a subset  
$A_{\e, \cx} \subset (\mathcal F^-(\cx)\cap   \mathcal R_\e)^2$ depending measurably on $\cx$ and of measure
\begin{equation}
(\m^{\mathcal F^-}_\cx\otimes \m^{\mathcal F^-}_\cx)(A_{\e, \cx} )\geq 1-4\sqrt{\e}
\end{equation}
(we just used  
$(\m^{\mathcal F^-}_\cx\otimes \m^{\mathcal F^-}_\cx)((\mathcal F^-(\cx)\cap   \mathcal R_\e)^2)  \geq (1-\sqrt \e)^2 > 1- 2\sqrt{\e}$).
Finally, Fubini's theorem and Lemma \ref{lem:trivial}~(1)  
provide a set   $\mathcal G_\e \subset\mathcal R''_\e $ such that 
\begin{enumerate}[(a)]
\item  $\m(\mathcal G_\e)\geq 1-2\e^{1/4}$ 
\item for every 
$\cx\in \mathcal G_\e$, $W^s_\loc(\cx)$ has size $2r(\e)$;
\item for every 
$\cx\in \mathcal G_\e$,  there exists a measurable set  
$\mathcal{G}_{\e, \cx} \subset \mathcal F^-(\cx)$ with 
 $ \m^{\mathcal F^-}_\cx(\mathcal{G}_{\e, \cx}) \geq 1-2\e^{1/4}$  such that  for every $\cy$ in $\mathcal{G}_{\e, \cx}$, 
$W^s_\loc(\cy)$ has size $\geq r(\e)$ and, viewed as a subset of $X$,
\begin{itemize} 
 \item  it is tangent to $W^s_\loc(\cx)$  to order $k_0$ at $x$, 
\item    $\osc_{k_0,x, r(\e)} (W^s_\loc(\cx), W^s_\loc(\cy)) \geq c(\e)$.
\end{itemize}
\end{enumerate} 
Note that $\cx\notin \mathcal G_{\e,\cx}$: indeed, when the local stable manifolds vary continuously, 
one can think of $A_{\e, \cx}$ as the complement of a small neighborhood of the diagonal in $\Omega\times \Omega$.
 

\begin{figure}[h]
\input{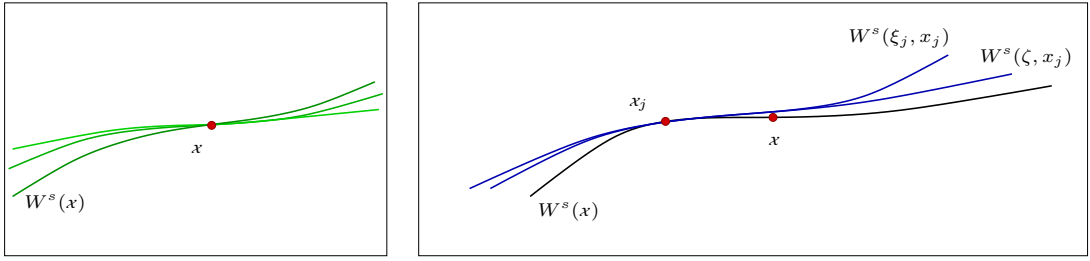}
\caption{
{\small{On the left, a generic point $\cx$ with the local stable manifolds $W^s_\loc(\xi_i,x)$ for distinct $(\xi_i)_{i\geq 0}$ (see Step 1). 
On the right, the choice of the sequence $(\zeta, x_j)$ gives a family of local stable manifolds (see Step 2).}}}
\end{figure}

\smallskip 

{\bf{Step 2.--}} To make the argument more transparent, we  first show that the fiber  entropy 
vanishes. 

Let $\eta^s$ be a Pesin partition subordinate to 
local stable manifolds in $\cX$. By 
Corollary \ref{cor:rokhlin_zero_entropy} it is enough to show that for $\m$-almost every $\cx$, 
$\m(\cdot \vert \eta^s(\cx))$ is atomic (hence  concentrated at $x$).  
Assume by contradiction that this is not the case. Therefore for 
$\e>0$ small enough there exists $\cx  = (\xi, x)\in \mathcal G_\e$ 
such that $\m(\cdot \vert \eta^s(\cx))_{\vert \eta^s(\cx)\cap \mathcal G_\e}$ is non-atomic, and
there exists an infinite  sequence of  
pairwise distinct points $\cx_j=(\xi, x_j)$  in $\mathcal G_\e\cap \eta^s(\cx)$ converging to $\cx$. 
Then with $\mathcal{G}_{\e, \star}$ as in Property~(c) of the  definition of  $\mathcal G_\e$, we have 
$\m^{\mathcal F^-}_{\cx_j}(\mathcal{G}_{\e, \cx_j})\geq 1-2{\e}^{1/4}$ for every $j$. 

Identifying all $\mathcal {F}^-(\cx_j) $  with $\Sigma^u_\loc(\xi)$, by Lemma \ref{lem:trivial}~(2) we can find 
$\zeta\in \Sigma^u_\loc(\xi)$  
such that $(\zeta, x_j)$   belongs to $\mathcal G_{\e,(\zeta, x_j)}$ for infinitely many $j$'s. 
Along this subsequence  the local stable manifolds $W^s_\loc(\zeta, x_j)$ form a sequence of disks of uniform size $r = 2r(\e)$ at $x_j$. 
Two such local stable manifolds are either pairwise disjoint or coincide along an open subset 
because they are associated to the same itinerary $\zeta$. 

Let us now use the notation from Corollary~\ref{cor:bls_quantitative} and Lemma \ref{lem:unique_intersection}.
We know that   $W^s_{r(\e)}(\zeta, x_j)$ is    tangent to $W^s_{r(\e)}(\xi, x)$ at $x_j$ to order $k_0$, with 
$\osc_{k_0,x_j, r(\e)} (W^s_{r(\e)}(\cx), W^s_{r(\e)}(\zeta, x_j)) \geq c(\e)$; so, by Lemma \ref{lem:unique_intersection},
$W^s_{r(\e)}(\zeta, x_j)$ and $W^s_{r(\e)}(\zeta, x_{j'})$ are disjoint as soon as $\dist_X(x_{j}, x_{j'}) < \beta r(\e)$. Finally, 
if $j$ and $j'$ are large enough, then $\dist_X(x_{j}, x_{j'}) < \alpha r(\e)$ and the $C^1$ distance 
between $W^s_{r(\e)}(\zeta, x_j)$ and $W^s_{r(\e)}(\zeta, x_{j'})$  is smaller than $\delta$; thus, 
Corollary~\ref{cor:bls_quantitative} asserts that  $W^s_{r(\e)}(\zeta, x_j)$ and $W^s_{r(\e)}(\zeta, x_{j'})$ cannot both 
be tangent to  $W^s_{r(\e)}(\xi, x)$. This is a contradiction, and we conclude  that 
the fiber entropy of  $\m$  vanishes.
  
\smallskip 

{\bf{Step 3.--}} We now prove the almost sure  invariance.

As in \cite[Eq. (11.1)]{br}  we consider a measurable partition 
$\mathcal P$ of $\cX$ with the property that  for $\m$-almost every $(\xi, x)$, 
\begin{equation}
\Sigma^s_\loc(\xi) \times W^s_{r(\xi, x)}(\xi, x) \subset \mathcal P (\xi, x) \subset \Sigma^s_\loc(\xi) \times W^s (\xi, x).
\end{equation} 
The existence of such a partition is guaranteed, for instance, by Lemma \ref{lem:stable_partition}. 
By  \cite[Prop 11.1]{br}(\footnote{Brown and Rodriguez-Hertz make it clear that this result holds  for an arbitrary smooth random dynamical system on a compact manifold.}),  to show that $\mu$ is almost surely invariant it is enough to prove that: 
\begin{equation}\label{eq:contradict}
\text{for }\m\text{ almost every }\xi, \ \m(\, \cdot\,  \vert \mathcal P(\xi, x)) \text{ is  
concentrated on  }\Sigma^s_\loc(\xi) \times \set{x}.
\end{equation}
By contradiction, assume that~\eqref{eq:contradict} fails. By contraction along the stable leaves, it follows that 
almost surely  $\Sigma^s_\loc(\xi) \times \set{x}$ is contained in
\begin{equation}
\supp\lrpar{\m(\cdot \vert \mathcal P(\xi, x))_{\vert \mathcal P(\xi, x) \setminus \Sigma^s_\loc(\xi) \times \set{x}}}
\end{equation}
(this  is identical to the argument of Corollary \ref{cor:rokhlin_zero_entropy}). 
In particular  for small $\e$ we can find $\cx = (\xi,x)\in \mathcal G_\e$ and a sequence of points 
$ \cx_j = (\xi_j, x_j)\in \mathcal G_\e$ such that 
$\cx_j$ belongs to $\mathcal P(\cx) \cap \mathcal G_\e$, $x_j\neq x$ and 
$(x_j)$ converges to $x$ in $X$. We can also assume that the $x_j$ are all distinct. 
 By definition of $\mathcal G_\e$,  $\m^{\mathcal F^-}_{\cx_j}\lrpar{\mathcal{G}_{\e, \cx_j}}\geq 1-2{\e}^{1/4}$  for every $j$. 
For $(\xi, \zeta)\in \Sigma^2$, set 
\begin{equation}
[\xi, \zeta] = \Sigma^u_\loc(\xi)\cap \Sigma^s_\loc(\zeta);
\end{equation} 
that is, $[\xi, \zeta] $ is the itinerary with the same past as $\xi$ and the same future as $\zeta$. As above, 
identifying the atoms of the partition $\mathcal F^-$ with $\Omega$,  
Lemma \ref{lem:trivial}~(2) provides an infinite subsequence   $(j_\ell)$ and for every $\ell$ an itinerary
$\zeta_{j_\ell}\in \Sigma^u_\loc (\xi_{j_\ell})$ such that 
$\cy_{j_\ell}:= 
(\zeta_{j_\ell}, x_{j_\ell})$ belongs to $\mathcal{G}_{\e, \cx_{j_\ell}}$ and all the 
$\zeta_{j_\ell}$ have the same future, that is 
$\zeta_{j_\ell}$ is of the form $[\xi_{j_\ell}, \zeta]$ for a fixed $\zeta$. 
By definition,   
\begin{align}\label{eq:osc}
\inter_{x_{j_\ell}}(W^s_\loc(\cx_{j_\ell}), W^s_\loc(\cy_{j_\ell}) )& = k_0 \\
\osc_{k_0, x_{j_\ell}, r(\e)} (W^s_\loc(\cx_{j_\ell}), W^s_\loc(\cy_{j_\ell})) & \geq c(\e).
\end{align}
In addition the disks $\pi_X(W^s_\loc(\cy_{j_\ell}))$ are pairwise disjoint or locally coincide  
because the $x_{j_\ell}$ are distinct 
and the $\zeta_{j_\ell}$ have the same future. 
Moreover, since $\cx_{j_\ell}$ belongs to  $\mathcal P(\cx)$, 
 $W^s(\cx_{j_\ell})$      coincides with  $W^s(\cx)$.   
Therefore, the $\pi_X(W^s_\loc(\cy_{j_\ell}))$ form a sequence of disjoint
disks of size $2r(\e)$ at $x_j$, all tangent to $\pi_X(W^s_\loc(\cx))$ to order $k_0$,   with osculation bounded 
from below by $c(\e)$. Since this sequence of disks is continuous and $(x_j)$ converges towards $x$,    Lemma 
\ref{lem:unique_intersection}  and  Corollary~\ref{cor:bls_quantitative} provide a contradiction,
 exactly as in Step 2. This completes the proof of the theorem.    \qed


\section{Stiffness}\label{sec:stiffness}

Here we study Furstenberg's stiffness  property for automorphisms of compact K\"ahler surfaces, thereby proving  Theorem \ref{mthm:stiffness}. 
Our first results in \S\ref{subs:stiffness_elementary}  
deal  with  elementary subgroups of $\Aut(X)$. The argument relies on the classification of such elementary groups together with
general group-theoretic criteria for stiffness; these criteria are recalled in 
\S~\ref{subs:stiffness_generalities} and~\ref{subs:inducing}. 
Theorem~\ref{thm:stiffness_real}  concerns the much more interesting case of non-elementary subgroups; 
its proof combines all results of the previous sections with the work of Brown and Rodriguez-Hertz \cite{br}.  

\subsection{Stiffness}\label{subs:stiffness_generalities}
Following Furstenberg \cite{furstenberg_stiffness},
a random dynamical system $(X, \nu)$ is  \textbf{stiff} if any $\nu$-stationary measure is almost surely invariant; equivalently, 
every ergodic stationary measure is almost surely invariant. 
This property can conveniently be expressed in terms of $\nu$-harmonic functions on $\Gamma$. 
Indeed if $\xi\colon X\to \R$ is a continuous function and $\mu$ is $\nu$-stationary, then 
$\Gamma\ni g\mapsto \int_X \xi(gx) \, d\mu(x)$ is a bounded, continuous, right $\nu$-harmonic function on $\Gamma$; thus 
proving that $\mu$ is invariant amounts to proving that such harmonic functions are constant. 
Stiffness can also be defined for group actions: a group 
$\Gamma$ \textbf{acts stiffly} on $X$ if and only if  $(X, \nu)$ is stiff for every probability measure $\nu$ on 
$\Gamma$ whose support generates $\Gamma$; in this definition, the measures $\nu$ can also be restricted to specific
families, for instance symmetric finitely supported measures, or measures satisfying some moment condition.
There  are some general criteria ensuring stiffness directly  from the properties of $\Gamma$. 
A first case is when $G$ is a topological  group acting continuously on $X$ and 
$\Gamma\subset G$ is     relatively compact. Then $\Gamma$ acts stiffly on $X$: this follows from the maximum principle for harmonic functions on $\overline \Gamma$ (see also \cite[Thm 3.5]{furstenberg_stiffness}). Another important case for us is that of Abelian and nilpotent groups.

\begin{thm}\label{thm:raugi}
Let $G$ be a locally compact, second countable, topological group. Let $\nu$ be a probability measure on 
$G$. If $G$ is nilpotent of class $\leq 2$, then any measurable, $\nu$-harmonic, and bounded function $\varphi\colon G\to \R$ is 
constant; thus, every measurable action of such a group is stiff. 
\end{thm}

We only stated the simplest result sufficient for our paper, but this theorem holds for most nilpotent groups without any assumption 
on the nilpotent class.  It is due to Dynkin and Malyutov for any finitely generated nilpotent group, and to Guivarc'h for a large class of locally compact nilpotent groups;  the case of Abelian groups is the famous Blackwell-Choquet-Deny theorem. We refer to~\cite{Guivarch:1973} for a proof 
(\footnote{The proof in~\cite{Raugi:2004} is not correct (Lemma 2.5 there is  false) but it
works perfectly, and is quite short, if the support of $\nu$ is countable 
or if the nilpotency class is $\leq 2$. See the introduction of~\cite{Raugi:2004} for a summary of previous results.}). 
We shall apply Theorem~\ref{thm:raugi} to subgroups $A\subset \Aut(X)$; what we implicitly do is first replace $A$ by
its closure in $\Aut(X)$ to get a locally compact group, and then apply the theorem to this group.

\subsection{Subgroups and hitting measures}\label{subs:inducing}
A basic tool 
is the {\bf hitting measure} on a subgroup, which we briefly introduce now (see 
\cite[Chap. 5]{benoist-quint_book} 
for details). 
Let $G$ be a locally compact second countable topological  group. 
A notion of length can be defined in this context as follows: given
 a neighborhood $U$ of the unit element, for any $g\in G$, $\length_U(g)$ is the least integer $n\geq 1$ 
 such that $g\in U^n$. By definition  a probability measure $\nu$ on $G$ 
  has a finite first moment (resp. a finite exponential moment)  if 
 $\int \length_U(g) \, d\nu(g)<\infty$ (resp. if  $\int \exp(\alpha \length_U(g)) \, d\nu(g)<\infty$ for some $\alpha >0$). This condition does not depend on the choice of $U$. 

Let $\nu$ be a probability measure on $G$, and consider the left random walk on $G$ governed by  $\nu$. 
Given a subgroup  $H\subset G$, for  $\omega=(g_i)\in G^\N$, define the  hitting time 
\begin{equation}
T (\omega) = T_{H}(\omega) := \min\set{n\geq 1\; ; \; g_n\cdots g_1\in H}.
\end{equation} 
If $T$ is almost surely finite we say that $H$ is \textbf{recurrent}
  and  the distribution of $g_{T(\omega)}\cdots g_1$ is by definition 
the {\bf{hitting measure}} of $\nu$ on $H$, which will be denoted $\nu_{H}$. 
The key  property of $\nu_{H}$ is that if 
$\varphi: G\to \R$ is a $\nu$-harmonic function, then $\varphi\rest{H}$ 
is also $\nu_H$-harmonic.
Therefore, if  $\mu$ is a $\nu$-stationary measure, then it is also $\nu_{H}$-stationary. 
Conversely, any bounded $\nu_{H}$-harmonic function $h$ on $H$ admits a unique extension 
$\widetilde h$ to a bounded $\nu$-harmonic function on $G$; this extension is defined by the formula 
\begin{equation}\label{eq:doob}
\widetilde h(x) = \ee_x(h(g_{T_{x,H}(\omega)}\cdots g_1x))  = \int h(g_{T_{x,H}(\omega)}\cdots g_1x)\, d\nu^\N(\omega)
\end{equation}
 where the stopping time $T_{x, H}$ is defined by 
$T_{x, H}(\omega) =    \min\set{n\geq 0 \; ; \;  g_n\cdots g_1x\in H}$. 
The uniqueness comes from Doob's optional stopping theorem, which 
asserts that if $(M_t)_{t\geq 0}$ is a bounded martingale and $T$ is a stopping time which is almost surely finite then $\ee(M_T) = \ee(M_0)$. 
Thus,   any bounded $\nu$-harmonic function $h$ on $G$ satisfies Formula~\eqref{eq:doob}. 

If $[G:H]<\infty$ then $H$ is recurrent and its stopping time admits an exponential moment. It 
follows that $\nu_{H}$ has a finite first (resp.  exponential) moment  
if and only if $\nu$ does. 

Likewise, assume that   $H$ is a normal subgroup of $G$ with 
$G/H$ isomorphic to $\Z$,  
and  that  $\nu$ is symmetric with a finite first moment.   
Then, the projection ${\overline{\nu}}$ of $\nu$ on $G/H$ is  symmetric with a finite first moment,   
so   the random 
walk governed by ${\overline{\nu}}$ on $G/H\simeq \Z$ 
is recurrent (see the Chung-Fuchs Theorem in~\cite[\S 5.4]{Durrett:BookV5} or \cite{Chung-Fuchs}) 
and $H$ is recurrent. 
 
\begin{lem}  \label{lem:stiff_finite_extension} 
Let $\nu$ be a probability measure on $\Aut(X)$ and 
$\Gamma'$ be a closed  subgroup which is recurrent for
the random walk induced by $\nu$. Let $\nu'$ be the induced measure on $\Gamma'$.  
If  $(X, \nu')$ is stiff then $(X, \nu)$ is stiff as well. This holds in particular if:
\begin{enumerate}[\em (i)]
\item either $[\Gamma_\nu:\Gamma']<\infty$ 
\item or $\Gamma'$ is a normal subgroup of $\Gamma_\nu$ with $\Gamma_\nu/\Gamma'$ isomorphic to $\Z$, and $\nu$ is symmetric with a finite first moment.
\end{enumerate}
\end{lem}

\begin{proof}
Let $\mu$ be a $\nu$-stationary measure on $X$. Then $\mu$ is $\nu'$-stationary, hence by stiffness it is $\Gamma'$-invariant. 
Therefore  for every Borel set $B\subset X$, the function 
$\Gamma\ni g \mapsto \mu(g\inv B)$ is a bounded $\nu$-harmonic function which is constant on $\Gamma'$ so by the 
uniqueness of harmonic extension  it is constant, and $\nu$ is $\Gamma$-invariant.
\end{proof}

\subsection{Elementary groups}\label{subs:stiffness_elementary}
Recall that $\Aut(X)$ is a topological group for the topology of uniform convergence and is 
in fact a complex Lie group (with possibly infinitely many connected components).
Let $\Aut(X)^\circ$ be the connected component of the identity in $\Aut(X)$ and 
\begin{equation}
\Aut(X)^\# = \Aut(X)/\Aut(X)^\circ.
\end{equation}
Let $\rho: \Aut(X)\to \GL(H^*(X;\Z))$ be the natural homomorphism; its image is  $\Aut(X)^* = \rho(\Aut(X))$  (see~\S~\ref{par:Hodge_decomposition});
is kernel contains $\Aut(X)^\circ$ and a theorem of Lieberman~\cite{lieberman} shows that
  $\Aut(X)^\circ$ has finite index in $\ker(\rho)$.
If $\Gamma$ is a subgroup of $\Aut(X)$, we set $\Gamma^* = \rho(\Gamma)$. 

\begin{thm}\label{thm:stiffness_elementary_groups}
Let $X$ be a compact K\"ahler surface. Let  $\nu$ be a symmetric probability measure on $\Aut(X)$ satisfying the moment condition~\eqref{eq:moment}. 
If $\Gamma_\nu$ is elementary and  $\Gamma_\nu^*$ is infinite, then $(X, \nu)$ is stiff. 
\end{thm}

Note that stiffness can fail when $\Gamma_\nu^*$ is finite: see Example~\ref{eg:stiffness_elementary} below. 
The proof relies on the   classification of elementary subgroups of $\Aut(X)$ 
(see \cite[Thm 3.2]{Cantat:Milnor}, \cite{favre_bourbaki}): 
if $\Gamma_\nu$ is elementary and $\Gamma_\nu^*$ is infinite there exists a finite index subgroup $A^*\subset \Gamma_\nu^*$ which is  
\begin{enumerate}[(a)]
\item either cyclic and generated by a loxodromic map;
\item or a free Abelian group of parabolic transformations possessing a common isotropic line; in that case, there 
is a genus $1$ fibration $\tau\colon X\to S$, onto a compact Riemann surface $S$, such that $\Gamma_\nu$ permutes
the fibers of $\tau$.  
\end{enumerate}
Denote by $\rho_{\Gamma_\nu} \colon \Gamma_\nu\to \Gamma_\nu^*$ the restriction of $\rho$ 
to $\Gamma_\nu$. We distinguish two cases.

\begin{proof}[Proof when the kernel of $\rho_{\Gamma_\nu}$ is finite] 
Let $A$ be the pre-image of $A^*$ in $\Gamma$; it fits into an
exact sequence $1\to F\to A\to A^*\to 0$ with $F$ finite, so a classical group theoretic lemma 
(see Corollary 4.8 in \cite{cantat-guirardel-lonjou}) asserts that 
$A$ contains a finite index, free Abelian 
subgroup $A_0$, such that $\rho_{\Gamma_\nu}(A_0)$ has finite index in  $A^*$.  
 Since $A_0$ is Abelian,  Theorem~\ref{thm:raugi} shows that the action of $(A_0,\nu_{A_0})$ on $X$ is stiff. The index
of $A_0$ in $\Gamma$ being finite,  Lemma~\ref{lem:stiff_finite_extension} concludes the proof. 
 \end{proof}

\begin{proof}[Proof when the kernel of $\rho_{\Gamma_\nu}$ is infinite] In case (a), $X$ is a torus $\C^2/\Lambda$ and $\ker(\rho_{\Gamma_\nu})$ is a group of translations of $X$
(see Proposition~\ref{prop:discrete}). Let $A\subset \Gamma_\nu$ be the pre-image of $A^*$; setting $K=\ker(\rho_{\Gamma_\nu})$, 
we obtain an exact sequence $0\to K\to A\to A^*\to 0$, with $A\subset \Gamma_\nu$ of finite index, $A^*\simeq \Z$ generated by 
a loxodromic element, and $K\subset X$ an infinite group of translations. Since $\nu$ is symmetric, the measure $\nu_A$ is also 
symmetric; since $\nu_A$ satisfies the moment condition~\eqref{eq:moment}, its projection on $A^*$ has a first moment (note
that if $f$ is loxodromic, then $\log(\norm{(f^*)^n})\asymp \abs{n}$). Since $K$ is Abelian, its action on $X$ is stiff; thus, 
as in Lemma~\ref{lem:stiff_finite_extension}.{(ii)}, the action of $A$ on $X$ is stiff.
Since $A$ has finite index in $\Gamma$, 
the action of $\Gamma$ on $X$ is stiff too by Lemma~\ref{lem:stiff_finite_extension}.{(i)}.

In case (b), we apply Proposition~\ref{pro:parabolic_infinite}. So, either $X$ is a torus, or the action of $\Gamma_\nu$ on the base $S$ of its invariant fibration $\tau\colon X\to S$ has finite order.
In the latter case, a finite index subgroup $\Gamma_0$  of $\Gamma$ preserves each fiber of $\tau$; then, $\Gamma_0$ contains 
a subgroup of index dividing $12$ acting by translations on these fibers. This shows that $\Gamma$ is virtually Abelian; in particular, 
$\Gamma$ is stiff. 
The last case is when the image of $\Gamma$ in $\Aut(S)$ is infinite and $X$ is a torus $\C^2/\Lambda_X$. 
Then, $S=\C/\Lambda_S$ is an elliptic curve and $\tau$ is induced by a linear projection $\C^2\to \C$, say the 
projection $(x,y)\mapsto x$. Lifting $\Gamma$ to $\C^2$, and replacing $\Gamma$ by a finite index subgroup
if necesssary, its action is by affine transformations of the form 
\begin{equation}
{\tilde{f}} \colon (x,y)\mapsto (x+a, y+m  x +b)
\end{equation}
with $m$ in  $\C^*$, and $(a,b)$ in $\C^2$. This implies that $\Gamma$ is a nilpotent group 
of length $\leq 2$; by Theorem~\ref{thm:raugi}  it   also acts stiffly and we are done.
\end{proof}

\begin{eg}\label{eg:stiffness_elementary}
If $X=\P^2(\C)$, its group of automorphism is $\PGL_3(\C)$ and for most choices of $\nu$ there is a unique 
stationary measure, which is not invariant; the dynamics is proximal, and this is opposite to stiffness (see~\cite{furstenberg_stiffness}).
If $X=\P^1(\C)\times C$, for some algebraic curve $C$, then $\Aut(X)$ contains $\PGL_2(\C)\times \Aut(C)$; if $\nu$ 
is a probability measure on $\PGL_2(\C)\times \{\id_C\}$, then in most cases the stationary measures are again non invariant.  
\end{eg}

\begin{pro}\label{pro:stiffness_elementary2}
Let $X$ be a complex projective surface, and $\Gamma$ be a subgroup of $\Aut(X)$ such that 
$\Gamma^*$ is finite. If $\Gamma$ preserves a probability measure, whose support   is 
Zariski dense in $X$, then the action of $\Gamma$ on $X$ is stiff.
\end{pro}

The main examples we have in mind is when the invariant measure is given by a volume form, or by an 
area form on the real part $X(\R)$ for some real structure on $X$, with $X(\R)\neq \emptyset$.

\begin{proof}
Replacing $\Gamma$ by a finite index subgroup we may assume that
 $\Gamma\subset\Aut(X)^\circ$. Denote by $\mu$ the invariant measure.
Let $G$ be the closure (for the euclidean topology) of $\Gamma$ in the Lie group $\Aut(X)^\circ$; then $G$ is a real Lie group preserving $\mu$. 

Let $\alpha_X\colon X\to A_X$ be the Albanese morphism of $X$. There is a homomorphism 
of complex Lie groups $\tau\colon \Aut(X)^\circ\to \Aut(A_X)^\circ$ such that $\alpha_X\circ f =  \tau(f)\circ \alpha_X$ 
for every $f$ in $\Aut(X)^\circ$. 

Pick a very ample line bundle $L$ on $X$, denote by 
$\P^N(\C)$ the projective space $\P(H^0(X,L)^\vee)$, where $N+1=h^0(X, L)$, and by 
$\Psi_L\colon X\to \P^N(\C)$ the Kodaira-Iitaka embedding of $X$ given by $L$.
By hypothesis, $(\Psi_L)_*\mu$ is not supported by a hyperplane of 
$\P^N(\C)$.

{\bf{Step 1.}}— Suppose $\tau(G)=1$. Since $\Pic^0(X)$ and $A_X$ are dual to each other, $G$ 
acts trivially on $\Pic^0(X)$ and $L$ is $G$-invariant, that is $g^*L=L$ for every $g\in G$. Thus 
there is a homomorphism $\beta\colon G\to \PGL_{N+1}(\C)$ such that $\Psi_L\circ g = \beta(g)\circ \Psi_L$ for every $g\in L$. 
If $G$ is not compact, there is a sequence of elements $g_n\in G$ going to infinity in $\PGL_{N+1}(\C)$: in the KAK 
decomposition $g_n=k_n a_n k_n’$, the diagonal part $a_n$ goes to $\infty$. 
Then, any probability measure on $\P^N(\C)$ which is invariant under all $g_n$ is 
supported in a proper projective subspace of $\P^N(\C)$, and this contradicts our preliminary remark. 
So, $G$ is compact in that case.

{\bf{Step 2.}}— Now, assume that $\tau(G)$ is infinite.  Identifying $\Aut(A_X)^\circ$ with $A_X$,
$\tau(\Aut(X)^\circ)$ is a complex algebraic 
subgroup of the torus $A_X$, of positive dimension since it contains $\tau(G)$. If the kernel of $\tau$ is finite, then $\Aut(X)^\circ$ is compact and virtually Abelian; thus, we may assume $\dim(\ker(\tau))\geq 1$. In particular the fibers of
$\alpha_X$ have positive dimension, $\dim(\alpha_X(X))\leq 1$ and $\alpha_X(X)$ is a 
curve, which is elliptic because it is invariant under the action of $\tau(\Aut(X)^\circ)$. 
Then, the universal 
property of the Albanese morphism implies $\alpha_X(X) = A_X$. 
 In particular, $\alpha_X$ is a submersion, for its critical values form a proper, $\tau(\Aut(X)^\circ)$-invariant subset of $A_X$. 
Thus, $X$ is a $\P^1(\C)$-bundle over $A_X$ because the fibers of $\alpha_X$ are smooth, are invariant under the action of 
$\ker(\tau)$, and can not be elliptic since otherwise $X$ would be a torus.
From~\cite[Thm 3]{Maruyama} 
(see also~\cite{Loray-Marin, Potters} for instance), there are two cases:
\begin{enumerate}
\item either $X=A_X\times \P^1(\C)$,  $\Aut(X)=\Aut(A_X)\times \PGL_2(\C)$ and we deduce as in the first step that $G$ is a compact group;
\item or $\Aut(X)^\circ$ is Abelian.
\end{enumerate}
In both cases  stiffness follows, and we are done. \end{proof}

\begin{rem}
Pushing the analysis further, it can be shown that, under the assumptions Proposition~\ref{pro:stiffness_elementary2},
$\Gamma$ is relatively compact. Indeed in the   last considered case, 
if $\Gamma$ is not bounded it can be  deduced  from \cite[Thm 3]{Maruyama} that there are elements 
with  wandering dynamics: all orbits in some Zariski open subset converge 
towards a section of $\alpha_X$. This contradicts the invariance of~$\mu$.
\end{rem}

\subsection{Invariant algebraic curves II} \label{par:invariant_curves2}
Let us start with an example. 
 \begin{eg}[See also~\cite{invariant}]\label{eg:mobius}
Consider an elliptic curve $E=\C/\Lambda$ and the Abelian surface $A=E\times E$. The group $\GL_2(\Z)$ determines a non-elementary 
group of automorphisms of $E\times E$ of the form $(x,y) \mapsto (ax+by, cx+dy)$. The involution $\eta=-\id$ 
generates a central subgroup of $\GL_2(\Z)$, 
hence $\PGL_2(\Z)$ acts on the (singular) Kummer surface $A/\eta$. Each singularity gives rise to a smooth $\P^1(\C)$ in the minimal 
resolution $X$ of $A/\eta$, the group $\{ B\in \PGL_2(\Z)\; ; \; B\equiv \id \mod 2\}$  preserves each of these $16$ rational curves, 
and its action on these curves is given by the usual linear projective action of $\PGL_2(\Z)$ on $\P^1(\C)$. In particular, it is proximal and strongly 
irreducible so it  admits  a unique, non-invariant, stationary measure. 
\end{eg}

The next result shows that when $\nu$ is symmetric, every 
non-invariant stationary measure is similar to  the previous example. 
 
\begin{pro}\label{pro:stiffness_curves}
Let  $(X, \nu)$ be a random holomorphic dynamical system, with
$\nu$ symmetric.  
Let $\mu$ be an ergodic $\nu$-stationary measure giving positive mass to some proper Zariski closed subset of $X$. 
Then $\mu$ is supported on a $\Gamma_\nu$-invariant proper Zariski closed subset and
\begin{enumerate}[\em (a)]        
\item either  $\mu$ is invariant; 
\item or the Zariski closure of $\supp(\mu)$   is a finite, disjoint union of smooth rational curves $C_i$, the stabilizer of $C_i$ in $\Gamma$ 
induces a strongly irreducible and  proximal subgroup of $\Aut(C_i)\simeq \PGL_2(\C)$, and $\mu(C_i)^{-1}\mu\rest{C_i}$ is the unique stationary measure of this group of Möbius transformations. 
\end{enumerate} 
Moreover, if $(X,\nu)$ is non-elementary, the curves $C_i$ have negative self-intersection and can be contracted on cyclic quotient singularities.
\end{pro}

Note that no moment assumption is assumed here. 
Before giving the proof, let us briefly discuss the question of stiffness for Möbius actions on 
$\pu(\C)$. Let $\nu$ be a symmetric  
measure on $\PGL_2(\C)$. 
As already said, by  Furstenberg's theory, if  
$\Gamma_{\nu}$ is strongly irreducible and unbounded  it admits a unique stationary measure, and this measure is not invariant.
Otherwise, any  $\nu$-stationary measure is invariant because 
\begin{itemize}
\item either $\Gamma_{\nu}$ is relatively compact and stiffness follows from 
\cite[Thm. 3.5]{furstenberg_stiffness};
\item or $\Gamma_{\nu}$ admits an invariant set made of two points, then $\Gamma_\nu$ is virtually Abelian and 
stiffness follows from Theorem~\ref{thm:raugi};
\item or $\Gamma_{\nu}$ is conjugate to a subgroup of the affine group $\Aff(\C)$ with no fixed point. 
\end{itemize}
In the latter case after conjugating $\Gamma_\nu$ to a subgroup of $\Aff(\C)$ 
we can write  any  $g\in \Gamma_{\nu}$  as $g(z)  = a(g)z+b(g)$. 
If $a(g)\equiv 1$ then $\Gamma_{\nu}$ is Abelian and we are done. Otherwise $\Gamma_{\nu}$ is 
merely  solvable and we apply the following lemma which follows from a result of    
Bougerol and Picard  
\begin{vcourte}
(see \cite[Thm. 2.4]{bougerol-picard};  a self-contained proof is provided
 in \cite{cantat-dujardin:vlongue}).  
\end{vcourte}
\begin{vlongue}
(see \cite[Thm. 2.4]{bougerol-picard}).  
\end{vlongue}

\begin{lem}
Let $\nu$ be a symmetric probability measure on $\Aff(\C)$.
If no point of $\C$ is fixed by $\nu$-almost every $g$, 
then the only $\nu$-stationary probability on  $\P^1(\C)$ is the point mass at $\infty$. 
\end{lem}

\begin{vlongue} 
\begin{proof}   
Assume by contradiction that there exists a stationary  
measure $\mu$ such that $\mu(\C) = 1$ and  $\mu(\{\infty\})=0$. 
If  $\Gamma_\nu$ is abelian, it is made of translations because it has no fixed point in $\C$; on the other hand 
 if $\Gamma_\nu$
is not abelian, its derived subgroup contains a non-trivial translation. Thus, in any case $\Gamma_\nu$ contains a non-trivial translation, and we infer that $\Gamma_\nu$ does not preserve any finite measure on $\C$. In particular  $\mu$ is not invariant. 

 
Let now $r_n$ be the right random walk associated to $\nu$ on $\Aff(\C)$. Put $\nu^\infty = \sum_{k=0}^\infty 2^{-{k+1}} \nu^{*k}$.
A classical martingale convergence 
argument (see \cite[Lem. II.2.1]{bougerol-lacroix})
provides a measurable set $\Omega_0$ with $\nu^\N(\Omega_0) = 1$ such that, for all $\omega\in \Omega_0$,
\begin{itemize}
\item    $r_n(\omega)_* \mu$ converges toward a probability measure 
$\mu_\omega$  and $\mu = \int \mu_\omega d\nu^\N(\omega)$;
\item for  $ \nu^\infty$-almost every $ \gamma $, 
$r_n(\omega)_*\gamma_*\mu$ converges towards the same limit  $\mu_\omega$. 
\end{itemize}
Since 
$\mu = \int \mu_\omega d\nu^\N(\omega)$,   we have $\mu_\omega(\C)= 1$ almost surely. 
Now, assume that for  some    $\omega\in \Omega_0$, $r_n(\omega)$ does not go to $\infty$  in 
$\PGL_2(\C)$. Extracting a convergent subsequence $r_{n_j}(\omega)\to r$, we infer that  $\gamma_*\mu  = \gamma'_*\mu = (r\inv)_*\mu_\omega$ for $(\nu^\infty\times \nu^\infty)$-almost-every $(\gamma, \gamma')$; hence $\mu$ is $\Gamma_\nu$-invariant, a contradiction. Thus  $r_n(\omega)$ goes to $\infty$ in $\PGL_2(\C)$ for almost every $\omega$. 

Suppose    that $(a(r_n(\omega)), b(r_n(\omega)))$ is unbounded in $\C^2$ for  
a subset $\Omega_0'\subset \Omega_0$ of positive measure. Set 
\begin{equation}
\tilde r_{n} (\omega)= \unsur{\max(\abs{a(r_{n}(\omega))}, \abs{b(r_{n}(\omega))})} r_{n}(\omega)
\end{equation}
and extract a subsequence ${n_{j}}$ so that   $\tilde r_{n_{j}}(\omega) \to \ell(\omega)$, where $\ell(\omega)$ is an affine endomorphism of~$\C$. 
If $\ell(\omega)(z)\neq 0$ then $r_{n_{j}}(\omega)(z)\to \infty$. Since $r_{n_{j}}(\omega)_*\mu\to \mu_\omega$ and  $\mu_\omega(\C)= 1$, we deduce that $\mu(\ell(\omega)^{-1}\{0\}) =1$. This is  a contradiction because $\mu$ is not concentrated at a single point.
Thus,   $(a(r_n(\omega)), b(r_n(\omega)))$ is almost surely bounded. Since $r_n(\omega)$ goes to $\infty$ in 
$\PGL_2(\C)$, $a(r_n(\omega))$ goes to 0 almost surely, in contradiction with the symmetry of $\nu$. This concludes the proof.
\end{proof}
\end{vlongue}

\begin{proof}[Proof of Proposition~\ref{pro:stiffness_curves}] If $\mu$ has an atom then, by ergodicity, $\mu$ is supported on a finite orbit 
and it is invariant. So we now assume that $\mu$ is atomless. By ergodicity, $\mu$ gives full mass to a $\Gamma_\nu$-invariant curve $D$; 
let $C_1, \ldots , C_n$ be its irreducible 
components. Let $\Gamma'$ be the finite index subgroup of $\Gamma_\nu$ stabilizing each $C_i$
and $\nu'$ be the hitting measure induced by   $\nu$ on $\Gamma'$; it is symmetric, 
$\mu$  is   $\nu'$-stationary, and so are its restrictions $\mu\rest{C_i}$, for each $C_i$.

If the genus of (the normalization of) $C_1$ is positive, then $\Gamma'\rest{C_1}\subset \Aut(C_1)$ is virtually Abelian, hence $\mu\rest{C_1}$ 
is $\Gamma'$-invariant. Since $\mu$ is ergodic, $\Gamma_\nu$ permutes transitively the $C_i$, and arguing as in Lemma \ref{lem:stiff_finite_extension}, we
see that $\mu$ is $\nu$-invariant as well. 
Now, assume that the normalization $\hat{C_1}$ is isomorphic to $\P^1(\C)$. If $C_1$ is not smooth, or if it intersects another 
$\Gamma_\nu$-periodic curve, then the image of $\Gamma'$ in $\Aut(\hat{C_1})\simeq \PGL_2(\C)$ is not strongly irreducible, and 
the discussion preceding this proof shows
 that $\mu$ is $\Gamma'$-invariant. Again, this implies that $\mu$ is $\Gamma_\nu$-invariant. The same 
holds if $\Gamma'$ is a bounded subgroup of  $\Aut(\hat{C_1})$. The only possibility left is that
$C_1$ is smooth, disjoint from the other periodic curves, and  $\Gamma'$ induces a strongly irreducible subgroup of  $\Aut(C_1)$. 
Since $\Gamma_\nu$ permutes transitively the $C_i$, conjugating the dynamics of 
the groups $\Gamma'\rest{C_i}$, the same property 
holds for each $C_i$. 

If $\Gamma_\nu$ is non-elementary,  Lemma~\ref{lem:periodic_curves} shows that $C_i^2=-m$ for some $m>0$, 
which does not depend on $i$ because $\Gamma_\nu$ 
permutes the $C_i$ transitively. Then, the $C_i$ being disjoint, one can contract them simultaneously, each of the contractions leading to a quotient singularity
$(\C^2,0)/\langle \eta\rangle$ with $\eta(x,y)=(\alpha x, \alpha y)$ for some root of unity $\alpha$ of order $m$ (see~\cite[\S III.5]{BHPVDV}).
\end{proof}



\subsection{Non-elementary groups: real dynamics}
 We now consider general non-elementary actions.  
 As explained in the introduction, so far our results are restricted to   
subgroups  of $\Aut(X)$ preserving a totally real surface $Y$. We further assume that there exists a 
$\Gamma_\nu$-invariant volume form on $Y$; this is automatically the case if 
$X$ is an Abelian, a K3, or an Enriques surface 
\begin{vcourte} 
(see \cite{invariant}). 
\end{vcourte} 
\begin{vlongue} 
(see Lemma \ref{lem:volume_Y}). 
\end{vlongue} 
Note that, \emph{a posteriori},  the results of 
\begin{vcourte} 
\S \ref{sec:rigidity} and~\cite{invariant} 
\end{vcourte}  
\begin{vlongue} 
\S \ref{sec:parabolic} and \ref{sec:rigidity} 
\end{vlongue}  
suggest that measures  supported on a totally real surface and invariant under a non-elementary subgroup of $\Aut(X)$ 
tend to be absolutely continuous, unless they are supported by a curve or a finite set. 
\begin{vlongue}
We saw in Example~\ref{eg:mobius}   that 
stiffness can fail in      presence of 
invariant rational curves along which the dynamics is that of a proximal and strongly irreducible
random product of M\"obius transformations. The
next theorem shows that for actions preserving a totally real surface, 
 this  obstruction to stiffness is the only one. 
\end{vlongue}
 
\begin{thm}\label{thm:stiffness_real}
Let $(X, \nu)$ be a non-elementary  random holomorphic dynamical system satisfying the 
moment condition \eqref{eq:moment}. 
Assume that 
 $Y\subset X$ is  a $\Gamma_\nu$-invariant totally real 2-dimensional smooth 
 submanifold   such  that the action of  
$\Gamma_\nu$ on $Y$ preserves a probability measure $\vol_Y$  equivalent to the Riemannian volume on $Y$.
Then, 
every ergodic   stationary measure $\mu$ on $Y$ is:
\begin{enumerate}[\em (a)]        
\item  either almost surely invariant,
\item or  supported on a $\Gamma_\nu$-invariant algebraic curve. 
\end{enumerate}
In particular if there is no $\Gamma_\nu$-invariant  curve then   $(Y, \nu)$ is stiff. 
Moreover, if the fiber entropy  of $\mu$ is positive, then $\mu$ is the restriction of 
$\vol_Y$ to a subset of positive volume.
\end{thm}
Recall   from  Lemma \ref{lem:periodic_curves}  that 
$\Gamma_\nu$-invariant curves can be contracted. For 
the induced random dynamical system on the  
resulting   singular surface, stiffness holds unconditionally.  
If furthermore $\nu$ is symmetric then the result 
can be made more precise by applying  Proposition \ref{pro:stiffness_curves}.

\begin{proof}[Proof of Theorem~\ref{thm:stiffness_real}]  
We split the proof in two steps. 

{\bf{Step 1.--}} Let $\mu$ be an ergodic stationary measure supported on $Y$. We assume that $\mu$ is not invariant, 
and we want to prove that it is supported on a $\Gamma_\nu$-invariant curve. Since the action is 
volume preserving, its Lyapunov exponents satisfy $\lambda^{-}+\lambda^{+}=0$ (see Lemma~\ref{lem:sum_exponents}). 
The invariance principle (Theorem~\ref{thm:invariance_principle}) shows that $\mu$  is hyperbolic: indeed
 $\mu$ is almost surely invariant when $\lambda^{-}\geq 0$. We can therefore apply Theorem 3.4 of \cite{br} to obtain the following trichotomy:
\begin{itemize} 
\item[(1)] either $\mu$ has finite support, so it is invariant;
\item[(2)] or the   distribution  of Oseledets stable directions is non-random;
\item[(3)] or $\mu$ is almost surely invariant and absolutely continuous with respect to $\vol_Y$: even more, it is the restriction 
of $\vol_Y$ to a  subset of positive volume.
\end{itemize}
Since $\mu$ is not invariant,  we are in case (2).  Theorem~\ref{thm:alternative_stable} then
 implies that $\mu$ is supported
on an invariant algebraic curve. This concludes the proof of the first assertions in Theorem~\ref{thm:stiffness_real}, including the 
stiffness property   when $\Gamma$ has no periodic curve.

{\bf{Step 2.--}} It remains to prove the last assertion. 
Let then $\mu$ be an ergodic stationary measure with $h_\mu(X,\nu)>0$. In the above trichotomy, (1) is now excluded. To  exclude the alternative~(2), by Theorem~\ref{thm:alternative_stable},  
 it suffices to show that $\mu$ is not supported on an invariant curve. 
By Proposition \ref{pro:margulis-ruelle} ({i.e.} the fibered Margulis-Ruelle inequality),
$\mu$ is hyperbolic. If $\mu$ is supported on an algebraic 
curve, the proof of Corollary \ref{cor:nevanlinna} leads to the following alternative: either $\mu$ is atomic or the Lyapunov exponent along that curve is negative. In the latter case $\mu$ is proximal along that curve and its stable conditionals are points. In both cases the fiber entropy would vanish, in 
contradiction with our hypothesis, so  $\mu$ is not supported on an algebraic curve, as desired. \end{proof}


\begin{vlongue}

\section{Subgroups with parabolic elements}\label{sec:parabolic}

We say that  $\Gamma\subset \Aut(X)$ is {\bf{twisting}}  if it contains a  parabolic automorphism (this terminology is
justified below). 
This section investigates the dynamics of 
$(X, \nu)$ when $\Gamma_\nu$ is non-elementary and twisting. 
Under this assumption invariant measures can be classified (Theorem \ref{thm:classification_invariant}): 
they are either hyperbolic or carried by some proper algebraic subset (Theorem \ref{thm:hyperbolic}). 
 
\begin{rem}
In many examples for which  $\Aut(X)$ contains a non-elementary group, $\Aut(X)$ contains also 
a parabolic automorphism (see the examples in \S\S\ref{par:Wehler}--\ref{par:Coble-Blanc}). 
So, if we are interested in random 
dynamical systems for which $\Gamma_\nu$ has finite index in $\Aut(X)$, 
the twisting assumption is quite  natural. Also, if $\Aut(X)$ is
both twisting and non-elementary, then there are thin subgroups $\Gamma\subset \Aut(X)$ with the same property: one 
can take two  
parabolic automorphisms $g$ and $h$ generating a non-elementary group, and set $\Gamma=\langle g^m, h^n\rangle$
for large integers $m$ and $n$.  
\end{rem}
\end{vlongue}

\begin{vlongue}
\subsection{Dynamics of parabolic automorphisms}\label{par:parabolic_automorphisms} 
Recall from \S\ref{par:parabolic_basics} that if $h$ is  
a parabolic automorphism of a compact K\"ahler surface $X$, it 
preserves a unique genus $1$ fibration, given by the fibers of a rational map
$\pi_h\colon X\to B$. In particular there is an automorphism $h_B$ of $B$ such that 
\begin{equation}
\pi\circ h=h_B\circ \pi.
\end{equation}
Moreover, if $X$ is not a torus there exists an integer $m>0$ such that $h^m$ preserves every fiber of $\pi$ and acts by translation on every  smooth fiber ({Proposition \ref{pro:parabolic_infinite}}). 
As shown in Lemma~\ref{lem:halphen_flat}, $h$ behaves like a ``complex Dehn twist'', acting by translations along the fibers of $\pi$, with a  shearing property in the transversal direction. 
This twisting property justifies the vocabulary introduced for ``twisting groups''. When $X$ is 
rational, 
the invariant fibration comes from a Halphen pencil of $\P^2_\C$ (see~\cite{Cantat:SLC}); this is why parabolic automorphisms are also called {\bf{Halphen twists}}.

Let $h$ be a parabolic automorphism with $h_B=\id_B$. The critical values of $\pi$ form a finite
subset $\mathrm{Crit}(\pi)\subset B$; we denote its complement by $B^\circ$.
Each fiber $X_w:=\pi^{-1}(w)$, $w\in B^\circ$, is a smooth curve of genus $1$, isomorphic to $\C/L(w)$ for some 
lattice $L(w)=\Z\oplus \Z\tau(w)$; and $h$ induces a translation $h_w(z)=z+t(w)$ of $X_w$, for some $t(w)\in \C/L(w)$. 
The points $w$ for which $h_w$ is periodic are characterized by the relation $t(w)\in \Q\oplus \Q\tau(w)$. If 
\begin{equation}
t(w)-(a+b\tau(w))\in  \R \cdot (p+q\tau(w))
\end{equation}
for some $(a,b)\in \Q^2$ and $(p,q)\in \Z^2$, the closure of $\Z t(w)$ in $\C/L(w)$ is an Abelian Lie group of dimension $1$, 
isomorphic to $\Z / k\Z \times \R /\Z$ for some $k>0$; then, the closure of each orbit of $h_w$ is a union of $k$ circles. Locally in $B^\circ$ this occurs along a countable union of analytic curves $(R_j)$. Otherwise, the orbits of $h_w$ are dense in $X_w$, and the unique $h_w$ invariant probability
measure is the Haar measure on $X_w$. 

Now, assume that $Y\subset X$ is a real analytic subset of $X$ of real codimension $2$, and that $h$ preserves $Y$; 
for instance $h$ may preserve  a real structure on $X$, and $Y$   be a connected component of  $X(\R)$. 
Then,  $\pi(Y)\subset B$ is (locally) contained in the  curves $R_j$. The smooth fibers $\pi_{\vert Y}^{-1}(w)$, 
 for $w\in \pi(Y)\setminus \mathrm{Crit}(\pi)$, are unions of circles along which the orbits of  
$h_w$ are either dense (for most $w\in \pi(Y)$) or finite (for countably many $w\in \pi(Y)$).  

\begin{lem}\label{lem:halphen_flat} Assume that $h_B$ is the identity. Let $U\subset B^\circ$ be a simply connected open subset.
There is a countable union of analytic curves $R_j\subset U$, such that 
\begin{enumerate}[\em (1)]
\item $h$ acts  by translation on each fiber $X_w=\pi^{-1}(w)$, $w \in U$;

\item for $w\in U\setminus\cup_j R_j$, the action of $h$ in the fiber $X_w$ is a totally irrational 
translation (it is uniquely ergodic, and its orbits are dense in $X_w$); 

\item for $w$ in some countable subset of $U$, the orbits of $h_w$ are finite;

\item if the orbits of $h_w$ are neither dense nor finite, then $w\in \cup_j R_j$ and the closure of each orbit of $h_w$ 
is dense in a finite union of circles;

\item there is a finite subset $\mathrm{Flat}(h)\subset U$ such that for  $x\notin \pi\inv\lrpar{
\mathrm{Flat}(h)}$
\[
\lim_{n\to \pm \infty} \norm{D_xh^n}\to +\infty 
\]
locally uniformly in $x$; more precisely for every 
 $v\in T_xX \setminus T_x X_{\pi(x)}$, 
 $\norm{D_x h^{n}(v)}$ grows linearly while 
 $\frac{1}{n}\pi_*(D_xh^{n}(v))$ converges to $0$. 
\end{enumerate}
Moreover, if $h$ preserves a $2$-dimensional real analytic subset $Y\subset X$, then
\begin{enumerate}[\em (6)]
\item  $\pi$ induces on $Y$ a singular fibration whose generic leaves are union of (one or two)  circles, 
and there exists an integer $m\in\{1,2\}$ such that $h^m$ preserves these circles and is uniquely ergodic along these circles except countably many of them.  
\end{enumerate} 
\end{lem}

This lemma is proven in~\cite{cantat_groupes, invariant}; Property(5) is the above mentioned twisting property of $h$.

\subsection{Classification of invariant measures} 
In this section, we review the classification of invariant ergodic probability measures for  twisting non-elementary groups of
automorphisms; we refer to \cite{cantat_groupes, invariant} for details and examples. 
If $X$ is a real K3 or Abelian surface and $X(\R)\neq \emptyset$
there is a unique section of the canonical bundle of $X$ which, when restricted to $X(\R)$, induces a 
positive area form of total area $1$; we denote this area form by $\vol_{X(\R)}$. 
\end{vlongue}
\begin{vlongue}
The associated probability measure 
is invariant under the action of $\Aut(X_\R)$, the subgroup of $\Aut(X)$ preserving the real structure.  
\end{vlongue}
\begin{vlongue}
In fact, such a smooth invariant probability measure exists on any totally real invariant surface 
(see \cite[\S 5]{invariant}): 

\begin{lem}\label{lem:volume_Y}
Let $X$ be an Abelian surface, or a K3 surface, or an Enriques surface with universal cover $X'$. 
Let $Y\subset X$ be a (real) surface of class $C^1$. Let $\Aut(X;Y)$ be the 
subgroup of $\Aut(X)$ preserving $Y$. If $Y$ is totally real,  
 $\Omega_X$ (resp. $\Omega_{X'}$) induces a smooth $\Aut(X;Y)$-invariant probability measure $\vol_Y$ on $Y$. \end{lem}

Note that there indeed exists examples of 
subgroups preserving a totally real surface $Y\subset X$ which is not a real form of $X$ (see \cite[\S 6]{invariant}). 
\end{vlongue}

\begin{vlongue}
The classification of invariant measures then reads as follows. 
\end{vlongue}

\begin{vlongue}
\begin{thm}\label{thm:classification_invariant}
Let $X$ be a compact K\"ahler surface. Let $\Gamma$ be a twisting non-elementary subgroup of $\Aut(X)$. 
Let $\mu$ be a $\Gamma$-invariant ergodic probablity measure on $X$. 
Then, $\mu$ satisfies one and only one of the following properties. 

\begin{enumerate}[{\em (a)}]
\item $\mu$ is the average on a finite orbit of $\Gamma$;
\item $\mu$ is supported by a $\Gamma$-invariant curve $D\subset X$;
\item  there is a $\Gamma$-invariant proper algebraic subset $Z$ of $X$, and a $\Gamma$-invariant, totally real, real analytic submanifold $Y$ of 
$X\setminus Z$ such that {\em {(1)}} $\mu(Z)=0$, {\em {(2)}} the support of $\mu$ is a union of finitely many connected components of $Y$, {\em {(3)}} $\mu$ 
is absolutely continuous with respect to the Lebesgue measure on $Y$, and {\em {(4)}} the density of $\mu$ with respect to any real analytic
area form on $Y$ is real analytic;
\item there is a $\Gamma$-invariant proper algebraic subset $Z$ of $X$ such that {\em {(1)}} $\mu(Z)=0$, {\em {(2)}} the support of $\mu$ is equal to $X$, {\em {(3)}} $\mu$ is absolutely continuous with respect to the Lebesgue measure on $X$, and {\em {(4)}} the density of $\mu$ with respect to any real analytic
volume form on $X$ is real analytic on $X\setminus Z$.
\end{enumerate}
If $X$ is not a rational surface, then in case \emph{(c)} (resp. \emph{(d)}) we can further conclude that the invariant measure is locally  proportional  to 
$\vol_Y$ (resp. equal to $\vol_X$). 
\end{thm}

The reason why we say that $\mu$ is proportional to $\vol_Y$ (and not equal to it) in the last sentence is because $\mu$ may be equal to zero on some components of $Y\setminus Z$. This theorem is a combination of Theorem~1.1 and \S~5.3 of \cite{invariant}. 
Let us also point out the following corollary of the proof. 

\begin{cor} \label{cor:real_invariant}
Let $\Gamma\leq \Aut(X)$ be as in Theorem~\ref{thm:classification_invariant}.  Assume furthermore that 
$X$ and $\Gamma$ are defined over $\R$ and $\Gamma$ does not preserve any proper Zariski closed subset of $X$. 
Then any  $\Gamma$-invariant ergodic measure supported on $X(\R)$ is supported by 
a union $X(\R)'=\cup_j X(\R)_j$ of connected components $X(\R)_j$ of $X(\R)$,  and is locally given by 
positive real analytic $2$-forms on $X(\R)'$. If $X$ is not rational, $\mu$ is equal to  the restriction of $\vol_{X(\R)}$ to $X(\R)'$, 
up to  some normalizing factor. 
 \end{cor}
\end{vlongue}

\begin{vlongue}
Using this classification we can now sharpen the conclusion of Theorem \ref{thm:stiffness_real} in the presence of parabolic automorphisms.  
When $Y= X(\R)$, the  statement   can also 
be combined with Corollary \ref{cor:real_invariant} to get an even more precise result.

\begin{cor}\label{cor:real_stationary}  
Let $(X, \nu)$ be a random holomorphic dynamical system on a compact K\"ahler surface, satisfying 
\eqref{eq:moment} and such that  $\Gamma_\nu$ is twisting and non-elementary. 
Let $Y\subset X$ be a $\Gamma_\nu$-invariant,  smooth,  totally real surface    such  that, on $Y$, $\Gamma_\nu$ preserves a probability measure $\vol_Y$  equivalent to the Riemannian volume.

Then  up to a positive multiplicative factor, every ergodic   stationary measure $\mu$ supported on $Y$ is :  
\begin{itemize}
\item  either the counting measure on a finite orbit;
\item or supported on a $\Gamma_\nu$-invariant algebraic curve; 
\item or  the restriction of $\vol_Y$ to a $\Gamma_\nu$-invariant 
open subset of $Y$ whose boundary is piecewise smooth.   
\end{itemize}
\end{cor}

In the last alternative, the boundary is obtained by intersecting an algebraic curve $D\subset X$ with $Y$; it may have a finite number
of singularities.

\begin{proof}
We just have to repeat the proof of Theorem \ref{thm:stiffness_real},  by incorporating the classification given in Theorem \ref{thm:classification_invariant}.  Note that $Y$ is automatically real analytic in this case. 
 \end{proof}
\end{vlongue}
 
\begin{vlongue}
\subsection{Hyperbolicity of the invariant volume} \label{subs:hyperbolicity}
It is a fundamental  (and mostly open) problem in conservative dynamics to show the typicality of non-zero Lyapunov exponents 
on a set of positive Lebesgue measure. In deterministic dynamics, a recent breakthrough is the work of 
Berger and Turaev \cite{berger-turaev}. Adding some randomness makes such a hyperbolicity result easier to obtain:
see \cite{blumenthal-xue-young} for random perturbation of the standard map, and \cite{barrientos-malicet, obata-poletti}
for random conservative diffeomorphisms on (closed real) surfaces. The  results  of Barrientos and Malicet or 
Obata and Poletti~\cite{barrientos-malicet, obata-poletti} are perturbative in nature and do not give explicit examples. 
Here the high rigidity of complex algebraic automorphisms
will be sufficient to show that   twisting, non-elementary, random dynamical systems $(X,\nu)$ automatically 
satisfy some non-uniform hyperbolicity with respect to the volume. 

\begin{thm}\label{thm:hyperbolic}
Let $X$ be a compact K\"ahler surface,
and let $\Gamma$ be a non-elementary, twisting subgroup of $\Aut(X)$. Let $\mu$ be an ergodic  $\Gamma$-invariant measure 
giving no mass to proper Zariski closed subsets of $X$ (\footnote{Hence by Theorem \ref{thm:classification_invariant},
$\mu$ is equivalent to  $\vol_X$ or $\vol_Y$ for some real analytic invariant surface with boundary.}). 
Then  for every  probability measure $\nu$ on $\Aut(X)$  
satisfying    the moment condition~\eqref{eq:moment} and such that $\Gamma_\nu=\Gamma$, 
$\mu$ is hyperbolic and the fiber entropy $h_\mu(X,\nu)$ is positive.\end{thm}
\end{vlongue}


\begin{vlongue}
The same argument leads to a variant of this result when $\Gamma_\nu$ contains a Kummer example. 
\end{vlongue}

\begin{vlongue}
Before stating our next result, let us recall the definition of classical Kummer examples (see also Example~\ref{eg:mobius}, and~\cite[\S 1.3]{cantat-dupont} for a more general definition).
Let $A = \C^2/\Lambda$ be a complex torus and let $\eta$ be the involution given by $\eta(z_1, z_2) = (-z_1, -z_2)$, which has $16$ 
fixed points. Then $A/\langle\eta\rangle$ is a  
 surface with $16$ singular points, and resolving  these singularities (each of them requires a single blow-up)
yields a  so-called \textbf{Kummer  surface} $X$: a K3 surface with $16$ disjoint nodal curves. 
Let $f_A$ be a loxodromic  automorphism of $A$ which is induced by a linear transformation of $\C^2$ preserving $\Lambda$;
then $f_A$ commutes to $\eta$ and descends to an automorphism $f$ of~$X$; such automorphisms will be referred to as 
\textbf{classical Kummer examples}. Of course, they preserve the canonical volume $\vol_X$. Notice that the Kummer surface $X$
 also supports automorphisms which are not coming from automorphisms of $A$ (see~\cite{Keum-Kondo} and~\cite{Dolgachev-Keum} for instance). 

\begin{thm} \label{thm:hyperbolic_kummer} 
Let $(X, \nu)$ be a non-elementary  random dynamical system on a Kummer  K3 surface satisfying 
\eqref{eq:moment} and such that  $\Gamma_\nu$ contains a classical Kummer example. Then   any ergodic $\Gamma_\nu$-invariant 
measure giving no mass to proper Zariski closed subsets of $X$ is hyperbolic and has positive  fiber entropy. 
\end{thm}

In this   statement  we do not assume that $\Gamma_\nu$ contains a parabolic element. 
In Theorem \ref{thm:rigidity_kummer} below, we classify  invariant probability measures which are 
supported on an invariant, real analytic, and  totally real surface $Y$, when $\Gamma_\nu$ contains a Kummer example.

Theorems \ref{thm:hyperbolic} and \ref{thm:hyperbolic_kummer} will be proven in \S\ref{subs:proof_hyperbolic}. 
\end{vlongue}

\begin{vlongue}
\subsection{Ledrappier's invariance  principle and   invariant measures on $\pp T\cX$} \label{subs:ledrappier_invariance_principle} 
This paragraph contains preliminary results for the proof of Theorems \ref{thm:hyperbolic} and \ref{thm:hyperbolic_kummer}. 
Our presentation is inspired by the exposition of \cite{barrientos-malicet}. It is similar in spirit
 to that of \cite{obata-poletti}, which relies on the ``pinching and twisting'' formalism of Avila and Viana 
(see \cite{viana} for an introduction\footnote{Beware that the word ``twisting'' has a different meaning there.}). 
Most of this discussion is valid for a random holomorphic  dynamical system on an arbitrary  complex surface (not necessarily compact), 
satisfying \eqref{eq:moment}.

Let $\mu$ be an ergodic $\nu$-stationary measure.  
We introduce the projectivized tangent bundles $\P T \cX_+  = \Omega \times \P TX$ and $\pp T\cX = \Sigma\times \P TX$. The tangent bundles $TX$ and $\P TX$ admit measurable trivializations over a set of full measure. 
Consider any probability measure $\hat\mu$ on  $\P TX$ that is stationary under the random dynamical system induced by $(X,\nu)$ on $\P TX$ 
and whose projection on $X$ coincides with $\mu$, i.e.
$\pi_*\hat\mu= \mu $ where $\pi\colon \pp TX\to X$ is the natural projection. Such measures always exist.
Indeed 
the set of probability measures on $\P TX$  projecting to $\mu$ is compact and convex, and it is non-empty since it contains
the  measures $\int \delta_{[v(x)]} d\mu(x)$ for any measurable section
$x\mapsto [v(x)]$ of $\P TX$; thus, the operator $\int \P (Df) \,d\nu(f)$ has a fixed point in that set. The stationarity of $\hat\mu$ is equivalent to the invariance of $\nu^\N\times \hat\mu$
 under  the transformation ${\widehat{F}_+}\colon \Omega\times \pp TX\to \Omega\times \pp TX$ defined by 
\begin{equation}
{\widehat{F}_+}(\omega, x, [v])= (\sigma(\omega), f^1_\omega(x), \P(D_xf^1_\omega)[v])
\end{equation}
for any non-zero tangent vector $v\in T_xX$. We denote by $\hat\mu_x$ the family of probability measures -- on the fibers $\pp T_x X$  of $\pi$ --
given by the disintegration of $\hat\mu$  with respect to $\pi$;
the conditional measures of $\nu^\N\times\hat\mu$ with respect to the projection 
$\P T\cX\to X$ are given by $\hat\mu_{\omega,x} = \nu^\N\times \hat\mu_x$. 
 
\begin{rem}
Even when  $\mu$ is $\Gamma_\nu$-invariant, this construction only provides a stationary measure on 
$\pp TX$.  This is exactly what    happens for twisting non-elementary subgroups:
 indeed  we will show in  \S\ref{subs:proof_hyperbolic} that projectively invariant measures do not exist in this case. 
\end{rem}
\end{vlongue}

\begin{vlongue}
 The tangent action of our random dynamical system  gives rise to a stationary product of 
matrices in $\GL(2, \C)$. To see this, fix a measurable trivialization $P\colon TX\to X\times \C^2$, given 
by linear isomorphisms  $P_x\colon T_xX\to \C^2$, which conjugates the action of $DF_+$ to that of a linear
  cocycle $A:\cX_+ \times \C^2\to \cX_+ \times \C^2$ over $(\cX_+, F_+, \nu^\N\times \mu)$. 
 In this context, 
Ledrappier establishes  in \cite{ledrappier_stationary} the following ``invariance principle''.

\begin{thm}\label{thm:ledrappier_invariance_principle}
If $\lambda^{-}(\mu) = \lambda^{+}(\mu)$, then for any  stationary measure $\hat \mu$ on $\pp TX$ projecting   to $\mu$, we have $\P(D_xf)_*\hat\mu_x = \hat\mu_{f(x)}$
for $\mu$-almost every $x$ and $\nu$-almost every $f$.
\end{thm}

From this point the second main ingredient of the proof is 
 a classification of such projectively invariant measures; this is where we follow \cite{barrientos-malicet}. 
To explain this result a bit of notation is required. Let $V$ and $W$ be   hermitian vector spaces of dimension 2. 
Endow the projective lines $\P(V)$ and $\P(W)$ with their respective Fubini-Study metrics. If $g\colon V\to W$ 
is a linear isomorphism,  set  
\begin{equation}
\llbracket g \rrbracket= \norm{ \P (g) }_{C^1}
\end{equation}
where  $\P(g) \colon \pp(V)\to \pp(W)$
is the projective linear map induced by $g$ and $\norm{\cdot}_{C^1}$
is the maximum of the norms of $D_z\P(g) \colon T_z\pp(V)\to T_{\P gz)}\pp(W)$ with respect to the Fubini-Study metrics. 
Let us fix  two  isometric isomorphisms $\iota_V\colon V\to \C^2$  
and $\iota_W\colon W\to \C^2$ to the standard
hermitian space $\C^2$.  If we denote by $\iota_W\circ g\circ \iota_V^{-1} = k_1ak_2$ the KAK decomposition of
$\iota_W\circ g\circ \iota_V^{-1}$ in $\PSL(2, \C)$, we get 
$
 \llbracket g \rrbracket  = \norm{a}^2   =  \norm{\iota_W\circ g\circ \iota_V^{-1}}^2
$
  where 
$\norm{\cdot}$ is the  matrix norm in $\PSL_2(\C)=\SL_2(\C)/\langle \pm \id\rangle$ associated to the Hermitian norm in $\C^2$.
In particular:
\begin{itemize}
\item[(a)] $\llbracket g  \rrbracket=1$ if and only if $\P (g)$ is an isometry from $\pp(V)$ to $\pp(W)$;
\item[(b)] for a sequence $(g_n)$ of linear maps $V\to W$, 
$\llbracket g_n\rrbracket$ tends to $+\infty$ as $n$ goes to $+\infty$ if and only if $\P(\iota_W\circ g\circ \iota_V^{-1})$ 
tends  to $\infty$ in  $\PSL_2(\C)$.
\end{itemize}
\end{vlongue}


\begin{vlongue}
We are now ready to state the classification of projectively invariant measures. 

\begin{thm}\label{thm:classification_proj_invariant}
Let  $(X, \nu)$ be a random dynamical system on a complex surface and let $\mu$ be an ergodic stationary measure. Let
$\hat\mu$ be a stationary measure on $\P TX$ such that $\pi_*{\hat{\mu}}=\mu$ and $(\P D_xf)_*\hat\mu_x = \hat\mu_{f(x)}$
for $\mu$-almost every $x$ and $\nu$-almost every $f$. Then, exactly one of the following two properties is satisfied: 
 \begin{enumerate}[\em (1)] 
 \item For $(\nu^\N\times \mu)$-almost every
 $\cx   = (\omega, x)$, the sequence $\llbracket D_xf^n_\omega\rrbracket$  
is unbounded  and then:
 \begin{itemize}
 \item[(1.a)] either there exists a measurable $\Gamma_\nu$-invariant family of lines  
  $E(x)\subset T_xX$ such that   $\hat\mu_x=\delta_{[E(x)]}$ for $\mu$-almost every $x$;
  \item[(1.b)] or there exists a measurable $\Gamma_\nu$-invariant family of pairs of  lines  
  $E_1(x), E_2(x) \subset T_xX$ and positive  numbers $ \lambda_1, \lambda_2$ with 
  $\lambda_1+\lambda_2 =1$  such that    $\hat\mu_x=\lambda_1\delta_{[E_1(x)]} + \lambda_2\delta_{[E_2(x)]}$  for $\mu$-almost every $x$;
 \end{itemize}
\item The projectivized tangent action of $\Gamma_\nu$ is reducible to a compact  group, that is 
there exists a measurable   trivialization  
of the tangent bundle $(P_x: T_xX \to \C^2)_{x\in X}$,    such that for every $f\in \Gamma_\nu$ and every 
$x$, $\P\lrpar{P_{f(x)}\circ D_xf\circ P_x\inv}$ belongs to the unitary group $\mathsf{PU}_2(\C)$. 
\end{enumerate}
\end{thm}
\end{vlongue}

\begin{vlongue}
In assertion (1.b), the pair is not naturally ordered, i.e. there is no natural distinction of $E_1$ and $E_2$, 
the random dynamical system may a priori permute these lines. The proof is obtained by adapting the 
arguments of  \cite{barrientos-malicet}  to the complex case. 
Details are given in Appendix~\ref{app:barrientos_malicet}. 
\end{vlongue}

\begin{vlongue}
\subsection{Proofs of Theorems \ref{thm:hyperbolic} and \ref{thm:hyperbolic_kummer}}\label{subs:proof_hyperbolic}  

\subsubsection{Proof of Theorem~\ref{thm:hyperbolic}}   
By Theorem~\ref{thm:classification_invariant}, 
$\mu$ is either equivalent  to the Lebesgue measure on $X$, or to the $2$-dimensional Lebesgue measure on some components of an invariant totally real surface $Y\subset X$.  
Let us assume, by contradiction, that $\mu$ is not hyperbolic. Hence its Lyapunov exponents vanish, 
and by Theorem~\ref{thm:ledrappier_invariance_principle} and Theorem \ref{thm:classification_proj_invariant}, there 
 is a measurable set $X'\subset X$ with $\mu(X') = 1$ such that one of the following properties is satisfied along $X'$:   \label{alternative_abc}
 \begin{itemize}
\item[(a)]  there is a measurable $\Gamma_\nu$-invariant line field $E(x)$; 
\item[(b)]  there exists  a measurable $\Gamma_\nu$-invariant splitting $E(x)\oplus E'(x)=T_xX$ of the tangent bundle; 
here, the invariance should be taken in the following weak sense: an element $f$ of $\Gamma_\nu$ maps $E(x)$ to  
$E(f(x))$ or $E'(f(x))$;
\item[(c)] there exists a measurable trivialization $P_x\colon T_xX\to \C^2$  such that in the corresponding coordinates
 the projectivized 
differential $\P(Df_x)$, $f\in \Gamma_\nu$,  takes its values in $\mathsf {PU}_2(\C)$. 
\end{itemize}
\end{vlongue}

\begin{vlongue}
Fix a small $\e>0$. By Lusin's  
 theorem, there is a compact set $K_\e $  with  $\mu(K_\e)>1-\e$ 
 such that the data   $x\mapsto E(x)$, resp.  $x\mapsto (E(x), E'(x))$ or $x\mapsto P_x$ in the respective 
 cases (a,b,c)  
 are continuous on $K_\e$. 
 In particular, in case (c), the norms of $P_x$ and $P_x^{-1}$ are bounded by some uniform constant $C(\e)$ 
on $K_\e$; hence, if $g\in \Gamma_\nu$ and $x$ and $g(x)$ belong to 
$K_\e$,   $\llbracket Dg_x \rrbracket$ is bounded by $C(\e)^2$. 

Fix a pair of Halphen twists $g$ and $h\in \Gamma_\nu$ with distinct invariant fibrations $\pi_g\colon X\to B_g$ and $\pi_h\colon X\to B_h$
respectively (see Lemma~\ref{lem:pairs_of_twists}). In a first stage assume that $X$ is not a torus: then by Proposition \ref{pro:parabolic_infinite}  we may  assume that 
$g$ and $h$  preserve every fiber of their respective invariant fibrations (see Section~\ref{par:parabolic_automorphisms}). 

First assume that $\mu$ is absolutely continuous with respect to the Lebesgue measure on $X$, with a positive real analytic density
on the complement of some invariant, proper, Zariski closed subset. Since the invariant fibration is holomorphic, the disintegration $\mu_b$ of 
$\mu$ is absolutely continuous on almost every fiber $\pi_h^{-1}(b)$. Thus,    
 there exists a fiber $\pi_h\inv(b)$  such that   (1) the Haar measure of 
$K_\varepsilon \cap \pi_h\inv(b)$ is positive,  
(2)  
$b \notin \mathrm{Flat}(h)$ and (3) 
the dynamics of $h$ in $\pi_h\inv(b)$ is uniquely ergodic (see Lemma~\ref{lem:halphen_flat}). 
Then we can pick $x\in \pi_h\inv(b)$ such that 
$(h^{k}(x))_{k\geq 0}$ visits $K_\e$ infinitely many times. The fifth assertion of Lemma~\ref{lem:halphen_flat}
rules out case (c) because the twisting property implies that the projectivized derivative $\llbracket Dh^n_x \rrbracket$ tends to infinity, while it should be bounded along the sequence of times 
 $n$ for which $h^n(x)\in K_\e$.
Case (b)  is also excluded: under the action of $h^n$, tangent vectors  projectively 
converge to the tangent space of the fibers, so the only possible invariant subspace is $\ker(D\pi_h)$.
Thus we are in case (a) and moreover  $E(x) = \ker D_x\pi_h$  for $\mu$-almost every $x$. But then, 
using $g$ instead of $h$ and the fact that $\mu$ does not charge the algebraic curve along which the fibrations 
$\pi_g$ and $\pi_h$ are tangent, we get a contradiction. 
\begin{vcourte}
This shows that (a) does not hold either, and proves that $\mu$ is hyperbolic.
\end{vcourte}
\begin{vlongue}
This shows that alternative (a) does not hold either, and this contradiction proves that $\mu$ is hyperbolic.
\end{vlongue}

\begin{vlongue}
If $\mu$ is supported by a $2$-dimensional real analytic subset $Y\subset X$, the same proof applies, 
except that we disintegrate
$\mu$ along the  singular  foliation of $Y$  by circles induced by $\pi_h$ and use the fact that a generic leaf is a 
circle along which
  $h$ is   uniquely ergodic (see Lemma~\ref{lem:halphen_flat}.{(6)}). 

If $X$ is a torus,
then its tangent bundle is trivial and the differential of an automorphism is 
constant. In an appropriate basis, the differential of a Halphen twist $h$  is of the form 
\begin{equation}
\begin{pmatrix} 1 & \alpha \\ 0 & 1 \end{pmatrix} \text{ with } \alpha\neq 0.
\end{equation}
 Thus we are in case (a) with $E(x) = \ker D_x\pi_h$ for $\mu$-almost every $x$. Using another 
twist $g$ transverse to $h$ we get a contradiction as before.

Since $\mu$ is invariant then the invariant 
measure $\m$ on $\cX$ is 
equal to $\nu^\Z\times \mu$. In both cases $\mu\ll \vol_X$ and  $\mu\ll \vol_Y$.  
The absolute continuity of the foliation by  local Pesin  unstable manifolds implies that 
the unstable conditionals of $\m$ cannot be atomic (see the classical argument showing
 that an absolutely continuous invariant measure has the SRB property, as in \cite{ledrappier_sinai}). 
Thus positivity of the entropy follows from Corollary \ref{cor:rokhlin_zero_entropy}, and the 
proof of Theorem \ref{thm:hyperbolic} is complete. 
\qed 
\end{vlongue}

\begin{vlongue}
\subsubsection{Proof of Theorem \ref{thm:hyperbolic_kummer}}\
  The proof is similar to that of Theorem \ref{thm:hyperbolic} so we only sketch it. Assume by contradiction that $\mu$ is 
not hyperbolic;  since $X$ is a K3 surface, Corollary~\ref{cor:sum_exponentsK3} shows that the sum of the Lyapunov exponents of $\mu$ vanishes; thus, each of them is equal to $0$, and one of the alternatives of Theorem \ref{thm:classification_proj_invariant} holds, referred to as (a), (b), (c) on page \pageref{alternative_abc}.  
By assumption, $\Gamma_\nu$ contains a   map $f$ which is uniformly
hyperbolic in  some Zariski open set $U$, which is thus of full $\mu$-measure. 
 We denote by $x\mapsto E^u_f(x)\oplus E^s_f(x)$ 
the associated splitting of $TX\rest{U}$. 
Since $f$ is uniformly expanding/contracting on $E^{u/s}_f$, alternative (c) is not possible. 

If 
alternative (a) holds, then $E(x)$ being $f$-invariant on a set of full 
measure, it must coincide with $E^u_f$ or $E^s_f$, say with $E^u_f$.  By continuity any $g\in \Gamma_\nu$ preserves 
$E^u_f$ pointwise. 
On the other hand, $E^u_f$ is everywhere tangent 
to an $f$-invariant (singular) holomorphic foliation $\mathcal F^u$, induced by a linear foliation on the torus $A$ given by the Kummer structure.
Every leaf of that foliation, except for a finite number of them, 
is biholomorphically equivalent to $\C$, and the Ahlfors-Nevanlinna currents of these entire curves 
are all equal to the unique closed positive current $T^+_f$ that satisfies $\M(T^+_f)=1$ and $f^*T^+_f=\lambda(f)T^+_f$ with $\lambda(f)>1$.
Now, pick any element $g$ of $\Gamma_\nu$. Since $g$ preserves the line field $E^u_f$,  $g$ preserves $\mathcal F^u$ as well, hence also the ray $\R_+[T^+_f]$, 
contradicting the non-elementary assumption. 

Finally, if alternative (b) holds, any $g\in \Gamma_\nu$ preserves 
$ \{E^u_f(x), E^s_f(x)\}$ on a set of full measure so by the continuity of the hyperbolic splitting it must either preserve or 
swap these directions. Passing to an index $2$ subgroup  both directions are preserved, and we are back to case (a). 
\qed 
\end{vlongue}

\section{Measure rigidity}\label{sec:rigidity}

In view of the results of 
\begin{vcourte}
Section~\ref{sec:stiffness} and~\cite{invariant}, 
\end{vcourte}
\begin{vlongue} 
Sections~\ref{sec:stiffness} and~\ref{sec:parabolic},
\end{vlongue}
it is natural to ask for a classification 
of invariant measures when $\Gamma$ does not contain parabolic elements.  
\begin{vlongue} 
The results in this section  belong to a thread of measure  rigidity results starting with 
Rudolph's theorem \cite{rudolph}  on Furstenberg's $\times 2\times 3$ conjecture, here in a non-linear and non-commutative setting. 
\end{vlongue} 
If $\mu$ is a probability measure on $X$, 
we denote by $\Aut_\mu(X)$ the group of automorphisms of $X$ preserving $\mu$.

\begin{thm}\label{thm:rigidity}
 Let $f$ be an automorphism of a compact K\"ahler surface $X$, preserving a totally real and real 
 analytic surface $Y\subset X$. 
 Let $\mu$ be an ergodic  $f$-invariant measure on $Y$ with positive entropy. Then 
 \begin{enumerate}[\em (a)]
 \item either $\mu$ is absolutely continuous with respect to the Lebesgue measure on $Y$;
 \item or $\Aut_\mu(X)$ is virtually cyclic. 
 \end{enumerate}
 If in addition the Lyapunov exponents of $f$ with respect to  $\mu$ satisfy $\lambda^-(f,\mu)+ \lambda^+(f,\mu) \neq 0$, then 
 case~(a) does not occur, so $\Aut_\mu(X)$ is virtually cyclic. 
\end{thm}

This result, and its proof, may be viewed as a counterpart,  in our setting, to Theorems 5.1 and 5.3 of \cite{br}; again the possibility 
of invariant line fields is ruled out  by using the complex structure.  As before the typical case to keep in mind is 
when $X$ is a projective surface defined over $\R$ and $Y=X(\R)$. Observe that  by ergodicity, if $f$ preserves a smooth 
volume $\vol_Y$, then  in case (a) $\mu$ will be the restriction of $\vol_Y$ to an 
$\Aut_\mu(X)$-invariant Borel set of positive volume.  

\begin{proof}[Proof of Theorem \ref{thm:rigidity}]  Since it admits a measure  of  positive entropy, $f$ is a loxodromic transformation. 
By the Ruelle-Margulis inequality  $\mu$ is hyperbolic with respect to $f$ and it does not charge any point, 
nor any piecewise smooth curve: indeed, the entropy of a homeomorphism of the circle or the interval  
is equal to zero. 

We first assume that $X$ is projective; non-projective surfaces will be studied at the end of the proof. 
For $\mu$-almost every $x\in X$, the stable manifold  $W^s(f,x)$ is an entire curve in $X$ which is 
either transcendental or contained in a periodic rational curve (see \cite[Thm. 6.2]{Cantat:Milnor}). 
Since $f$ has only finitely many invariant algebraic curves 
(see \cite[Prop. 4.1]{Cantat:Milnor}) and $\mu$ gives no mass to curves,  $W^s(f,x)$ is $\mu$-almost surely transcendental; then, 
the only Ahlfors-Nevanlinna current associated to $W^s(f,x)$ is $T^+_f$; similarly, the Ahlfors-Nevanlinna currents of the unstable manifolds
give $T^-_f$. (This is the analogue in deterministic 
dynamics of Theorem \ref{thm:nevanlinna}.)
Fix   $g\in \Aut_\mu(X)$ and set $\Gamma :=\langle f,g \rangle$. Our first goal is to prove the following:

\smallskip 

\noindent{\bf{Alternative}}: {\emph{either $\Gamma^*$
is virtually cyclic and preserves $\{ \P [T^+_f], \P[T^-_f]\}\subset \fr \Hyp_X$; or $\mu$ is absolutely continuous with respect to the 
Lebesgue measure on $Y$.}} 

\smallskip 

Let $Y'\subset Y$ be the union of the connected components of $Y$ of positive $\mu$-measure. 
The measure  $\mu$ does not charge any analytic subset of $Y$ of dimension $\leq 1$; thus, 
by analytic continuation, any $h\in \Gamma$ preserves $Y'$. So, without loss of generality we  can replace $Y$ by $Y'$.

We divide the argument into several cases according to the existence or non-existence of certain $\Gamma$-invariant line fields. 
In the first two cases we will conclude  that $\Gamma$ is elementary. 
In the third case,  $\mu$ will be absolutely continuous with respect to the Lebesgue measure on $Y$; then by the Pesin formula 
its Lyapunov exponents satisfy  
$\lambda^+(f, \mu)  = -\lambda^-(f, \mu)= h_\mu(f)$ so   when 
$\lambda^+ (f, \mu)+ \lambda^-(f, \mu)\neq 0$, Case~3  is actually impossible. 
 
 \smallskip

{\bf Case 1.-- } {\emph{There exists a $\Gamma$-invariant measurable line field.}} Specifically, we mean a measurable field of complex lines
$x\mapsto E(x) \in \P(T_xX)$, defined on a set of full $\mu$-measure,  such that $ D_xh(E(x)) =   E(h(x))$ for every 
$h\in  \Gamma$ and almost every $x\in X$; since $\mu$ is supported on the totally real surface $Y$, 
the field of real lines $E(x)\cap T_xY\subset T_xY$ is also invariant, and determines $E(x)$. Now, $\mu$ being ergodic and hyperbolic for $f$,  the Oseledets theorem shows that either $E(x) = E^s_f(x)$ $\mu$-almost everywhere or $E(x) = E^u_f(x)$ $\mu$-almost everywhere. Changing $f$ into $f^{-1}$ if necessary, we may assume that  $E(x) = E^s_f(x)$.


Consider the automorphism $h = g\inv f g \in \Aut_\mu(X)$. 
Since $h$ is conjugate to $f$,  $\mu$ is also ergodic and hyperbolic for $h$. Thus, either $E^s_h(x) = E^s_f(x)$
for $\mu$-almost every $x$ or $E^u_h(x) = E^s_f(x)$ for $\mu$-almost every $x$. 

\begin{lem}\label{lem:equal_stable_manifolds}
If there is a measurable set $A$ of positive measure  along which 
$E^s_h(x) = E^s_f(x)$ (resp. $E^u_h(x) = E^s_f(x)$), then $W^s(f,x) = W^s(h, x)$ for almost every $x$ in $A$ (resp. $W^u(h,x) = W^s(f, x)$).
\end{lem}

Let us postpone the proof of this lemma and conclude the argument. Suppose first  that 
$E^s_h(x) = E^s_f(x)$ on a subset $A$ with $\mu(A)>0$. Then $T^+_f = T^+_h$ because for $\mu$-almost every $x$, the unique 
Ahlfors-Nevanlinna current associated to the (complex) stable manifold $W^s(f,x)$  (resp. $W^s(h, x)$)  is $T^+_f$ (resp. $T_h^+$). Since $T^+_{h} = \M(g^*T^+_f)\inv g^* T^+_f$,
we see that $g$, and therefore $\Gamma$ itself, preserve the line   $\R [T^+_f]\subset H^{1,1}(X)$.
Since $[T^+_f]^2=0$,  $\Gamma$ fixes a point  $\P[T^+_f]$ of the boundary $\fr\Hyp_X$, so it is  
elementary. Since in addition $\Gamma$ contains a loxodromic element,  Theorem~3.2 of \cite{Cantat:Milnor}  shows that 
$\Gamma^*$ is virtually cyclic. 

Now, suppose that $E^u_h(x) = E^s_f(x)$ on $A$. Then, $T^-_h=T^+_f$ and the group 
generated by $f$ and $h$ is elementary. Since it contains a loxodromic element  
  \cite[Thm 3.2]{Cantat:Milnor} says 
that $\langle f^*, h^*\rangle$ is virtually cyclic and fixes also $\P[T^-_f] \in\fr \Hyp_X$.  This implies that $g$, hence
$\Gamma$, preserves   the pair of boundary points $\{ \P [T^+_f], \P[T^-_f]\}\subset \fr \Hyp_X$.
Thus, in both cases $\Gamma^*$ is virtually cyclic and preserves $\{ \P [T^+_f], \P[T^-_f]\}\subset \fr \Hyp_X$.

\begin{proof}[Proof of Lemma \ref{lem:equal_stable_manifolds}] The argument is similar  to that of Theorem \ref{thm:alternative_stable}, 
in a simplified setting, so we only sketch it. 
For $\mu$-almost every $x$, $W^s(f,x)$ and  $W^s(h, x)$ are tangent at $x$. 
Assume by contradiction that there exists a measurable subset 
$A'$ of $A$ of positive measure such that $W^s(f,x) \neq W^s(h, x)$ for 
every $x\in A'$. 
Then for   small  $\e>0$ there exists two positive constants  $r=r(\e)$ and  $c=c(\e)$, an integer $k\geq 2$, 
 and  a measurable subset $\mathcal G_\e\subset A'$  such that $\mu(\mathcal G_\e)>0$ and 
 \begin{itemize} 
\item[-]  $W_\loc^s(f,x)$ and  $W_\loc^s(h, x)$ are well defined and of size $r$ for every $x\in \mathcal G_\e$,
\item[-]  $W_\loc^s(f,x)$ and  $W_\loc^s(h, x)$ depend continuously on $x$ on $\mathcal G_\e\subset X$,
\item[-]  $\inter_x(W^s_\loc(f,x), W^u_\loc(f,x)) = k $ for every $x\in \mathcal G_\e$,
\item[-]  and $\osc_{(k, x, r)} (W^s_{r}(f,x), W^s_{r}(h,x))\geq c$ for every $x\in \mathcal G_\e$.
\end{itemize}
Indeed, to get the first and second properties, one intersects $A'$ with a large Pesin set $\mathcal R_\e$. 
On $A'\cap \mathcal R_\e$ the  multiplicity of intersection $x\mapsto \inter_x(W^s_\loc(f,x), W^u_\loc(f,x))$
is semi-continuous, so we can find   $k\geq 2$ and a subset $\mathcal R'_\e \subset (A'\cap \mathcal R_\e)$ of positive measure 
such that 
\begin{vcourte}
\begin{equation}
\inter_x(W^s_\loc(f,x), W^u_\loc(f,x)) = k
\end{equation}
\end{vcourte}
\begin{vlongue}
\begin{equation}
\inter_x(W^s_\loc(f,x), W^u_\loc(f,x)) = k
\end{equation}
\end{vlongue}
for every $x\in \mathcal R'_\e$. Thus, the $k$-th osculation number
is well defined, and the last property holds on a subset $\mathcal G_\e\subset \mathcal R'_\e$ of positive measure if $c$ is small.

Let 
$\eta^s$ be a Pesin partition subordinate to the local stable manifolds of $f$.  
Since $h_\mu(f)>0$ the    conditional measures   $\mu(\cdot\vert \eta^s)$   are non-atomic. 
Thus there exists $x\in {\mathcal G}_\e$ such that  
$x$ is an accumulation point of $\supp\big(\mu(\cdot \vert  \eta^s(x))\rest{{\mathcal G}_\e\cap  \eta^s(x)}\big)$. Fix a neighborhood $N$ of $x$ such that 
$W^s_r(f,x)\cap W^s_r(h,x)\cap N = \set{x}$, and then pick a sequence 
$(x_j)$ of points in ${\mathcal G}_\e\cap  \eta^s(x)\cap N$ converging to $x$. 
The local stable manifolds $W^s_r(h,x_j)$ form a sequence of disks of size $r$ at $x_j$, each of them tangent to 
$W^s_r(f,x)$ (at $x_j$), and all of them disjoint from $W^s_r(h,x)$ (because  $x_j$ does not belong to $W^s_r(h,x)$).
This contradicts Corollary \ref{cor:bls_quantitative}, and the proof is complete. 
 \end{proof}
 

  \medskip 
  
{\bf Case 2.-- } {\emph{There is a pair of distinct measurable line fields $\set{E_1(x), E_2(x)}$   invariant under $\Gamma$}}. Again 
by the Oseledets theorem applied to $f$, necessarily $\{E_1(x), E_2(x)\} =  \{E^s_f(x), E^u_f(x)\}$. 
For $\mu$-almost every $x$, $g(  \{E^s_f(x), E^u_f(x)\})  = \{E^s_f(g(x)), E^u_f(g(x))\}$. 
As before,  consider   $h = g\inv f g \in \Aut_\mu(X)$.  Since $h$ is conjugate to $f$, it is hyperbolic and ergodic with 
respect to $\mu$, and   $\{E^s_f(x), E^u_f(x)\} = \{E^s_h(x), E^u_h(x)\}$ for almost every $x$. 
Replacing $h$ by $h\inv$ if necessary, there exists a set $A$ of positive measure for which $E^s_h(x) = E^s_f(x)$, 
and we conclude as in Case 1. 
   
\medskip 

{\bf Case 3.-- } {\emph{There is no $\Gamma$-invariant line field or pair of line fields.}}
In other words, Cases~1 or~2 are now excluded. This part of the argument is identical to the proof of \cite[Thm 5.1.a]{br}.

First, we claim that there exists $g_1, g_2 \in \Gamma$ and a subset $A$ of positive measure such that 
$D_xg_1(E^s_f(x))  \notin \{E^s_f(g_1(x)), E^u_f(g_1(x))\}$ and $D_xg_2(E^u_f(x))\notin  \{E^s_f(g_2(x)), E^u_f(g_2(x)\}$ for every $x$ in $A$. 
Indeed since we are not in Case~2 (possibly switching $E^u_f$ and $E^s_f$) there exists $g_1 \in \Gamma $ 
and a set $A$ of positive measure  such that for $x\in A$, $D_xg_1(E^s_f(x))  \nsubset E^s_f(g_1(x))\cup E^u_f(g_1(x))$. 
Since we are not in Case~1, there exists  $g\in \Gamma $ and a set $B$ of positive measure such that for $x\in B$, 
$D_xg(E^u_f(x))\neq E^u_f(g(x))$. If $D_xg(E^s_f(x))\in  \{E^s_f(g(x)), E^u_f(g(x))\}$ on a subset $B'$ of $B$ of positive measure, 
then choose $k>0$ and $\ell>0$ such that $\mu(f^\ell(A)\cap B')>0$ and $\mu(f^k(g(f^\ell(A)))\cap A)>0$  and define $g_2 = g_1f^k gf^\ell$; otherwise,
set $g_2= gf^\ell$ with $\ell$ such that $\mu(f^\ell(A)\cap B)>0$. Then change $A$ into $A=A\cap f^{-\ell}(B')$ (resp. $A\cap f^{-\ell}(B)$).

Denote by $\Delta$ the simplex $\set{(a, b,c, d)\in (\R_+^*)^4\; ; \;   a+b+c+d =1}$. 
For $\alpha=(a, b,c, d)$ in $\Delta$, let $\nu_\alpha$ be the probability measure $\nu_\alpha = a \delta_f + b\delta_{f\inv} +  c \delta_g + d \delta_{g\inv}$.
Then $\mu$ is   $\nu_\alpha$-stationary and since $\mu$ is $f$-ergodic and $\nu_\alpha(\set{f})>0$, it is also ergodic as a 
$\nu_\alpha$-stationary measure (see~\cite[\S 2.1.3]{benoist-quint_book}). 
\begin{vcourte}
Since we are not in Cases 1 or 2 and $\mu$ is hyperbolic for $f$, the invariance principle of 
Ledrappier~\cite{ledrappier_stationary} implies that the Lyapunov exponents of $\mu$, viewed 
as a $\nu_\alpha$-stationary measure, satisfy 
$\lambda_\alpha^-(\mu)< \lambda_\alpha^+(\mu)$. (see Section 13.2.2 of~\cite{br} and also~\cite{barrientos-malicet}; more precise statements and proofs can be found in~\cite[Thms. 11.10 and 11.11]{cantat-dujardin:vlongue}, and also~\cite{hyperbolic})
\end{vcourte}
\begin{vlongue}
Since we are not in Cases 1 or 2 and $\mu$ is hyperbolic for $f$, Theorems \ref{thm:ledrappier_invariance_principle} and   \ref{thm:classification_proj_invariant}
imply  that   the Lyapunov exponents of $\mu$, viewed as a $\nu_\alpha$-stationary measure, satisfy 
$\lambda_\alpha^-(\mu)< \lambda_\alpha^+(\mu)$. 
\end{vlongue}

\begin{lem}
There exists a choice of $\alpha\in \Delta$ such that $\mu$ is a hyperbolic $\nu_\alpha$-stationary measure, 
i.e. $\lambda_\alpha^-(\mu)< 0 < \lambda_\alpha^+(\mu)$
\end{lem}

\begin{proof} This is automatic when $f$ and $g$ are volume preserving because $\lambda_\alpha^-(\mu)=- \lambda_\alpha^+(\mu)$
in that case. For completeness, let us copy the proof given in \cite[\S 13.2.4]{br}.
The assumptions of Case~3 and the strict inequality $\lambda^-(\mu)< \lambda^+(\mu)$ imply that 
\begin{equation}
\alpha \in \Delta \mapsto ( \lambda^-_\alpha(\mu), \lambda^+_\alpha(\mu))\in \R^2
\end{equation} 
is  continuous (see \cite[Prop. 13.7]{br} or \cite[Chap. 9]{viana}). 
Since $\lambda^-_\alpha(\mu)< \lambda^+_\alpha (\mu)$ for every $\alpha\in \Delta$, 
one of $\lambda^-_\alpha$ and $\lambda^+_\alpha$ is non zero. Furthermore,  $\mu$ being invariant, the involution 
$(a,b,c,d) \mapsto (b,a,d,c)$ interchanges the Lyapunov exponents. It follows that 
$P = \set{\alpha\in \Delta, \lambda^+_\alpha>0}$ and $N = \set{\alpha\in \Delta, \lambda^-_\alpha<0}$ are non-empty 
open subsets of $\Delta$ such that $P\cup N = \Delta$. The connectedness of $\Delta$ implies $P\cap N\neq \emptyset$, as was to be shown. \end{proof} 

Fix $\alpha\in \Delta$ such that $\mu$ is hyperbolic as a $\nu_\alpha$-stationary measure. The assumptions of Case~3 imply that  the 
stable directions depend  on the itinerary so the main result of \cite{br}
shows that $\mu$ is fiberwise SRB (on the surface $Y$), that is, the unstable conditionals of the measures $\mu_\cx$ (here $\mu_\cx = \mu$)
are given by the Lebesgue measure (in some natural affine parametrizations of the unstable manifolds by the real line $\R$). 
Since $\mu$ is invariant, we can revert the stable and unstable directions by applying the argument to $F\inv$, 
and we conclude that the stable conditionals are given by the Lebesgue measure as well. The absolute continuity property of the stable and unstable laminations then implies that $\mu$ is absolutely continuous with respect to the Lebesgue measure on $Y$. 

\smallskip

\noindent{\bf{Conclusion.--}} Assume  that $\mu$ is not absolutely continuous with respect to the Lebesgue 
measure on $Y$. The above alternative holds for all subgroups $\Gamma= \langle f, g \rangle$, 
with $g\in \Aut_\mu(X)$ arbitrary. Therefore, 
if $X$ is projective, then 
  $\Aut_\mu(X)^*$ preserves $\{ \P [T^+_f], \P[T^-_f]\}\subset \fr \Hyp_X$,
which implies that $\Aut_\mu(X)^*$ is virtually cyclic.  By  
Lemma~\ref{lem:virtually_cyclic_non_projective}, $\Aut_\mu(X)^*$ is
also virtually cyclic  when $X$ is not projective. It remains to prove that $\Aut_\mu(X)$ itself  is virtually cyclic.
If not, then $\Aut(X)^\circ$ is infinite, $X$ is a torus $\C^2/\Lambda$
(see Proposition~\ref{prop:discrete}), and $\Aut_\mu(X)\cap \Aut(X)^\circ$ is a normal subgroup of $\Aut_\mu(X)$ 
containing infinitely many  translations. 
This group  is a closed subgroup of the compact Lie group $\Aut(X)^\circ=\C^2/\Lambda$; 
thus, its  connected component  of the identity 
is a (real) torus $H\subset \C^2/\Lambda$ of positive dimension. This torus $H$ is invariant under the action of $f$ by conjugacy. Since
$X=\C^2/\Lambda$, $f$ is a complex linear Anosov diffeomorphism of $X$, and it follows  that $\dim_\R(H)\geq 2$. Being 
$H$-invariant, $\mu$ is then absolutely continuous with respect to the Lebesgue measure of $Y$; 
this contradiction completes the proof. 
\end{proof} 

\begin{vlongue}
\begin{rem}
Theorem \ref{thm:rigidity} can be extended to the case of singular analytic subsets  $Y$, after minor adjustments of the proof, because $\mu$ cannot charge its singular locus.
\end{rem}
\end{vlongue}
It is natural to expect that the positive entropy assumption in Theorem~\ref{thm:rigidity} could be replaced by 
a much weaker assumption, namely,  ``{\emph{$\mu$ gives no mass to  proper Zariski closed subsets}}''. 
\begin{vlongue}
In  full generality    this seems to exceed  the scope of techniques of this paper, however we are able to deal with a special case. 
\end{vlongue}
\begin{vcourte}
We are able to deal with the following special case, which shows that the stiffness Theorem \ref{thm:stiffness_real}   takes 
 a particularly strong form when $\supp(\nu)$ contains a Kummer example
\end{vcourte}

\begin{thm}\label{thm:rigidity_kummer}
Let  $f$ be a Kummer example  
on a compact K\"ahler surface $X$. Let $\mu$ be an atomless,  $f$-invariant, and ergodic 
probability measure that 
is supported on a totally real, real 
analytic   surface $Y\subset X$. 
If $g\in \Aut(X)$ preserves $\mu$, then:
\begin{enumerate}[\em (a)]
\item either $\mu$ is absolutely continuous with respect to $\vol_Y$; 
\item or $\langle f, g \rangle$ is virtually isomorphic to $\Z$. 
\end{enumerate}
\end{thm}
 
\begin{vlongue}
Thus, as in the case of subgroups containing parabolic transformations, the stiffness Theorem \ref{thm:stiffness_real}   takes 
a particularly strong form when $\supp(\nu)$ contains a Kummer example.
\end{vlongue}
 
\begin{proof}
Let us start with a preliminary remark. Assume that $\mu(C)>0$ for some irreducible curve $C\subset X$; since $\mu$ does not 
charge any point the support of $\mu_{\vert C}$ is Zariski dense in $C$, and $C$ is an $f$-periodic curve. But $f$ being a Kummer 
example, such a curve is a rational curve $C\simeq \P^1(\C)$ (obtained by blowing-up a periodic point of a linear Anosov map on a torus), 
on which $f$ has a north-south dynamics; thus, all $f$-invariant 
probability measures on $C$ are atomic, and we get a contradiction. This means that the assumption ``{\emph{$\mu$ has no atom}}'' 
is equivalent to the assumption ``{\emph{$\mu$  gives no mass to  proper Zariski closed subsets of $X$}}''.
Now, we follow step by step the proof of Theorem \ref{thm:rigidity}, only insisting on the required modifications. 
Since $\mu$ does not charge any curve, we can contract all $f$-periodic curves, and lift $(f,\mu)$ to $({\tilde{f}}, {\tilde{\mu}})$, where ${\tilde{f}}$ is a 
linear Anosov diffeomorphism  of some compact torus $\C^2/\Lambda$ and ${\tilde{\mu}}$ is an ${\tilde{f}}$-invariant probability measure
 (see~\cite{Cantat-Zeghib} for details on Kummer examples). We deduce that $\tilde{\mu}$ is hyperbolic 
for ${\tilde{f}}$ and then, coming back to $X$, that $\mu$ is hyperbolic for $f$. Case~3 of the proof of Theorem \ref{thm:rigidity} only requires hyperbolicity of $\mu$ so it carries over without modification. In Cases~1 and 2 we have to show that if 
 $\Gamma = \langle f,g \rangle$ preserves a measurable line field or a pair of measurable line fields then $\Gamma^*$ is elementary. 
 In either case we consider $ h = g f g\inv$ and up to  possibly replacing $E^u_f$ by $E^s_f$ and $h$ by $h\inv$,  we have  
 $E^s_f(x) = E^s_h(x)$  on a set of positive  measure. But now $f$ and $h$ are Kummer examples so their respective stable  
 foliations $\mathcal F^{s} _f$ and $\mathcal F^{s} _h $ are (singular) holomorphic foliations. From the previous reasoning 
 $\mathcal F^{s} _f$ and $\mathcal F^{s} _h $ are tangent on a set of positive $\mu$-measure; thus, $\mathcal F^{s} _f = \mathcal F^{s} _h$ because the support of  $\mu$ is Zariski dense.   
\begin{vlongue} 
Then, we conclude exactly as in Theorem \ref{thm:hyperbolic_kummer}. 
\end{vlongue}
\begin{vcourte}
Moreover, every leaf of this foliation, except a finite number of algebraic leaves, is parametrized by $\C$ and the Ahlfors-Nevanlinna currents
 of these entire curves are all equal to the unique closed positive current $T^+_f$ that satisfies $\M(T^+_f)=1$ and $f^*T^+_f=\lambda(f)T^+_f$.
This implies that $g^*$ preserves $\R_+[T^+_f]$ or permute it with $\R_+[T^-_f]$. Thus, a  subgroup $\Gamma_0\subset \Gamma$ of index $\leq 2$
preserves $\R_+[T^+_f]$ and by~\cite[Thm. 3.2]{Cantat:Milnor}, $\Gamma$ is virtually cyclic.
\end{vcourte} 
\end{proof}
 
We expect that  most results in this paper can be extended
 to polynomial automorphisms of $\R^2$. 
This is indeed the case for Theorem \ref{thm:rigidity}, 
with essentially the same proof.

\begin{thm}\label{thm:rigidity_henon}
Let $f$ be a   polynomial  automorphism of $\R^2$. 
Let $\mu$ be an ergodic  $f$-invariant measure   with positive entropy supported on $\R^2$. 
If $g\in\Aut(\R^2)$  satisfies $g_*\mu   = \mu$,  then: 
\begin{enumerate}[\em (a)]
\item either $f$ and $g$ are conservative and $\mu$ is the restriction of  $\leb_{\R^2}$ to a Borel set of positive measure invariant under $f$ and $g$;
\item or  the group generated by $f$ and $g$ is solvable and virtually cyclic; in particular, there exists $(n,m)\in \Z^2\setminus{\set{(0,0)}}$ such that $f^n = g^m$. 
\end{enumerate}
\end{thm}

\begin{rem}
With the techniques developed in \cite{Cantat:BHPS}, the same result applies to the dynamics of 
${\mathsf{Out}}({\mathbb{F}}_2)$ acting on the real 
part of the character surfaces of the once punctured torus. 
\end{rem}

\begin{proof}
We briefly explain the modifications required to adapt the proof of Theorem~\ref{thm:rigidity}, and leave  the details  to the reader.  
We freely use standard facts from the dynamics of automorphisms of $\C^2$. 
Let    $f$ and $g$ be as in the statement of the theorem, and set $\Gamma=\langle f, g \rangle$. 
Since its entropy  is positive, $f$ is of H\'enon type in the sense of~\cite{lamy}: this means that $f$ is conjugate to a composition of 
generalized H\'enon maps, as in~\cite[Thm. 2.6]{friedland-milnor}. Thus, the support of $\mu$  is a compact subset of $\C^2$, because the 
basins of attraction of the line at infinity for $f$ and $f^{-1}$ cover the complement of a compact set; 
moreover, as in Theorem~\ref{thm:rigidity}, $\mu$ cannot charge any proper Zariski closed subset.

Let $\gamma$ be an arbitrary element of $\Gamma$; 
then  $h:=\gamma\inv f  \gamma$ is also  of H\'enon type. 
We run through Cases 1, 2 and 3 as in the proof of Theorem \ref{thm:rigidity}. Case 3 is treated exactly  in the same way as above and implies that $\mu$ is absolutely continuous. This in turn implies that the Jacobian of $f$, a constant $\Jac(f)\in \C^*$ since $f\in\Aut(\C^2)$, is equal to $\pm1$; and since $\mu$ is ergodic for $f$, 
it must be the restriction of $\leb_{\R^2}$ to some $\Gamma$-invariant subset. In Cases 1 and 2, arguing as before and keeping the 
same notation, we arrive at $W^s(h,x) = W^s(f,x)$ or $W^u(f,x)$ on a set of positive measure. 
For a H\'enon type automorphism of $\C^2$, the closure of any stable manifold is 
equal to the forward Julia set  $J^+$, and $J^+$ carries a unique positive closed current $T^+$ of mass 
$1$ relative to the Fubini Study form in $\P^2(\C)$ (see \cite{sibony}). So we 
infer that $T^+_h = T^+_f$ or $T^+_h =T^-_f$; as a consequence, the Green functions of $f$ and $h$ satisfy $G^+_h=G^+_f$
or $G^+_h=G^-_f$, respectively.

Automorphisms of $\C^2$ act on the Bass-Serre tree of $\Aut(\C^2)$, each $u\in \Aut(\C^2)$
giving rise to an isometry $u_*$ of the tree. If $u$ is of H\'enon type,  then $u_*$ is loxodromic;
its axis $\Ax(u_*)$ is the unique $u_*$-invariant geodesic, and $u_*$ acts as a translation along it. From~\cite[Thm.~5.4]{lamy}, $G^+_h = G^+_f$ implies  $\Ax(h_*)=\Ax(f_*)$; changing $f$ into
$f\inv$, $G^+_h =G^-_f$ gives $\Ax(h_*)=\Ax(f_*\inv)=\Ax(f_*)$ because $\Ax(f_*\inv)=\Ax(f_*)$. 
Since $\gamma_* \Ax(f_*)=\Ax(h_*)$, we see that $\Gamma$ preserves $\Ax(f_*)$; so, all
$u\in \Gamma$ of H\'enon type satisfy $\Ax(u_*)=\Ax(f_*)$. From \cite[Prop. 4.10]{lamy}, we conclude that $\Gamma$ is solvable and virtually 
cyclic. \end{proof}

\begin{vlongue}
\newpage
\end{vlongue}
\appendix

\section{General compact complex surfaces}\label{par:appendix_non_kahler}

Here, we  study the concept of non-elementary groups of automorphisms on (non Kähler) compact complex surfaces. We show that the two possible definitions of non-elementary group are equivalent and force the surface to be Kähler. 

Let $M$ be a compact manifold. We say that a  group $\Gamma$ of 
homeomorphisms of $M$ is {\bf{cohomologically non-elementary}} if its 
image $\Gamma^*$ in $\GL(H^*(M;\Z))$ contains a non-Abelian free subgroup, and that 
$\Gamma$ is {\bf{dynamically non-elementary}} if it contains a non-Abelian free group $\Gamma_0$ 
such that the topological entropy of every $f\in \Gamma_0\setminus \{ \id\}$ is positive. 
When $M$ is a compact Kähler surface and $\Gamma\subset \Aut(M)$,  
Theorem~3.2 of \cite{Cantat:Milnor} and the fact that parabolic automorphisms have zero entropy 
imply  that $\Gamma$ is non-elementary (in the sense of Section~\ref{par:hyp_X}) if and only if
it is cohomologically non-elementary, if and only if it is dynamically non-elementary.

\begin{lem}\label{lem:non-elementary_free_groups}
Let $M$ be a compact manifold, and   $\Gamma$ be a subgroup of $\Diff^\infty(M)$. If $\Gamma$ is cohomologically non-elementary, then $\Gamma$ is dynamically non-elementary. 
\end{lem}

\begin{proof} We split the proof in two steps, the first one concerning groups of matrices, and 
the second one concerning topological entropy.

\smallskip

{\bf{Step 1.-}}  {\emph{$\Gamma^*$ contains a free subgroup $\Gamma_1^*$, all of whose non-trivial elements  
have spectral radius larger than $1$}}.

\smallskip

The proof uses basic ideas involved in Tits's alternative, here in the simple case of subgroups of $\GL_n(\Z)$.
Let $N$ be the rank of $H^*_{t.\!f\!.}(M;\Z)$, where $t.\!f\!.$ 
stands for ``torsion free''. Fix a basis of this free $\Z$-module. 
Then $\Gamma^*$ determines a subgroup of $\GL_N(\Z)$. Our assumption implies that the 
derived subgroup of $\Gamma^*$ contains a non-Abelian free group $\Gamma_0^*$ of rank~$2$.  

If all (complex) eigenvalues 
of all elements of $\Gamma_0^*$ have modulus $\leq 1$, then by Kronecker's lemma all of them are roots 
of unity. This implies that $\Gamma_0^*$ contains a finite index nilpotent subgroup (see 
Proposition 2.2 and Corollary 2.4 of~\cite{Benoist:Grp_Disc}), contradicting the existence of a non-Abelian free subgroup. 
Thus, there is an element $f^*$ in $\Gamma_0^*$ with a complex eigenvalue of modulus  $\alpha>1$. 
Let $m$ be the number
of eigenvalues of $f^*$ of modulus $\alpha$, counted with multiplicities. Consider the linear representation of $\Gamma_0^*$ 
on $\bigwedge^m H^*(M;\C)$; the action of $f^*$ on this space has a unique dominant eigenvalue, of 
modulus $\alpha^m$; the corresponding eigenline determines an attracting fixed point for $f^*$ in the
projective space $\P(\bigwedge^m H^*(M;\C))$; the action of $f^*$ on this topological space is proximal. 

Let 
\begin{equation}
\{0\}=W_0 \subset W_1\subset \cdots \subset W_k \subset W_{k+1}=\bigwedge^m H^*(M;\C)
\end{equation}
be a Jordan-H\"older sequence for the representation of $\Gamma^*$: the subspaces $W_i$ are 
invariant, and the induced representation of $\Gamma^*$ on $W_{i+1}/W_i$ is irreducible for all $0\leq i\leq k$. 
Let $V$ be the quotient space $W_{i+1}/W_{i}$ in which the eigenvalue of $f^*$ of modulus $\alpha^m$
appears. Since $\Gamma_0^*$ is contained in the derived subgroup of $\Gamma$, the linear 
transformation of $V$ induced by $f^*$ has determinant $1$; thus, $\dim(V)\geq 2$. Now, 
we can apply Lemma~3.9 of~\cite{Benoist:Grp_Disc} to (a finite index, Zariski connected subgroup of) $\Gamma^*_{0}\rest{V}$:
changing $f$ is necessary, both $f^*\rest{V}$ and $(f^{-1})^*\rest{ V}$ are proximal, and there is 
an element $g^*$ in $\Gamma^*$ that maps the attracting fixed points $a^+_f$ and $a^-_f\in \P(V)$ of $f^*\rest{ V}$ and $(f^*\rest{V})\inv$ to two distinct
points (i.e. $\{ a^+_f, a^-_f\}\cap \{ g(a^+_f), g(a^-_f)\} =\emptyset$) ; then, by the ping-pong lemma, large powers of $f^*$ and $g^*\circ f^*\circ (g^*)^{-1}$ 
generate a non-Abelian free group $\Gamma_1^*\subset \Gamma^*$ such that each element 
$h^*\in \Gamma_1^*\setminus \{\id\}$ has an attracting fixed point in $\P(V)$. This implies that every element of $\Gamma_1^*\setminus \{\id\}$
has an eigenvalue of modulus $>1$ in $H^*(M;\C)$. 

\smallskip

{\bf{Step 2.-}}  Since $\Gamma_1^*$ is free, there is a free subgroup  $\Gamma_1\subset \Gamma$ such that the homomorphism
$\Gamma_1\mapsto \Gamma_1^*$ is an isomorphism. By Yomdin's theorem~\cite{yomdin}, all elements of $\Gamma_1\setminus \{\id\}$ have positive entropy, and we are done. 
\end{proof}

\begin{thm}
Let $M$ be  a compact complex surface, and  $\Gamma$ be a subgroup of $\Aut(M)$.
Then, $\Gamma$ is cohomologically non-elementary if and only if it is dynamically non-elementary. 
If such a subgroup exists, then  $M$ is a projective surface. 
 
\end{thm}

\begin{proof}
Indeed it was shown in~\cite{Cantat:CRAS} that every compact complex surface possessing
an automorphism of positive entropy is K\"ahler. Thus, the first assertion follows from Lemma~\ref{lem:non-elementary_free_groups} and  Theorem~3.2 of \cite{Cantat:Milnor}, and the second one  follows from  Theorem~\ref{thm:X-is-projective2}.
\end{proof}

\section{Strong laminarity of  Ahlfors currents}\label{par:appendix_ahlfors}

In this appendix, we sketch the proof of Lemma~\ref{lem:ahlfors_current}, explaining how to adapt arguments of 
\cite{bls,Dujardin:Laminar2003,isect}, written for $X=\P^2(\C)$, to our context.  
 
\begin{proof}[Proof of Lemma \ref{lem:ahlfors_current}] 
 Let $(\Delta_n)$ be a sequence of unions of disks, 
as in the definition of injective Ahlfors currents, such that 
$\unsur{\M(\Delta_n)}\set{\Delta_n}$ converges to $T$.
 Since $X$ is projective we can choose a finite family of meromorphic fibrations $\varpi_i: X\dasharrow \pu$ such that 
\begin{itemize}
\item the general fibers of $\varpi_i$ are smooth curves of genus $\geq 2$;
\item for every $x\in X$, there are at least two of the fibrations $\varpi_{i}$, denoted for simplicity 
by  $\varpi_{1}$ and $\varpi_{2}$,
which are well defined in some neighborhood $U_x$ of  $x$ ($x$ is not 
a base point of the corresponding pencils), satisfy $(d\varpi_{1}\wedge d\varpi_{2})(x)\neq 0$ (the fibrations are transverse), 
and for which the fibers $\varpi_{k}^{-1}(\varpi_{k}(x))$ containing $x$ are smooth.
\end{itemize}
 If we blow-up the base points of $\varpi_{k}$, $k=1,2$, we obtain a new surface $X'\to X$ on which each $\varpi_{k}$ lifts to a 
regular fibration $\varpi_{k}'$; the open neighborhood $U_x$ is isomorphic to its preimage in $X'$ so, when 
working on $U_x$, we can do as if the two fibrations $\varpi_{k}$ were local submersions with smooth fibers
of genus $\geq 2$.

To construct $T_r$, we follow the proof of \cite[Proposition 4.4]{isect} (see also \cite[Proposition 3.4]{Dujardin:Laminar2003}). The construction
will work as follows: we fix a sequence $(r_j)$ converging to zero, and for 
every $j$ we extract from $\unsur{\M(\Delta_n)}\set{\Delta_n}$ a current 
$T_{n,r_j}$ made of disks of size $\approx r_j$ which are obtained from $\Delta_n$ by only keeping graphs of size $r_j$ 
over one of the projections 
$\varpi_i$.  

By a covering argument, it is enough to
 work locally near a point $x$, with two projections $\varpi_{1}$ and $\varpi_{2}$ as above.
Let $S\subset \C$ be the unit square $\{x+{\mathsf{i}} y\; ; \; 0\leq x \leq 1, \; 0\leq y\leq 1\}\simeq [0,1]^2$. To simplify the exposition, we may assume that 
\begin{equation}
\quad \varpi_{k}(U_x)= S\subset \C\subset \P^1(\C)\quad ({\text{for }} \; k=1,2).
\end{equation}
Set $r_j=2^{-j}$ and consider the subdivision $\qq_j$ of $S\simeq [0,1]^2$ into $4^j$ squares $Q$ of size $r_j$.
A connected component of $\Delta_n\cap \varpi_{k}^{-1}(Q)$, for such a small square $Q$, is  called a graph
(with respect to $\varpi_{k}$) if it lifts to a local section of the fibration $\varpi_{k}'\colon X'\to \P^1(\C)$ above~$Q$.
Then, we fix $j$,  intersect $\Delta_n$ with  $\varpi_{k}^{-1}(Q)$, and keep only the components 
of $\varpi_{k}\inv(Q\cap \Delta_n)$, $Q\in \qq_j$ which are graphs with respect to $\varpi_{k}$. 
Such a family of graphs is normal because the fibers of $\varpi'_{k}$ have genus $\geq 2$ (compare to Lemma~3.5 of \cite{Dujardin:Laminar2003}). 

This being done, we can copy the proof of \cite[Proposition 4.4]{isect}.
Letting $n$ go to $+\infty$ and extracting a converging subsequence, we obtain a uniformly laminar current 
$T_{\qq_j, k}\leq T$.   Away from the base points of $\varpi_{k}$, $T_{\qq_j, k}$ is made  
of disks of size $\asymp r_j$ which are limits of disks 
contained in  the $\Delta_n$.
Combining the two currents $T_{\qq_j, k}$,   we get a  current $T_{r_j}\leq T$  
which is uniformly laminar in every cube $\varpi_{1}\inv(Q)\cap \varpi_{2}\inv(Q')$, $Q,Q'\in\qq_j$, and 
such that 
\begin{equation}\label{eq:strong_laminar}
\langle T-T_{r_j}, \varpi_{1}^* {\kappa_{\P^1}}+ \varpi_{1}^* {\kappa_{\P^1}}\rangle
\leq \langle T-T_{\qq_j, 1}, \varpi_{1}^* {\kappa_{\P^1}}\rangle + \langle T-T_{\qq_j, 2}, \varpi_{2}^* {\kappa_{\P^1}}\rangle, 
\end{equation}
where ${\kappa_{\P^1}}$ is the Fubini-Study form.
By definition, $T$ will be  strongly approximable if locally $\M(T - T_{r_j})\leq O(r_j^2)$. 
Using the fact that $\varpi_{1}^* {\kappa_{\P^1}}+ \varpi_{1}^* {\kappa_{\P^1}}\geq C \kappa_0$ and the Inequality~\eqref{eq:strong_laminar},
it will be enough to show that  
$\langle T  - T_{\qq_j, k}, \varpi_{k}^*{\kappa_{\P^1}}\rangle  = O(r_j^2)$  for $k=1,2$. This itself reduces
to counting (with multiplicity) the number of ``good components''  of $\Delta_n$ for the projections $\varpi_{k}: \Delta_n \to \qq_j$ that is, 
the components above the squares $Q$ of $Q_j$ that are kept in the above contruction of $T_{\qq_j, k}$ (the graphs relative
to $\varpi_{k}$).

The counting argument  is identical to 
\cite[\S 7]{bls}, except that we apply the Ahlfors theory of covering surfaces to a union of disks, not just one.
For notational ease, set $\varpi = \varpi_{k}$, $r=r_j$  and $\qq = \qq_j$; $\qq$ is a subdivision of $S\simeq [0,1]^2$ by squares
of size $2^{-j}$.   
We decompose $\qq$ as a  union of four non-overlapping subdivisions $\qq^{\ell}$, $\ell = 1, 2, 3, 4$; by this we mean that for each $\ell$,
the squares $Q\in \qq^\ell$ have disjoint closures $\overline Q$.  Fix such an $\ell$
and let $q = \# \qq^\ell= 4^{j-1}$. Applying  Ahlfors' theorem  to each of the disks constituting $\Delta_n$ and 
summing over these disks, we deduce 
  that the number of good components $N(\qq^{\ell})$ satisfies (\footnote{The term $(q-4)$ instead of $(q-2)$ in 
  \cite{bls} is due to the fact that we are projecting on $\P^1$ and not on $\C$.})
\begin{equation}\label{eq:good_components}
N(\qq^{\ell}) \geq (q-4) \area_{\pu} (\Delta_n) - h \length_{\pu}(\fr\Delta_n),
\end{equation}
where $\area_{\pu}$ (resp. $\length_{\pu}$) is the area of the projection $\varpi(\Delta_n)$ (resp. length of $\varpi(\fr\Delta_n)$), counted with multiplicity, and $h$ is a constant that depends only on the geometry of $\qq^\ell$. Dividing by $\area_{\pu} (\Delta_n)$,
using $\length_{\pu}(\fr\Delta_n) = o(\area_{\pu} (\Delta_n))$, which is guaranteed by Ahlfors' construction, 
and letting $n$ go to $+\infty$,  we obtain
\begin{equation}\langle  T_{\qq}\rest{\qq^{\ell}}, \varpi^*{\kappa_{\P^1}}\rangle \geq (q-4) r^2 = 
\area_{\pu}\lrpar{ \bigcup\nolimits_{S\in \qq^\ell} S } -4 r^2. \end{equation}
Finally, summing from $\ell=1$ to $4$, we see that, relative to $\varpi^*\kappa_{\pu}$, the  mass lost   by discarding  the bad components  of size $r$ in $T$ is of order $O(r^2)$: this is precisely   the required estimate.

Let us now justify the geometric intersection statement, following step by step the proof of \cite[Thm. 4.2]{isect}:  
let $S$ be a current with continuous normalized potential on $X$; we have to show that $S\wedge T_r$ 
increases to $S\wedge T$ as $r$ decreases to 0. Again the result is local so we work near $x$, use 
the projections $\varpi_1$ and $\varpi_2$, and keep notation as above. Given squares $Q, Q'\in \qq$ and a real number 
$\lambda<1$, we denote by $\lambda Q$ the homothetic of $Q$ of factor $\lambda$ 
with respect to its center,  and by $C(Q, Q')$ the cube $\varpi_{1}\inv(Q)\cap \varpi_{2}\inv(Q')$.
Fix $\e>0$. We want to show that for   $r\leq r(\e)$, the mass of $(T-T_r)\wedge S$ is smaller than $\e$. 
The first observation  is that there exists $\lambda(\e)\in (0,1)$, independent of $r$, such that 
translating $\qq$ if necessary,    the mass of 
$T\wedge S$ concentrated in ${\bigcup_{Q, Q'}  C(Q, Q')\setminus C(\lambda Q, \lambda Q')}$ is smaller than 
$ \e/2$
 (see \cite[Lem. 4.5]{isect}). Fix such a $\lambda$. 
 It only remains to  estimate the mass of $(T-T_r)\wedge S$ in 
 $\bigcup_{Q, Q'}    C(\lambda Q, \lambda Q')$. 
In such a cube $C(\lambda Q, \lambda Q')$ the argument presented in~\cite[pp. 123-124]{isect}, based on an integration by parts,  gives the estimate
 \begin{equation} \int_{ C(\lambda Q, \lambda Q')} (T-T_r)\wedge S \leq 
 C(\lambda)  \moco(u_S, r) \frac{1}{r^2} \M\lrpar{(T-T_r)\rest{ C(  Q,   Q')}}, 
 \end{equation}
 where $\moco(u_S, r)$ is the modulus of continuity of the potential $u_S$ of $S$. To conclude, we  sum 
 over all squares $Q, Q'$ and use the estimate $M(T-T_r) = O(r^2)$ to get that 
 \begin{equation}
 \M\lrpar{(T-T_r)\rest{ \bigcup_{Q, Q'}    C(\lambda Q, \lambda Q')}}\leq C\omega(u_S, r).
 \end{equation}
This is smaller than $\e/2$ if $r\leq r(\e)$. 
\end{proof}

\begin{vlongue}
\section{Proof of Theorem \ref{thm:classification_proj_invariant}} \label{app:barrientos_malicet}

Let us consider a random dynamical system $(X, \nu)$ and $\mu$ an ergodic stationary measure, 
as in Theorem \ref{thm:classification_proj_invariant}. We keep the notation from \S \ref{subs:ledrappier_invariance_principle}.  

We say that a sequence of real numbers $(u_n)_{n\geq 0}$  {\bf{almost converges towards $+\infty$}} if for every 
$K\in \R$, 
 the set $L_K=\set{ n\in \N\; ; \; u_n\leq K }$ has an asymptotic lower density
\begin{equation}
\underline\dens(L_K):=\liminf_{n\to +\infty}\left( \frac{\sharp (L_K\cap [0,n])}{n+1}\right)
\end{equation}
which is  equal to $0$: $\underline \dens(L_K)=0$ for all $K$.

\begin{lem}\label{lem:alternative_bdd}
The set   of points $\cx=(\omega,x)$ in $\X_+$ such that $\llbracket D_xf^n_\omega\rrbracket$ 
almost converges towards $+\infty$ on $\pp (T_xM)$ is $F_+$-invariant. In particular, by ergodicity, 
\begin{enumerate}[\em (a)]
\item either $\llbracket D_xf^n_\omega\rrbracket$  almost converges towards $+\infty$ for $(\nu^\N\times \mu)$-almost every $(\omega, x)$;
\item or, for $(\nu^\N\times \mu)$-almost every $(\omega, x)$, there is a sequence $(n_i)$ with positive lower density along which 
$\llbracket  D_xf^{n_i}_\omega \rrbracket$ is  bounded.
\end{enumerate}
\end{lem}

The proof is straightforward. 
We are now ready for the proof of Theorem \ref{thm:classification_proj_invariant}. 
Let us first  emphasize one delicate issue: in Conclusion~(1) of the theorem,  
it is important that the directions $E$ (resp. $E_1$ and $E_2$) only depend on $x\in X$ (and not on 
$\cx = (x,\omega)\in \cX_+$). Likewise in Conclusion~(2), the trivialization $P_x$ should depend  only on $x$. 
This justifies the inclusion of a detailed proof of Theorem \ref{thm:classification_proj_invariant}, since in
 the slightly different setting of       
\cite{barrientos-malicet}, the authors did not have to check this point carefully.

We fix a measurable trivialization $P\colon TX\to X\times \C^2$, given 
by linear isometries  $P_x\colon T_xX\to \C^2$, where $T_xX$ is endowed with the hermitian form $(\kappa_0)_x$,  and 
$\C^2$ with its standard hermitian form. This trivialization conjugates the action of $DF_+$ to that of a 
cocycle $A\colon \cX_+ \times \C^2\to \cX_+ \times \C^2$ over $F_+$.  
We denote by $A_\cx\colon \set{\cx}\times \C^2\to \set{F_+(\cx)}\times \C^2$ the induced linear map; observe that 
$A_\cx = A_{(\omega, x)}$ depends only on $x$ and on the first coordinate $f_\omega^1 = f_0$ of $\omega$. 
Using  $P$ we transport the measure $\hat\mu$ to a measure,  
still denoted by $\hat\mu$, on the product space $X\times \pp^1(\C)$. By our invariance assumption, its
disintegrations $\hat\mu_\cx=\hat\mu_x$ satisfy $(\P A_\cx)_*\hat\mu_\cx=\hat\mu_{F_+(\cx)}=\hat\mu_{f^1_\omega(x)}$.

\smallskip

\textbf{The bounded case. --} In this paragraph we show
  that in the essentially bounded case (b) of Lemma \ref{lem:alternative_bdd}, 
Conclusion~(2) of Theorem \ref{thm:classification_proj_invariant} holds.  We streamline the argument following the proof of   
\cite[Prop. 4.7]{barrientos-malicet} which deals with  the more general case of $\GL(d, \R)$-cocycles,
 and is  itself   a variation on previously known ideas  (see e.g.~\cite{Arnold-Nguyen-Oseledets, Zimmer:Israel1980}). 

Set $G =\PGL(2, \C)$, and  define the $G$-extension   ${\widetilde{F}}_+$ of
$F_+$  on $\X_+\times G$ by 
\begin{equation}
{\widetilde{F}}_+(\cx, g)=(F_+(\cx), \P(A_\cx)  g)=((\sigma(\omega), f^1_\omega(x)), \P(A_{(\omega, x)})  g)
\end{equation}
for every $\cx=(\omega,x)$ in $\X_+$ and $g$ in $G$; thus ${\widetilde{F}}_+$  is given by $F_+$ on $\X_+$ 
and is the multiplication by  $\P(A_\cx)$ on $G$.
Since $\P(A_{(\omega, x)})$ depends on $\omega $ only through its first coordinate, 
${\widetilde{F}}_+$ can be interpreted as the skew product 
map associated to a random dynamical system on $X\times G$. 
Denote by $\mathcal P$ the convolution operator associated to this random dynamical system; 
thus $\mathcal P$ 
acts on probability measures on $X\times G$. Let $\mathrm{Prob}_\mu(X\times G)$ 
the set of probability measures 
on $X\times G$ projecting to $\mu$ under the natural map $X\times G \to X$. Since $\mu$ is stationary, 
$\mathcal P$ maps $\mathrm{Prob}_\mu(X\times G)$ to itself.

Recall that by assumption 
there is a set $E$ of positive measure in $\X_+$, 
a compact subset $K_G$ of $G$,  and a positive real number $\varepsilon_0$ such that 
\begin{equation}
\underline{\dens}\set{ n\; ; \; \P(A^{(n)}_\cx)\in K_G }\geq \varepsilon_0
\end{equation}
for all $\cx$ in $E$.  

\begin{lem} There exists  an ergodic, stationary, Borel probability measure $\widetilde{\mu}_G$ on $X\times G$ 
with marginal measure $\mu$ on $X$.
\end{lem}


\begin{proof} (See \cite[Prop. 4.13]{barrientos-malicet} for details).
Let $\widetilde{\mu}_G$ be any cluster value of the sequence of probability measures  
\begin{equation}
\unsur{N}\sum_{i=0}^{N-1} \mathcal P^i (\mu\times \delta_{1_G}).
\end{equation} 
By the boundedness assumption, $\widetilde{\mu}_G$ has mass $M\geq \e_0$ 
and is stationary (i.e. $\mathcal P$-invariant). Standard arguments show that its projection on the 
first factor is equal to $M\mu$.  We renormalize it 
to get a probability measure and using  the ergodic decomposition and  the ergodicity of $\mu$,  we may replace 
it by an ergodic stationary measure in $\mathrm{Prob}_\mu(X\times G)$.
\end{proof}

Denote by $\widetilde \m_G = \nu^\N\times \widetilde \mu_G$ the  ${\widetilde{F}}_+$-invariant measure associated to 
$\widetilde{\mu}_G$. 
The action of ${\widetilde{F}}_+$ on     $\cX_+\times G$ (resp. of the induced random dynamical system on 
$X\times G$) commutes to the action of $G$ 
by right multiplication, i.e. to the diffeomorphisms $R_h$, $h\in G$, defined by
\begin{equation}
R_h(\cx, g)= (\cx, g  h).
\end{equation}
Slightly abusing notation we also denote by $R_h$ the analogous map on $X\times G$. 
The next lemma combines classical arguments due to Furstenberg and Zimmer. 

\begin{lem}
Let $\widetilde\mu_G$ be a Borel stationary measure on $X\times G$ with marginal $\mu$ on $X$.  
Set 
\[
H =\set{ h \in G \; ; \; (R_h)_*{\widetilde{\mu}}_G  = \widetilde \mu_G}
= \set{ h \in G \; ; \; (R_h)_*{\widetilde{\m}}_G  =\widetilde \m_G }.
\]
Then $H$ is a compact subgroup of $G$ and there is a measurable function $Q\colon X \to G$ 
such that the cocycle $B_\cx = Q_{f^1_\omega(x)}^{-1}\times \P(A_\cx) \times Q_x$ takes its values in $H$
for $(\nu^\N\times \mu)$-almost every $\cx$. 
\end{lem}

\begin{proof}  
Clearly, $H$ is a closed subgroup of $G$. 
If $H$ were not
bounded then, given any compact subset $C$ of $G$, we could find a sequence $(h_n)$ of elements of $H$ 
such that the subsets $R_{h_n}(C)$ are pairwise disjoint. Choosing  $C$ 
such that $X\times C$ has positive $\widetilde\mu_G$-measure, 
we would get a contradiction with  the finiteness of $\widetilde\mu_G$. So $H$ is a compact subgroup of $G$. 

We say that a point $(x,g)$ in $X\times G$ is generic  
if for $\nu^N$-almost every $\omega$, 
\begin{equation}\label{eq:generic_G}
\frac{1}{N}\sum_{n=0}^{N-1}  \varphi\lrpar{{\widetilde{F}}^n_+(\omega, x, g) } \underset{N\to\infty}\longrightarrow 
\int_{\cX_+\times G} \varphi \; d  \widetilde\m _G 
\end{equation}
for every compactly supported continuous function on $\cX_+\times G$.
The Birkhoff ergodic theorem 
provides a Borel set $\mathcal E$ of full $\widetilde \mu_G$-measure made of generic points. 
Now  if   $(x, g_1)$ and $(x, g_2)$ belong to $\mathcal E$,  
writing $g_2=g_1  h=R_h(g_1)$ for $h=g_1^{-1}  g_2$, we get that   $h$ is an 
element of $H$. 

Given $g\in G$, define ${\mathcal{E}}_x\subset G$ to be the set of elements $g\in G$ such that $(x,g)$ is generic. Then there 
exists a   measurable section
 $X\ni x\mapsto Q_x  \in  G$ such that $Q_x\in {\mathcal{E}}_x$ for almost all~$x$. By definition of $\mathcal{E}_x$, 
$(\omega, x, Q_x)$ satisfies \eqref{eq:generic_G} for $\nu^N$-almost every $\omega$.  
Then for  $\nu$-almost every $f_0 = f_\omega^1$, 
  by $\widetilde F_+$-invariance of the set of Birkhoff generic points we infer that 
  $(f_\omega^1(x),  \P(A_\cx) Q_x)$ belongs to $\mathcal E$. Since $(f_\omega^1(x), Q_{f_\omega^1(x)})$ belongs to $\mathcal E$ 
  as well, it follows that $Q_{f_\omega^1(x)}^{-1}  \P(A_\cx)  Q_x$ is in $H$. 
We conclude that the cocycle 
$B_\cx = Q_{f^1_\omega(x)}^{-1}\times \P(A_\cx) \times Q_x$ takes its values in $H$ for almost all $\cx$, as claimed. 
\end{proof}

Note that the map $x\mapsto Q_x$ lifts to a measurable map $x\mapsto Q'_x\in \GL_2(\C)$.  
Conjugating $H$ to a subgroup of $\mathsf{PU}_2$ by some element $g_0\in G$, we can now readily conclude from the two previous lemmas that when 
$\llbracket  D_xf^{n}_\omega \rrbracket$ is essentially bounded, Conclusion~(2) of Theorem~\ref{thm:classification_proj_invariant} holds
(the $P_x$ are obtained by composing the $Q'_x$ with a lift of $g_0$ to $\GL_2(\C)$).

\smallskip

\textbf{The unbounded case. --}  Now, we suppose that $\llbracket D_xf^n_\omega\rrbracket$ is essentially unbounded (alternative (a) of Lemma \ref{lem:alternative_bdd}), and adapt the results of \cite[\S 4.1]{barrientos-malicet} to the complex setting to arrive at one of the Conclusions~(1.a) or (1.b)
of Theorem~\ref{thm:classification_proj_invariant}.
 The main step of the proof is the following lemma.

\begin{lem}\label{lem:1/2}
Let $A$ be a measurable $\GL(2, \C)$ cocycle over $(\cX_+, F_+, \nu^\N\times \mu)$
 admitting a projectively  invariant family of probability 
measures  $\lrpar{\hat \mu_x}_{x\in X}$ such that almost surely $\llbracket A_{\cx}^{(n)}\rrbracket$ almost converges to 
infinity. Then for almost every $x$, $\hat\mu_x$ possesses an atom of mass at least $1/2$; more precisely: 
\begin{itemize}
\item either $\hat\mu_x$ has a unique atom $[w(x)]$ of mass $\geq 1/2$, that depends measurably on $x\in X$;
\item or $\hat\mu_x$ has a unique pair of atoms   
 of mass $1/2$, and this (unordered) pair depends measurably on $x\in X$. 
\end{itemize}
\end{lem}

For the moment, we take  this result for granted and proceed with the proof. 
By ergodicity, the number of atoms of $\hat\mu_x$  and the list of their masses  are constant on a set
of full measure.  A first possibility is that $\hat\mu_x$ is almost surely the single point mass $\delta_{[w(x)]}$; this corresponds to (1.a).
A second possibility is that $\hat\mu_x$ is the sum of two point masses of mass $1/2$; this corresponds to~(1.b). 
In the 
remaining cases, 
 there is exactly one atom of mass $1/2\leq \alpha <1$ at a point $[w(x)]$. 
Changing the trivialization $P_x$, we can suppose that $[w(x)]=[w]=[1:0]$. Then we
write $\hat\mu_x=\alpha \delta_{[1:0]}+\hat\mu_x'$, and apply Lemma \ref{lem:1/2} to the family of measures  
$\hat\mu_\cx'$ (after normalization to get a probability measure). 
 We deduce that almost surely $\hat\mu_x'$ admits  an atom  of mass $\geq (1-\alpha)/2$. 
Two cases may occur:  
\begin{itemize}
\item $\hat\mu_x'$ has a unique atom of mass $\beta \geq (1-\alpha)/2$,
\item $\hat\mu_x'$ has two atoms of mass $(1-\alpha)/2$.
\end{itemize}
The second one is  impossible, because changing the trivialization, we would have
$\hat\mu_x=\alpha \delta_{[1:0]}+\frac{1-\alpha}{2} (\delta_{[-1:1]}+\delta_{[1:1]})$, and the invariance 
of the finite set $\set{ [1:0], [-1:1], [1:1] }$ would imply that the cocycle $\P (A_\cx)$ stays in a finite subgroup of $\PGL_2(\C)$, contradicting 
the unboundedness assumption. 

If  $\hat\mu_\cx'$ has a unique atom of mass $\beta \geq (1-\alpha)/2$, we change $P_x$ to put it at $[0:1]$
(the trivialization $P_x$ is not an isometry anymore).
We repeat the argument with  $\hat\mu_x=\alpha \delta_{[1:0]}+\beta \delta_{[0:1]}+\hat\mu_x''$. 
If $\beta = 1-\alpha$, i.e. $\hat\mu_x''=0$, then we are done. Otherwise 
  $\hat\mu_x''$ has one or two atoms of mass $\gamma \geq (1-\alpha-\beta)/2$, and we change
$P_x$ to assume that one of them is $[1:1]$ and the second  one --provided it exists-- 
is $[\tau(x):1]$; here, $x\mapsto \tau(x)$ is a
complex valued measurable function. 
Endow the projective line $\pp^1(\C)$ with the coordinate $[z:1]$; then   $\P(A_\cx)$ is of the 
form $z\mapsto a(\cx) z$. Since $\P (A_\cx) \lrpar{\set{1, \tau(x)}} = \lrpar{\set{1, \tau(F_+(\cx))}}$, we infer that:
\begin{itemize}
\item either $a(\cx)1 = 1$ and $\P(A_\cx)$ is the identity;
\item  or $a(\cx)1 = \tau(\pi_X(F_+(\cx))) $ and  $a(\cx)\tau(x) =1$ in which case  $\tau(\pi_X(F_+(\cx))) = \tau(x)\inv$. 
\end{itemize}
Thus we see that along the  orbit of $\cx$, $a(F^n_+(\cx))$ takes at most two values $\tau(\pi_X(F^n_+(\cx)))^{\pm 1}$, and $\llbracket A^{(n)}_\cx\rrbracket$ is bounded, which is contradictory. This concludes the proof. \qed

\begin{proof}[Proof of Lemma \ref{lem:1/2}]
Let $r$ and $\e$ be small positive real numbers.
Let $\mathrm{Prob}_{r, \e}(\P^1(\C))$ be the set of probability measures $m$
on $\P^{1}(\C)$ such that $\sup_{x\in \pu} m(B(x, r))\leq 1/2-\e$, 
where the ball is with respect to some fixed Fubini-Study metric. 
This is a compact subset of the space of probability measures on $\pu$. 
The set 
\begin{equation}
G_{r, \e} = \set{\gamma\in \PGL(2, \C),\ \exists m_1, m_2 \in \mathrm{Prob}_{r, \e}(\P^1(\C)), \ \gamma_*m_1 =m_2}
\end{equation}
is a bounded subset of $\PGL(2, \C)$. Indeed otherwise there would be an unbounded sequence $\gamma_n$ together 
with sequences $(m_{1, n})$ and $(m_{2, n})$ in $\mathrm{Prob}_{r, \e}(\P^1(\C))$ such that 
$(\gamma_n)_*m_{1, n} =  m_{2, n}$. Denote by $\gamma_n=k_n a_n k'_n$  the KAK decomposition of $\gamma_n$ in $\PGL(2, \C)$,
with $k_n$ and $k'_n$ two isometries for the Fubini-Study metric; since $\gamma_n$ is unbounded, we can extract a
subsequence such that the measures $(k'_n)_*m_{1, n}$ and $(k_n^{-1})_*m_{2, n}$ converge in 
$\mathrm{Prob}_{r, \e}(\P^1(\C))$ to two measures $m_1$ and $m_2$,
while the diagonal transformations $a_n$ converge locally uniformly 
on $\P^1(\C)\setminus\set{[0:1]}$ to the  constant map $\gamma: \P^1(\C) \setminus\set{[0:1]} \mapsto \set{[1:0]}$. Then
\begin{equation}
\gamma_*\lrpar{m_{1\vert \P^1(\C)\setminus\set{[0:1]}}} =  m_1(\P^1(\C)\setminus\set{[0:1]}) \delta_a \leq m_2;
\end{equation}
since $m_1$  belongs to  $\mathrm{Prob}_{r, \e}(\P^1(\C))$, $m_1(\P^1(\C)\setminus\set{[0:1]}) \geq 1/2+\e$, hence 
$m_2\geq (1/2+\e) \delta_a$, in contradiction with $m_2\in \mathrm{Prob}_{r, \e}(\P^1(\C))$. This proves that $G_{r,\e}$ is bounded.

To prove the lemma, let us consider the ergodic dynamical system $\P DF_+$, and the family of conditional 
probability measures $\hat\mu_\cx$ for the projection $(\omega, x,v)\mapsto \cx=(\omega,x)$.
If there exist  $r, \e>0$ such that $\hat \mu_\cx $ belongs to $\mathrm{Prob}_{r, \e}(\P^1(\C))$ 
for $\cx$ in some positive measure subset $B$ then, by ergodicity, for almost every $\cx\in \X_+$ there exists a set of integers $L(\cx)$ of 
 positive density such that for $n\in L(\cx)$,  $F_+^n(\cx)$ belongs to $B$, hence $A^{(n)}_{\cx} $ belongs to $G_{r, \e}$ (\footnote{We are slightly abusing here when the Fubini-Study metric depends on $x$, for instance when $P_x$ is not an isometry; however restricting to subset of large positive measure the metric $(P_x)_*(\kappa_0)_x$ is uniformly comparable to a fixed Fubini-Study metric.}). 
From the above claim we deduce that   $\llbracket A^{(n)}_\cx \rrbracket$ is uniformly bounded for $n\in L(\cx)$, a contradiction. 
Therefore for every $r, \e>0$, the measure of $\set{\cx, \ \hat \mu_\cx \in \mathrm{Prob}_{r, \e}(\P^1(\C))}$ is equal to $0$; it  follows that 
for almost every $\cx$, $\hat\mu_\cx$ possesses an atom of mass at least $1/2$.

If there is a unique atom of mass $\geq 1/2$, this atom determines a measurable map $\cx\mapsto [w(\cx)]\in \P T_xX$; since $\hat \mu_\cx$ does not depend on $\omega$, $[w(\cx)]$ depends only on $x$, not on $\omega$. If there are generically two atoms of mass $\geq 1/2$, then 
both of them has mass $1/2$, and the pair of points determined by these atoms depends only on $x$. 
 \end{proof}
 
\newpage
\end{vlongue}

 \bibliographystyle{plain}
\bibliography{biblio-serge}

\end{document}